\documentclass[article,11pt]{article}

\usepackage[english]{babel}
\usepackage{inputenc}
\usepackage[T1]{fontenc}
\usepackage{geometry}
\usepackage{xcolor}
\usepackage{enumerate}
\usepackage{amsmath,amsthm,amssymb,amsfonts,bbm,bm,nicefrac,mathtools}
\usepackage{notations}
\usepackage{hyperref}

\let\svthefootnote\thefootnote
\newcommand\freefootnote[1]{%
  \let\thefootnote\relax%
  \footnotetext{#1}%
  \let\thefootnote\svthefootnote%
}

\usepackage{graphicx}
\usepackage{subcaption}
\usepackage{authblk}
\usepackage[font=footnotesize]{caption}
\definecolor{ps}{RGB}{0,0,200}

\newcommand{\longcomment}[1]{}
\usepackage[numbers,sort&compress]{natbib}
\begin{document}

\title{
%A Precise Characterization of Finite-Dimensional Posterior Marginals for GLMs under Proportional Asymptotics or 
Characterizing Finite-Dimensional Posterior Marginals in
High-Dimensional GLMs via Leave-One-Out}
%\title{Characterizing Finite-Dimensional Posterior Marginals in High Dimensions}
%\\ or ``under Proportional Asymptotics via Leave-one-Out}
\author[1]{Manuel S\'aenz\textsuperscript{\textdagger}}
\author[2]{Pragya Sur\textsuperscript{\textdagger}}
\affil[1]{Departamento de Matem\'atica y Ciencias, Universidad de San Andres}
\affil[2]{Department of Statistics, Harvard University}

\maketitle
%Despite extensive advances in high-dimensional Bayesian inference, majority of prior work is restricted to the ultra-high-dimensional regime that requires sparsity assumptions for suitable inference. %consider the setting where 
%We characterize the behavior of finite-dimensional marginals of the posterior distribution in high-dimensional generalized linear models (GLMs). Specifically, we consider the proportional asymptotics regime, where the number of samples and features diverge at a comparable rate. 
%We characterize the behavior of finite-dimensional marginals of the posterior distribution in high-dimensional generalized linear models (GLMs). Specifically, we consider the proportional asymptotics regime, where the number of samples and features diverge at a comparable rate.
\begin{abstract}
We investigate Bayes posterior distributions in high-dimensional generalized linear models (GLMs) under the proportional asymptotics regime, where the number of features and samples diverge at a comparable rate. Specifically, we characterize the limiting behavior of finite-dimensional marginals of the posterior.  
%While the proportional asymptotics regime has been extensively studied in frequentist statistics, it has only recently received attention in the Bayesian context.
We establish that the posterior does not contract in this setting. Yet, the finite-dimensional posterior marginals converge to Gaussian tilts of the prior, where the mean of the Gaussian depends on the true signal coordinates of interest. Notably, the effect of the prior survives even in the limit of large samples and dimensions. %Notably, the Gaussian tilt depends on the true signal coordinates of interest. We show that the posterior does not contract in this setting, yet its finite-dimensional marginals retain information about the signal coordinates. 
We further characterize the behavior of the posterior mean and demonstrate that the posterior mean can strictly outperform the maximum likelihood estimate in mean-squared error in natural examples.  Importantly, our results hold regardless of the sparsity level of the underlying signal.
%As a consequence of our results, we can study the frequentist performance of the posterior mean. In particular, we demonstrate that the posterior mean can strictly outperform the maximum likelihood estimate in mean-squared error in natural examples. 
On the technical front, we introduce leave-one-out strategies for studying these marginals that may be of independent interest for analyzing low-dimensional functionals of high-dimensional signals in other Bayesian inference problems.

.\freefootnote{\textsuperscript{\textdagger}Contact: saenzm@udesa.edu.ar, pragya@fas.harvard.edu.}
\end{abstract}

    %\tableofcontents

\section{Introduction}\label{sec:intro}
High-dimensional Bayesian inference has garnered significant recent attention. Early work considered high-dimensional Bayesian regression where the dimension diverges at a rate  slower than the sample size, and explored when low-dimensional results continue to hold in such settings \cite{ghosal1997normal,ghosal1999asymptotic,ghosal2000asymptotic}. Following this, extensive prior work studied the ultra-high-dimensional regime where the number of features $p$ is much larger than the number of samples $n$, and the underlying signal is sufficiently sparse or approximately sparse \cite{jiang07,castillo2015bayesian,narisetty_he_14,song2023nearly,atchade17,ningetal20,gaoetal20,wei2020contraction}. This line of work identified conditions under which the posterior contracts, is asymptotically normal or mixture of normals, and Bernstein-von-Mises theorems (BvMs) hold. In this paper, we deviate from both these settings and consider the high-dimensional regime where the number of features and samples diverge at a rate comparable to each other, also known as the proportional asymptotics regime. 

This regime has become popular in modern high-dimensional statistics due to the following attractive features. 
First, asymptotic theory established under this regime exhibits remarkable finite sample behavior, as demonstrated extensively in  prior work  \cite{donoho2009message,el2018impact,el2013robust,sur2019likelihood,sur2019modern,candes2020phase,zhao2022asymptotic}. Second, the regime allows one to move away from traditional strong sparsity assumptions on the underlying signal, and develop rigorous procedures that remain valid under sparse to dense signals, leading to methods with impressive practical performance \cite{sur2019modern,celentano2020lasso,bellec2019biasing,zhong2022empirical,li2023spectrum,song2024hede,chen2024method}. Finally, this regime allows for the precise characterization of new high-dimensional phenomena that may be difficult to track or explain using other mathematical frameworks \cite{sur2019modern,mignacco2020role,hastie2022surprises,song2024generalization}. 
%In turn, this leads to the development of novel methods with impressive practical performance \cite{celentano2020lasso,bellec2019biasing,li2023spectrum,zhong2022empirical}. 
In light of these attractive properties, several prior work has studied this regime in a frequentist setting. %(c.f.~\cite{montanari2022short,feng2022unifying} and references cited therein). 

%[MAYBE a line logistic likelihood stuff here would be good.] 

In the Bayesian context, a long literature in statistical physics and information theory has studied this regime in the so-called matched setting, where the statistician knows the true prior underlying the signal \cite{zdeborova2016statistical,barbier2019optimal,montanari2024friendly,feng2022unifying}. However, from a statistical perspective, the mismatched setting, where the statistician does not know the true underlying prior, is certainly more interesting. Only recently, \cite{barbier2021performance} studied a high-dimensional linear model in a mismatched setting, establishing properties of the log-normalizing constant and mean-squared error of Bayesian estimators in presence of a Gaussian prior. In this paper, we provide for the first time precise characterization of finite-dimensional marginals of the posterior in a high-dimensional Bayesian generalized linear model under the proportional asymptotics regime in the more realistic mismatched setting. Understanding marginals requires a more nuanced analysis than that for understanding the log-normalizing constant or mean-squared errors. For the latter, an ``average case'' analysis of the model suffices. In our case, a more nuanced leave-one-out technique is necessary that we introduce in our work. 
We summarize our major contributions below. 
%\textcolor{red}{To take pass on below list, especially in light of \cite{lee2025clt}.}
%\ps{ERvisit the below later.}

\begin{enumerate}
\item We provide the first characterization of finite-dimensional marginals of the posterior in the aforementioned setting (Theorem \ref{thm:marginals}). Specifically, we uncover that, unlike in the low-dimensional or ultra-high-dimensional regression literature in statistics, the posterior marginals are Gaussian tilts of the prior, where we precisely pin down the tilting function. The mean of the Gaussian tilt is a constant times the true signal and additional random noise variables. 
\item  Our result also shows that the effect of the prior does not wash away, even in the limit of large samples and dimensions. 
 This is a critical observation---since this can pave the path for designing priors that help in high-dimensional learning, and constructing  credible intervals with low width. 
 \item Furthermore in this regime, posterior contraction does not occur (Theorem \ref{thm:contraction}), yet our posterior marginal characterization can provide first steps toward conducting inference on pre-specified finite number of coordinates of the signal.
\item As a corollary, our result provides a characterization of the posterior mean in the aforementioned high-dimensional regime (Corollary \ref{cor:marginals_bayes}). Once again, the effect of the prior survives even asymptotically. %\textcolor{red}{Should I elaborate on this? To take a call at the end.}\ms{Maybe it is a good idea. It emphasizes that Bayesian and frequentist methods will differ in this regime. In particular, when data is scarse, a good prior can help A LOT. Do you think this is a reasonable storytelling?}
\item Crucially, based on our aforementioned characterization, we construct an explicit setting where the posterior mean strictly outperforms the MLE (c.f., the left panel of Figure \ref{fig:mse_graphs2}), even in a setting where the MLE is well-defined.\footnote{Note, although we consider a high-dimensional regime where the number of features $p$ and samples $n$ both diverge, we allow them to grow proportionally so that $p/n \rightarrow \kappa$. For $\kappa \in (0,1)$ the setting is still high-dimensional in that high-dimensional phenomena would emerge but the MLE can still be well-defined for some subset of $\kappa$ values in this range. The exact subset would depend on the GLM considered.} This regime uncovers a striking departure from classical behavior: the posterior mean no longer converges to the MLE, even in the limit of large samples. Indeed, this divergence is favorable for Bayesian methods, since the MLE is known to perform poorly in high dimensions \cite{sur2019modern}. But it also consolidates the fact that one should not pursue classical BvMs in this regime
%should look different from classical BvMs
since the MLE is strictly sub-optimal. To overcome issues with the MLE, a debiased MLE was introduced in \cite{sur2019modern}. We illustrate in the right panel of Figure \ref{fig:mse_graphs2} that a function of the posterior mean can strictly outperform this debiased MLE.
%\textcolor{red}{Write Figure 1b as a takeaway. Very possible to beat the MLE even with very simple priors. Another along the line of famous Stein shrinkage phenomenon plus literature.}
\item In terms of technical contributions, this paper provides the first leave-one-out analysis for a high-dimensional Bayesian model with a planted signal and shows how such leave-one-out techniques can be used to investigate posterior marginals. 
In the frequentist literature, leave-one-out analyses for understanding fluctuations and other properties of estimators have been introduced for high-dimensional linear and logistic regression \cite{el2013robust,el2018impact,sur2019modern}, treatment effect estimation \cite{jiang2022new} (see \cite{chen2021spectral} for other applications). To the best of our knowledge, in the Bayesian literature, leave-one-out ideas have been employed only for global null models in the statistical physics and probability literature \cite{talagrand2010mean,barbier2022marginals}. These settings are distinct from settings of interest to the statistician where there is an unknown underlying signal of interest in the data generating process. Presence of such a planted signal, especially in conjunction with the non-linear link functions in GLMs, raises significant technical difficulties that we overcome in this manuscript. 
\end{enumerate}

In recent work, a regime where the posterior does not contract was studied in the context of regression problems through the lens of  Variational Bayes \cite{mukherjee2022variational,mukherjee2023mean,lee2025clt}, but their framework does not cover ours. In fact, in our setting, their feature dimension can grow at most negligible to the sample size---thus they work in  a setting complementary to ours. The proportional asymptotics regime has been studied for empirical Bayes problems in recent literature, including interesting algorithmic advances \cite{zhong2022empirical,fan2023gradient,fan2025dynamical}. However, these works do not study the behavior of finite-dimensional posterior marginals, and it remains unclear how their tools can be utilized to address this question (see also \cite{sterzinger2023diaconis}). Our results also have interesting connections to the literature on Gaussian sequence models \cite{johnstone2001distribution} and non-parametric Bayes \cite{regazzini2003distributional,james2009posterior,rousseau2011asymptotic,de2013asymptotic,ghosal2017fundamentals}. We discuss these connections to prior work in detail in Section \ref{sec:results}. 

Finally, our leave-one-out analyses provide an analogue of classical local asymptotic analysis in high dimensions. Our key proof idea relies on tilting the full posterior with respect to leave-one-out posteriors that we introduce in Section \ref{sec:model}. Subsequently, we establish that the tilt is asymptotically Gaussian,  as in traditional Le Cam  theory \cite{le1956asymptotic,van2002statistical}. Our proof heuristics section (Section \ref{sec:pfoutline}) elaborates on this connection. The analogy with local asymptotic analyses can be particularly observed in \eqref{eq:fullnlooo} and \eqref{eq:loo2}, where we express the full posterior in terms of tilts of the leave-a-variable-out and leave-an-observation-out posteriors respectively. We believe this leave-one-out framework can be used to study many other problems in Bayesian inference where low-dimensional functionals of the high-dimensional signal are of interest. Given the emerging interest in this area in the high-dimensional Bayesian community \cite{castillo2024variational,lee2025clt}, our work here is meant to provide a general-purpose tool for studying such problems in Bayesian inference. 

%\textcolor{red}{Write two lines about the culture of studying coefficients or low-dimensional functionals in the recent literature; and that we are providing a general math tool for this. }

The rest of this manuscript is organized as follows: Section \ref{sec:model} introduces our formal setup and key ingredients required for our leave-one-out analyses. Section \ref{sec:results} introduces our main results and provides a short discussion on immediate applications. Section \ref{sec:pfoutline} describes our proof outline---given the technical nature of the paper, we believe this is the most accessible description of our overarching leave-one-out strategies. Finally, we conclude with a discussion of future directions in Section \ref{sec:discussion}. 
    \section{Formal Setup and Leave-one-Out Posteriors}\label{sec:model}
%\noindent
%\subsection{Formal setup}
%\subsection{{\large Formal Setup}}
We work in a high-dimensional regime where the number of samples $n$ and features $p(n)$ both diverge with $p(n)/n \rightarrow \kappa >0$. That is, we consider a sequence of problem instances 
%\begin{equation*}
 $  \{\bm{y}(n),\bm{X}(n),\bbeta_{\star}(n),\bm{e}(n) \}_{n \geq 1},$
%\end{equation*}
where $\bm{y}(n)\in \mathbb{R}^{n}$ denotes the vector of $n$ outcomes, $\bm{X}(n) \in \mathbb{R}^{n \times p}$ the design matrix, $\bbeta_{\star}(n) \in \mathbb{R}^{p}$ the unknown signal and $\bm{e}(n) \in \mathbb{R}^n$ the unobserved noise variables. Moving forward, we suppress the dependence on $n$ when it is clear from context. We assume that the samples are i.i.d.~with the $i$-th sample satisfying the following single-index model
\begin{equation}\label{eq:model}
    y_i = f(\bX_i^{\top}\bm{\beta}_{\star},e_i), 
\end{equation}
where $\bX_i$ is the $i$-th row of $\bX$, $f(\cdot)$ is some non-linear function and $e_i\sim \text{Unif}[0,1]$ is independent of $\bX_i$. As a running example, consider logistic regression where $f(\bX_i^{\top}\bbeta_\star,e_i) = \mathbb{I}_{\{\bX_i^{\top}\bbeta_\star \geq \sigma^{-1}(e_i)\}}$ and $\sigma(x) = e^x/(1+e^x)$ is the sigmoid function.

%We consider the setting where the true model link $f(\cdot)$ and the measure $\pi(\cdot)$ are, possibly, not known to the statistician. We therefore allow for a mismatch between the model and the data generating mechanism.
%\ps{Does it matter to us whether f is unknown or not, I feel } \note{It does not have to be unknown, but we include that case. We should frame it differently.}

We seek to study properties of the posterior distribution when the statistician posits an i.i.d.~prior $\mu(\cdot)$ on the signal coordinates and assumes the data is generated from a canonical generalized linear model (GLM). 
%To this end, we assume the statistician posits an i.i.d.~prior $\mu(\cdot)$ on the coordinates and calculates the posterior distribution. In this paper, we will characterize properties of finite-dimensional marginals of this posterior. For s
%she computes the posterior distribution using a factorised prior (given by the product of some distribution $\mu(\cdot)$) in a Bayesian estimation scheme. We seek to characterize properties of the posterior in the aforementioned setting. 
For simplicity, we assume that the covariates satisfy $\bX_i \stackrel{\text{i.i.d.}}{\sim} \mathcal{N}(\bm{0},\bm{I}/n)$. Though this  assumption is stylized, we will see that new high-dimensional phenomena emerge in our aforementioned setting even in this simple setting. Defining functions 
\begin{equation}\label{eq:def_tT_u}
    \tilde{T}_i(\cdot) := T \circ f(\cdot,e_i) \,\, \text{and} \,\, u_i(x,y) := x \tT_i(y) - A(x),
\end{equation}
for some function  $T(\cdot)$, the log-likelihood of our canonical GLM takes the form
\begin{equation}\label{eq:glmhamilton}
    \cL_{n,p}(\bB) := \sum_{i \in[n]} \tilde{T}_i(\bX_i^{\top}\bbeta_{\star})\bX_i^{\top}\bB - A(\bX_i^{\top}\bB) = \sum_{i \in[n]} u_i(\bX_i^{\top}\bB,\bX_i^{\top}\bbeta_{\star}),
\end{equation}
%where we defined the random functions $\tilde{T}_i(\cdot) := T \circ f(\cdot,e_i)$ and $u_i(x,y) := x \tT_i(y) - A(x)$. 
for $\bB\in\R^p$ and suitable log-normalizing constant $A(\cdot)$. 

We often analyze the joint behavior of independent samples from the posterior. For a given realization of the data, we denote $l$ independent samples from a posterior\footnote{We will use different posteriors for different parts of our calculations; when we refer to samples from a posterior, it should be clear from context which posterior is involved; where this is not the case, we will specify explicitly.} by $\bbeta^{(1)},\bbeta^{(2)}, \hdots, \bbeta^{(l)}$. When the superscript is omitted, we assume it equals $1$, i.e.~$\bbeta$ is used to denote $\bbeta^{(1)}$. 
Note the log-likelihood involves $\bbeta_\star$ and $\bm{b}$ through the inner products with $\bX_i$'s. We introduce a specific notation for these inner products for convenience.  For $m\in[l]$, $\ba_\star$ and $\ba_m$ will denote the random vectors in $\R^n$ given by 
\begin{equation}\label{eq:a_is}
\ba_\star=(\bX_1^{\top}\bbeta_\star,\hdots,\bX_n^{\top}\bbeta_\star), \qquad \ba_m = (\bX_1^{\top}\bbeta^{(m)},\hdots,\bX_n^{\top}\bbeta^{(m)})
%a_{\star,i}:=\bX_i^\top\bbeta_\star, \,\, a_{l,i}:=\bX_i^\top\bbeta^{(l)},
 \end{equation}
As before, when the $m$ subscript is omitted, we assume it equals $1$. That is, $\bm{a}$ refers to $\bm{a}_1$. Let $a_{\star,i}$ and $a_{m,i}$ denote the $i$-th coordinates of $\ba_{\star}$ and $\ba_{m}$ respectively. 
%respectively. 

%\ps{Which one is $a_1$ really? I added the last above line, then can we delete the below.}for example, $\ba$ and $a_n$ will be used to refer to $\ba_1$ and $a_{1,n}$.

%\ps{Why do we need the below? Suggest move later whenever we first need. Or put in abbreviation para}In a similar way, we will denote deterministic vectors in $\R^p$ by $\bB^{(1)}$, $\bB^{(2)}$, $\bB^{(3)}$, etc.

Toward understanding the behavior of the posterior distribution, one of our central results quantifies the expected value of a function $g(\cdot)$ applied to $l$ independent samples from the posterior, for any integer $l$. Note, any such expectation involves two sources of randomness---one coming from the  data, and the other coming from sampling from the posterior, given any realization of the data. We will distinguish these sources of randomness, using $\mathbb{E}(\cdot)$ to denote the expectation with respect to the data randomness and $\langle\cdot \rangle$ to denote the expectation with respect to the posterior. Thus, for a function $g:\R^{pl}\mapsto\R$, we define
\begin{equation}\label{eq:posterior_expectation}
    \left\langle g(\bbeta^{(1)},\hdots, \bbeta^{(l)})\right\rangle = \frac{1}{Z_n}\int  g(\bbeta^{(1)},\hdots,\bbeta^{(l)}) \exp\left\{\sum_{m\in[l]} \cL_{n,p}(\bbeta^{(m)})\right\}\prod_{m\in[l]}\prod_{j\in[p]}\mu(d\beta_j^{(m)}),
\end{equation}
where $Z_n$ denotes the appropriate normalizing constant. 
%\ms{Add integral definition of expectations wrt leave one out}

%\vspace{0.3cm}
\subsection{Leave-one-Out Posteriors}\label{sec:LOOs}
\noindent

Two critical components in our mathematical analysis are the \emph{leave-an-observation-out} and \emph{leave-a-variable-out} posteriors. The former is the posterior distribution obtained using $n-1$ observations instead of all $n$. It corresponds to the log-likelihood
\begin{equation}\label{eq:looposterior}
    \logl_{n-1,p}(\bB) = \sum_{i \in[n-1]} u_i(\bX_i^\top\bB,\bX_i^\top\bbetas)
\end{equation}
with i.i.d.~prior $\mu$ applied to each of the $p$ coordinates. If $\bbeta^{(1)},\dots,\bbeta^{(l)}$ are $l\geq1$ independent samples from the leave-an-observation-out posterior, we will let, with some abuse of notation, $\ba_1,\dots,\ba_l\in\R^{n}$ be defined as the vectors with $i$-th coordinates $\bX_i^\top\bbeta^{(1)},\dots,\bX_i^\top\bbeta^{(l)}$, respectively. Again, dropping the $m$ superscript corresponds to taking $m=1$.
%\ps{My suggestion is to keep $k,r$ the same. Any issues or inconsistency with that later? Else I will change this.}. 
Then for $g:\R^{n(l+1)}\mapsto\R$ we denote by $\langle g(\ba_\star,\ba_1,\dots,\ba_l)\rangle_o$ the expectation of $g(\ba_\star,\ba_1,\dots,\ba_l)$ under the leave-an-observation-out posterior (given the observed data). Thus, formally
\begin{equation}\label{eq:mean_observation}
    \thermal{g(\ba_\star,\ba_1,\dots,\ba_l)}_o := \frac{\int g(\bX_1^\top\bbetas,\dots,\bX_{n}^\top\bbeta^{(l)}) \prod_{m\in[l]}\Big\{\exp\{\logl_{n-1,p}(\bbeta^{(m)})\}\prod_{j=1}^p \mu(d\beta^{(m)}_j)\Big\}}{\int \prod_{m\in[l]}\Big\{\exp\{\logl_{n-1,p}(\bbeta^{(m)})\} \prod_{j=1}^p \mu(d\beta^{(m)}_j)\Big\}}.
\end{equation}
%\textcolor{red}{Write what above would mean when we have independent copies like explained in eqn \ref{eq:posterior_expectation}}\ms{Done.}

Analogously, the leave-a-variable-out posterior is the posterior distribution obtained on dropping one of the $p$ variables while observing all $n$ samples. Fix $j_0\in[p]$. We will use the following convention. Given a vector $\bB\in\R^p$, we will denote by $\tbB\in\R^{p-1}$ the vector obtained by dropping the $j_0$-th coordinate. Thus, $\tilde{\bX}_i,\tilde{\bbeta_\star}$ denote the vectors $\bX_i,\bbeta_{\star}$ respectively, with the $j_0$-th component removed. Define the leave-a-variable-out log-likelihood as
\begin{equation}\label{eq:lovposterior}
    \logl_{n,p-1}(\tbB) = \sum_{i \in[n]} u_i(\tbX_i^\top\tbB,\tbX_i^\top\tbbetas),
\end{equation}
where $\tbB \in \R^{p-1}$. The leave-a-variable-out posterior is a measure on $\R^p$ corresponding to the log-likelihood $\logl_{n,p-1}(\cdot)$ and prior $\mu^{\otimes p}$. Let $\bbeta^{(1)},\dots,\bbeta^{(l)}$ be $l\geq1$ samples from the leave-a-variable-out posterior. Then, for every integrable $f:\R^{p(l+1)}\mapsto\R$, we will define expectations under the leave-a-variable-out posterior (given the observed data) to be
\begin{equation}\label{eq:mean_variable}
    \thermal{f(\bbetas,\bbeta^{(1)},\dots,\bbeta^{(l)})}_v := \frac{\int f(\bbetas,\bbeta^{(1)},\dots,\bbeta^{(l)})\prod_{m\in[l]} \Big\{\exp\{ \logl_{n,p-1}(\tbbeta^{(m)}) \} \prod_{j\in[p]} \mu(d\beta^{(m)}_j)\Big\}}{\int \prod_{m\in[l]} \Big\{\exp\{ \logl_{n,p-1}(\tbbeta^{(m)}) \} \prod_{j\in[p]} \mu(d\beta^{(m)}_j)\Big\}}.
\end{equation}
%\textcolor{red}{Write what above would mean when we have independent copies like explained in eqn \ref{eq:posterior_expectation}}
%\textcolor{red}{When we use this in appendix, we sometimes use this for an $\ell$-independent copy version, should we give that a separate notation or do an abuse of notation?}\ms{Now it is a multi replica definition as we agreed.}
Observe that, if $\bbeta$ is a sample from the leave-a-variable-out measure, then $\tbbeta$ is distributed according to a posterior on $\R^{p-1}$ corresponding to log-likelihood $\logl_{n,p-1}(\cdot)$ and prior $\mu^{\otimes p-1}$ and is independent of $\beta_{j_0}$ (the $j_0$-th coordinate of $\bbeta$), which is distributed according to the prior $\mu$. In analogy to the previous notations, we will let $\tbas$ and $\tba_l$ be the random vectors in $\mathbb{R}^n$ given by 
\begin{equation}\label{eq:ta_is}
 \tbas =(\tbX_1^\top \tbbeta_\star,\hdots, \tbX_n^\top\tbbeta_\star)^{\top}, \qquad \tba_{\ell} = (\tbX_1^\top \tbbeta^{(\ell)},\hdots, \tbX_n^\top\tbbeta^{(\ell)})^{\top},
 \end{equation}
 where, again, $\tbbeta^{(1)},\dots,\tbbeta^{(l)}$ are obtained by dropping the $j_0$-th coordinate from samples $\bbeta^{(1)},\dots,\bbeta^{(l)}$ drawn from  the leave-a-variable-out posterior.
%\textcolor{red}{We want to change this to make RHS $\prod_{j=1}^p \mu(db_j)$ and LHS $\langle f(\bbeta) \rangle_v$ but we should make this change only after checking the proofs.}\ms{I think this is an old comment addressing how to define the lavo posterior which we have already agreed and implemented.}

As with the notations introduced in \eqref{eq:mean_observation} and \eqref{eq:mean_variable}, whenever a new relevant posterior arises in our analysis, we refer to its expectation (given the observed data) to be $\thermal{\cdot}$, adding a suitable subscript to distinguish it from other posterior expectations previously defined. The need to introduce this notation will transpire later when we discuss our proof strategies. 

% For any vector $\bm{v} \in \mathbb{R}^p,$ we define its empirical distribution to be the probability distribution that assigns mass $1/p$ to each coordinate of the vector. \ms{Should we keep this definition of the empirical distribution? We only use it here in the hypothesis a couple of lines below and we write the formula down anyway.}  
\subsection{Assumptions on prior, likelihood and signal}
 We work with the following regularity conditions on the prior, likelihood and underlying signal. In Section \ref{sec:applications}, we discuss examples where our theory applies---we show that either the following assumptions apply directly or that the example can be well-approximated by a model where the following apply.
% \textcolor{red}{TILL HERE ON FINAL PASS.}
%We later discuss examples where these directly hold or some variant of these allow us to cover the respective example. 
\begin{hypothesis}\label{hyp1}
    Our assumptions on the prior, likelihood and signal are as follows:
    \begin{enumerate}[(i)]
        \item $\log \mu:\R\mapsto\R$ is a $\mathcal{C}^2$ and strongly concave function;
        \item $\tT_i:\R\mapsto\R$ is almost surely (a.s.) $\mathcal{C}^3$ and there exists a constant $d_1 > 0$ s.t. $||\tT'_i||_\infty,||\tT''_i||_\infty$, $||\tT'''_i||_\infty < d_1$ a.s.;
        \item $A:\R\mapsto\R$ is a non-negative, $\mathcal{C}^3$ and convex function s.t. $||A''||_\infty\leq1$ and $||A'''||_\infty < d_2$ for some $d_2 > 0$;
        \item the signal vectors $\bbetagt\in\R^p$ define a deterministic sequence such that $\sup_{n\geq1}||\bbetas||_\infty<\infty$, and as $n \rightarrow \infty$, the empirical measure $p^{-1}\sum_{j=1}^p \delta_{\beta_{\star, j}}(\cdot)$ converges weakly to some limiting distribution $\pi(\cdot)$, with second moments converging; in fact, we assume $\mbox{Var}(\bX_i^{\top}\bbeta_{\star}) = ||\bbetagt||^2/n$ converges to some finite constant of the form $\kappa \gamma^2 > 0$.
        \item for $i\in[n]$, $u_i(x,y)$ is an a.s. non-positive function (recall the definition of $u_i(\cdot,\cdot)$ from \eqref{eq:def_tT_u});
      \item and, for $k \in \mathbb{N}$, $\E(n^{-1}\sum_{i\in[n]}\tT^2_i(0))^{2k}$ is a bounded sequence.
    \end{enumerate}
\end{hypothesis}

In the former, all \emph{almost sure} statements are with respect to the randomness associated to the noise variables $e_1,\dots,e_n$. Furthermore, the measure $\pi(\cdot)$ in Hypothesis 1(iv) may place a positive mass at zero which would correspond to a fraction of the coordinates of $\bbeta_\star$ being zero, i.e.~signal vectors with such linear sparsity are allowed within our framework. In this sense, our framework accommodates both sparse and dense signals. We mention the aforementioned conditions as an assumption, but later we will show that all our examples satisfy these conditions. Finally, in addition to Hypothesis \ref{hyp1}, our main result requires some control on the moments of   $a_i$ as summarized below.

\begin{hypothesis}\label{hyp2}
    With $\mathbb{E}(\cdot)$ and $\thermal{\cdot}$ as above, we further assume that, for any $k \in \mathbb{N}$, there exists some $C_{k} > 0$ s.t. for all $n\geq1$ we have $\E\thermal{a_{1,1}^k} \leq C_k$.
\end{hypothesis}

\begin{remark}[establishing Hypothesis \ref{hyp2}]
    Note that in contrast to Hypothesis \ref{hyp1}, Hypothesis \ref{hyp2} might not be straightforward to check for a given model. In Appendix \ref{app:conv_mon_control}, we present general results that allow to check this, and show for our main examples that this indeed holds. 
\end{remark}
%\ms{Although this is presented as Hyp, make clear this is proved in the paper!!!}

\subsection{Auxiliary random measures and fixed point equations}\label{sec:auxiliary}

%{\color{blue}REMARK: I think this paragraph is a bit confusing as it is because we do not determine the convergence of $\bX_i^{\top}\bbeta_{\star},\bX_i^{\top}$ under the full posterior but only under the loo posterior. But this is not clear here. The role of Proposition 4 is a bit more convoluted than this. I don't really know how to convey the message of this paragraph without risking being to vague.}
Our main result (Theorem \ref{thm:marginals}) states that the finite-dimensional posterior marginals behave as Gaussian tilts of the prior. To describe the mean and variance of the tilt, 
we require two auxiliary measures that we introduce here. These measures arise repeatedly in our analysis. 
Before we describe these, note that the posterior is proportional to $\exp\{\cL_{n,p}(\bB)\} \prod_{j\in[p]}\mu(b_j)$ where the  likelihood $\mathcal{L}_{n,p}(\bB)$ from \eqref{eq:glmhamilton} depends on the data through the quantities $\{\bX_i^{\top}\bbeta_{\star},\bX_i^{\top} \bB\}_{i\in[n]}$. Thus, for a sample $\bbeta$ from the posterior distribution, if we can quantify the limit of the joint distribution of $(\bX_i^{\top}\bbeta_{\star},\bX_i^{\top} \bbeta)$ as $n,p \rightarrow \infty$ and $p/n \rightarrow \kappa$, we can establish the behavior of posterior marginals using this characterization. 
Later, in Proposition \ref{prop:hat_fix_point}, we do exactly that: we describe this asymptotic joint distribution. The main message from the proposition is that $(\bX_i^{\top}\bbeta_{\star},\bX_i^{\top} \bbeta)$ asymptotically converges to $(\theta_\star,\theta)\in\R^2$ where
%\ps{I am writing this interpretation here, but this may actually be quite wrong. I guess I am confused because your description below keeps $\xi_B$ free instead of taking it to be Gaussian. So my way of writing here may be wrong or misplaced, but we can discuss together once you are back if there is a better way to do this.}
\begin{equation}\label{eq:thetathetast}
 \begin{bmatrix}  \theta_\star \\ \theta \end{bmatrix} \equiv \begin{bmatrix}
     \theta_\star \\ \theta (\xi_B) 
\end{bmatrix} := \begin{bmatrix}
    0 \\ \sqrt{\kappa (v_B- c_B)} \xi_B 
\end{bmatrix} + \mathcal{N} \begin{bmatrix}
    \boldsymbol{0},\kappa\boldsymbol{\Sigma}
\end{bmatrix}, \qquad \text{with} \qquad \bSigma = \begin{bmatrix}
\gamma^2 \quad c_{B B_\star} \\ 
c_{B B_{\star}} \quad c_B
\end{bmatrix}.
\end{equation}
%\ms{Minor corrections in the formula above: $\kappa$ appeared to times in the Normal distribution and the $\kappa$ should not be square-rooted.}
The argument in $\theta(\xi_B)$ is used to denote that $\theta$ is a function of the dummy variable $\xi_B$. Recall that $\gamma^2$ was defined in Hypothesis \ref{hyp1}(iv).
Above, $v_B,c_B,c_{BB_\star}$
%we introduce the measure That is, we define
%\begin{equation}\label{eq:density_h}
%  p_{h}(b) \propto e^{-\frac{1}{2v}(b - m)^2} \mu(d b);
%\end{equation}
%where $m = \alpha B_\star + \sigma Z$ with $B_\star \sim \pi(\cdot)$ \ps{Which one do you use for this small $\pi(\cdot)$ or $\Pi(\cdot)$?} and $Z \sim \mathcal{N}(0,1)$ independent of everything else. 
%an independent standard Normal random variable. 
%As before, we will let $\thermal{\cdot}_h$ and $\E_{Z,B_\star}(\cdot)$ to be the expectations with respect to $p_h$ and $(Z,B_\star)$ respectively. 
%The following three quantities will be central in our characterization of the posterior marginals:
%\begin{equation}\label{eq:FPE_all1}
 %   \begin{dcases}
 %       v_B := \E_{Z,B_\star}\thermal{\beta^2}_{h} \\
 %       c_B := \E_{Z,B_\star}\thermal{\beta}^2_{h} \\
 %       c_{BB_\star} := \E_{Z,B_\star}\thermal{\beta B_\star}_{h}
%    \end{dcases}
%\end{equation}
are constants that we define later. Roughly, they correspond to the (normalized) norm of a typical sample from the posterior, the inner product between two independent samples from  the posterior, and the inner product between a sample from the posterior and the true signal (see Proposition \ref{prop:conv_order_param}). 
%sample and the true signal, and as follows. 
%That is, $v_B,c_B, c_{BB_\star}$  correspond to the (normalized) norm of a typical sample from a posterior, the inner product between two independent samples and  the inner product between a sample from the posterior and the true signal. 
%Alternately

Note that an alternate way of representing $\theta_\star$ and $\theta(\xi) $ would be via the following representations: 
\begin{equation}\label{eq:thetasearlier}
    \theta(\xi_{B}) := \sqrt{\kappa(v_{B}-c_{B})}\xi_{B}+\sqrt{\kappa c_{B}} z_{BB_{\star}} \,\,\, \mbox{ and } \,\,\, \theta_\star := \sqrt{\kappa\left(\gamma^2 - \frac{c_{BB_\star}^2}{c_B}\right)}\xi_{B_{\star}} +c_{BB_{\star}}\sqrt{\frac{\kappa}{c_B}}z_{BB_{\star}}
\end{equation}
if $c_B >0$ or $\theta(\xi_{B}) := \sqrt{\kappa v_{B}}\xi_{B}$ and $\theta_\star := \sqrt{\kappa\gamma^2}\xi_{B_{\star}}$ otherwise. 
Here, $\xi_{B_{\star}}$ and $z_{BB_\star}$ are independent standard Gaussian random variables. For now, we keep $\xi_B$ free. 

Now let $e \sim \mbox{Unif}[0,1]$ independent of everything else and define $\tT(\theta_\star):=T\circ f(\theta_\star,e)$. Define the following random measure (random since it depends on the random variables $e$, $\xi_{B_\star}$ and $z_{B B_\star}$)
 with density 
\begin{equation}\label{eq:density_s}
    p_{s}(\xi_B) \propto \exp\left\{\tT(\theta_\star)\theta(\xi_B) - A(\theta(\xi_B)) - \frac{\xi_B^2}{2}\right\}
\end{equation}
%As before, we will denote by $\thermal{\cdot}_{d_2}$ the expectation with respect to the random measure of density $\rho_{d_2}(\xi_{B})$ defined up to normalization by 
Let $\thermal{\cdot}_s$ denote the expectation under $p_s$. We will also let $\E_{G \otimes e}(\cdot)$  denote the expectation with respect to $\xi_{B_\star}$, $z_{BB_\star}$, and $e$. Here $G$ is used to denote the 2-dimensional standard Gaussian vector $(\xi_{B_\star},z_{B B_\star})$. Notice that $p_{s}$ is the posterior corresponding to a canonical GLM likelihood (recall the form from \eqref{eq:glmhamilton}) where $\theta_\star, \theta(\xi_B)$ respectively play the roles of $\bX_i^{\top}\bbeta_\star,\bX_{i}^{\top}\bbeta$ and a Gaussian prior is posited on $\xi_B$. %Using this measure, define the constants $(r_1,r_2,r_3)$ as 
%To define the constants of interest  $\alpha,\sigma,v$ that show up in \eqref{eq:density_hj} and \eqref{eq:density_h}, we need one final ingredient.
 If $\mathcal{L}_s$ denotes the log-likelihood corresponding to the density $p_s$, then the scores with respect to both $\theta_\star$ and $\theta$ will serve critical in our later derivations; we denote this by 
 %\textcolor{red}{$r_1$ is measuring in some sense how different from a certain exponential family this measure is. To revisit if I wish to comment about this}
\begin{equation}\label{eq:scalarscores}
    S_{\theta_\star}(\xi_B) := \partial_{\theta_\star} \mathcal{L}_s = \tTp(\theta_\star)\theta(\xi_B) \qquad \text{and} \qquad S_{\theta}(\xi_B) := \partial_{\theta}\mathcal{L}_s = \tT(\theta_\star)-A'(\theta(\xi_B)).
\end{equation}
With these definitions in place, define the constants $r_1,r_2,r_3 $ as
\begin{equation}\label{eq:FPE_all2}
    \begin{dcases}   
        r_1 := \E_{G \otimes e} \Big[\thermal{A''(\theta(\xi_B))}_s - \text{Var}_{\thermal{\cdot}_s}(S_\theta(\xi_B))\Big]\\
        r_2 := \E_{G\otimes e}\left[\text{Cov}_{\thermal{\cdot}_s}(S_{\theta_\star}(\xi_B);S_{\theta}(\xi_B))\right]\\
        r_3 := \E_{G\otimes e}\left[\thermal{S_{\theta}(\xi_B)}^2_{s}\right];
    \end{dcases}
\end{equation}
%\textcolor{red}{ON MAR 18, WE AGREE WE KEEP ABOVE AS IS, BUT THAT MEANS EVERYTHING ELSE INVOLVING $r_1,r_2,r_3$ NEEDS TO CHANGE. BECAUSE REST IS BASED ON OLDER EQUATIONS WHICH USED TO LOOK NASTIER.}
where $\text{Var}_{\thermal{\cdot}_s}$ and $\text{Cov}_{\thermal{\cdot}_s}$ denote, respectively, the variance and covariance under the measure with density $p_s$. 
Note these constants depend on $v_B,c_B,c_{BB_\star}$ that appeared in \eqref{eq:thetathetast}. We now introduce a different auxiliary measure that determines these.  

Recalling from Hypothesis \ref{hyp1}(iv) that $\pi(\cdot)$ denotes the limiting empirical distribution of the signal sequence $\bbeta_\star(n)$,  we define the measure
\begin{equation}\label{eq:density_h}
  p_{h}(b) \propto e^{-\frac{1}{2v}(b - m)^2} \mu(d b) 
\end{equation}
with $B_{\star}\sim \pi(\cdot)$, $Z \sim \mathcal{N}(0,1)$ independent of everything else, and 
%\textcolor{red}{MANUEL says change should be $\alpha=(t_\gamma + r_2)/r_1, \sigma = \sqrt{r_3}/r_1, v=1/r_1$ PRAGYA says this should be checked, $\alpha$ does not matfch for me, like where does $\bar{m}$ go?}\ms{I checked this and corrected it. What I was saying when we talked, was correct.}
\begin{equation}\label{eq:alphasv}
m = \alpha B_\star + \sigma Z, \quad \alpha := \frac{r_2 + t_\gamma}{r_1}, \quad \sigma := \frac{\sqrt{r_3}}{r_1}, \quad v := \frac{1}{r_1}, \quad t_\gamma = \mathbb{E}[\tT'_1(\sqrt{\kappa}\gamma Z)]
\end{equation}

%an independent standard Normal random variable. 
%\textcolor{red}{We both agree above is all correct now. Note that $\gamma^2$ should be the one normalized by $p$.}
We denote $\thermal{\cdot}_h$ and $\E_{Z,B_\star}(\cdot)$ to be the expectations with respect to $p_h$ and $(Z,B_\star)$ respectively.\footnote{With some abuse of notation, we will use $\thermal{\cdot}_h$ to denote expectation under multiple independent copies of $p_h(\cdot)$ on some cases later, but this would either be clarified explicitly or be clear from context.} Then, let
\begin{equation}\label{eq:FPE_all1}
    \begin{dcases}
        v_B := \E_{Z,B_\star}\thermal{\beta^2}_{h} \\
        c_B := \E_{Z,B_\star}\thermal{\beta}^2_{h} \\
        c_{BB_\star} := \E_{Z,B_\star}\thermal{\beta B_\star}_{h}
    \end{dcases}
\end{equation}
Note the RHS of \eqref{eq:FPE_all1} depends on $\alpha,\sigma, v$ (which in turn depends on $r_1,r_2,r_3$), whereas the RHS of \eqref{eq:FPE_all2} depends on $v_B,c_B,c_{BB_\star}$. Thus,  \eqref{eq:FPE_all2} and \eqref{eq:FPE_all1} provide a coupled system of equations inter-dependent on each other. We assume throughout that this coupled system admits a unique solution. Strategies for proving uniqueness of such systems of equations have appeared in the prior frequentist literature \cite{bellec2023existence,montanari2025generalization,li2023spectrum}. Proving them in our setting would require case-by-case analysis of GLMs, which we refrain from diving into since it draws significant attention away from our main contributions. Numerically, we find that if our other assumptions are met, the system can be solved with extremely high precision. The code for solving this system of equations can be found in \cite{github_manu}.

    \section{Main Results and Applications}\label{sec:results}
\subsection{Main Results}
%\ps{Need a result on the full posterior mean}
%\ps{This section needs mass re-writing: goal should be to directly talk about the finite-dimensional marginals result, that contraction does not occur, and the posterior mean result as a consequence--also want for the full posterior mean.}
Our main result characterizes the behavior of finite-dimensional marginals of the posterior distribution in the setting of Section \ref{sec:model}. Before we delve into this, we highlight a major difference between our setting and the more traditional high-dimensional Bayesian literature that posits sparsity-type assumptions on the underlying signal, and considers the setting where the number of features is much larger than the sample size. A central picture that emerges in this  traditional regime is that the $p$-dimensional posterior measure $\proba(\cdot|\bX,\by)$ concentrates around the true signal---a phenomenon known as posterior consistency. Formally, posterior consistency states that for all $\epsilon >0 $, the probability that the posterior distribution assigns to balls (in a suitable metric) centered at $\bbeta_{\star}$ with radius $\epsilon$ converges almost surely or in probability to $1$ asymptotically. In our regime, a different picture emerges---the posterior no longer concentrates around the truth. In fact the following holds. 

%\textcolor{red}{Chenge statement for means of general separable functions and then MSE is particular case.}
%Proposition \ref{prop:conv_order_param} readily implies the non-contraction of the posterior measure.
\begin{theorem}\label{thm:contraction}
In the setting of Section \ref{sec:model}, let $f:\R^l\mapsto\R$ grow at most polynomially. Then we have that
\begin{equation*}
    \lim_{n\rightarrow \infty, p/n \rightarrow \kappa} \E \thermal{\frac{1}{p}\sum_{j\in[p]}f(\beta^{(1)}_j,\dots,\beta^{(l)}_j)} = \E_{Z,B_\star}\thermal{f(\beta^{(1)},\dots,\beta^{(l)})}_h,
\end{equation*}
where on the left, $\beta_j^{(1)},\beta_j^{(2)},\hdots, \beta_{j}^{(l)}$ denote the $j$-th coordinates of $l$ independent samples from the posterior and on the right, the expectation is over $\beta^{(1)},\hdots,\beta^{(l)}$ drawn from $l$ independent copies of $p_h(\cdot)$ defined in \eqref{eq:density_h}. 
%\textcolor{red}{The h notation on RHS neds clarification.}
In particular,
\begin{equation*}
    \lim_{n\rightarrow \infty, p/n \rightarrow \kappa} \E \thermal{(p^{-1}\|\bbeta - \bbeta_\star \|^2-c)^2} \rightarrow 0,
\end{equation*}
where $\bbeta$ denotes a sample from the posterior and $c := \gamma^2 + v_B - 2 c_{BB_\star}$, with $\gamma^2$ defined as in Hypothesis \ref{hyp1}(iv) and $v_B, c_{BB_\star}$ defined as in \eqref{eq:FPE_all1}. 
\end{theorem}

\begin{figure}[ht]
  \centering
  \includegraphics[width=0.6\textwidth]{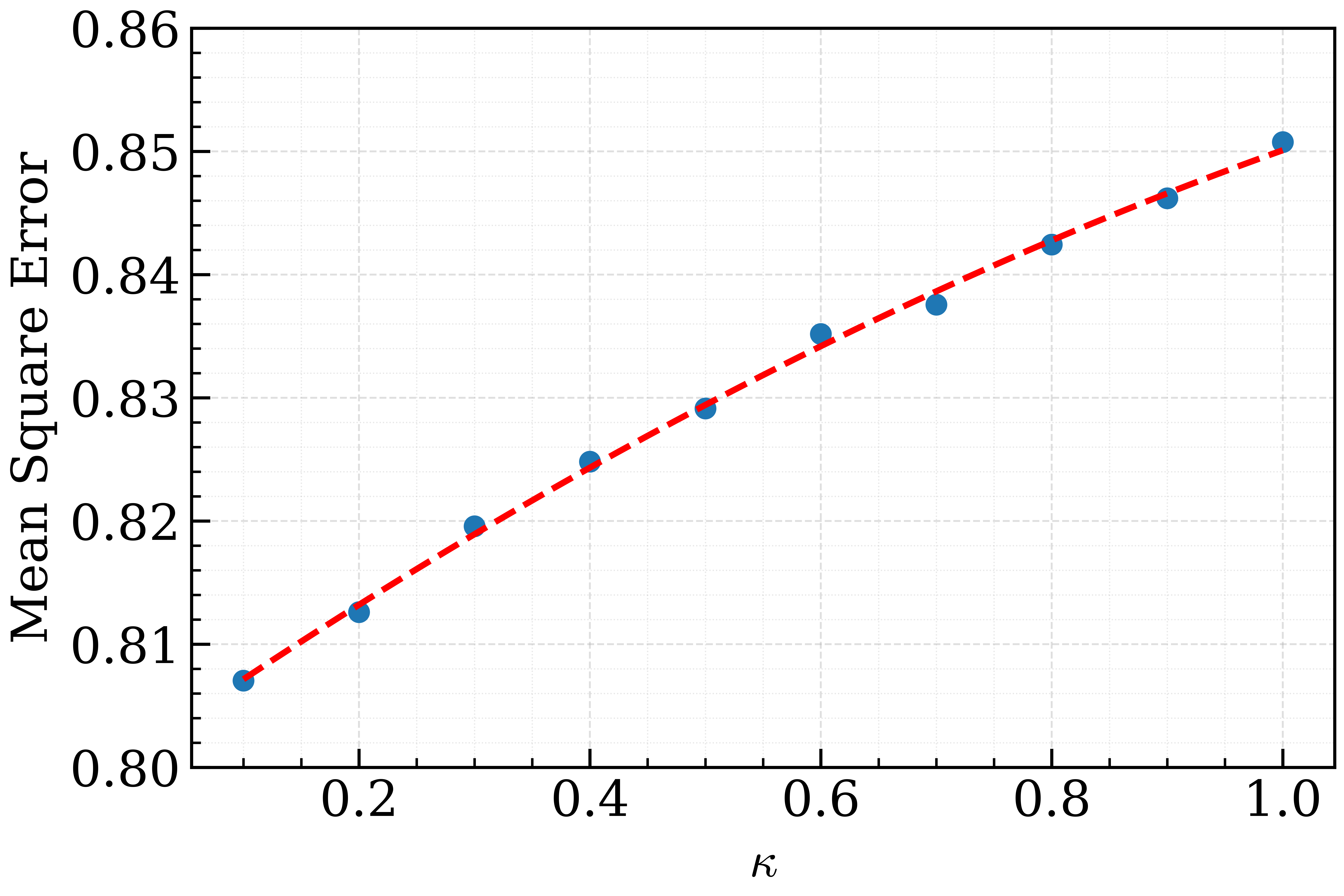}
  \caption{The asymptotic mean squared error, constant $c$ of Theorem \ref{thm:contraction}, as a function of the aspect ratio $\kappa$ for a logistic regression model with Beta prior and signal. Specifically, we choose $\mu=\text{Beta}(2,2)$ and $\pi=\text{Beta}(2,5)$. Observe the increase in MSE with increase in $\kappa$. Computing $c$ requires  numerically calculating $v_B$ and $c_{BB_\star}$, which in turn requires solving the fixed point equations \eqref{eq:FPE_all2} and \eqref{eq:FPE_all1}.  Precise details on how these equations are solved numerically is provided in Appendix \eqref{app:computation} and linked in \cite{github_manu}.}
  \label{fig:mse_graphs}
  %\textcolor{blue}{Figure \ref{fig:mse_graphs} displays the asymptotic mean squared error, contant $c$ of Theorem \ref{thm:contraction}, as a function of the aspect ratio $\kappa$. In the left panel, the MSE is shown for a $\mathrm{Beta}(2,2)$ prior and a $\mathrm{Beta}(2,5)$ signal, illustrating the monotone deterioration of performance as $\kappa$ increases. There, we present the values of the MSE of Bayes mean as predicted by Theorem \ref{thm:contraction}. Each numerical computation was repeated $5$ times. These predictions were compared with MCMC simulations of $p=500$ and $n = \lfloor \kappa p\rfloor$, $5$ simulations for each value of $\kappa$. All predictions fall within error margin of the MCMC simulations. }
\end{figure}

%\textcolor{red}{A picture for $c$ as a function of $\kappa$ growing should be included here in this section for whatever setting is doable. So we can at least validate numerically that $c>0$ as soon as $\kappa > 0$}\ms{This figure has been added in the applications section.}
%We defer the formal description of the constant $c$ to the Appendix.
The proof of this theorem can be found in Appendix \ref{sec:proof_contraction}.
Typically, $c > 0$ as long as $\kappa >0$.  We demonstrate this in Figure \ref{fig:mse_graphs}, where we plot $c$ as a function of $\kappa$ for a logistic model with Beta prior. Observe that $c$ is positive and increases as a function of the dimension to sample size ratio $\kappa$, suggesting that the larger the dimensionality the larger the mean squared error of a typical sample from the posterior.
 %\textcolor{red}{TILL HERE ON FINAL PASS.} \ps{We need a plot of $c$ as a function of $p/n$ in a logistic model - beta prior setting, keeping everything else  fixed}. Crucially, $c$ increases as a function of the dimensionality factor, the larger the number of features compared to the sample size, the larger this value. 
 The main takeaway is that since the $L_2$ distance between a sample from the posterior and the true signal converges in $L_2$ to a positive constant, the full $p$-dimensional posterior cannot concentrate around the truth. %\ps{If we could come up with sufficient condition under which we can mathematically show that $c >0$, that would be really nice.} 
 This highlights the stark contrast between the traditional ultra-high-dimensional (parametric) regime  and the proportional asymptotics regime. We remark that in non-parametric and semi-parametric models, there are multiple instances in prior literature where the posterior is inconsistent in a frequentist sense \cite{diaconis1986consistency,kim2001posterior,ghosal2010dirichlet,miller2014inconsistency,alamichel2024bayesian}. In that sense, our parametric setting behaves similar to these models rather than the high-dimensional parametric ultra-sparse models. Recently, \cite{mukherjee2022variational,lee2025clt} studied such an inconsistency regime as well, but in our framework, their results remain within the $p=o(n)$ regime. 
 %\textcolor{red}{Call Sumitda to clarify details about a particular point in their work; same for Subhabrata's work.}

%\textcolor{red}{Figure 1 description, anything else to include here or in caption?}
Despite this negative picture about posterior consistency, perhaps fascinatingly, the finite-dimensional marginals of the posterior still contain information about the corresponding coordinates of the true signal. We formalize this in our main result below. 
%\textcolor{red}{TILL HERE ON SECOND PASS}
%\textcolor{red}{TILL HERE ON FINAL PASS.}
%\textcolor{red}{CLarify h notation in previous theorem then go to the next theorem.}
\begin{theorem}\label{thm:marginals}
In the setting of Section \ref{sec:model},
   % Assume Hypothesis \ref{hyp1},\ref{hyp2}, and \ref{hyp3} hold \ps{I am in favor of removing Hypothesis 2 if indeed we have a proof of this later} and let $(v_B,c_B,c_{BB_\star},r_1,r_2,r_3)$ be the unique solution of \eqref{eq:FPE_all1} and \eqref{eq:FPE_all2}.  
   suppose $\bbeta$ denotes a sample from the posterior distribution. For each $r\in\mathbb{N}$ and set of distinct indices $j_1,\dots,j_{r}\in\mathbb{N}$, let $(\beta_{j_1},\dots,\beta_{j_{r}})\in\R^r$ denote the random vector formed by the $j_1,\dots,j_{r}$-th coordinates of $\bbeta$. Then
    \begin{equation*}
      (\beta_{j_1},\hdots, \beta_{j_r}) \stackrel{\textrm{d}}{\rightarrow} (B_{j_1},\hdots, B_{j_r}),
    \end{equation*}
    where $(B_{j_1},\hdots,B_{j_r})$ has a joint density on $\R^r$ proportional to $p_{h,{j_1}} \times \dots \times p_{h,{j_r}}$, where $p_{h,j}$ is given by \begin{equation}\label{eq:density_hj}
  p_{h,j}(b) \propto e^{-\frac{1}{2v}(b - m_{j})^2} \mu(d b), \quad \text{with} \quad m_j = \alpha \beta_{\star,j}+\sigma Z_j.
\end{equation}
Above $\beta_{\star, j}$ denotes the $j$-th coordinate of the true signal $\bbeta_{\star}$, the variables $Z_j$ are i.i.d. standard Gaussian independent of everything else, and $\alpha,\sigma, v$ are as in  \eqref{eq:alphasv}.   
    %   with the values of $(\alpha,\sigma,v)$ in \eqref{eq:density_hj} given by $\alpha = \frac{t_\gamma+r_3}{r_1+r_2}$, $\sigma = \frac{\sqrt{r_2}}{r_1+r_2}$, $v = \frac{1}{r_1+r_2}$, and $t_\gamma$ as in \eqref{eq:def_tp}.
\end{theorem}
%\emph{Arrange this way: (i) Gaussian calculation, contrast etc, (ii) Gaussian sequence model comments (iii) Nonparametric Bayes comments (more details than below), (iv) MLE comparison comments (v) Discussion -- what do we need to estimate to make our results useful, optimal priors.}
The proof is deferred to Section \ref{sec:proof_marginals}. The theorem says that the finite-dimensional posterior marginals can be expressed as Gaussian tilts of the prior. Crucially, the tilt itself has a random mean, that is in the form of a constant times the corresponding true signal coordinate plus additional Gaussian noise. Our proof heuristics section (Section \ref{sec:pfoutline}) further illuminates why the posterior marginals take this form. In particular, our proof uncovers interesting connections with classical Le Cam theory. For instance, we express the full posterior as tilts of the leave-a-variable-out and leave-an-observation-out posteriors introduced in Section \ref{sec:LOOs}. A key step in our proof is to show the tilt part converges to appropriate Gaussian variables, as in classical local asymptotic analysis of standard estimators such as the MLE \cite{van2000asymptotic}. We explain this in further detail in Section \ref{sec:pfoutline}. Here, we present some discussions surrounding the theorem.

%\textcolor{red}{This will go later, and other comments if any needs to be merged into our setting. Below discussion needs a pass and work.} 
To gain insights, we contrast our theorem with the case of a high-dimensional Gaussian linear model with Gaussian prior, where explicit closed-form formulae are available for the posterior. Specifically, if we assume the observed data comes from a linear model with $y_i = \bX_i^{\top}\bbeta_{\star} + \epsilon_i$ and $\epsilon_i \sim \mathcal{N}(0,\tilde{\sigma}^2\mathbb{I}_n)$. In our setting, this means $f(\bX_{i}^{\top}\bbeta_\star,e_i) = \bX_i^{\top}\bbeta_\star + \Phi_{\tilde{\sigma}}^{-1} (e_i)$ where $\Phi_{\tilde{\sigma}}^{-1}(\cdot)$ is the inverse Gaussian cdf of a $N(0,\tilde{\sigma}^2\mathbb{I}_n )$. Suppose that we consider an i.i.d. $\mathcal{N}(0,\tilde{v})$ prior on each coordinate. Although we can perform closed-form calculations here, we assume for simplicity $\tilde{\sigma}=\tilde{v} =1$. The posterior distribution at $\bbeta$ is proportional to 
\begin{equation}\label{eq:linmod} 
\exp\{-\frac{1}{2}\|\by - \bX\bbeta \|^2 - \underbrace{\frac{1}{2}\|\bbeta \|^2}_{\text{tilt from the prior}}\} = \exp\{-\frac{1}{2}\bbeta^{\top}(\bX^{\top}\bX + \bm{I})\bbeta + \by^{\top} \bX \bbeta\}, 
\end{equation}
which is a multivariate Gaussian $N(\bm{\mu}, \bm{\Sigma})$ with $\mu = (\bX^{\top}\bX + \mathbb{I}_p)^{-1}\bX^{\top}\by, \bm{\Sigma} = (\bX^{\top}\bX + \mathbb{I}_p)^{-1}.$

From here the asymptotic distribution of the marginals can be explicitly calculated. We show later (Appendix \ref{app:direct_integration_op}) that this closed-form calculation and our Theorem \ref{thm:marginals} indeed match. Here we show the form of the posterior to emphasize that it is indeed a tilt from the prior. Note that the posterior marginals in  a Gaussian linear model remains Gaussian due to both the 
prior and likelihood being Gaussian. Once we move beyond this simple example, the posterior naturally no longer remains Gaussian. 
Notably, the structure of the posterior being a tilt from the prior still remains, as outlined in  \eqref{eq:density_hj}. Theorem \ref{thm:marginals} shows that, even for GLMs where the likelihood is far from a Gaussian, the likelihood part still contributes a Gaussian density (with suitable mean and variance) in the limit. This is a non-trivial takeaway and we defer more detailed discussion for why this is the case to the proof heuristics section (Section \ref{sec:pfoutline}). Unlike many preceding parametric high-dimensional Bayesian settings, the prior effect does not wash away, even in the limit of large sample and dimensions, in our setting.

Among existing works, Theorem \ref{thm:marginals} is reminiscent of a few concrete settings studied in prior literature.  
 First, in \cite{johnstone2010high}, Gaussian sequence models were extensively studied. Letting the prior variance depend on the sample size, the author pinned down the precise relation between the prior variance versus the noise level in the data, as a function of the dimension, that determine when the effect of the prior washes away versus not, in the limit of large samples; the calculations there already allude to the fact that, in a parametric problem when the number of samples and dimensions are diverging at a comparable rate, traditional results such as Bernstein-von-Mises, and effect of the prior washing away, should no longer hold. However, \cite{johnstone2010high} heavily exploited the Gaussian sequence model structure, which ensured the posterior remains Gaussian (since the assumed priors were also Gaussian so the posterior has a neat closed form, analogous to \eqref{eq:linmod} but for the sequence model case). Whereas in our setting, Gaussianity of the posterior is not automatic, and we find that it is a Gaussian tilt of the prior instead.

Second, there are ample examples from the non-parametric Bayes, mixture model, and normal means literature where the effect of the prior does not wash away even in the limit of large samples \cite{regazzini2003distributional,james2009posterior,rousseau2011asymptotic,de2013asymptotic,ghosal2017fundamentals} (see also numerical observations reported in \cite{ray2022variational}). Third, there has been an emerging interest in studying low-dimensional parameters in high-dimensional models in the Bayesian inference literature \cite{castillo2024variational,lee2025clt}). We anticipate that our leave-one-out strategies, explained in Section \ref{sec:pfoutline}, will provide general-purpose tools to advance the study of low-dimensional functionals in high-dimensional Bayesian problems. Finally, a naive mean field and variational Bayes literature has emerged \cite{mukherjee2022variational,mukherjee2023mean,mukherjee2024naive}, where the posterior has been expressed as tilts of the prior. Translated to our setting, these papers cover the case of $p=o(n)$ which is lower dimensional than ours (see also \cite{katsevich2023improved,katsevich2024approximation}).
%and furthermore, provide average coverage guarantees which is different from studying finite-dimensional marginals. 
In fact, in our setting, naive mean field approximations fail \cite{qiu2023tap,qiu2023sub,fan2021tap,celentano2023mean}, thus prior naive mean field-based techniques \cite{mukherjee2022variational,mukherjee2024naive} no longer apply. In some sense, our paper is precisely complementary to this line of work. 

\begin{figure}[h!]
  \centering
  % First row
    \includegraphics[width=0.48\linewidth]{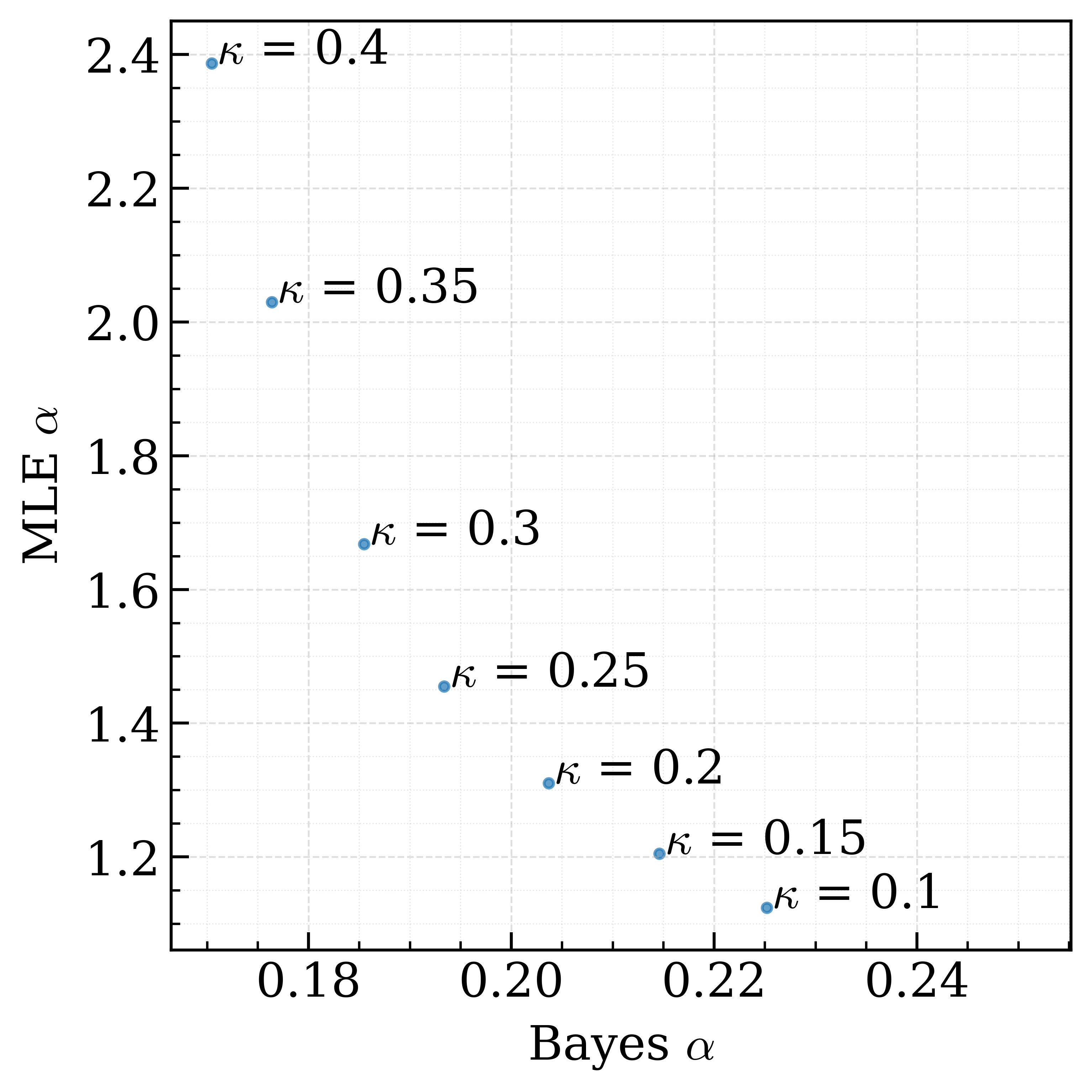}\hfill
    \includegraphics[width=0.48\linewidth]{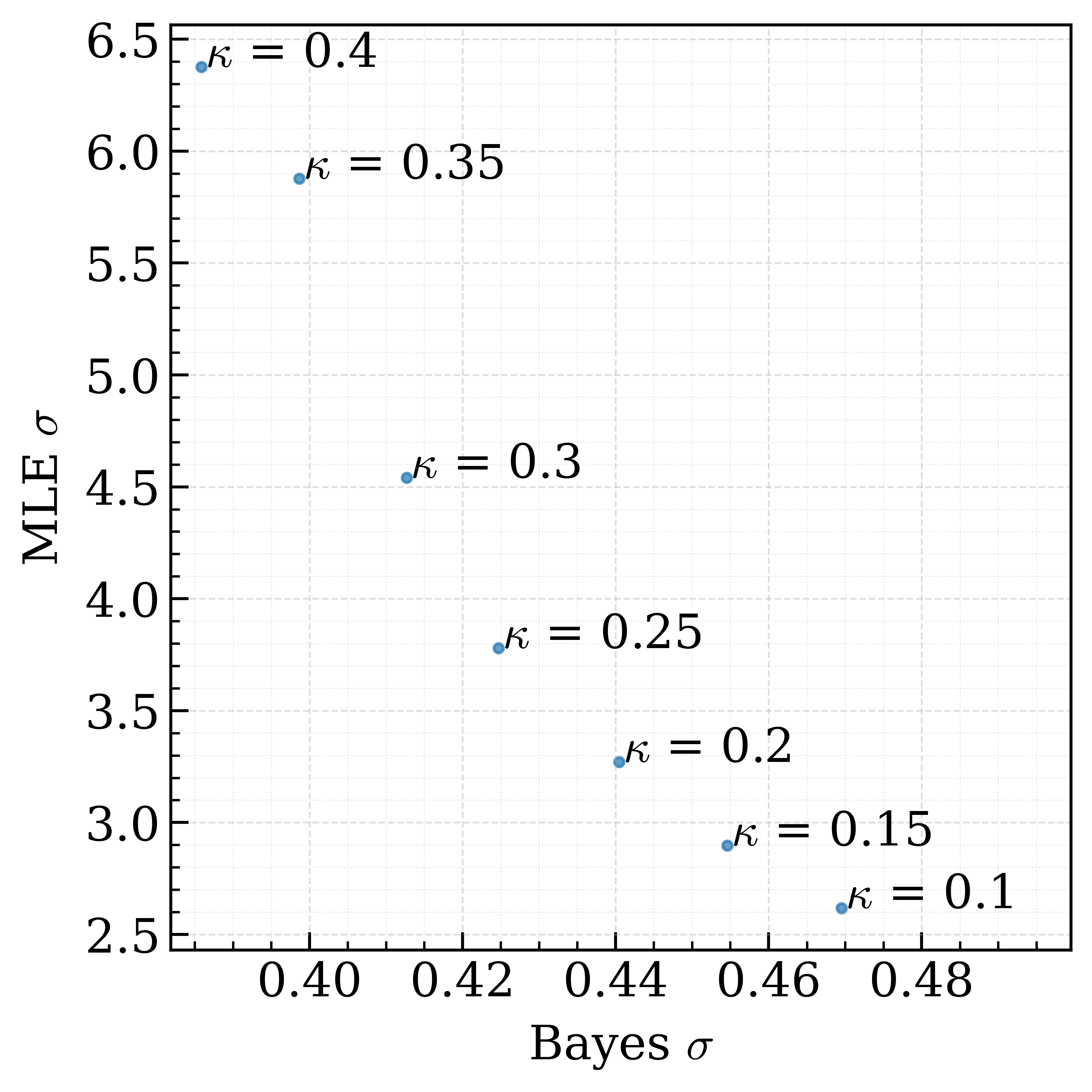}
  \caption{Comparison of the coefficients $\alpha$ and $\sigma$, from \eqref{eq:mle_coef} and \eqref{eq:bayes_mean_coef}, for the MLE and the posterior mean corresponding to various values of $\kappa$. Here, $\pi(\cdot)$ is taken to be a Rademacher distribution and the prior $\mu(\cdot)$ is taken to be standard Gaussian. We choose such a setting with stark difference between $\pi(\cdot)$ and $\mu(\cdot)$ to illustrate that our results allow for such mismatch; this is in fact one of the strengths of our approach.  Interestingly, between the MLE and the posterior mean,  the growth trends for $\alpha$ and $\sigma$ see a reversal as a function of $\kappa$.}
\label{fig:comparison_mle_bayes}
\end{figure}
  %Here the prior considered is a standard Normal distribution and the signal is Rademacher multiplied by a constant.} 
  %Left panel, comparison of coefficients $\alpha$ for various values of $\kappa$. Right panel, comparison of coefficients $\sigma$ for various values of $\kappa$.}

%Finally, the proportional asymptotics regime has been studied for empirical Bayes problems extensively in recent literature, including interesting algorithmic advances in this front \cite{}. That said, these works do not study behavior of the posterior distribution and it is unclear how to generalize their tools to study the posterior.

%In contrast, \cite{zhou} operates in a
%Bayes optimal regime where the true signal prior is known to the statistician, again a regime different from ours. 
%[\textcolor{red}{CHECK SUBHA and ZHOU work details and re-arrange last line if required.] 
%[\textcolor{red}{CITE Subha and Zhou's work earlier}] 
%[CITE GHOSHAL AND VdV, DeBlasi et al from Marta's email]].

One critical aspect of our result is the form the mean of the Gaussian in the tilt takes on. One would naturally inquire if the finite-dimensional posterior marginals can be connected to the finite-dimensional marginals of the MLE in these models in the same sense that a classical BvM theorem does. At a first pass, we should not expect any traditional BvM to hold here, following observations already pointed out in \cite{cox1993analysis,freedmanwaldled, johnstone2010high}. However, some connections can still be seen between the posterior marginals and the MLE for models that are well-studied in this proportional asymptotics regime in the frequentist literature. For instance, for logistic models, it is known by now \cite{sur2019modern,zhao2022asymptotic} that the $j_0$-th marginal of the MLE satisfies
\begin{equation}\label{eq:mle_coef}
    \hat{\beta}_{{\rm MLE},j_0} \xrightarrow{d} \alpha_{\rm MLE} \beta_{\star,j_0} + \sigma_{\rm MLE} Z',
\end{equation}
 $Z' \sim \mathcal{N}(0,1)$, independent of everything else. An analogous result is known for finite-dimensional marginals \cite{sur2019modern,zhao2022asymptotic} but we choose a single coordinate here for simplicity. Above, $\alpha_{\rm MLE},\sigma_{\rm MLE}$ are constants, with $\alpha_{\rm MLE}$ related to the projection of the MLE in the direction of the true signal, and $\sigma_{\rm MLE}$ given by the norm of the residual. 
 
 By our Theorem \ref{thm:marginals}, the mean of the Gaussian tilt for a posterior marginal takes on a similar form but with different constants $\alpha,\sigma$. The presence of the prior in the limit \eqref{eq:density_hj} makes it hard to relate these constants to simple objects such as projection of the posterior sample in the direction of the signal, etc., as was possible in case of the MLE. It is still worthwhile to compare how similar or different these constants are. Note that if the prior is taken to be Gaussian, then even in logistic models the posterior marginal limit in \eqref{eq:density_hj} simplifies---comparing the constants for this setting with the logistic MLE constants from \eqref{eq:mle_coef} is insightful. Additionally, performing this comparison at the level of means of the posterior marginal is particularly instructive. We present this comparison next, but we first mention the following corollary that describes posterior marginals at the level of means. 

 \begin{corollary}\label{cor:marginals_bayes}
Under the conditions of Theorem \ref{thm:marginals}, the posterior mean $\thermal{(\beta_{j_1},\hdots,\beta_{j_r})}$ of the r-tuple $(\beta_{j_1},\hdots,\beta_{j_r})$ converges in distribution toward a random vector $\bm{m}\in\R^r$ with independent coordinates $k\in[r]$ distributed according to
    \begin{equation*}
        m_k \overset{d}{=} \int B \, p_{h,{j_k}}(d B),
    \end{equation*}
with $p_{h,{j_k}}$ defined as in \eqref{eq:density_hj}.
\end{corollary}
%We present this comparison in [CITE Figure]
The proof of this corollary can be found in Section \ref{sec:proofcor1}.
As mentioned earlier,  $p_{h,j_k}$ critically retains the effect of the prior. However, if the prior is Gaussian, it serves conjugate to the Gaussian tilt part and the limiting mean distribution simplifies. In fact, for a  Gaussian prior, Corollary \ref{cor:marginals_bayes} can be re-written as
%we can see that there are constants can be integrated giving that, for proper constants $\alpha_{\rm bayes}$ and $\sigma_{\rm bayes}$
\begin{equation}\label{eq:bayes_mean_coef}
    \thermal{\beta_{j_0}} \xrightarrow{d} \alpha_{\rm Bayes} \beta_{\star,j_0} + \sigma_{\rm Bayes} Z
\end{equation}
where $Z \sim \mathcal{N}(0,1)$, independent of everything else and we choose $r=1$ for simplicity.  Here the constants $\alpha_{\rm Bayes},\sigma_{\rm Bayes}$ depend on the choice of the GLM and other problem parameters such as the dimension to sample size ratio $\kappa$ and the signal norm $\gamma^2$ (this may also be thought of as a signal-to-noise ratio parameter). 

In the specific case where the underlying GLM is taken to be logistic regression, we compare the Bayes constants from \eqref{eq:bayes_mean_coef} with the MLE constants from \eqref{eq:mle_coef} in Figure \ref{fig:comparison_mle_bayes} and the left plot of Figure \ref{fig:mse_graphs2}. For simplicity, we take the prior to be standard Gaussian. Also, note that if the dimensionality is sufficiently large compared to the sample size, the MLE would not exist, since the data would be perfectly linearly separable. We restrict to the regime where the MLE exists as identified in \cite{candes2020phase}. 
%Note that both the multiplicative bias $\alpha$ and the standard deviation $\sigma$ increase as a function of $\kappa$ and $\gamma$ for both the MLE and the posterior mean. Therefore the same trend is observed for the asymptotic MSE shown in Figure \ref{fig:mse_graphs} since here we simply plot $\alpha^2+\sigma^2$. But 
Notably, the MLE values for $\alpha,\sigma$ are significantly higher than the Bayes values.  This  contrasts sharply with classical statistics or other parametric frameworks where the posterior mean asymptotically behaves like the MLE. Remarkably, this difference persists even though we operate within a well-behaved parametric setting---one as simple as logistic regression. 
%Thus, our high-dimensional regime uncovers rather different phenomena than those traditionally familiar to Bayesian statisticians, warranting further investigation.  
% although we operate within a nice parametric setting, something as simple as logistic regression. In that sense, this high-dimensional regime uncovers rather different phenomena than what Bayesian statisticians may be traditionally familiar with and thus merits additional investigations. 
\begin{figure}[ht]
  \centering
  \includegraphics[width=0.48\textwidth]{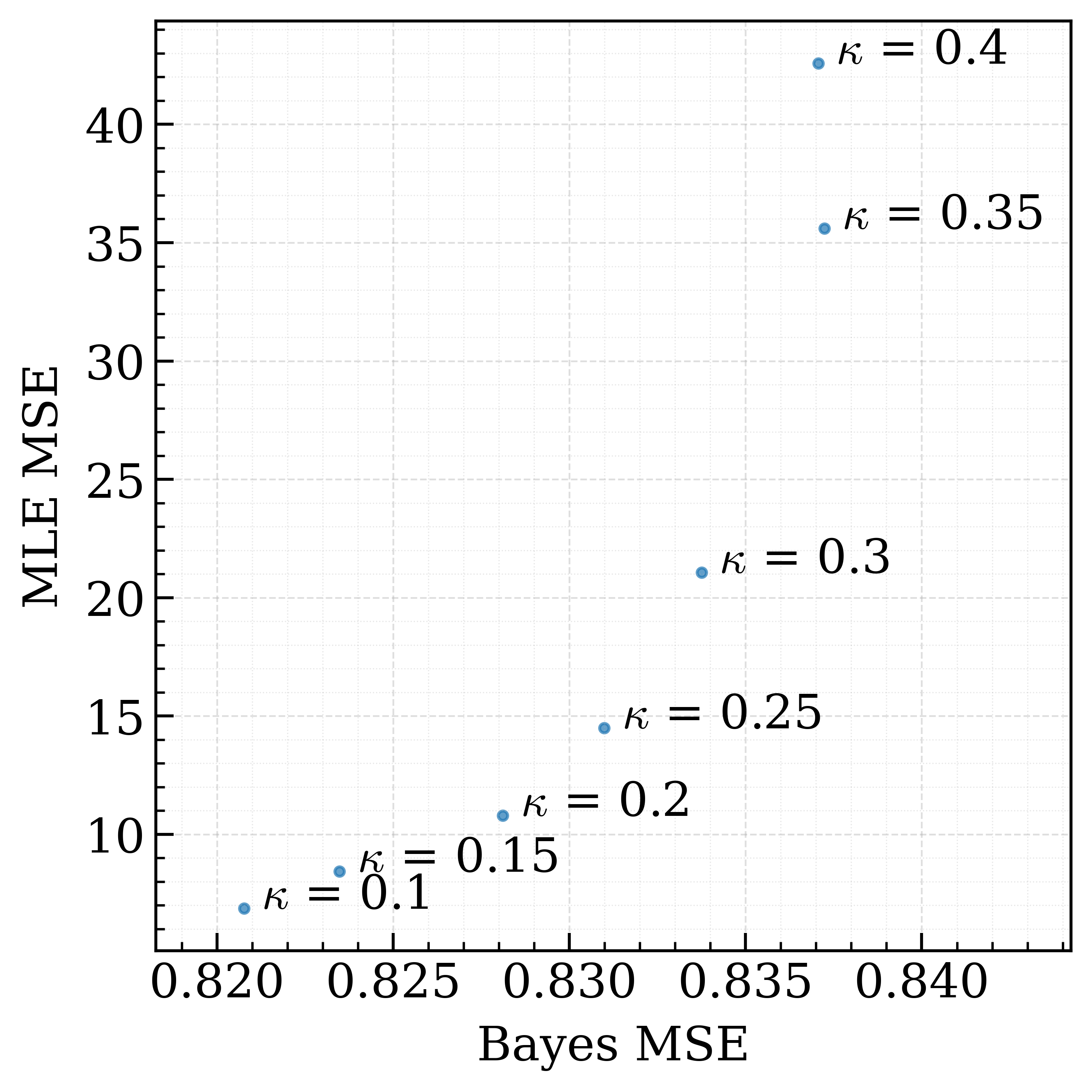}\hfill
  \includegraphics[width=0.48\textwidth]{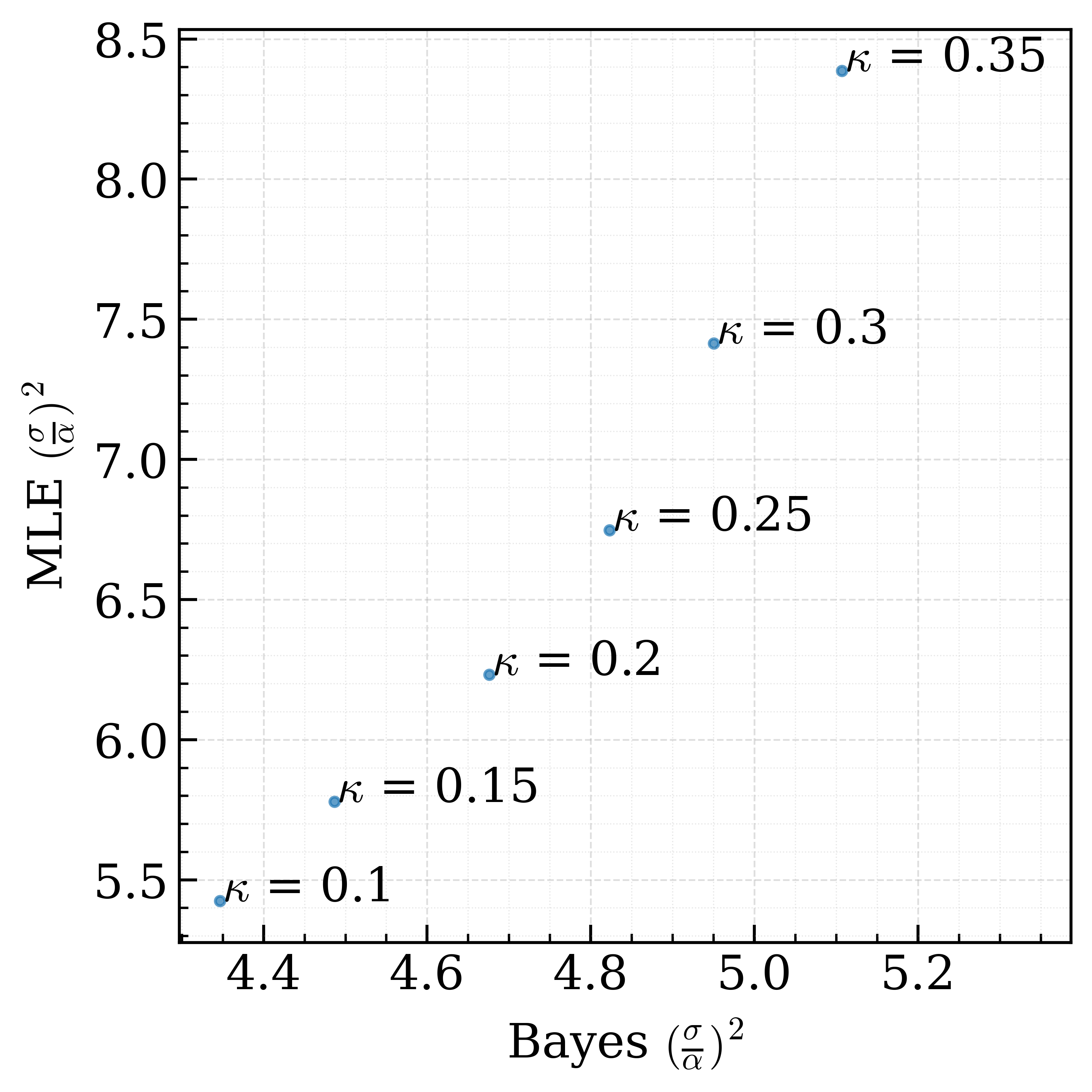}
  \caption{The left panel plots the asymptotic MSE of an MLE coordinate versus that of the posterior mean corresponding to that coordinate. The setting is the same as for Figure \ref{fig:comparison_mle_bayes}. 
Due to the relations \eqref{eq:mle_coef} and \eqref{eq:bayes_mean_coef}, the MSEs plotted on the left are given by  $(1-\alpha)^2+\sigma^2$. The right panel presents the asymptotic MSE of debiased versions, obtained by dividing the MLE and the posterior mean by their respective $\alpha$ coefficients. Thus the MSE plotted on the right is given by $(\sigma/\alpha)^2$. In both cases, the MLE performs worse than the corresponding Bayes procedure. Across the board, the MSEs increase with increase in $\kappa$. For the Bayes procedures, the unbiasedness on the right comes at the cost of significant increase in variance.}
%\textcolor{red}{The description for the right figure needs more details on how the MSE is actually computed. Like in terms of the $\alpha,\sigma$ constants that have now been introduced in the test, what exactly is the y-axis and what exactly is the x-axis.}}
  \label{fig:mse_graphs2}
\end{figure}

In fact, in our regime, beating the MLE with an exceptionally simple Bayes estimator---the posterior mean with a Gaussian prior---is easy. The left panel in Figure  \ref{fig:mse_graphs2} compares the asymptotic MSE of  the MLE versus the posterior mean for a marginal. A striking feature is the pronounced gap between the two values: across the full range of $\kappa$ values considered, the Bayes estimator exhibits substantially smaller error than the MLE. 
%This highlights the improvement in performance achieved by incorporating prior information in our high-dimensional regime. 
Notably, one does not even need a sophisticated prior to see such gains in performance. 
Given the issues with the MLE pointed out in \cite{sur2019modern}, this may be unsurprising. But, this reinforces the understanding that seeking a BvM type result with the MLE as the center would be unreasonable to pursue in this regime; since the MLE itself is sub-optimal. Identifying the correct centering frequentist estimator would be an interesting direction to pursue in the future.  

%\textcolor{blue}{Integrate with previous writing:
%We finally compare Bayes mean and Maximum Likelihood Estimation of a fixed coordinate in the context of a standard Normal prior and Rademacher signal multiplied by a constant. In Figure \ref{fig:comparison_mle_bayes}, we compare coefficients of equations \eqref{eq:bayes_mean_coef} and \eqref{eq:mle_coef}, of Bayes mean and MLE, for various values of $\kappa$ and $\gamma^2$ in this setting. This allows to visually contrast the corresponding Bayes and MLE single coordinate estimators under different scenarios.
%}

%Figure blah shows --- figure description.

In \cite{sur2019modern} the authors proposed a debiased MLE that removes the multiplicative bias in \eqref{eq:mle_coef} by rescaling the MLE with a consistent estimator of $\alpha_{\rm MLE}$. One might ask if this rescaled MLE is a better choice for comparing with the posterior mean. But, this comparison is unfair since the posterior mean is biased per \eqref{eq:bayes_mean_coef}. This is perhaps unsurprising given observations already seen in \cite{sur2019modern,sur2019likelihood}, which show that traditionally unbiased estimators no longer remain so in this high-dimensional regime. Such bias in the posterior mean has also been observed in semi-parametric Bayesian causal inference problems \cite{ray2020semiparametric}; for us, this is seen purely due to high-dimensional effects. Since there is no apriori assumed sparsity in the true underlying signal (it could be sparse or dense), the cumulative contribution of all coordinates incurs the bias in the mean for the posterior marginals, in a spirit similar to \cite{sur2019modern}.  

The corresponding debiased version of the posterior mean---the one obtained by dividing the mean by $\alpha_{\rm Bayes}$---would exhibit an asymptotic MSE given by $\sigma_{\rm Bayes}^2/\alpha_{\rm Bayes}^2$. In the right panel of Figure \ref{fig:mse_graphs2}, we compare this MSE with the debiased MLE MSE given by $\sigma_{\rm MLE}^2/\alpha_{\rm MLE}^2$. Once again, the Bayes procedure outperforms the debiased MLE. That said, the debiased Bayes procedure performs worse than the original posterior mean as is evident on comparing the range of the x-axis values between the left and the right panels. This is unsurprising given the phenomenon of Stein shrinkage that we famously know in statistics. The unbiasedness in the debiased Bayes procedure costs sufficient raise in the variance. Hence we do not pursue this debiasing direction further.  %\textcolor{red}{Fill in takeaways from new plot; for now wrote a line and commented it out use it later if simulation outcomes are consistent with what I expect.}

Finally, that the effect of the prior survives even in the limit of large samples and dimensions, may be perceived as a blessing of dimensionality, since it implies we can design data-driven priors that help in high-dimensional learning and inference. 

%that the Bayes MSE is strictly lower. Thus, this regime uncovers interesting phenomena where the Bayes procedure beats the MLE strictly even where the MLE is well-defined. 

Next we turn to explaining the main distributional limit that arises in Theorem \ref{thm:marginals} further. Critically, the limit involves constants $\alpha,\sigma,v$ that were defined in \eqref{eq:alphasv} based on the systems introduced in \eqref{eq:FPE_all2} and \eqref{eq:FPE_all1}. The quantities that appear in these systems possess specific interpretations. In the remainder of this section, we provide alternate ways of expressing the constants appearing in \eqref{eq:FPE_all2} and \eqref{eq:FPE_all1} that yields additional insights into how these emerge in our results. 
To this end, define 
 for $m,m'\in \mathbb{N}$,
\begin{equation} \label{eq:overlaps}
    Q_{mm'} := \frac{1}{p} \bbeta^{(m)\top}\bbeta^{(m')}, \qquad Q_{m\star} := \frac{1}{p} \bbetas^\top \bbeta^{(m)},
\end{equation}
where $\bbeta^{(m)},\bbeta^{(m')}$ denote two independent samples from the posterior. Proposition \ref{prop:conv_order_param} below, that forms one of our main intermediate results, shows that the constants \eqref{eq:FPE_all1} can be interpreted as limits of these quantitires, in the sense that for any $m \neq m' \in \mathbb{N}$
\[Q_{m m} \rightarrow v_B, \quad Q_{m m'} \rightarrow c_B, \quad Q_{m\star} \rightarrow c_{BB_\star},\]
in an $L_2$ convergence sense. That is, $v_B,c_B, c_{BB_\star}$  correspond respectively to the (normalized) norm of a typical sample from the posterior, the inner product between two independent samples and  the inner product between a sample from the posterior and the true signal. Interestingly, the other set of constants $r_1,r_2,r_3$, critical for defining the mean and variance parameters $\alpha,\sigma,v$ in our main result, Theorem \ref{thm:marginals}, \eqref{eq:density_hj}, can be interpreted using versions of the score functions introduced in \eqref{eq:scalarscores}.
%omponents of the score  corresponding to a model which has the same likelihood as in \eqref{}

To see this, define $\bS^{(m)}, \bS_{\star}^{(m)} \in \mathbb{R}^n$ to be the vectors with $i$-th entry given by
\begin{equation}\label{eq:scores}
    S^{(m)}_i:= \tT_i(\bX_{i}^{\top}\bbeta_\star)  - A'(\bX_i^{\top}\bbeta^{(m)})  \,\,\mbox{ and } \,\, S_{\star,i}^{(m)} :=\tT'_i(\bX_i^{\top}\bbeta_\star)\bX_i^{\top}\bbeta^{(m)}.
\end{equation}
where $A'(\bX_i^{\top}\bbeta^{(m)}) = dA/dx |_{x = \bX_i^{\top}\bbeta^{(m)}}$ and $\tT'_i(\bX_{i}^{\top}\bbeta_\star) =\
d\tT_i/dx|_{x = \bX_i^{\top}\bbeta_\star}$. In other words $\bS^{m}, \bS_\star^{m}$ are vector versions of the scalar scores introduced in \eqref{eq:scalarscores}, with $\bX_i^\top\bbeta_\star$ and $\bX_{i}^{\top}\bbeta^{(m)}$ substituted for $\theta_\star$ and $\theta$ respectively. 
We then define, 
%\begin{equation} \label{eq:overlaps}
%    Q_{mm'} := \frac{1}{n} \bbeta^{(m)\top}\bbeta^{(m')}, \,\, Q_{m\star} := \frac{1}{n} \bbetas^\top \bbeta^{(m)},
%\end{equation}
\begin{equation}\label{eq:overlaps2}
    \tQ_{m m'}:=n^{-1} \bS^{(m)\top}\bS^{(m')},\,\, {M}_{mm'}:=n^{-1}\bS^{(m)\top}\bS_\star^{(m')} \,\, \bar{Q}_{mm'}:=n^{-1}\bS^{(m)\top}_\star \bS^{(m')}_\star.
\end{equation}
\begin{equation}\label{eq:converseSstar}
 \mbox{ and } \,\, a_{dp} :=\E_{G\otimes e}\thermal{A''(\theta)}_s,
\end{equation}
%\textcolor{red}{This $a_{dp}$ definition should be moved somewhere else, I realize that now, will do on my final manuscript pass. Just revisit this once, where $a_{dp}$ def should appear first.}
where recall that $\langle \cdot\rangle_s$ denotes expectations under the measure $p_s$ introduced in \eqref{eq:density_s}.
With these definitions, we will see below that the constants $r_1,r_2,r_3$ from \eqref{eq:FPE_all1} are intricately tied to inner products between the scores introduced above.  We will sometimes refer to the quantities introduced in \eqref{eq:overlaps} and \eqref{eq:overlaps2} as the \emph{order parameters}. 

\begin{proposition}\label{prop:conv_order_param}
  In the setting of Section \ref{sec:model}, for any $m \neq m' \in \mathbb{N}$, we have that
  %Assume Hypothesis \ref{hyp1}, \ref{hyp2}, and \ref{hyp3}; and let $\thermal{\cdot}_s$ be the expectation defined by \eqref{eq:density_s} with parameters given by $(v_B,c_B,c_{BB_\star})$. Then, for every $m\neq m'\geq1$,
    \begin{equation*}
         \lim_{n\to\infty} \E\thermal{A''(a_{m,1})}=a_{dp},
    \end{equation*}
    \begin{equation*}
        \lim_{n\to\infty} \E\thermal{(Q_{mm} - v_B)^2}, \,\, \lim_{n\to\infty} \E\thermal{(Q_{m m'} - c_B)^2}, \,\, \lim_{n\to\infty} \E\thermal{(Q_{m\star} - c_{BB_\star})^2} = 0,
    \end{equation*}
    \begin{equation*}
        \lim_{n\to\infty} \E\thermal{(\tilde Q_{m m} - \tv)^2}, \,\, \lim_{n\to\infty} \E\thermal{(\tilde Q_{m m'} - r_3)^2}, \,\, \mbox{ and } \lim_{n\to\infty} \E\thermal{(M_{mm} - M_{mm'} - r_2)^2} = 0;
    \end{equation*}
    where $\tv := a_{dp} + r_3 - r_1$. 
    %\textcolor{red}{Definitions of order parameters of dim p should be normalized by $p$ and same for $\gamma^2$; $r_2$ sign may be wrong above, and will need a fix if wrong.}
%    \textcolor{red}{CHECK IF CONSISTENT WITH CURRENT SYS OF EQNS FOR $r_1,r_2,r_3$. Also it does not match with what Appendix C.4 for instance says, according to that the $\tilde{m}$ definition is different --- this needs fixing. We should also have the $\bar{Q}_{mm'}$'s defined here, at one place instead of having them later. There is a point to discuss whether we need this proposition really here, or it should be somewhere else? WE just }\ms{Now the definitions of $r_1,r_2,r_3$ are consistent. I don't know if it is a good idea to define the $\bar{Q}_{mm'}$ as these nor their limits appear in the main results. It will be confusing, I think. But we can discuss it.}
\end{proposition}
%\textcolor{red}{I should write a description of prop 1 in proof outline, or maybe here itself. It is very easy to check heuristically.}
Thus, the mean and variance parameters in our main results $\eqref{eq:density_hj}$,  which are defined in terms of $r_1,r_2,r_3$ from \eqref{eq:FPE_all2}, essentially arise as limits of  \eqref{eq:overlaps2}.      Proposition \ref{prop:conv_order_param} forms a crucial first step for our proofs. Its proof is deferred to Section \ref{sec:proof_main_result}. 

\subsection{Applications}\label{sec:applications}
%\textcolor{red}{Plots we need total : (a) QQ plot for logistic with Beta (done) (b) To do: QQ plot Linear with Gaussian. (c) To do: QQ plot  Linear with Beta (d) MSE for logistic with Gaussian (done)-beautify (e) TODO: Compare logistic with Beta constants in the mean part of the Gaussian part of posterior with analogous $\alpha$'s and $\sigma$'s from PS equations.}
Before we delve into the main results of the paper, we provide specific GLM examples where our results apply. 
%some concrete generalized linear models that are widely used and that can be worked out with our results. 
We complement the logistic regression example below with numerical experiments showing the efficacy of our results.
%In the first two examples, we present numerical experiments to visualize some of the consequences of the corresponding asymptotics. 
%\ps{It will be easier for reader if this section is written using $y_i,\bX_i$ notation rather than $a_i$ notation. Also each example should start with the $y_i, \bX_i$ relation, write the likelihood given the model, and then conclude what the $\tilde{T}$ and $A$'s are supposed to be, Currently this is written a bit in the reverse order. Also I like $a_i$'s referring to $\bX_i^{\top}\bbeta^{(m)}$'s and not both this and $\bX_i^{\top}\bm{b}$. For the latter, maybe we can use full hand instead of the $a_i$ short form? }
%\ms{I will update this in my next version of the section.}
%\ps{I will edit this bit more in my next pass, content at the moment seems okay-ish, we should remove multinomial example, and add a picture for probit and multinomial. The QQ plots and density plots are both very nice.}

\paragraph{Example 1: Linear regression---prototype for checking validity} Perhaps the simplest example of a canonical GLM is the classical linear model. Although any prior that satisfies the conditions in Hypothesis \ref{hyp1} works for our theory, we  mainly discuss the case of the Gaussian linear model since here closed-form expressions allow us to provide independent validation of the accuracy of our results. Thus we take this example as a means to cross-check our theory, and present non-trivial examples later in this section. 

For this example, our observed data arises from the linear model
\begin{equation*}
    y_i = \bX_i^\top\bbetas+z_i
\end{equation*}
where $z_1,\dots,z_n$ are i.i.d. standard Gaussian random variables (we take the variance of these to be $1$ for simplicity). The log-likelihood is given by 
\begin{equation*}
    \logl_{n,p}(\bB) := -\frac{1}{2}\sum_{i\in[n]}\left(y_i-\bX_i^\top\bB\right)^2
\end{equation*}
so that matching with our formulation in \eqref{eq:glmhamilton} we have that $\tT_i(x)=f(x,e_i), \,\, f(x,e_i)= x + \Phi^{-1}(e_i)$, where $\Phi(\cdot)$ is the cdf of a standard Gaussian and $A(x) = x^2/2$. 
%\ps{There is a small repeat from the previous section (to be fixed).}

It is easy to check that Hypothesis \ref{hyp1} holds for this model. With a bit 
of work it can be proved, as a consequence of \cite[Lemma 5.1]{barbier2021performance} with $u(x) = -\frac{x^2}{2}$, that Hypothesis \ref{hyp2} also holds. We now show how our fixed point equations look for this model. In particular, given these explicit forms for $\tT_i$ and $A$, \eqref{eq:FPE_all2} simplifies to yield: 

\begin{equation}\label{eq:FPE_all2_linear}
    \begin{cases}        
        r_1 = 1 + \E_{G \otimes e}[\thermal{\theta(\xi)}^2_s] - \E_{G \otimes e}\thermal{\theta^2(\xi)}_s \\
        r_2 = r_1 - 1 \\
        r_3 = \E_{G \otimes e}\thermal{\theta(\xi)-\theta_\star-\Phi^{-1}(e)}^2_s,
    \end{cases}
\end{equation}
where $\theta(\xi),\theta_\star$ are given by \eqref{eq:thetasearlier} with $\xi$ used for $\xi_B$ to shorten the notation.  
%\textcolor{red}{I agree with the definition of $r_3$ in \eqref{eq:matchingr1to3}, which means the third equation above should be $r_3 = \E_{G \otimes e}\thermal{\theta(\xi)-\theta_\star-\Phi^{-1}(e)}^2_s.$}\ms{I agree and have updated it.}

For a linear model, \eqref{eq:FPE_all2_linear} can be further simplified by  integrating $\E_{G \otimes e}[\thermal{\theta^2(\xi)}_s]$ and $\E_{G \otimes e}\thermal{\theta(\xi)}^2_s$ to yield closed form formulae in terms of $v_B$, $c_B$, and $c_{BB_\star}$. In contrast,  \eqref{eq:FPE_all1} cannot in general be integrated due to the presence of the prior. However, in the special case of the Gaussian linear model, these can also be integrated reducing \eqref{eq:FPE_all2},\eqref{eq:FPE_all1} to an explicit system of equations. This is summarized in the next proposition whose proof is delayed to Appendix \ref{app:linear_regression}.

\begin{proposition}\label{prop:lin_reg_eqs}
For linear regression,  \eqref{eq:FPE_all2_linear} simplifies to
    \begin{equation*}
        \begin{dcases}
            r_1 = \frac{1}{\kappa(v_B-c_B)+1} \\
            r_2 = - \frac{\kappa(v_B-c_B)}{\kappa(v_B-c_B)+1} \\
            r_3 = \frac{\kappa(\gamma^2+c_B -2 c_{BB_\star})+1}{(\kappa(v_B-c_B)+1)^2}.
        \end{dcases}
    \end{equation*}
    Furthermore, if $\mu(\cdot)$ is standard Gaussian and $\pi(\cdot)$ is a centered distribution, we have
    \begin{equation*}
        \begin{dcases}
            v_B = \frac{1}{r_1 + 1} + c_B \\
            c_B = \frac{(r_2+1)^2 \gamma^2 + r_3}{(r_1 + 1)^2} \\
            c_{BB_\star} = \frac{(r_2+1) \gamma^2}{r_1 + 1},
        \end{dcases}
    \end{equation*}
    which implies that
    \begin{equation}\label{eq:r1r2}
        r_1 = r_2+1 = \sqrt{\left(\frac{\kappa}{2}\right)^2 + 1} - \frac{\kappa}{2}.
    \end{equation}
\end{proposition}

%\ps{Write a couple of lines explaining direct computations in appendix}
The Gaussian linear model is particularly instructive since the posterior can be explicitly analyzed in this setting. Specifically, the posterior takes the form \eqref{eq:linmod}. In Appendix \ref{app:direct_integration_op} we analyze this posterior directly using random matrix theory tools. We use this to check the validity of  our results presented in the preceding subsection. Specifically, we check that the posterior limit is indeed in the form given by Theorem  \ref{thm:marginals} and that the expressions for $r_1,r_2$ that our theory predicts, i.e. \eqref{eq:r1r2}, match with those obtained from direct random matrix theory computations.

\paragraph{Example 2: Binary Outcome GLMs} We now consider GLMs with binary outcomes defined as follows
\begin{equation*}
y_i = \mathbb{I}_{\{\sigma(\bX_i^\top\bbetas) \geq e_i \}},
\end{equation*}
with $e_i$ an independent $\mbox{Unif}[0,1]$ random variable.

When $\sigma(x)=e^x/(1+e^x)$, the sigmoid function, this yields logistic regression, whereas for $\sigma(x) = \Phi(x)$, the cdf of a standard normal random variable, this leads to probit regression. %\textcolor{red}{Probit regression is as popular as logistic regression in statistics. Should we write a comment somewhere in this section that our results can also cover that by suitable approximation strategies? What extra do we need to check for this probit case beyond what you already do for the logistic case?}\ms{Proving the results hold is probably very little work. Simulations are another story.}
The log-likelihood becomes
\begin{equation*}
    \logl_{n,p}(\bB) = \sum_{i\in[n]} \{y_i \log\frac{ \sigma(\bX_i^\top\bB)}{1-\sigma(\bX_i^\top\bB)} + \log(1-\sigma(\bX_i^\top\bB))\}.
\end{equation*}
We show the applicability of our theory for logistic regression where the likelihood simplifies to
\begin{equation*}
    \logl_{n,p}(\bB) =  \sum_{i\in[n]} \{y_i \bX_i^\top\bB - \log(1+e^{\bX_i^\top\bB})\},
\end{equation*}
although other binary outcome GLMs would also fall in our framework.
Note that
\begin{equation*}
\tT_i(x) = \mathbb{I}_{\{x \geq \sigma^{-1}(e_i)\}} \,\, \mbox{ and } \,\, A(x) = \log(1+e^{x}).
\end{equation*}

%Although the model functions lack the smoothness required by Hypothesis \ref{hyp1}, i.e. 
Importantly, $\tT$ is not differentiable, and therefore does not satisfy Hypothesis \ref{hyp1}(ii). However, it can be arbitrarily approximated by smooth functions. To this end, consider a smooth approximation of $\mathbb{I}_{\{x\geq0\}}$ given by $f_\delta(x):=(\tanh(x/\delta)+1)/2$, where $\delta>0$ is small.\footnote{In simulations, we consider $\delta=1/1000$.} Define 
\begin{equation*}
    \tT_{\delta,i}(x) := f_\delta(x-\sigma^{-1}(e_i)).
\end{equation*}
%Notice that $\sigma^{-1}(e_i)$ has a standard logistic distribution.
When $\delta\approx0$, we have $\tT_{\delta,i}\approx\tT_i$ (see  Appendix \ref{app:smooth_app} for suitable formalization).
%we formally prove closeness of the two models in terms of statistics of interest to us. 
%that the posterior marginals under the smoothed model and the original logistic model become arbitrarily close as $\delta \to 0$.
%prove, in a precise way, that our main results applied to the smooth model extrapolate to logistic regression when $\delta\to0$. } \textcolor{red}{The approximation is written differently than the first paragraph of Appendix I. Should we write it here to match that first paragraph?}\ms{Yes, they should be made to match.}\ms{I changed the name of the smooth function, the position of $\sigma$ in its definition and removed the $\lambda$ parameter. I removed the "As explained in Section 3.2" of the appendix. Because there was a loop of references. I just kept the "In Appendix \ref{app:smooth_app}, we prove..." of this section.} 
The log-likelihood in the smoothed model is given by
\begin{equation*}
    \logl_{n,p,\delta}(\bB) :=  \sum_{i\in[n]} \tT_{\delta,i}(\bX_i^\top\bbetas) \bX_i^\top\bB - \log(1+e^{\bX_i^\top\bB}).
\end{equation*}
It is easy to check that Hypothesis \ref{hyp1} holds for the smoothed model. By Proposition \ref{prop:control_mom_conv} (Appendix \ref{app:conv_mon_control}), Hypothesis \ref{hyp2} also holds. Thus our main results hold for the smoothed model.
%is means that we are within the range of models covered by our main results. Moreover, by Lemma \ref{lem:ord_param_smooth} of Appendix \ref{app:smooth_app}, the order parameters for the approximate model with approximation parameter $\delta >0$, approach the corresponding asymptotic order parameters for logistic regression when $\delta \to 0$. As in 
In Proposition \ref{prop:marginal_smooth}, we establish that the posterior marginals under the smoothed model approach those under  logistic regression as $\delta\to 0$. Thus, the posterior marginals in logistic regression can be studied using our results through this approximation strategy. Similar approximations can be invoked for other binary outcome GLMs.  Any log-concave prior would fall in our framework. As a use-case example, we consider a ${\rm Beta}(\alpha,\beta)$ prior with $\alpha > 1, \beta>1$. 

Figure \ref{fig:marginals}  compares the theoretical posterior marginal distribution for a fixed coordinate, as predicted by Theorem \ref{thm:marginals}, with its empirical counterpart obtained from MCMC simulations. In the left panel, the QQ-plot shows a close alignment between the theoretical and empirical quantiles, indicating good agreement between the two distributions. The right panel further illustrates this correspondence by juxtaposing the marginal density obtained from our theory with the histogram of MCMC samples. A notable feature is that the theoretical curve closely tracks the empirical distribution across the full support, including in the vicinity of the true parameter value, marked by the red dashed line. 
The figure illustrates the applicability of our results in  finite samples and dimensions.

\begin{figure}[h!]
  \centering
  \includegraphics[width=0.48\textwidth]{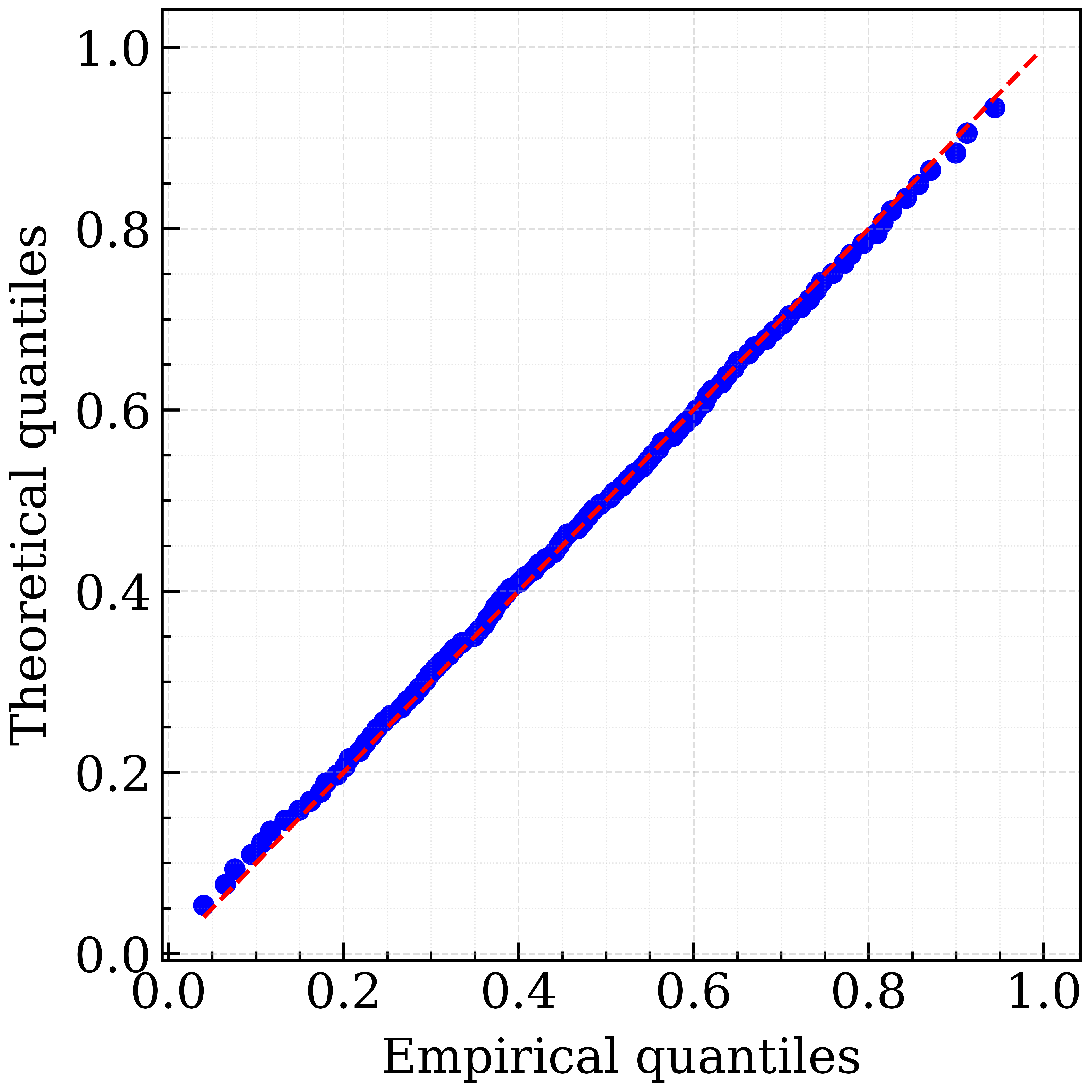}\hfill
  \includegraphics[width=0.48\textwidth]{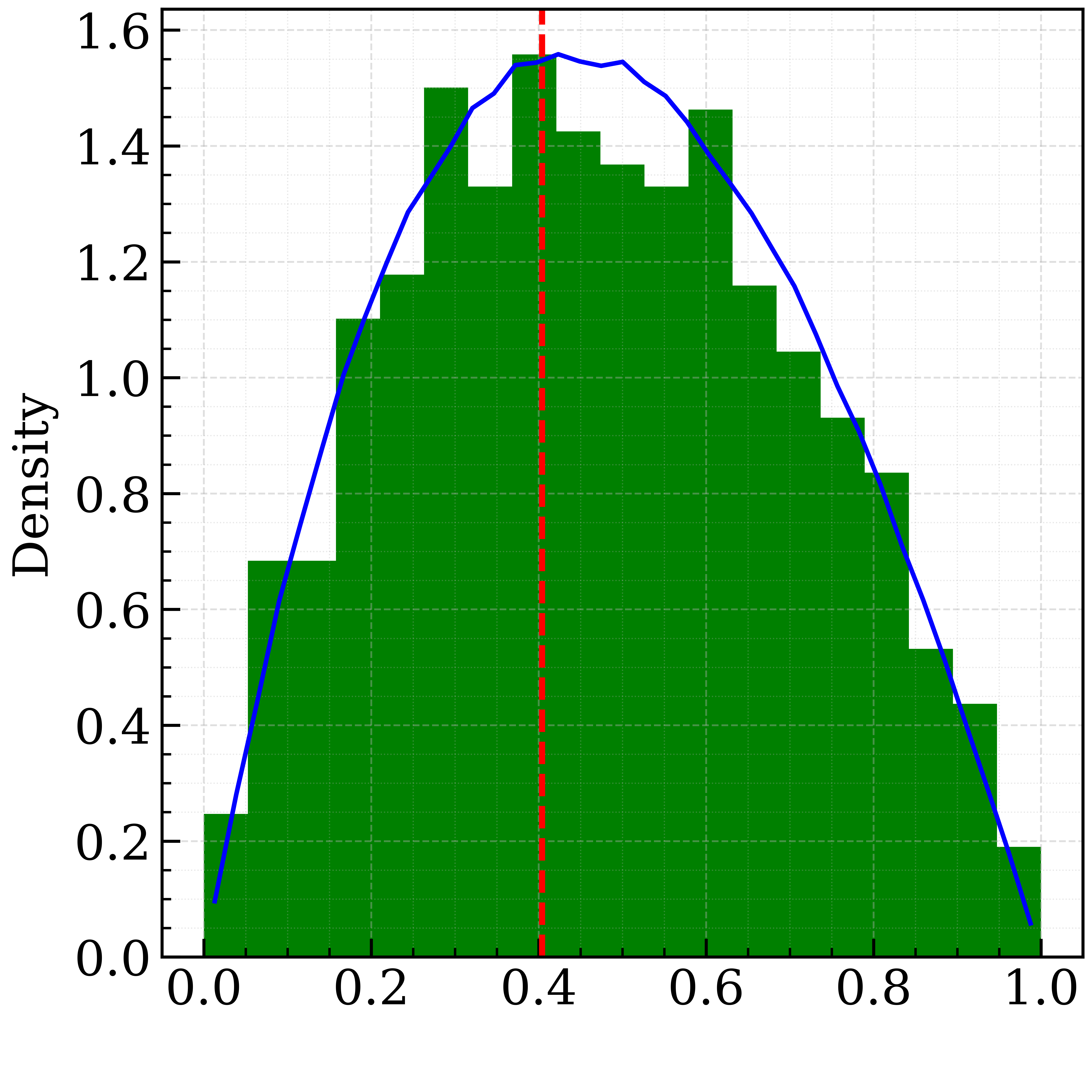}
  \caption{Comparison of posterior marginals as predicted by Theorem \ref{thm:marginals} and MCMC simulations. The setting is that of Figure \ref{fig:mse_graphs} and $\kappa=1$. For the MCMC simulations, $n=1000$, the number of chains was taken to be $4$, and $1000$ draws were produced with a tune of $2000$. The left panel shows QQ-plots for theoretical quantiles, obtained using Theorem \eqref{thm:marginals} versus empirical quantiles obtained from MCMC. The right panel shows the predicted marginal density from our theory versus the MCMC histogram; the true signal value $\beta_{\star,j_0}$ is marked using a red dashed line.}
  \label{fig:marginals}
\end{figure}
%\textcolor{red}{Honestly, Figure 3 looks legit; like MCMC and our theory matches here. I would have felt much more comfortable if figure 1b and Figure 2 were in the same setting. There were some numerical reasons why this is not the case although I dont remember why. Remind me please?}\ms{Just that the hardware I was using was old and not capable of running a lot of Beta-Beta simulations.}

%\ms{Here I commented out the Logistic regression with Gaussian prior section.}
%\paragraph{Example 3: Logistic regression with Gaussian prior.} Here we consider the same setting as above but for standard Gaussian signal and prior. In the absence of mismatch, the simplifications discussed in \ref{app:Bayes_optimal} hold. This simplifies the numerical complexity of solving equations \eqref{eq:FPE_all1}--\eqref{eq:FPE_all2}. 

%Because the prior is centered, in this case, if either $\lambda = 0$ or $1/\kappa \approx 0$ the performance of the posterior mean should be trivial. That is, the expected mean square error should be close to $1$. In Figures \ref{fig:logistic_normal_normal_kappa} and \ref{fig:logistic_normal_normal_snr} we present simulations for the mean square error of the posterior mean as a function of $1/\kappa$ and $\lambda$, respectively. As we can appreciate, the error approaches $1$ and the origin. Moreover, as expected, the behavior in both cases is monotonically decreasing as a function of the free variable.

%\ps{I would really get rid of logistic with Gaussian and instead give a couple of simulations for binomial with beta prior or probit with beta.}

\paragraph{Example 3: Binomial regression} Binomial regression with log-concave priors is naturally covered by our results by extending the calculations from logistic regression. Specfically, in binomial regression, for some fixed $m \in \mathbb{N}$, $y_i \sim \mathrm{Bin}(m,p(\bX_i^{\top}\bbeta_\star))$; where $p(\bX_i^{\top}\bbeta_\star) = \sigma(\bX_i^{\top}\bbeta_\star)$ with $\sigma(\cdot)$ the sigmoid function. Thus, the log-likelihood is
\begin{equation*}
    \logl_{n,p}(\bB) = \sum_{i\in[n]} \{y_i \bX_i^\top\bB - m \log(1+e^{\bX_i^\top\bB})\}.
\end{equation*}
and the approximations from the prior example can be naturally generalized. 
\section{Proof heuristics}\label{sec:pfoutline}

%\ms{This is not so much a proof outline but the heuristics of the proof. Should we change the name of the section?}
%\ms{This section had been heavily rewritten since oyur first pass.}

Our proof strategy for characterizing the posterior marginals relies on relating the posterior measure to the leave-a-variable-out measure, \eqref{eq:mean_variable}, through an appropriate Radon-Nikodym derivative and then approximating this derivative in the proportional high-dimensional limit. We now explain this in further detail. 

To study a posterior marginal, say the $j_0$-th marginal, we proceed as follows. Recall that for a vector $\bm{v} \in \mathbb{R}^p$, we denote $\tilde{\bm{v}} \in \mathbb{R}^{p-1}$ and $ v_{j_0}$ to be the vector obtained on removing the $j_0$-th coordinate, and the $j_0$-th coordinate, respectively. Then, by a Taylor expansion, the log-likelihood at $\bm{b}$ admits the following decomposition:
\begin{equation}
    \logl_{n,p}(\bB) = \logl_{n,p-1}(\tbB) + \Delta\logl(b_{j_0},\tbB) + E_n(\bB), \qquad \text{where} \label{loovdecomposition}
\end{equation}
\begin{equation}\label{eq:DeltaL}
  \Delta \logl(b_{j_0},\tbB) :=  Y_1(\tbB,\tbbetas) b_{j_0} + Y_2(\tbB,\tbbetas) \beta_{\star,{j_0}} - Y_3(\tbB) b^2_{j_0}  + Y_4(\tbbetas) \beta_{\star,{j_0}}b_{j_0}  + Y_5(\tbB,\tbbetas) \beta_{\star,{j_0}}^2
\end{equation}
and the $Y_i$'s are given by 
\begin{equation}\label{eq:Y1thro5}  
    \begin{dcases}
        Y_1(\tbB,\tbbetas) := \sum_{i \in[n]}X_{i{j_0}}\left( \tT_i(\tXibs)  - A'(\tbX_i^{\top}\tbB)\right)\\
        Y_2(\tbB,\tbbetas) := \sum_{i \in[n]}X_{i{j_0}}\tTp_i(\tXibs)\tbX_i^{\top}\tbB\\
        Y_3(\tbB) := \frac{1}{2}\sum_{i \in[n]} X_{i{j_0}}^2A''(\tbX_i^{\top} \tbB)\\
        Y_4(\tbbetas) := \sum_{i \in[n]} X_{i{j_0}}^2  \tTp_i(\tXibs)\\
        Y_5(\tbB,\tbbetas) := \frac{1}{2}\sum_{i \in[n]}  X_{i{j_0}}^2 \tTdp_i(\tbX_{i}^{\top}\tbbeta_{\star})\tbX_i^{\top}\tbB 
    \end{dcases}.   
\end{equation}
The arguments in the $Y_i$'s are spelled out to showcase their dependence on $\tbB,\tbbeta_\star$ and that these do not depend on $\beta_{\star,j_0}$ or $b_{j_0}$. In the sequel, whenever it is clear from context, we will suppress the arguments of the $Y_i$'s. 
Further, $E_n(\bB)$ equals the error term arising from the Taylor expansion, defined in \eqref{eq:taylor_error}. Since it involves sums of terms proportional to $X_{ij_0}^3$, $E_n(\bB)$ is of order $\mathcal{O}(n^{-1/2})$. Above, the notation $\Delta \cL(b_{j_0},\tbB)$ is used to emphasize that in this term, $b_{j_0}$ and $\tbB$ are ``disentangled'' in that $b_{j_0}$ appears only as coefficients to the $Y_i$'s and $\tbB$ appears only in the $Y_i$'s.

Recall that $\thermal{\cdot}_v$ stands for expectation with respect to the leave-a-variable-out posterior, defined in \eqref{eq:mean_variable}. By \eqref{loovdecomposition}, we have that, for $f:\R\mapsto\R$, 
\begin{equation}\label{eq:fullnlooo}
    \begin{split}
        \thermal{f(\beta_{j_0})} & = \frac{\int f(\beta_{j_0}) \exp\{\logl_{n,p}(\bbeta)\}\prod_{j=1}^p \mu(d\beta_j)}{\int \exp\{\logl_{n,p}(\bbeta)\}\prod_{j=1}^p \mu(d\beta_j)} \\
        & = \frac{\int f(\beta_{j_0}) \exp\{\logl_{n,p-1}(\tbbeta) + \dlogl(\beta_{j_0},\tbbeta) + E_n(\bbeta)\}\prod_{j=1}^p \mu(d\beta_j)}{\int \exp\{\logl_{n,p-1}(\tbbeta) + \dlogl(\beta_{j_0},\tbbeta) + E_n(\bbeta)\}\prod_{j=1}^p \mu(d\beta_j)}\\
        & = \frac{\frac{\int f(\beta_{j_0}) \exp\{\logl_{n,p-1}(\tbbeta) + \dlogl(\beta_{j_0},\tbbeta) + E_n(\bbeta)\}\prod_{j=1}^p \mu(d\beta_j)}{\int \exp\{\logl_{n,p-1}(\beta_{j_0},\tbbeta)\}\prod_{j=1}^p\mu(d\beta_j)}}{\frac{\int \exp\{\logl_{n,p-1}(\tbbeta) + \dlogl(\beta_{j_0},\tbbeta) + E_n(\bbeta)\}\prod_{j=1}^p \mu(d\beta_j)}{\int \exp\{\logl_{n,p-1}(\beta_{j_0},\tbbeta)\}\prod_{j=1}^p\mu(d\beta_j)}}\\
        & = \frac{\thermal{f(\beta_{j_0})\exp\{\dlogl(\beta_{j_0},\tbbeta) + E_n(\bbeta)\}}_v}{\thermal{\exp\{\dlogl(\beta_{j_0},\tbbeta) + E_n(\bbeta)\}}_v} = \thermal{f(\beta_{j_0})\cD_n}_v,
    \end{split}
\end{equation}
where we defined\footnote{Our notational convention is that dummy variables used for arguments of log-likelihoods and any term without integrals are represented by $\bB$ and its variants, whereas variables over which integration is performed are represented using $\bbeta$ and its variants.}
\begin{equation}\label{eq:def_deriv_variable}
    \cD_n := \frac{\exp\{\dlogl(b_{j_0},\tbB) + E_n(\bB)\}}{\thermal{\exp\{\dlogl(\beta_{j_0},\tbbeta) + E_n(\bbeta)\}}_v}.
\end{equation}
Note $\cD_n$ is  the Radon-Nikodym derivative between the full posterior and the leave-a-variable-out posterior, corresponding to \eqref{eq:posterior_expectation} and \eqref{eq:mean_variable} respectively. This calculation is analogous to how local asymptotic distributions are calculated in classical asymptotic statistics \cite{van2000asymptotic}.

The proof then involves establishing that, in the setting of Section \ref{sec:model} and under $\thermal{\cdot}_v$, $\cD_n$ converges in the proportional high-dimensional limit to the random Gaussian tilt \eqref{eq:density_hj}, which is a function of the $j_0$-th coordinate of the true signal $\beta_{\star,j_0}$ and additional gaussian noise $Z_{j_0}$. Note that, as a function of $b_{j_0}$, $\Delta \cL(b_{j_0},\tbB)$ contains only linear and quadratic terms. Completion of squares naturally yields a Gaussian density in $b_{j_0}$. To see that this density takes the specific form outlined in \eqref{eq:density_hj}, in the limit of large sample and dimensions, requires analyzing the limiting behavior of the $Y_i$'s under $\thermal{\cdot}_v$. These limits are established in Lemma \ref{lem:conv_dist_Z}. Plugging in the form of these limits, we obtain that 
\begin{equation}\label{loocentre}
\cD_n \approx \frac{\exp\{-\frac{1}{2v}(b_{j_0}-m_{j_0})^2\}}{\int \exp\{-\frac{1}{2v}(\beta_{j_0}-m_{j_0})^2 \mu(d\beta_{j_0})},   
\end{equation}
where $m_{j_0}$ is defined as in \eqref{eq:density_hj}. 
Thus, essentially, as in classical Le Cam theory, the Radon Nikodym derivative $\cD_n$ is also asymptotically Gaussian in our setting. In some sense, our leave-one-out technology allows to build an analogue of classical Le Cam type theory in high dimensions.

This argument can be generalized to studying a single coordinate simultaneously of $l$ independent samples from the posterior. In fact, in its most general form, we obtain Proposition \ref{prop:fix_point} that  formalizes our aforementioned argument for a multivariate function applied to a pre-fixed coordinate of the true signal and $l$-independent samples from the posterior. 
\begin{proposition}\label{prop:fix_point}
In the setting of Section \ref{sec:model},
 let $f:\R^{l+1} \to \R$ be continuous and polynomially bounded. For any $j\in \mathbb{N}$, we have that
    \begin{equation*}        \lim_{n\to\infty}\E\thermal{f(\beta_{\star,j},\beta_{j}^{(1)},\dots,\beta_{j}^{(l)})} = \E_{Z}\thermal{f(\beta_{\star,j},\beta_{j}^{(1)},\dots,\beta_{j}^{(l)})}_{h,j},
    \end{equation*}
 where on the left, $\beta_{j}^{(1)},\beta_{j}^{(2)},\hdots, \beta_{j}^{(l)}$ refer to the $j$-th coordinates of $l$ independent samples from the posterior and on the right $\thermal{\cdot}_{h,j}$ refers to expectation over $\beta^{(1)}_j,\hdots,\beta^{(l)}_j$ drawn from  $l$ independent copies of $p_{h,j}$ defined in  \eqref{eq:density_hj}.  Furthermore, $Z=(Z_1,\hdots,Z_l) \sim \mathcal{N}(0,I_l)$ refers to the Gaussians arising in the respective means for these copies. 
%density $\prod_{m=1}^l p(b_m)$ with $p(b_m)$ proportional to 
  %  with $\E_{Z}(\cdot)$ and $\thermal{\cdot}_{h,j}$ as in Section \ref{sec:results}. \textcolor{red}{Check exactly where $\E_Z$ is defined.}\textcolor{red}{Needs changing back to what it was before, will get to this after Appendix B.}
\end{proposition}

The proof is deferred to Section \ref{sec:proof34}. Proposition \ref{prop:fix_point} is a key result since it directly yields a characterization of  moments of all orders for marginals corresponding to a posterior sample. From here, establishing our main result on the limiting behavior of a posterior marginal is  straightforward, and described in Section \eqref{sec:proof_marginals}. The necessitaty of establishing Proposition \eqref{prop:fix_point} in the general version we presented above, instead of simply for $l=1$, becomes clear on examining our proofs. 

%\textcolor{red}{I wrote the following for r coordinates. Quick question: to do the thing for r coordinates rigorously, would you actually need what I wrote below, or can PRoposition 3 somehow be used smartly to get there? Nothing about the $r>1$ coordinate case was written anywhere in this proof heuristics section, but I am wondering what I wrote is an overkill or the right way to rigorize this. We dont have to show the rigorization for $r > 1$, but I wrote it once to have everything written out.}\ms{As it is right now, I do think what you wrote is needed. It might be possible, I think, to prove some kind of asymptotic independence through concentrations of functions of the form $p^{-1}\sum_j f(\beta^{(1)}_j)f(\beta^{(2)}_j)$ which should follow as for the order parameters. But, at this point, I think your argument should be there.}

To extend to multiple but finitely many coordinates, e.g, to $r$ coordinates as in the statement of Theorem \ref{thm:marginals}, we adopt a leave-r-variables-out strategy. To this end, the starting point would be a log-likelihood decomposition of the form
\begin{equation*}
    \logl_{n,p}(\bB) = \logl_{n,p-r}(\bB_{-[j_1,\hdots,j_r]}) + \Delta\logl(b_{j_1},b_{j_2},\hdots,b_{j_r},\bB_{-[j_1,\hdots,j_r]}) + E_n(\bB_{-[j_1,\hdots,j_r]}), \label{loovdecomposition}
\end{equation*}
where $\bB_{-[j_1,\hdots,j_r]}$ refers to all but the $j_1,j_2,\hdots,j_r$-th coordinates of $\bB$ and 
 $\Delta\logl$ is given by a suitable analogue of \eqref{eq:DeltaL} and \eqref{eq:Y1thro5} that isolates the effect of the $r$ coordinates from the rest.  
As long as $r$ is finite, the arguments for our analyses remain essentially the same as $r=1$.

In the aforementioned proof description, the main detail that was skipped pertained to how the limits of the $Y_i$'s are derived under $\thermal{\cdot}_v$. We describe this step next. 
%The step that remains is to find the values of the constants that appear in the limits of $Y_1,\dots,Y_5$. 
Note that under $\thermal{\cdot}_v$, the $Y_i$'s are functions of $ \tilde{\bX_i}^{\top}\tilde{\bbeta}_\star, \tilde{\bX_i}^{\top}\tilde{\bbeta}$, where $\tilde{\bbeta}$ refers to all but the $j_0$-th coordinate of a sample from the  leave-a-variable-out posterior introduced in \ref{eq:mean_variable}. Thus, understanding the joint distribution of these under $\thermal{\cdot}_v$ should allow us to characterize the limiting behavior of the $Y_i$'s. Our main observation is that for any function $g$, asymptotically, 
\begin{equation}\label{eq:loorelation}
\thermal{g(\tbX_i^{\top}\tbbeta)}_v \approx \thermal{g(\bX_i^{\top}\bbeta)}.
\end{equation}
That is, in the limit of large sample and dimensions, functions of $\tbX_i^{\top}\tbbeta$, where $\tbbeta$ is all but the $j_0$-th coordinate of a sample from the  leave-a-variable-out measure, behave similarly to functions of $\bX_i^{\top}\bbeta$, where $\bbeta$ is a sample from the full posterior. To understand the behavior of the $Y_i$'s, we will therefore study the RHS in \eqref{eq:loorelation}.
%Thus, \textcolor{red}{TILL HERE ON FINAL PASS.} 
%Thus, as mentioned above, the constants to be determined are then asymptotic values of the means of functions of the fitted values $\tbX_1^\top\tbbeta,\dots,\tbX_n^\top\tbbeta$ under the leave-a-variable-out posterior measure. Now recall from \eqref{eq:a_is} that $a_{\star,i}:=\bX_i^\top\bbeta_\star$ and $a_{l,i}:=\bX_i^\top\bbeta^{(l)}$, where for $l\geq1$, $\bbeta^{(1)},\bbeta^{(2)},\dots,\bbeta^{(l)}$ denote independent samples from the posterior. Thus, $a_{i,\star},a_{i,l}$ are the analogous of $\tilde{\bX_i}^{\top}\tilde{\bbeta}_\star, \tilde{\bX_i}^{\top}\tilde{\bbeta}$ calculated from $n$ samples, as opposed to $n-1$ samples. But as $n \rightarrow \infty$, asymptotically these  should admit the same characterization. Therefore, to establish these quantities, it will suffice to find the limiting values of the order parameters \eqref{eq:overlaps2} under the full posterior.  For doing this, we use a leave-an-observation-out argument similar to the leave-a-variable-out above. We first note that, for all $g:\R\mapsto\R$,
Next observe that
\begin{equation}\label{eq:loo2}
    \thermal{g(\bX_i^\top\bbeta)} = \frac{\thermal{g(\bX_i^\top\bbeta)\exp\{u_i(\bX_i^\top\bbeta,\bX_i^\top\bbetas)\}}_o}{\thermal{\exp\{u_i(\bX_i^\top\bbeta,\bX_i^\top\bbetas)\}}_o} = \thermal{g(\bX_i^\top\bbeta)\bar{\cD}_n}_o,
\end{equation}
where we define
\begin{equation}\label{eq:def_deriv_variable2}
    \bar{\cD}_n := \frac{\exp\{u_i(\bX_i^\top\bB,\bX_n^\top\bbetas)\}}{\thermal{\exp\{u_i(\bX_i^\top\bbeta,\bX_n^\top\bbetas)\}}_o}.
\end{equation}
Note  $\bar{\cD}_n$ is the Radon-Nikodym derivative between the full posterior and the leave-an-observation-out posterior defined in \eqref{eq:mean_observation}.
%measures $\proba(\cdot)$ and $\proba_o(\cdot)$ induced by \eqref{eq:posterior_expectation} and \eqref{eq:mean_variable}, respectively. 
Again, the remaining argument involves computing the limit of $\bar{\cD}_n$ under $\thermal{\cdot}_o$. In fact, we show in Subsection \ref{subsec:prop2}, c.f., \eqref{eq:torefearlier} that asymptotically
%Recall definitions \eqref{eq:thetathetast} of $\theta$ and $\theta_\star$. \textcolor{red}{TILL HERE ON FINAL PASS.} By Proposition \ref{prop:RS} and Lemma \ref{lem:normal_approx_dist} it follows that
\begin{equation}\label{looosimplify}
   \thermal{g(\bX_i^\top\bbeta) \bar{\cD}_n}_o \approx \frac{\int g(\theta) \exp\{u(\theta,\theta_\star)\}\phi(\xi_B)d\xi_B}{\int \exp\{u(\theta,\theta_\star)\} \phi(\xi_B)d\xi_B},
\end{equation}
where $\theta,\theta_\star$ are as  defined in \eqref{eq:thetathetast},  $u(\theta,\theta_\star)=\tT(\theta_\star)\theta-A(\theta)$, and $\phi(\cdot)$ denotes the standard Gaussian density. Subsequently, using the definitions of $\theta$ and $\theta_\star$, the RHS of \eqref{looosimplify} can be expressed in terms of the measure $p_s(\cdot)$ introduced in \eqref{eq:density_s}. 

Proposition \ref{prop:hat_fix_point} states this formally for the more general case where we consider a multivariate function of $a_{\star,i},a_{1,i},\hdots,a_{l,i}$ where recall from  \eqref{eq:a_is} that $a_{\star,i}=\bX_i^{\top}\bbeta$ and $a_{m,i}=\bX_i^{\top}\bbeta^{(m)}$ with $m \in [l]$ and $\bbeta^{(1)},\hdots,\bbeta^{(l)}$ independent samples from the full posterior. 
%\bX_i^{\top}\bbeta_{\star},\bX_i^{\top}\bbeta^{(1)}, \hdots, \bX_i^{\top}\bbeta^{(l)}$, where $\bbeta^{(1)},\hdots,\bbeta^{(l)}$ denote $l$ independent samples from the posterior.
\begin{proposition}\label{prop:hat_fix_point}
   In the setting of Section \ref{sec:model},
%{\color{blue}that
%    \begin{equation*}
 %       \lim_{n\to\infty} \E\thermal{Q_{11}} = v_B, \,\, \lim_{n\to\infty} \E\thermal{Q_{12}} = c_B, \,\, \lim_{n\to\infty} \E\thermal{Q_{1\star}} = c_{BB_\star},
  %  \end{equation*}
    %for some constants $v$, $m$, and $q$,} 
 let $g:\R^{l+1} \to \R$ be  continuous and polynomially bounded. Then for any $i \in [n]$
    \begin{equation}\label{eq:prop4}
    \lim_{n\to\infty}\E\langle g(a_{\star,i},a_{1,i},\hdots,a_{l,i})\rangle = \E_{G \otimes e}\thermal{g(\theta_\star,\theta_1,\hdots, \theta_l)}_s,
    \end{equation}
   where, for $m\in[\ell]$, when $c_B > 0$,
\begin{equation}\label{eq:thetas}
    \theta_m := \sqrt{\kappa(v_{B}-c_{B})}\xi_{B}^{(m)}+\sqrt{\kappa c_{B}}z_{BB_{\star}}  
    \mbox{ and }  \theta_\star := \sqrt{\kappa\left(\gamma^2 - \frac{c_{BB_{\star}}^2}{c_B}\right)}\xi_{B_{\star}} +c_{BB_\star}\sqrt{\frac{\kappa}{c_B}}z_{BB_{\star}},
\end{equation}
else $  \theta_m = \sqrt{\kappa v_B} \xi_B^{(m)}$ and $ \theta_\star := \sqrt{\kappa \gamma^2} \xi_{B_\star}$.
Here, $\xi_{B_\star},z_{BB_\star}$ are i.i.d.~standard Gaussian variables, independent of everything else and $G=(\xi_{B_\star},z_{BB_{\star}})$. Furthermore, in \eqref{eq:prop4},  $\thermal{\cdot}_s$ denotes expectation over  $\xi_B^{(1)},\hdots, \xi_B^{(l)}$ that have density proportional to 
$$\prod_{m=1}^l \exp\left\{\tT(\theta_\star)\theta(\xi_B^{(m)}) - A(\theta(\xi_B^{(m)})) - \frac{(\xi_B^{(m)})^2}{2}\right\},
  \quad \text{where} \quad \tT(\theta_\star) = T \circ f(\theta_\star,e),$$ and $e \sim \text{Unif}(0,1)$ independent of everything else.
\end{proposition}
The proof is deferred to Section \ref{sec:proof34}. 
Propositions \ref{prop:fix_point} and \ref{prop:hat_fix_point} can be thought of as contrasts to each other. The key idea underlying the former is to relate the full posterior (involving $n$ samples and $p$ variables) to the leave-a-variable-out posterior where one of the variables has been dropped from the log-likelihood. In contrast, the key idea underlying Proposition \eqref{prop:hat_fix_point} is to relate the full posterior and the leave-an-observation-out posterior, where a sample is dropped but all variables are retained. Such leave-one-out arguments have appeared in the prior statistics and optimization literature for frequentist problems \cite{el2018impact,bean2013optimal,sur2019modern,chen2021spectral} and in the statistical physics/spin glasses literature either for models where the statistician knows the true signal distribution and uses it during the fitting process (the Bayes optimal setting) or for global null models, where there is no underlying signal \cite{talagrand2003spin,zdeborova2016statistical,mezard1987spin,montanari2024friendly}. A major technical contribution of our paper is to rigorously develop this machinery, for the first time, for commonly used statistical models, e.g., GLMs, under a general  Bayesian setting (and not the Bayes optimal setting) for which prior leave-one-out approaches do not apply; significant additional heavy-lifting is necessary in this case, as is evident from our proofs. 

%\ps{Should I expand on this last point a bit ? -- Contemplate on this on the next and final pass. Somehow emphas}
%\textcolor{red}{REvisit earlier in the paper, and if things are repeat, just skip the literature review here. Or keep this in color, and use in last round to integrate earlier.}

%{\color{blue} Apr 7: I have already checked the consistency of Sections 2,3,4 with the notations and definitions agreed.}

\section{Discussion}\label{sec:discussion}
An outstanding question that emerges from this work is whether our posterior characterization can be used to provide intervals with valid frequentist coverage for finite-dimensional marginals of the true signal. Although we do not develop this here, with some work we believe this would be feasible. In particular, our main result shows that finite-dimensional marginals of the posterior are suitable Gaussian tilts of the prior where the tilt parameters are determined by a system of equations that we characterize. Utilizing this, one can develop a method of moments based approach where we equate empirical moments under the posterior distribution with population moments. We require three moments to estimate the parameters $\alpha, \sigma, v$ in our limiting distribution. Thus, for distributions where third moments exist, we can calculate data-driven estimates for these parameters. Subsequently, a function inversion technique should allow us to infer about finite-dimensional marginals of the true signal. 
%Such an approach was recently invoked by \cite{} for frequentist inference of doubly robust functionals in  GLMs.
Another direction of future work could involve understanding to what extent our current assumptions  can be relaxed. For instance, we believe the Gaussianity assumption on the covariates can be relaxed using independent universality arguments as done in the recent frequentist literature \cite{hu2022universality,liang2022precise,montanari2022universality,dudeja2022spectral,han2023universality,lahiry2023universality}. It would be important to establish this formally. 
%and a leave-one-out type approach can still be used to characterize the posterior distribution. We require the current conditions for certain technical aspects of our proofs and it would certainly be interesting to understand to what extend these can be relaxed. 
%Specifically, relaxing the log-concavity assumption on the prior would be an important direction for future work. 

%In fact, the only places in our proof where we use Gaussianity are \textcolor{red}{Manuel, very few places in the proof actually use this. Can you refer to the specific lemmas or results that rely on Gaussianity here so that it becomes clear where they are used?}\ms{Well, it is used in many places: concentration of free energy, exponential bounds for norm of vector $\bbeta$, exponential bound for $a_{1,i}$, etc. But in most cases it is used by convenience and proofs could be tackled differently if the vectors weren't Gaussian. As you point out, probably in most parts it could be replaced by sub-gaussianity. But this should be checked with a lot of care before statign it as a fact. In interpolation, Gaussianity is used in a fundamental way as Gaussian integration by parts is at its core. Here this is not the case.}. For all other arguments, sub-Gaussian tails, bounded second moments, etc.~suffice.

Finally, the fact that the effect of the prior does not wash away raises an interesting question: given an inference problem, could we identify the optimal prior that would lead to credible intervals (or transformations thereof) with valid frequentist coverage and minimal possible length? In settings where the prior effect has mattered in previous literature, authors have exploited this phenomenon for suitable Bayesian inference \cite{rousseau2011asymptotic,ascolani2023clustering}, but thus far, this has been understood for quite different settings than the one considered in our work. These aspects would be important to explore in the future. For this manuscript, we focused on introducing the leave-one-out technology for GLMs with planted signals (the case of no signal had already been done in \cite{talagrand2003spin}) and demonstrate how it can help characterize finite-dimensional marginals of high-dimensional posteriors without traditional sparsity-type assumptions on the true underlying signal. 

\section{Acknowledgements}
P.S.~and M.S.~would like to thank Jean Barbier for insightful discussions in the early stages of the project. P.S.~would like to thank Filippo Ascolani, Marta Catalano, Iain Johnstone, Sumit Mukherjee, Igor Pruenster, Subhabrata Sen and Surya Tokdar for helpful discussions and pointers to the literature. M.S.~would like to thank Ernesto Mordecki for valuable discussions. P.S.~was supported by the NSF CAREER Award.
%[Emphasize the beyond log-concave prior setting the mathematical difficulty] 
%[ESTIMATION OF CONSTANTS.]
\iffalse
\ms{MANU TODO:
\begin{itemize}
    \item Compute expression for c here, 
    \item produce a graph of c vs kappa, 
    \item write a proof for thm 1,
    \item computations Comparison 1
\end{itemize}
}

\ps{PRAGYA TODO:
\begin{itemize}
    \item Theorem discussion with non-parametrics,
    \item comparison with gaussian sequence model (Johnstone),
    \item comparison MLE (my paper),
    \item incorporating computations of Comparison 1 (if any (hopefully some)),
    \item write discussion
\end{itemize}
}
\fi
%\iffalse

%\fi

    \bibliographystyle{unsrtnat}
\bibliography{references}

    \appendix

    \section{Proofs of the main results: Proposition \ref{prop:conv_order_param} and Theorems \ref{thm:contraction} and \ref{thm:marginals}}\label{app:proof_mainres}
Before we start our proofs, we state some basic notations that would be useful throughout. 
%Recall that $\mu(\cdot)$ is the prior the statistician uses on each coordinate. 
Recalling the definitions of $\ba_\star$ and $\ba_\ell$ from \eqref{eq:a_is}, we  
define $\bA^{(l)}$, ${\bA'}^{(l)}$, and ${\bA''}^{(l)}$ to be the random vectors in $\R^n$ with coordinates given, for $i\in[n]$, by $A^{(l)}_i := A(a_{\ell,i})$, ${A'_i}^{(l)} := dA/dx|_{x=a_{\ell,i}}$, and ${A''_i}^{(l)} := d^2A/dx^2|_{x=a_{\ell,i}}$ respectively. Again, when the $\ell$ index is omitted, it is taken to be equal to $1$. Recalling from the discussions before \eqref{eq:glmhamilton} that $\tT_i(x) := T(f(x,e_i))$,
%, recall also the definitions of $a_{i,\star},a_{i,}$ from \eqref{eq:a_is}. As before, 
we denote by $\tbT,\tbTp\in\R^n$ the random vectors with coordinates $\tT_i := \tT_i(a_{\star,i})$ and $\tT'_i := d\tT_i/dx|_{x=a_{\star,i}}$, respectively. With these abbreviated notations, the log-likelihood \eqref{eq:glmhamilton} at $\bbeta^{(\ell)}$ can be re-expressed as 
$  \cL_{n,p}(\bbeta^{(\ell)}) = \sum_{i\in[n]} \tT_i a_{\ell,i} - A_i^{(\ell)}$ which for $\ell=1$ simplifies by our notations to to $\sum_{i\in[n]} \tT_i a_{i} - A_i$.
%\textcolor{red}{CONVENTION THROUGHOUT: FOR LIKELIHOOD DEFINITIONS USE A DUMMY $\boldsymbol{b}$ BUT FOR INTEGRALS UNDER MEASURES JUST USE $\bbeta$ since the integral makes it clear what the measure is under.}
\subsection{Proof of Proposition \ref{prop:conv_order_param}}\label{sec:proof_main_result}
The proof of Proposition \ref{prop:conv_order_param} relies on the following variants of Propositions \ref{prop:fix_point} and \ref{prop:hat_fix_point}. The difference between the original propositions and the following lies in the fact that the below assumes certain limits exist and then conclude relevant convergence results. For this subsection and Section \ref{sec:proofs}, we redefine constants used earlier in the manuscript as follows. Subsequently, in Proposition \ref{prop:FPidentify}, we establish that if these constants were defined as in Propositions \ref{prop:fix_point2} and \ref{prop:hat_fix_point2}, they also satisfy the system of equations based definitions provided in \eqref{eq:FPE_all2} and \eqref{eq:FPE_all1}. In the proof of Proposition \ref{prop:conv_order_param}, we will invoke the next two propositions along sub-subsequences.

\begin{proposition}[Variant of Proposition \ref{prop:fix_point}]\label{prop:fix_point2}
    In the setting of Section \ref{sec:model}, assume that the following limits hold,
    \begin{equation*}
        \lim_{n\to\infty} \E\thermal{\tilde Q_{11}} =: \tv, \,\, \lim_{n\to\infty} \E\thermal{\tilde Q_{12}} =: \tq ,\lim_{n\to\infty} \E\thermal{M_{11}} =: \tm, \,\,   \lim_{n\to\infty} \E\thermal{M_{12}} =: \bar{m} \\
    \end{equation*}
    \begin{equation*}
\lim_{n\to\infty}\E\thermal{A''(a_1)} =: a_{dp}, \,\, \lim_{n\to\infty} \E\thermal{\bar{Q}_{11}} =: \bar{v},\,\, \lim_{n\to\infty} \E\thermal{\bar{Q}_{12}} =: \bar{q} \,\, \lim_{n\to\infty}\frac{1}{2}\E\thermal{\tTdp_i(\bX_i^{\top}\tbbetas)\bX_i^{\top}\tbbeta} =: c_g,
    \end{equation*}
    where $\tilde Q_{mm'}, M_{mm'}, \bar{Q}_{mm'}$ are defined as in \eqref{eq:overlaps2}.
    Define 
    \begin{equation}\label{eq:r123}
    r_1 := a_{dp}+\tilde{q}-\tilde{v}, r_2:=\tilde{m}-\bar m, r_3 := \tilde{q}
    \end{equation}
     %\textcolor{red}{Why is $r_2$ definition even needed here?}
   % \textcolor{red}{To add $c_g$ here.} 
%where $a_{dp}$ is defined as in \eqref{eq:overlaps2}. 
 %  $$  \tv := \lim_{n\to\infty} \E\thermal{\tQ_{11}}, \,\,
  %  \tq := \lim_{n\to\infty} \E\thermal{\tQ_{12}}, \,\, \tm := \lim_{n\to\infty} \E\thermal{M_{11}}, \,\, \text{and} \,\, \bar{m} := \lim_{n\to\infty} \E\thermal{M_{12}}.$$  
%    for some constants $\tv, r_2, r_3, a_{dp}, \bar{v}, \bar{q}$}, and, 
For $l\geq1$, let $f:\R^{l+1} \to \R$ be a continuous and polynomially bounded function. We then have that, for every $j\geq1$,
    \begin{equation}\label{eq:RHSloov}
        \lim_{n\to\infty}\E\thermal{f(\beta_{\star,j},\beta_{j}^{(1)},\dots,\beta_{j}^{(l)})} = \E_{Z}\thermal{f(\beta_{\star,j},\beta_{j}^{(1)},\dots,\beta_{j}^{(l)})}_{h,j}
    \end{equation}
    where $Z=(Z_1,\hdots,Z_l) \sim \mathcal{N}(0,I_l)$  independent of everything else and $\thermal{\cdot}_{h,j}$ here refers to expectation under density $\prod_{m=1}^l p(b_m)$ with $p(b_m)$ proportional to 
    \begin{equation}\label{eq:pbdef}
    p(b_m) \propto e^{-(r_1/2)(b_m-m_{j})^2} \,\,\text{where} \,\,m_{j}=\frac{r_2+t_\gamma}{r_1}\beta_{\star,j}+\frac{\sqrt{r_3}}{r_1}{Z_m}.
    \end{equation}
\end{proposition}

\begin{proposition}[Variant of Proposition \ref{prop:hat_fix_point}]\label{prop:hat_fix_point2}
In the setting of Section \ref{sec:model}, assume that the following limits hold
    \begin{equation}\label{eq:intermlimit}
        \lim_{n\to\infty} \E\thermal{Q_{11}} =: v_B, \,\, \lim_{n\to\infty} \E\thermal{Q_{12}} =: c_B, \,\, \lim_{n\to\infty} \E\thermal{Q_{1\star}} =: c_{BB_\star}.
    \end{equation}
   % for some constants $v$, $m$, and $q$,} 
For $l\geq1$, let $g:\R^{l+1} \to \R$ be a continuous and polynomially bounded function. We then have that, for every $i\geq1$,
    \begin{equation}\label{eq:limloov}
    \lim_{n\to\infty}\E\langle g(a_{\star,i},\dots,a_{l,i})\rangle = \E_{G \otimes e}\thermal{g(\theta_\star,\theta_1,\hdots, \theta_l)}_s,
    \end{equation}
    where we define (for $m\in[\ell]$) when $c_B > 0$,
\begin{equation}\label{eq:thetas}
    \theta_m := \sqrt{\kappa(v_{B}-c_{B})}\xi_{B}^{(m)}+\sqrt{\kappa c_{B}}z_{BB_{\star}}  
    \mbox{ and }  \theta_\star := \sqrt{\kappa\left(\gamma^2 - \frac{c_{BB_{\star}}^2}{c_B}\right)}\xi_{B_{\star}} +c_{BB_\star}\sqrt{\frac{\kappa}{c_B}}z_{BB_{\star}}
\end{equation}
else $  \theta_m = \sqrt{\kappa v_B} \xi_B^{(m)}, \theta_\star := \sqrt{\kappa \gamma^2} \xi_{B_\star}$
Here, $\xi_{B_\star}, z_{BB_\star}$ are i.i.d.~standard Gaussian random variables, independent of everything else and $G=(\xi_{B_\star},z_{BB_{\star}})$. Furthermore, $\thermal{g(\theta_\star,\theta_1,\hdots, \theta_l)}_s$ denotes expectation of $g(\theta_\star,\theta_1,\hdots, \theta_l)$ with respect to $\xi_B^{(1)},\hdots, \xi_B^{(l)}$ that have density proportional to 
  $$\prod_{m=1}^l \exp\left\{\tT(\theta_\star)\theta(\xi_B^{(m)}) - A(\theta(\xi_B^{(m)})) - \frac{(\xi_B^{(m)})^2}{2}\right\}
  \quad \text{where} \quad \tT(\theta_\star) = T \circ f(\theta_\star,e),$$ and $e \sim \text{Unif}(0,1)$ independent of everything else.
  %Finally, $G$ denotes the joint dis
%} \textcolor{red}{Needs updating like previous proposition; mostly the $\langle \rangle_s$ definition should be explicitly clarified here. I will do this once Manuel elaborates on equation (70).}
\end{proposition}
%Note that $\theta_m$'s are copies of $\theta(\xi_B)$ from \eqref{eq:thetasearlier} with $\xi_B$ replaced by independent standard Gaussian variables.
We present the proofs for the aforementioned propositions in Section \ref{sec:proofs}. 
\begin{proposition}\label{prop:FPidentify}
Under the assumptions of Propositions \ref{prop:fix_point2} and \ref{prop:hat_fix_point2}, the constants $r_1,r_2,r_3$ from \eqref{eq:r123} and $v_B,c_B,c_{BB_\star}$ from \eqref{eq:intermlimit} satisfy the fixed point equations \eqref{eq:FPE_all2} and \eqref{eq:FPE_all1} respectively, and $a_{dp}$ satisfies \eqref{eq:overlaps2}. 
\end{proposition}

\emph{Proof of Proposition \ref{prop:FPidentify}.} For the constants $r_1,r_2,r_3$, this follows directly from  the definitions \eqref{eq:r123} on an application of \eqref{eq:limloov}. For the other set of constants we start from the definition \eqref{eq:intermlimit} and use 
\eqref{eq:RHSloov}. Analogous arguments hold for the constant $a_{dp}$
%\textcolor{red}{To finish. Worried about a scaling issue  while checking this even for $r_1$.}
%[NEEDS DISCUSSION AND ALSO SHOULD BE MOVED SOMEWHERE AFTER INTERMEDIATE PROP STATEMENTS?]
%\ps{I think there is some writing error here or I am missing something. The lemma and proposition deal with moments of $\beta_{j_0}$ and $a_n$'s which are all sample quantities, whereas (10) and (11) never even involve these. So is this a citation error to which lemma etc you wished to call? It might be best to write this more mathematically, spelling out exactly which means you are referring to, to avoid confusion, since we have means of so many things in the paper. Generally, I would suggest a rewrite of this proof using the math notations for what you are referring to, for ease of reading. For instance, $Q_11$ and the other order parameters are average quantities calculated. based on a posterior sample. Now Prop 1 tells you how to handle functions of one coordinate at a time while Proposition 2 tells us the same for functions of fitted values. We didnt write a proposition about average of separable functions applied to all coordinates, so I think there is a gap (not in the argument I know but in the current writing); so we should elaborate this proof in my opinion. }\ms{I added explicitly the reference ot the means I refer to. I also changed the wording a bit so that it is a bit clearer. The notations in A1 have been (up to possible mistakes) updated to the new notation.}

\emph{Proof of Proposition \ref{prop:conv_order_param}} Recall definitions \eqref{eq:overlaps} and \eqref{eq:overlaps2}. By Lemma \ref{lem:cont_mom} the expectations 
\begin{equation*}
    \E\thermal{Q_{11}}, \quad \E\thermal{Q_{12}}, \quad \text{and} \quad \E\thermal{Q_{1\star}}
\end{equation*}
define bounded sequences in $n$. By Hypothesis \ref{hyp2}, the same holds true for
\begin{equation*}
    \E\thermal{\tQ_{11}}, \quad \E\thermal{\tQ_{12}}, \quad \E\thermal{\bar{Q}_{11}}, \quad \E\thermal{\bar{Q}_{12}}, \quad \E\thermal{A''(a_1)}, \quad \E\thermal{M_{11}}, \quad \text{and} \quad \quad \E\thermal{M_{12}}.
\end{equation*}
Therefore, for every subsequence $(n_k)_{k\geq1}\subseteq\mathbb{N}$ there is a sub-subsequence $(n_{k_l})_{l\geq1}$ such that all these expectations converge to constants.

Let $(n_k)_{k\geq1}$ be some subsequence and $(n_{k_l})_{l\geq1}$ a sub-subsequence such that these expectations converge to finite limits. By Propositions \ref{prop:fix_point2}-\ref{prop:FPidentify}, the limit of these expectations must define solutions to the fixed point equations \eqref{eq:FPE_all2} and \eqref{eq:FPE_all1}. However, we operate in the regime where this system admits a unique solution, implying that these expectations converge to the same value for some sub-subsequence of each subsequence $(n_k)_{k\geq1}$. Hence, these expectations  converge along the entire sequence. Since the concentration in $L^2$ of these quantities is proved in Propositions \ref{prop:RS} and \ref{prop:conv_RS}, this completes the result.

\subsection{Proofs of Propositions \ref{prop:fix_point} and \ref{prop:hat_fix_point}}\label{sec:proof34}
Proofs of Propositions \ref{prop:fix_point} and \ref{prop:hat_fix_point} follow from Propositions \ref{prop:conv_order_param}, \ref{prop:fix_point2} and \ref{prop:hat_fix_point2}. 
\subsection{Proof of Theorem \ref{thm:contraction}}\label{sec:proof_contraction}
Recalling the definitions from \eqref{eq:overlaps}, 
notice that
\begin{equation*}
    \frac{1}{p}||\bbeta-\bbetas||^2 = Q_{11} + \gamma^2_n - 2 Q_{1\star}.
\end{equation*}
Then, if we let $c := v_B + \gamma^2 - 2c_{BB_\star}$,
\begin{equation*}
    \E\thermal{(p^{-1}||\bbeta-\bbetas||^2-c)^2} \leq 16 \left( \E\thermal{(Q_{11}-v_B)^2} + \E\thermal{(\gamma_n^2-\gamma)^2} + \E\thermal{(Q_{1\star}-2c_{BB_\star})^2} \right).
\end{equation*}
The result then follows by combining Hypothesis \ref{hyp1}(iv) and Proposition \ref{prop:conv_order_param}.

\subsection{Proof of Theorem \ref{thm:marginals}}\label{sec:proof_marginals}
By considering moments, this can be proved directly as a corollary of Proposition \ref{prop:fix_point}. Here we show the proof in the special case where $r=1$. The general case follows identically. Notice that the random density from \eqref{eq:density_hj} defined up to normalization by
\begin{equation*}
    p_{h,{j_1}}(db) \propto e^{-\frac{1}{2v}(b-m_{j_1})^2}\mu(db)
\end{equation*}
is strongly log-concave and thus is sub-exponential (see \cite[Theorem 5.1]{saumard2014log}). This implies that its moment generating function is defined in an open interval containing $0$ and thus $p_{h,{j_1}}$ is \emph{determined by its moments} (see \cite[Chapter 30]{billingsley2012probability}). Furthermore, by Lemma \ref{lem:cont_mom}, all the moments of $\beta_{j_1}$ are defined and by Proposition \ref{prop:fix_point} they converge to the moments of $p_{h,{j_1}}$. Then, by \cite[Theorem 30.2]{billingsley2012probability} the proof is complete. 
%conclude the first part of the corollary.

\subsection{Proof of Corollary \ref{cor:marginals_bayes}}\label{sec:proofcor1}

%\textcolor{red}{The following is proof of Corollary 1, we should make it so.}
For every $m\in\mathbb{N}$ let $g_m(x_1,\dots,x_m):=\prod_{l\in[m]}x_l$.  Similarly as in the previous proof, from Proposition \ref{prop:fix_point} applied to $g_{m}(\beta_{j_1}^{(1)},\dots,\beta_{j_1}^{(m)})$ we obtain that
\begin{equation*}
    \E\thermal{\beta_{j_1}}^m = \E\thermal{g_m(\beta_{j_1}^{(1)},\dots,\beta_{j_1}^{(m)})}\xrightarrow{n\to\infty}\E\left(\frac{\int b \, p_{h,{j_1}}(db)}{\int p_{h,{j_1}}(db)}\right)^m.
\end{equation*}
The result then follows, by \cite[Theorem 30.2]{billingsley2012probability}.

%\textcolor{red}{Write a line re proofs of props 3 and 4, and revisit statements for prop 5 and 6, specifically, what are the $v_B, c_B$ definitions? Through the FP equations or through limits of those quantities.}   
    \section{Leave-one-out arguments: Proofs of Propositions \ref{prop:fix_point2} and \ref{prop:hat_fix_point2}}\label{sec:proofs}
%We first present the proofs of Propositions \ref{prop:fix_point2} and \ref{prop:hat_fix_point2}.
\subsection{Proof of Proposition \ref{prop:fix_point2}}\label{sec:proof_fpe}
%\textcolor{red}{WE AGREED LIKELIHOOD DEFNITIONS HAVE $\bm{b}$'s BUT FOR CONVENIENCE INTEGRALS HAVE $\bbeta$'s. }

%\textcolor{red}{MANUEL, CAN YOU PLEASE KEEP YOUR CHANGES IN BLUE OR SOME OTHER COLOR SO I CAN READ THEM QUICK IN NEXT ROUND.}
%\ps{I feel we are repeating ourselves way too much. This part og this section, that is, everything before Subsection \ref{subsec:conditioning} should be integrated to the extent possible in Section \ref{sec:pfoutline} to the extent possible. I understand for somethings like the $Y_i^m$ the m-th replica version definition should stay here.But many/most of the other things could go earlier and be stated only once in the manuscript, and not twice.}\ms{I removed the paragraph redefining everything and just added equation references.}
Let $j_0\in[p]$ be the index of the marginal of interest. For any vector $\bB$, we will denote $\tilde{\boldsymbol{b}}$ and $b_{j_0}
$ to be the vector with the $j_0$-th coordinate removed and the $j_0$-th coordinate of $\bB$ respectively. We begin by expanding the log-likelihood into the following terms  %\eqref{eq:lovposterior}--\eqref{eq:mean_variable}. To avoid restating it for every result, we will assume that all lemmas in this appendix hold under Hypothesis \ref{hyp1}, \ref{hyp2}, and \ref{hyp3}.
%We will start by writing the relationship between the leave-a-variable-out log-likelihood and the original one. For this, we will Taylor expand the log-likelihood to order $O(n^{-1/2})$ which corresponds to a second order in the expansion. This yields
\begin{equation}\label{eq:lvarout}
    \begin{split}
      &  \logl_{n,p}(\bB)  =  \logl_{n,p-1}(\tbB) +  b_{j_0}\underbrace{\sum_{i \in[n]} X_{ij_0}(\tT_i(\tbX_i^\top\tbbetas) -A'(\tbX_i^\top\tbB))}_{Y_1}+ \beta_{\star,{j_0}} \underbrace{\sum_{i \in [n]}X_{i{j_0}}\tTp_i(\tbX_i^\top\tbbetas)\tbX_i^\top\tbB }_{Y_2} \\ &+  \beta_{\star,{j_0}}b_{j_0}\underbrace{\sum_{i \in [n]}X_{i{j_0}}^2\tTp_i(\tbX_i^\top\tbbetas)}_{Y_4}
       + \beta_{\star,{j_0}}^2\underbrace{\frac{1}{2}\sum_{i \in [n]}X_{i{j_0}}^2\tTdp_i(\tbX_i^\top\tbbetas)\tbX_i^\top\tbB }_{Y_5}  -b_{j_0}^2\underbrace{\frac{1}{2}\sum_{i \in [n]}A''(\tbX_i^\top\tbB)X_{i{j_0}}^2}_{Y_3} \Big] \nonumber + E_n(\bB)\\
        & = \logl_{n,p-1}(\tbB) + \Delta \logl(b_{j_0},\tbB) + E_n(\bB),
    \end{split},
\end{equation}
where recall that $\Delta \logl(b_{j_0},\tbB)$ and $Y_1,\hdots, Y_5$ were introduced in  \eqref{eq:DeltaL} and  \eqref{eq:Y1thro5} respectively. Above the error term $E_n$ is defined as follows
%\ps{One thing I noticed is the following: The $\tilde{T}$'s actually depend on $i$, since for each $i$, there is an external randomness $e_i$ on which it depends. THus this create any difficulties in the proof. I am not making the change in this pass in case I start doing it and it gets stuck somewhere, then might be trouble. But I will try to keep this in mind as I go along.}\ms{I think this is not a problem.} 
%Where in the last line, the last term is the error in the Taylor expansion and is equal to
\begin{equation}\label{eq:taylor_error}
    \begin{split}
        E_n (\bB):= & -\frac{b_{j_0}^3}{6}\sum_{i\in[n]} A'''(\xi_i) X_{i{j_0}}^3 + \frac{\beta_{\star,{j_0}}^3}{6}\sum_{i\in[n]} \tT_i'''(\chi_i) \tbX_i^\top\tbB X_{i{j_0}}^3 \\
        & \,\,\,\,\, + \frac{\beta_{\star,{j_0}}^2b_{j_0}}{2}\sum_{i\in[n]} \tT_i''(\tbX_i^{\top}\tbbeta_\star) X_{i{j_0}}^3+\frac{\beta_{\star,{j_0}}^3b_{j_0}}{6}\sum_{i\in[n]} \tilde{T}_i'''(\chi_i)X_{i{j_0}}^4
    \end{split}
\end{equation}
%\textcolor{red}{The difference is in the argument of $\tT''$ and that the arguments of $\tT'''$ in the second and fourth terms are supposed to be the same $\chi_i$. Also the first term has a negative sign. I checked many times I think what I wrote above is correct?}\ms{Yes, I agree}
with $\xi_i$ some value in the interval between $\tbX_i^\top\tbbeta$ and $\tbX_i^\top\tbbeta + X_{ij_0} \beta_{j_0}$ and $\chi_i$ a value in the interval between $\tbX_i^\top\tbbetas$ and $\tbX_i^\top\tbbetas + X_{ij_0} \beta_{\star,{j_0}}$. 
%\textcolor{red}{To remember that $Y_3$ and $Y_5$ have the halves and that should be carried through the proof.}

The first hurdle encountered in the proof of Proposition \ref{prop:fix_point2} is that $\beta_{j_0},Y_1,\dots,Y_5$ are unbounded under $\E\thermal{\cdot}$ and that $E_n$ cannot be a priori neglected. To mitigate this, in the next subsection, we present two intermediate results which establish that expectations under $\E\thermal{\cdot}$ can be well-approximated by expectations with respect to a  measure $\bE\thermal{\cdot}'_K$ that is much easier to handle. 
%\ps{define this fully here with self-consistent notation so nobody needs to look up anything else in the manuscript at this point. Also, why do you add a bar over the expectation, does it have a specific meaning?}\ms{Are you sure about moving this, the definition is a couple of lines below. And I think it will make it more confusing if I move it up because I need to introduce many auxiliary definitions without previous motivation.} 
This is done in Section \ref{subsubsec:approx_cavity}. The final technical component involves establishing the asymptotic properties of $Y_1,\dots,Y_5$; the result is stated and proved in Section \ref{subsubsec:asymp_Y_k}.

\subsubsection{Simplification of the leave-a-variable-out measure}\label{subsubsec:approx_cavity}

The simplification of the leave-a-variable-out measure will consist of two steps. We will first show that expectations under the full posterior can be asymptotically approximated by expectations conditional on certain quantities being bounded.
%, and show that the means of a given family of functions remain unchanged asymptotically. 
In a second step we will show that dropping the Taylor error terms $E_n$  does not substantially change these expectations. %substantially the means of these functions either.

%--------------------------------------------------------------------------------

\paragraph{Conditioning argument} 

%\ps{Goes to appendix with suitable placement}
Let $||\cdot||_{op}$ stand for the \emph{operator norm}. For $K>0$, consider the following events
\begin{equation}\label{eq:def_Ak}
    A_K:=\{|\beta_{j_0}|,|Y_1|,\dots,|Y_5|\leq K\} \,\, \mbox{ and } \,\, B:=\{||\bX_{\bullet j_0}||,||\bX||_{op}\leq3\}
\end{equation}
with
\begin{equation}\label{eq:def_X_j_0}
    \bX_{\bullet j_0} = (X_{1j_0},\dots,X_{nj_0}) \in \R^n.
\end{equation}
Notice that, if we define $F:\R^p\mapsto\R^6$ and $G:\R^{p\times n}\mapsto\R^2$ as
\begin{equation*}
    F(\bbeta) = (K-|\beta_{j_0}|,K-|Y_1|,\dots,K-|Y_5|) \quad \text{and} \quad G(\bX) = (3-||\bX_{\bullet j_0}||,3-||\bX||_{op}),
\end{equation*}
we can then write $A_K=\{F(\bbeta)\geq0\}$ and $B=\{G(\bX)\geq0\}$---this is the version of events we work with later in Appendix \ref{app:conditioning}.

%\textcolor{blue}{Manuel to clarify over email if he has already fixed this issue.}
Let $\bE(\cdot)$ denote expectation w.r.t. $(\bX,\be)$ conditional on event $B$. That is, for $f:\R^{n\times p}\times\R^n\mapsto\R$, $\bE(f(\bX,\be))$ is defined according to
\begin{equation}\label{eq:def_mean_bE}
    \bE(f(\bX,\be)) = \frac{\E(f(\bX,\be) \mathbb{I}_{B})}{\proba(B)}.
\end{equation}
Notice that this quantity is well defined as $\proba(B) > 0$. In fact, $\proba(B) \rightarrow 1$ in our high-dimensional setting. %\textcolor{red}{For any function f? I am confused...what goes to 1 exactly?} 
Next we take $K>0$ large enough and define for $f:\R\mapsto\R$, 
%\ps{I am confused here, are we conditioning on both $A_K$ and $B$ here? I feel then the below definition should make that explicit; there is no B below so how are we conditioning on B?}\ms{Yes, this mean expectation is defined under the conditional measure $\bE$. I have changed the wording to make this more clear. I also added a comment below pointing to this explicitely.}, 
%and K large just for defining this?}
\begin{equation}\label{eq:angleK}
    \thermal{f(\beta_{j_0})}_K := \frac{\thermal{f(\beta_{j_0})\mathbb{I}_{A_K}}}{\thermal{\mathbb{I}_{A_K}}}\mathbb{I}_B.
\end{equation}
%\textcolor{red}{PS to take pass on English below; and take final pass on this in final round to decide if $I_B$ stays or not.}
Notice that, 
%by the way it was defined, 
$\thermal{f(\beta_{j_0})}_K$ is non-zero only on the event B. Furthermore, on $B$  for sufficiently large $K>0$, a.s. $\thermal{\mathbb{I}_{A_K}} > 0$. Hence, \eqref{eq:angleK} is well-defined. 
%To see this, notice that conditional on $B$, for every centered ball $\{||\bB||\leq r\}\subseteq\R^p$ of radius $r > 0$, there is a sufficiently large $K>0$ such that $\{||\bB||\leq r\}\subseteq A_K$. By choosing $r>0$ large enough to have positive posterior mass and $K >0$ large enough so that $\{||\bB||\leq r\}\subseteq A_K$, we have that $\thermal{\mathbb{I}_{A_K}} > 0$. \ms{You were not convinced by this argument and I therefore rewrote it. I wouldn't include this. I just write it for you in case you still have doubts. To see the above just notice that if $\bB\in \{||\bB||\leq r\}$, then by Hypothesis \ref{hyp2} and the condition that $B$ holds
%\begin{equation*}
 %   |Y_1| \leq K_1 + K_2 ||\bX_{\bullet j_0}||\, ||\bX||_{op} ||\tbB|| \leq K_1 + 9 K_2 ||\tbB||.
%\end{equation*}
%This means that if we choose $K = K_1 + 9 K_2 r$, then $||\bB||\leq r$ implies $|Y_1| \leq K$. The same can be done with $Y_2,\dots,Y_5$ to prove the statement.
%}
The following lemma controls the difference of expectations of functions under $\E\thermal{\cdot}$ and $\bE\thermal{\cdot}_K$. This allows us to work, in the proof of Proposition \ref{prop:fix_point2}, under the latter measure instead of the original one.
\begin{lemma}\label{lem:aprox_K}
    For $f:\R^p\mapsto\R$ bounded, there are vanishing sequences $(\varepsilon_n)_{n\geq1}$ and $(\varepsilon'_K)_{K\geq1}$ such that
    \begin{equation*}
        |\E\thermal{f(\bbeta)} - \bE\thermal{f(\bbeta)}_K| \leq ||f||_\infty( \varepsilon_n + \varepsilon'_K).
    \end{equation*}
\end{lemma}
%\ms{If we follow your suggestion that each measure has a special notation for its samples, we will need to introduce many new notations here which are not used much. What should be do? For now I keep it like it is, only changing the notations when these are samples form the leave-one-out measures.}
The proof of this lemma is delayed to Appendix \ref{app:approximation_cavity_measure}.

% --------------------------------------------------------------------------------------------------------------

\paragraph{Taylor approximation} 

Given a function $f:\R^p\mapsto\R$, define the expectation $\thermal{\cdot}'_K$ according to
\begin{equation}\label{eq:mean_approx_K}
    \thermal{f(\bbeta)}'_K := \frac{\thermal{f(\bbeta) \mathbb{I}_{A_K} \exp\{-E_n\} }}{\thermal{\mathbb{I}_{A_K} \exp\{-E_n\}}}{\mathbb{I}_B}.
\end{equation}
That is, $\thermal{\cdot}'_K$ is the posterior expectation conditional on $A_K$ (and on $B$) with respect to a modified log-likelihood that drops the Taylor error terms $E_n$. The following lemma controls the difference of expectations of functions under $\thermal{\cdot}_K$ and $\thermal{\cdot}'_K$. 
%\textcolor{red}{We should be consistent in our definitions across \eqref{eq:angleK} and\eqref{eq:mean_approx_K}. The first one has $I_B$ within the definition, the second one doesnt. Should we add that back in the second one as well. I think nothing changes if we make this change, since these are used under $\bar{E}$ mostly which is conditional on B anyways. }\ms{I agree that this is mostly for consistency. But I would add it in the second one}
\begin{lemma}\label{lem:taylor_control}
    For every continuous $f:\R\mapsto\R$ there is a constant $C>0$ such that 
    \begin{equation*}
        \bE|\thermal{f(\beta_{j_0})}_K-\thermal{f(\beta_{j_0})}'_K|\leq \frac{C}{n^{1/2}}.
    \end{equation*}
\end{lemma}
%\textcolor{red}{I dont like that argument of f in Lemma 1 is  $\mathbb{R}^p$ but in Lemma 2 is $\mathbb{R}$.}
The proof of this lemma is also delayed to Appendix \ref{app:approximation_cavity_measure}.
%\textcolor{red}{TILL HERE.}

%\textcolor{red}{Till here.}

% --------------------------------------------------------------------------------------------------------------

\subsubsection{Asymptotics of $Y_1,\dots,Y_5$ under the leave-a-variable-out measure}\label{subsubsec:asymp_Y_k}
%\textcolor{red}{Lets change $\langle \rangle_v \propto e^{\mathcal{L}_{n,p-1}(\tbbeta)}\mu^{\otimes p}(\bbeta)$. So that the $j_0$th coordinate is a draw from the prior.+}

Let $\bbeta^{(1)},\hdots, \bbeta^{(\ell)}$ denote independent samples from the leave-a-variable-out measure defined in \eqref{eq:mean_variable}. This means that  $\tbbeta^{(1)},\hdots,\tbbeta^{(\ell)}$ denote all but the $j_0$-th coordinates of these and that $\beta_{j_0}^{(1)},\hdots, \beta_{j_0}^{(\ell)}$ are independent samples from the prior. In the next section, we will see that, proving Proposition \ref{prop:fix_point2} requires analyzing the limit of functions of the random vector $Z^{(m)}_n\in\R^6$ defined as $(\beta^{(m)}_{j_0},Y^{(m)}_1,Y^{(m)}_2,Y^{(m)}_3,Y^{(m)}_4,Y^{(m)}_5)$ where we choose $Y_1^{(m)} := Y_1(\tbbeta^{(m)},\tbbeta_\star), \dots, Y_5^{(m)} := Y_5(\tbbeta^{(m)},\tbbetas)$ defined as in \eqref{eq:Y1thro5}. 
%\textcolor{red}{This notation is actually good, that is, making the argument of $Y_i$'s explicit, lets keep this as is for the time being.}
%Here, 

\begin{lemma}\label{lem:equivalence_mean}
    Under the assumptions of Proposition \ref{prop:hat_fix_point}, for $r\geq1$, let $f_1,\dots,f_r:\R^2\mapsto\R$ be continuous and polynomially bounded functions and $\boldsymbol{V}^{(1)},\dots,\boldsymbol{V}^{(r)}\in\R^n$ be random vectors with  coordinates, for each $m\in[r]$ and $i\in[n]$, given by $V_i^{(m)} := f_m(a_{\star,i},a_{m,i})$. Then for every $m,m'\in[r]$, if $\hat Q_{m,m'}:=n^{-1}\boldsymbol{V}^{(m)\top}\boldsymbol{V}^{(m')}$, then
    \begin{equation*}
        \lim_{n\to\infty}\E\langle \hQ_{m,m'}\rangle = \hat{q} \quad \text{implies that} \quad \lim_{n\to\infty}\E\langle \hQ_{m,m'}\rangle_v = \hat{q}.
    \end{equation*}
\end{lemma}
\begin{proof}
    To see this, for every $t\in[0,1]$, let $\thermal{\cdot}_t$ be expectation with respect to the posterior corresponding to the log-likelihood
    \begin{equation*}
        \logl_{n,p,t}(\bB) := t (\Delta L + E_n) + \logl_{n,p-1}(\tbB).
    \end{equation*}
    That is, for each $f:\R^p\mapsto\R$, we define
    \begin{equation*}
        \thermal{f(\bbeta)}_t := \frac{\int f(\bbeta)\exp\{\logl_{n,p,t}(\bbeta)\} \prod_{j\in[p]}\mu(d\beta_j)}{\int \exp\{\logl_{n,p,t}(\bbeta)\} \prod_{j\in[p]}\mu(d\beta_j)}.
    \end{equation*}
    Clearly $\thermal{\cdot}_0$ is equal to $\thermal{\cdot}_v$ (expectation under the leave-a-variable-out posterior) and $\thermal{\cdot}_1$ to $\thermal{\cdot}$ (expectation under the full posterior). Now, notice that 
    \begin{equation*}
        \begin{split}
            \frac{d}{ds}\E\langle\hat{Q}\rangle_s\big|_{s=t} & = \E\left[\frac{d}{ds}\frac{\int \hat{Q}\exp{\{\logl_{n,p,s}(\bbeta)\}} \prod_{j\in[p]}\mu(d\beta_j)}{\int \exp{\{\logl_{n,p,s}(\bbeta)\}} \prod_{j\in[p]}\mu(d\beta_j)}\right]\Big|_{s=t} \\
            & = \E\left[\frac{\int \hat{Q}(\Delta L + E_n)\exp{\{\logl_{n,p,s}(\bbeta)\}} \prod_{j\in[p]}\mu(d\beta_j)}{\int \exp{\{\logl_{n,p,s}(\bbeta)\}} \prod_{j\in[p]}\mu(d\beta_j)}\right]\Big|_{s=t} \\
            & \quad \quad - \E\left[ \frac{\int \hat{Q}\exp{\{\logl_{n,p,s}(\bbeta)\}} \prod_{j\in[p]}\mu(d\beta_j)\int (\Delta L + E_n)\exp{\{\logl_{n,p,s}(\bbeta)\}} \prod_{j\in[p]}\mu(d\beta_j)}{\left(\int \exp{\{\logl_{n,p,s}(\bbeta)\}} \prod_{j\in[p]}\mu(d\beta_j)\right)^2} \right]\Big|_{s=t} \\
            & = \E\thermal{(\Delta L + E_n)(\hat{Q}-\langle \hat{Q} \rangle_t)}_t;
        \end{split}
    \end{equation*}
    where, for the first equality, we used \cite[Proposition A.2.1]{talagrand2010mean} to exchange differentiation and expectation and, for the second one, we used that $\frac{d}{ds}\logl_{n,p,s}(\bbeta) = \Delta L + E_n$. By straightforward adaptations of Propositions \ref{lem:2mon_Y} and \ref{prop:conv_RS} to $\thermal{\cdot}_t$ we have that the right hand side of this last equation is $o(1)$. From this and the Mean Value Theorem we conclude that
    \begin{equation*}
        \lim_{n\to\infty} \E[\langle\hat{Q}\rangle_v - \E\langle\hat{Q}\rangle ]= 0.
    \end{equation*}
\end{proof}

For each distinct pair $m,m'\in[l]$, let $\tilde{Q}_{mm}$, $\tilde{Q}_{mm'}$, $M_{mm}$, and $M_{mm'}$ as in \eqref{eq:overlaps2}, but with $\bbeta^{(1)},\hdots, \bbeta^{(\ell)}$ drawn from the leave a variable-out measure. Recall $t_{\gamma}$ from \eqref{eq:alphasv}. Furthermore, under the conditions of the Proposition,
by Lemma \ref{lem:equivalence_mean} we have that
\begin{equation}\label{eq:def_tp}
    \begin{split}
        &\lim_{n\to\infty} \E\thermal{A''(\tbX^\top_i\tbbeta)}_v = a_{dp}, \\ 
        &\lim_{n\to\infty} \E\thermal {(\tT'(\tbX_{i}^{\top}\tbbeta_\star)(\tbX_i^{\top}\tbbeta))^2}_v = \bar{v}\\
        &\lim_{n\to\infty} \E \thermal{\tT'(\tbX_i^{\top} \tbbeta_\star)^2 \tbX_i^{\top}\tbbeta^{(m)} \bX_i^{\top}\tbbeta^{(m')}}_v = \bar{q} \\
        &\lim_{n\to\infty}\E\thermal{\frac{1}{2}\tTdp_i(\tbX_i^{\top}\tbbetas)\tbX_i^{\top}\tbbeta}_v = c_g.
    \end{split}
\end{equation}
%where $a_{dp}$ is as defined in \eqref{eq:overlaps2}.

%\begin{equation}\label{eq:def_lim_overlaps}
%    \tv := \lim_{n\to\infty} \E\thermal{\tQ_{11}}, \,\,
%    \tq := \lim_{n\to\infty} \E\thermal{\tQ_{12}}, \,\, \tm := \lim_{n\to\infty} \E\thermal{M_{11}}, \,\, \text{and} \,\, \bar{m} := \lim_{n\to\infty} \E\thermal{M_{12}}.
%\end{equation}
%are the limits of the order parameters \eqref{eq:overlaps}--\eqref{eq:overlaps2}.
%\textcolor{red}{Definitions here and the proposition in main manuscript does not match.}
%\textcolor{red}{I havent checked details of this section properly and need to do that.}

Here we will state a lemma that characterizes the limit in distribution of the random vectors $Z_n^{(m)}$. For $m\in[l]$, let $\tilde \theta_m := \tilde \xi_m + \tilde{z}'$ and $\bar \theta_m := \bar \xi_m + \bar z$; with $(\tilde \xi_1,\dots,\tilde \xi_l,\bar\xi_1,\dots,\bar\xi_l)\in\R^{2l}$ a Gaussian random vector with covariance matrix given, for $m,m'\in[l]$, by
\begin{equation}\label{eq:covariance_thetas}
    \E(\tilde{\xi}_m\tilde{\xi}_{m'}) = (\tilde v - \tilde q) \delta_{mm'}, \,\, \E(\bar{\xi}_m\bar{\xi}_{m'}) = (\bar v - \bar q) \delta_{mm'}, \mbox{ and } \E(\tilde{\xi}_m\bar{\xi}_{m'}) = (\tilde m - \bar m) \delta_{mm'}
\end{equation}
and $(\tilde{z}',\bar z)\in\R^2$ 
%\textcolor{red}{Where were $\bar{v},\bar{q},\bar{m}$ defined?}\ms{Now they are defined in the equations above}
a Gaussian random vector independent of $(\tilde{\xi}_m,\bar{\xi}_m)$ with covariance matrix given by $\E {\tilde{z'}}^2 = \tilde q$, $\E {\bar{z}}^2 = \bar q$, and $\E({\tilde{z}}' \bar z) = \bar m$. 
%\textcolor{red}{You meant one of them is $\bar{m}$?}\ms{This is an index}
Then, for $m\in[l]$, let %\textcolor{red}{THIS section is generally very dense, if there is a place to see why the below $\zeta_m$'s should take the form they have based on the limits of the order parameters that would be helpful. Or some lines explaining this here would be helpful. Any de-densifying you can do would be helpful.}
\begin{equation}\label{eq:def_zeta}
    \zeta_m:=(\beta_m,\tilde{\theta}_m,\bar{\theta}_m,a_{dp}/2,t_\gamma,c_g)
\end{equation}
with $\beta_1,\beta_2,\dots$ independent samples from the prior $\mu(\cdot)$. We then have the following limiting characterization. \begin{lemma}\label{lem:conv_dist_Z}
Under the assumptions of Proposition \ref{prop:fix_point2}, for every $l\geq1$ and every continuous and bounded function $f:\R^l\mapsto\R$, we have that%\textcolor{red}{First time $\langle \rangle_v$ this is defined we need to clarify what l copies under that means, for instance, what the LHs below means.}
\begin{equation*}
        \E\thermal{f(Z_n^{(1)},\dots,Z_n^{(l)})}_v \xrightarrow{n\to\infty} \E(f(\zeta_1,\dots,\zeta_l)).
    \end{equation*}
  %  where $\E_{\zeta}(\cdot)$ stands for expectation with respect to $\zeta_1,\dots,\zeta_l$ as defined above.\textcolor{red}{Cant use $\zeta$ here tp denote $\zeta_1,\hdots,\zeta_\ell$ since $\zeta$ refers to $\zeta_1$ by our convention.}
\end{lemma}
The proof is deferred to Appendix \ref{subsec:limitloov}. 
%\textcolor{red}{The following should actually be shortened and moved to or integrated with proof outline, I will do it in my next pass, for now I keep this as is.}
 In Lemma \ref{lem:conv_dist_Z}, there are two kinds of limits involved---$Y_3$, $Y_4$, and $Y_5$ involve second moments of $X_{ij_0}$ and thus converge to constants. To prove these we will resort to Proposition \ref{prop:conv_RS} which establishes the concentration in $L^2$ of certain functions of the fitted values $\tbX_1^\top\tbbeta,\dots,\tbX_n^\top \tbbeta$. On the other hand, by the results of Appendix \ref{app:asymp_representation}, $Y_1$ and $Y_2$ converge to Gaussians. The convergence holds since $\tbbeta^{(1)},\dots,\tbbeta^{(l)}$ are independent of $\bX_{\bullet j_0}$ thus conditionally on $\tbbeta^{(1)},\dots,\tbbeta^{(l)}$, the random variables $\bX_{\bullet j_0}^\top\tbbeta^{(1)},\dots,\bX_{\bullet j_0}^\top\tbbeta^{(l)}$ are jointly Gaussian with a covariance matrix which converges to a suitable limit due, again, to Proposition \ref{prop:conv_RS}.

By a similar argument as in Lemma \ref{lem:cond_w_conv},  we can extend this convergence in distribution to convergence under $\bE\thermal{\cdot}_v$.
\begin{corollary}\label{cor:cond_conv_dist_Z}
    Let $l\geq1$ be some arbitrary integer. Then, for every continuous and bounded function $f:\R^l\mapsto\R$,
    \begin{equation*}
        \bE\thermal{f(Z_n^{(1)},\dots,Z_n^{(l)})}_v \xrightarrow{n\to\infty} \E(f(\zeta_1,\dots,\zeta_l)).
    \end{equation*}
\end{corollary}
%As before, the proof of Lemma \ref{lem:conv_dist_Z} is delayed to Appendix \ref{app:approximation_cavity_measure}.

% --------------------------------------------------------------------------------------------------------------

\subsubsection{Leave-a-variable-out argument}
%\textcolor{red}{Manuel says replica indices should be $m=1,2,\hdots,\ell$ this needs to be fixed and made consistent throughout. There are really an appalling number of notational inconsistencies in this paper now. I am starting to make this change this section onwards.}

Here we will present the proof for $\ell=1$, but the general case follows in an analogous manner. Furthermore, we will prove the result for continuous and bounded functions $f:\R\mapsto\R$. The extension to continuous polynomially bounded functions follows by appealing to the uniform integrability of $\beta_{j_0}$ under expectation $\E\thermal{\cdot}$, obtained as a consequence of Lemma \ref{lem:cont_mom}.Recall that our goal is to analyze the limit of $\mathbb{E} \langle f(\beta_{\star,{j_0}},\beta_{j_0}^{(1)},\hdots, \beta_{j_0}^{(\ell)})\rangle $ where $\beta_{j_0}^{(1)},\hdots, \beta_{j_0}^{(\ell)}$ denotes the $j_0$-th coordinate of $\ell$ independent samples from the posterior. That is, 
\[ \mathbb{E} \langle f(\beta_{\star,{j_0}},\beta_{j_0}^{(1)},\hdots, \beta_{j_0}^{(\ell)})\rangle = \mathbb{E}\left[\frac{\int f(\beta_{\star,{j_0}},\beta_{j_0}^{(1)},\hdots, \beta_{j_0}^{(\ell)})\prod_{m=1}^\ell e^{\cL_{n,p}(\bbeta^{(m)})} \mu^{\otimes p}(d\bbeta^{(m)}) }{\int \prod_{m=1}^{\ell}e^{\cL_{n,p}(\bbeta^{(m)})} \mu^{\otimes p}(d\bbeta^{(m)})}\right], \]
which recall from Section \eqref{sec:pfoutline} simplifies to the following 
%I introduced this new $\langle \rangle_{v ^{\otimes \ell} }$. We should discuss if this is worth using through the rest or not. I havent used it everywhere below since I wanted to have this discussion once.}
\begin{equation} \label{eq:interm}
\mathbb{E}\left[\frac{\Big\langle f(\beta_{\star,{j_0}},\beta_{j_0}^{(1)},\hdots,\beta_{j_0}^{(\ell)})\prod_{m=1}^{\ell} \exp \{\Delta \cL(\beta_{j_0}^{(m)},\tbbeta^{(m)})+E_n(\bbeta^{(m)})\} \rangle_{v}}{\langle  \prod_{m=1}^{\ell}\exp \{\Delta \cL(\beta_{j_0}^{(m)},\tbbeta^{(m)})+E_n(\bbeta^{(m)})\}\rangle_{v}}\right],
\end{equation}
%\textcolor{red}{We are saying we dont do the above $\langle \rangle_{v^{\otimes \ell}}$ but keep it $\langle \rangle_{v}$}
%\textcolor{red}{Please check the notation I introduce here for expectations under $\ell$ independent copies of a LOOV posterior. I think it is important to distinguish a single copy from $\ell$ independent copies. Can we do the same for leave an observation-out when that appears in calculations?}
%where $\langle \cdot \rangle_{v^{\otimes \ell}}$ denotes expectations under $\ell$ independent copies of the leave-a-variable-out posterior defined in \eqref{eq:mean_variable}. 
%$\frac{\thermal{f(\bbeta)\exp\{\dlogl_{n,p-1}(\beta_{j_0},\tbbeta) + E_n(\bbeta)\}}_v}{\thermal{\exp\{\dlogl_{n,p-1}(\beta_{j_0},\tbbeta) + E_n(\bbeta)\}}_v}$
%where $\langle \cdot\rangle_v$ denotes expectations under the leave-a-variable-out posterior defined in \eqref{eq:mean_variable}. 
Crucially, by definition, under the $\langle \cdot \rangle_v$ measure, $\tbbeta^{(m)}$'s are independent draws from a density proportional to $e^{\cL_{n,p-1}(\tbbeta^{(m)})}\mu^{\otimes p-1 }(\tbbeta^{(m)})$ whereas $\beta_{j_0}^{(m)}$'s are independent draws from the prior $\mu(\cdot)$. We work with this definition of the $\tbbeta^{(m)}$'s and $\beta_{j_0}^{(m)}$'s. Our goal is to then show that \eqref{eq:interm} converges to the RHS of \eqref{eq:RHSloov}.

To this end, recall the definition of $Z_n^{(m)}$ from Section \ref{subsubsec:asymp_Y_k} and that of $\Delta \cL(\beta_{j_0},\tbbeta)$ from \eqref{eq:DeltaL}. For convenience we redefine $\Delta \cL(\beta_{j_0}^{(m)},\tbbeta^{(m)})$ to be
%It is product product omg.
\begin{equation*}
    \Delta \logl(Z_n^{(m)}) :=  Y^{(m)}_1 \beta^{(m)}_{j_0} + Y^{(m)}_2 \beta_{\star,{j_0}} - \beta^{(m)2}_{j_0} Y^{(m)}_3 + \beta_{\star,{j_0}}\beta_{j_0}^{(m)} Y_4^{(m)} + \beta_{\star,j_0}^2 Y^{(m)}_5,
\end{equation*}
to emphasize the dependence of this quantity on the entries of $Z_n^{(m)}$. 
For $K$ large enough, define the event $A_K^{(m)}:=\{||Z_n^{(m)}||_\infty\leq K\}.$
%\textcolor{red}{TILL HERE.}
Notice that due to the definition of $A_K^{(m)}$, there is some small enough $\delta > 0$ such that $\delta e^{\Delta\logl(Z_n)}\mathbb{I}_{A_K}\leq 1/2$ 
%\textcolor{red}{Do you mean this inequality in an almost sure sense? If so, why is this true?}
%Although $\dlogl(Z_n^{(m)})$ and $\dlogl(\beta^{(m)}_{j_0})$ refer to the same expression, here we will write $\dlogl(Z_n^{(m)})$ to emphasize the dependence of this quantity over the vector $Z_n^{(m)}$
%\ps{What are $Z_n$'s and $\Delta \mathcal{L}(Z_n)$, also $A_K$ definition will naturally move up here when you try to define $\bE\thermal{\cdot}'_K$.}\ms{Added a comment explaining what this means.} 
Fix $\delta >0$ such that this holds. Let $L\geq1$.
%\textcolor{red}{To check the below once more, have some little confusion about the indicator $A_k$ placement,}For each $k\le L$ define
%\textcolor{red}{I am confused why you can find such a $\delta$ no matter the value of $e^{\Delta \cL(Z_n)}$ On $A_k$ $Z_n$ is bounded but do you have this due to the moment boundedness of the $\beta_{j_0}$'s? I dont agree right? E.g. normal distribution has moments bounded but its support is the entire real line? So there is no guarantee the $e^{\Delta}$ term couldnt be blowing up?} 
%\ps{We were using dummy index $l$ for replicas until some point. May be we keep that? From previous section it seems to me m should be reserved for moment calculations.}\ms{I think that, up to now, the index $m$ has been consistently used as the dummy index and $l$ was reserved for the largest index in a collection of replicas. So usually, $m$ was some index such that $m\leq l$. Do you still think we should change it?}
\begin{equation}\label{eq:gkdef}
   g_{k}(Z_n^{(1)},\dots,Z_n^{(k)}) := f(\beta^{(1)}_{j_0}) \exp\{\Delta\logl(Z_n^{(1)})\}\mathbb{I}_{A_K^{(1)}} \prod_{m=2}^{k} (1 - \delta \exp\{\Delta\logl(Z_n^{(m)})\}\mathbb{I}_{A_K^{(m)}}),
\end{equation}
%\textcolor{red}{Please clarify exponential exponent with additional parentheses.}
where we let $\prod_{j=2}^1 (\cdots) =1$. As per our Hypothesis \ref{hyp1}, the definition of $A_K^{(m)}$ and the fact that $f$ is continuous and bounded, the functions $g_k$ are also continuous and bounded. By \eqref{eq:mean_approx_K},   we have that
\begin{equation}\label{eq:cond_cavity}
\thermal{f(\beta_{j_0})}'_K =\frac{\thermal{f(\beta_{j_0})\exp\{-E_n\}\mathbb{I}_{A_K}}}{\thermal{\exp\{-E_n\}\mathbb{I}_{A_K}}}\mathbb{I}_B = \frac{\thermal{f(\beta_{j_0})\exp\{\Delta\logl(Z_n)\}\mathbb{I}_{A_K}}_v}{\thermal{\exp\{\Delta\logl(Z_n)\}\mathbb{I}_{A_K}}_v}\mathbb{I}_B,
\end{equation}
where the last equality follows by definition of $\langle \cdot \rangle_v $ from \eqref{eq:mean_variable}.
%\ms{This should be replaced with a simple computation using the derivative.}
%where for the second equality we used the fact that $\logl_{n,p}(\bbeta) = \logl_{n,p-1}(\tbbeta) + \dlogl(\beta_{j_0},\tbbeta) + E_n$,
%which is true since from definition we have that
%\begin{equation*}
%\begin{split}  \thermal{f(\beta_{j_0})\exp\{-E_n\}\mathbb{I}_{A_K}} & = \frac{1}{C_n}\int_{A_K} f(\beta_{j_0})\exp\{\logl_{n,p}(\bbeta)-E_n\} \prod_{j=1}^p \mu(d\beta_j) \\
%        & = \frac{1}{C_n}\int_{A_K} f(\beta_{j_0})\exp\{\logl_{n,p-1}(\tbbeta)+\dlogl(\beta_{j_0},\tbbeta)\} \prod_{j=1}^p \mu(d\beta_j) \\
%        & = \frac{1}{C_n}\thermal{f(\beta_{j_0})\exp\{\Delta\logl(Z_n)\}\mathbb{I}_{A_K}}_v,
%    \end{split}
%\end{equation*}
%where $C_n > 0$ is the normalizing constant. A similar derivation applies for the denominator, leading to \eqref{eq:cond_cavity}.%both numerator and e
%and a similar relation for the denominator.

%\ps{The below steps will require full elaboration. Note, currently such elaboration is given in the leave-an-observation-out section, however, we need to do that here to do justice to the below since we are switching the order of proofs of Propositions 1 and 2.}\ms{Done}
Recall the formula for the partial sums of the geometric series; that is, for all $r > 0$, 
    \begin{equation}\label{eq:geom_partial}
        \sum_{0\leq k \leq L-1} r^k = \frac{1-r^L}{1-r}.
    \end{equation}
By equation \eqref{eq:cond_cavity} and this last formula we have  
    
 \begin{equation}
        \begin{split}
            \bE\langle f(\beta_{j_0})\rangle'_K &= \bE\left[\thermal{ f(\beta_{j_0})}'_K \big(1-(1 -\delta\langle e^{\Delta\logl(Z_n)}\mathbb{I}_{A_K} \rangle_v )^{{L}}\big)\right] \\
            & \hspace{4cm} + \bE\left[\langle f(\beta_{j_0}) \rangle'_K(1 -\delta\langle e^{\Delta\logl(Z_n)}\mathbb{I}_{A_K} \rangle_v )^{{L}}\right]\\
            &= \delta\bE\left[\frac{\thermal{f(\beta_{j_0})e^{\Delta\logl(Z_n)}\mathbb{I}_{A_K}}_v}{\delta\thermal{e^{\Delta\logl(Z_n)}\mathbb{I}_{A_K}}_v}\big(1-(1 -\delta\langle e^{\Delta\logl(Z_n)}\mathbb{I}_{A_K} \rangle_v )^{{L}}\big)\right]\\
            & \hspace{4cm} + \bE\left[\langle f(\beta_{j_0}) \rangle'_K(1 -\delta\langle e^{\Delta\logl(Z_n)}\mathbb{I}_{A_K} \rangle_v )^{{L}}\right]\\
            &= \delta\bE\left[\thermal{f(\beta_{j_0})e^{\Delta\logl(Z_n)}\mathbb{I}_{A_K}}_v\frac{1-(1 -\delta\langle e^{\Delta\logl(Z_n)}\mathbb{I}_{A_K} \rangle_v )^{{L}}}{\delta\thermal{e^{\Delta\logl(Z_n)}\mathbb{I}_{A_K}}_v}\right] \\
            & \hspace{4cm} + \bE\left[\langle f(\beta_{j_0}) \rangle'_K(1 -\delta\langle e^{\Delta\logl(Z_n)}\mathbb{I}_{A_K} \rangle_v )^{{L}}\right]\\
            & = \delta\sum_{1\leq k\leq L} \bE \langle g_{k}(Z_n^{(1)},\dots,Z_n^{(k)}) \rangle_v + \bE\left[\langle f(\beta_{j_0}) \rangle'_K(1 -\delta\langle e^{\Delta\logl(Z_n)}\mathbb{I}_{A_K} \rangle_v )^{{L}}\right];\label{eq:cond_geom_series}
        \end{split}    \end{equation}
        where the first equality follows by adding and subtracting a term, the second one by using equation \eqref{eq:cond_cavity}, and the last one by using \eqref{eq:geom_partial} and the definition from \eqref{eq:gkdef}.
        %\textcolor{red}{Can you please fill steps that is the power becomes product of expectations etc?, I think one time showing the full calculation clearly that takes you from second to third line will be useful. Just for one time. }\ms{I added new lines to make this more clear and expanded the explanations below the equations.}
   % \textcolor{red}{It would be really helpful if the intermediate lines steps are actually filled in for the above calculation, I have once again forgotten the key things here, every time I need to check this from scratch, so best to write it once and for all?}
%    where in the last line we used the formula for the partial sums of the geometric series. \ps{Same comment here as before, bring in the elaboration steps from Proposition 1 proof, since we are switching the order of proofs of the propositions.}\ms{Done}

    %The key observation to prove this proposition is that the first term in the last line of \eqref{eq:cond_geom_series} is only a function of $Z_n^{(1)},\dots,Z_n^{(k)}$. To characterize the limit as $n\to\infty$ of this sum, it is then enough to recall that, 
 Using Corollary \ref{cor:cond_conv_dist_Z}, under $\bE\thermal{\cdot}_v$ we have that
    \begin{equation*}
        (Z_n^{(1)},\dots,Z_n^{(l)}) \xrightarrow{d} (\zeta_1,\dots,\zeta_l),
    \end{equation*}
    where, for $m\geq1$, $\zeta_m$ is as in \eqref{eq:def_zeta}. By Hypothesis \ref{hyp1} and the fact that $f$ and $\Delta\logl$ are continuous and bounded in $A_K$, the $g_k$ are all continuous and bounded. This and Corollary \ref{cor:cond_conv_dist_Z} imply that
    \begin{equation}\label{eq:cond_geom_sum_lim}
        \sum_{k\leq L} \bE \langle g_k(Z_n^{(1)},\dots,Z_n^{(k)}) \rangle_v \xrightarrow{n\to\infty} \sum_{k\leq L} \E g_k(\zeta_1,\dots,\zeta_k).
    \end{equation}
    
    For a proper constant $C > 0$ and assuming w.l.g. that $L$ is even, the following bound for the second term holds
   
    \begin{equation}\label{eq:cond_bound_error}
        \begin{split}
            \Big|\bE \left[ \langle f(\beta_{j_0}) \rangle'_K (1 - \delta \langle \exp \Delta\logl(Z_n)\mathbb{I}_{A_K} \rangle_v )^L\right]\Big| & \leq C \bE\Big[  (1 -\delta\langle \exp \Delta\logl(Z_n) \mathbb{I}_{A_K} \rangle_v )^L\Big]\\
            & = C \bE\Big[  \langle 1 -\delta \exp \Delta\logl(Z_n) \mathbb{I}_{A_K} \rangle_v^L\Big]\\
            & \leq C \bE \langle (1 - \delta\exp \Delta\logl(Z_n)\mathbb{I}_{A_K} )^L\rangle_v,
        \end{split}
    \end{equation}
    where in the first line we used that $f$ is bounded and in the last Jensen's inequality.  Let
    \begin{equation*}
        \tilde{A}_K := \{\|\zeta\|_\infty \leq K\},
    \end{equation*}
    and define
    \begin{equation*}
        \dlogl(\zeta_m) := -\frac{a_{dp}}{2}\beta_m^2 + \beta_m \tilde{\theta}_m + \beta_m \beta_{\star,j_0} t_\gamma + \beta_{\star,j_0} \bar{\theta}_m + c_g \beta_{\star,j_0}^2.
    \end{equation*}
  Here, $\dlogl(\zeta_m)$ is essentially $\dlogl(\beta_{j_0},\bbbeta)$ with $Y_1,\dots,Y_5$ replaced by their corresponding limits as in Lemma \ref{lem:conv_dist_Z}. Likewise, the event $\tilde{A}_K$ is the analogous of $A_K$ where $Y_1,\dots,Y_5$ have been replaced by their limits. In the same way as above, we then have that
   %\textcolor{red}{Question: What does $I_{A_K}$ in the limit mean? You mean the event on which $\| \zeta\|_\infty$ is bounded by K right? But we should have a different notation for this, like $A_\zeta$ or something like this, since $A_K$ refers to $\|Z_n\|_{\infty}$ being bounded right?}\ms{This is indeed ambiguous. I definied a new set and function for hte limits to distinguish them from the finite dimensional objects. I changed the notations of this section accordingly. I don't think this is used anywhere else, so we should be good.}
    \begin{equation}\label{eq:cond_n_lim_error}
        \bE \langle (1 - \delta\exp \{\Delta\logl(Z_n)\}\mathbb{I}_{A_K} )^L\rangle_v \xrightarrow{n\to\infty} \E (1 - \delta\exp \{\Delta\logl(\zeta_1)\}\mathbb{I}_{\tilde{A}_K} )^L.
    \end{equation}  
  %\textcolor{red}{Can you please define formally what you mean by $\Delta \cL(\zeta)$? Do you mean $\Delta \cL (\zeta_1)$?? I am very confused what the math expression for this is.}\ms{When a replica index is omitted, it is asumed to be $1$. But it is true that in this context, because notation is different, it is confusing. So I added it everywhere.}
    Thus, combining equations \eqref{eq:cond_geom_series} though \eqref{eq:cond_n_lim_error}, we have for every $L\geq 1$,
    \begin{equation}\label{eq:cond_limit1}
        \lim_{n\to\infty} \bE\langle f(\beta_{j_0})\rangle'_K = \delta\sum_{1\leq k\leq L} \E g_k(\zeta_1,\dots,\zeta_k) + \mathcal{O}\Big(\E(1 - \delta\exp \{\Delta\logl(\zeta_1)\}\mathbb{I}_{\tilde{A}_K} )^L\Big).
    \end{equation}
%    \textcolor{red}{The below needs a proper derivation, like can you show starting from the definition of $g_k(\cdot)$ that the LHS indeed converges to RHS as $L \rightarrow \infty$? Showing the derivation is critical since otherwise it seems out of the blue. Also, how exactly is the DCT being used. Dont you need to show the LHS is summable for infinite L?}\ms{In the new changes below, the computaiton is now more explicit.}
 
Now, by the definition of the functions $g_k$ and \eqref{eq:geom_partial} we have
    \begin{footnotesize}
        \begin{equation}\label{eq:bound-partial-sum}
        \begin{split}
            \delta \sum_{1\leq k\leq L} \E g_k(\zeta_1,\dots,\zeta_k) & = \delta \E_{\tilde{z}',\bar{z}}\left[\E_{\tilde \xi,\bar \xi}\int f(\beta)\exp\{\Delta\logl(\zeta_1)\}\mathbb{I}_{\tilde{A}_K}\mu(d\beta)\frac{1-\left(1-\delta\E_{\tilde \xi,\bar \xi}\int\exp\{\Delta\logl(\zeta_1)\}\mathbb{I}_{\tilde{A}_K}\mu(d\beta)\right)^L}{\delta\E_{\tilde \xi,\bar \xi}\int\exp\{\Delta\logl(\zeta_1)\}\mathbb{I}_{\tilde{A}_K}\mu(d\beta)}\right] \nonumber \\
            & = \E_{\tilde{z}',\bar{z}}\left[\frac{\E_{\tilde \xi,\bar \xi}\int f(\beta)\exp\{\Delta\logl(\zeta_1)\}\mathbb{I}_{\tilde{A}_K}\mu(d\beta)}{\E_{\tilde \xi,\bar \xi}\int\exp\{\Delta\logl(\zeta_1)\}\mathbb{I}_{\tilde{A}_K}\mu(d\beta)}\right] \nonumber\\
            & \hspace{0.5cm} -\E_{\tilde{z}',\bar{z}}\left[\frac{\E_{\tilde \xi,\bar \xi}\int f(\beta)\exp\{\Delta\logl(\zeta_1)\}\mathbb{I}_{\tilde{A}_K}\mu(d\beta)}{\E_{\tilde \xi,\bar \xi}\int\exp\{\Delta\logl(\zeta_1)\}\mathbb{I}_{\tilde{A}_K}\mu(d\beta)}\left(1-\delta\E_{\tilde \xi,\bar \xi}\int\exp\{\Delta\logl(\zeta_1)\}\mathbb{I}_{\tilde{A}_K}\mu(d\beta)\right)^L\right].
        \end{split}
     \end{equation}
    \end{footnotesize}
    We would like to bound the second term of the last line, which represents the error term resulting of the approximation of the full series $\delta \sum_{k\geq 1} g_k$ by its first $L$ terms $\delta \sum_{1\leq k \leq L} g_k$. For this we use that $f$ is a bounded function to get that, for some $K\geq0$,
    \begin{equation*}
        \left|\E_{\tilde{z}',\bar{z}}\left[\frac{\E_{\tilde \xi,\bar \xi}\int f(\beta)\exp\{\Delta\logl(\zeta_1)\}\mathbb{I}_{\tilde{A}_K}\mu(d\beta)}{\E_{\tilde \xi,\bar \xi}\int\exp\{\Delta\logl(\zeta_1)\}\mathbb{I}_{\tilde{A}_K}\mu(d\beta)}\left(1-\delta\E_{\tilde \xi,\bar \xi}\int\exp\{\Delta\logl(\zeta_1)\}\mathbb{I}_{\tilde{A}_K}\mu(d\beta)\right)^L\right]\right|
    \end{equation*}
    is smaller or equal to
    \begin{equation*}
        K \E_{\tilde{z}',\bar{z}}\left(1-\delta\E_{\tilde \xi,\bar \xi}\int\exp\{\Delta\logl(\zeta_1)\}\mathbb{I}_{\tilde{A}_K}\mu(d\beta)\right)^L.
    \end{equation*}
    Moreover, by Jensen's inequality, because the map $x\mapsto x^L$ is convex, we have that
    \begin{equation*}
        \begin{split}
            \E_{\tilde{z}',\bar{z}}\left(1-\delta\E_{\tilde \xi,\bar \xi}\int\exp\{\Delta\logl(\zeta_1)\}\mathbb{I}_{\tilde{A}_K}\mu(d\beta)\right)^L & = \E_{\tilde{z}',\bar{z}}\left(\E_{\tilde \xi,\bar \xi}\int(1-\delta\exp\{\Delta\logl(\zeta_1)\}\mathbb{I}_{\tilde{A}_K})\mu(d\beta)\right)^L \\
            & \leq \E\left(1-\delta\exp\{\Delta\logl(\zeta_1)\}\mathbb{I}_{\tilde{A}_K}\right)^L.
        \end{split}
    \end{equation*}
    Combining this with \eqref{eq:bound-partial-sum} we get that
    \begin{equation*}\label{eq:limit-plus-error-lavo}
    \begin{split}
        \delta \sum_{1\leq k\leq L} \E g_k(\zeta_1,\dots,\zeta_k) & = \E_{\tilde{z}',\bar{z}}\left[\frac{\E_{\tilde \xi,\bar \xi}\int f(\beta)\exp\{\Delta\logl(\zeta_1)\}\mathbb{I}_{\tilde{A}_K}\mu(d\beta)}{\E_{\tilde \xi,\bar \xi}\int\exp\{\Delta\logl(\zeta_1)\}\mathbb{I}_{\tilde{A}_K}\mu(d\beta)}\right] \\
        & \hspace{2cm} + \mathcal{O}\left(\E\left(1-\delta\exp\{\Delta\logl(\zeta_1)\}\mathbb{I}_{\tilde{A}_K}\right)^L\right).
    \end{split}
    \end{equation*}
   % \textcolor{red}{Proper elaboration of this last step---showing the intermediate steps---will be useful.} 
    Above, $(\tilde{\xi},\bar{\xi},\tilde{z}',\bar{z})$ have the same joint distribution as described for the $(\tilde{\xi}_1,\bar{\xi}_1,\tilde{z}',\bar{z})$'s in \eqref{eq:covariance_thetas}. We then conclude that, for every $L\geq1$,
    \begin{equation}\label{eq:cond_limit2}
        \lim_{n\to\infty} \bE\langle f(\beta_{j_0})\rangle'_K = \E_{\tilde{z}',\bar{z}}\left[\frac{\E_{\tilde \xi,\bar \xi}\int f(\beta)\exp\{\Delta\logl(\zeta_1)\}\mathbb{I}_{\tilde{A}_K}\mu(d\beta)}{\E_{\tilde \xi,\bar \xi}\int\exp\{\Delta\logl(\zeta_1)\}\mathbb{I}_{\tilde{A}_K}\mu(d\beta)}\right] + \mathcal{O}\Big(\E(1 - \delta\exp \Delta\logl(\zeta_1)\mathbb{I}_{\tilde{A}_K} )^L\Big).
    \end{equation}
    Note that, by the choice of $\delta$, we have that $1 - \delta\exp \Delta\logl(\zeta_1)\mathbb{I}_{\tilde{A}_K}$ is a.s smaller than $1/2$. Taking the limit $L\to\infty$ in \eqref{eq:cond_limit2}, by dominated convergence we then have that
    \begin{equation*}
        \E\thermal{f(\beta_{j_0})}'_K\xrightarrow{n\to\infty}\E_{\tilde{z}',\bar{z}}\left[\frac{\E_{\tilde \xi,\bar \xi}\int f(\beta)\exp\{\Delta\logl(\zeta_1)\}\mathbb{I}_{\tilde{A}_K}\mu(d\beta)}{\E_{\tilde \xi,\bar \xi}\int\exp\{\Delta\logl(\zeta_1)\}\mathbb{I}_{\tilde{A}_K}\mu(d\beta)}\right].
    \end{equation*}
%  \textcolor{red}{By $\Delta \cL$ on the RHS I assume you mean $\Delta \cL(\zeta)$?}\ms{I added explicitely the argument of this everywhere.} Furthermore, by taking the limit $K\to\infty$ we get
    \begin{equation}\label{eq:Kn_lim_mean}
        \lim_{K\to\infty}\lim_{n\to\infty} \E\thermal{f(\beta_{j_0})}'_K = \E_{\tilde{z}',\bar{z}}\left[\frac{\E_{\tilde \xi,\bar \xi}\int f(\beta)\exp\{\Delta\logl(\zeta_1)\}\mu(d\beta)}{\E_{\tilde \xi,\bar \xi}\int\exp\{\Delta\logl(\zeta_1)\}\mu(d\beta)}\right].
    \end{equation}
%    \textcolor{red}{It would be really helpful if intermediate steps for going from the above to the step below is shown. It is a simple gaussian integration argument, but it helps reviewers if we show steps so that they dont have to check accuracy of every step. Reviewers get really annoyed if they have to check accuracy of every step. Basically, I am asking to show the calculation explicitly, starting from recalling the $\Delta \cL(\zeta)$ definition, so reviewer can quickly see the calculation is correct.}
Notice that, by the definition of $\dlogl(\zeta_1)$ and those of $\tilde{\theta}_1$ and $\bar{\theta}_1$, the right hand side of \eqref{eq:Kn_lim_mean} is equal to 
    \begin{equation}\label{eq:aux_expect}
       \E_{\tilde{z}',\bar{z}}\left[\frac{\E_{\tilde \xi,\bar \xi}\int f(\beta)\exp\{ -\frac{a_{dp}}{2}\beta^2 + \beta (\tilde{\xi} + \tilde{z}') + \beta \beta_{\star,j_0} t_\gamma + \beta_{\star,j_0} (\bar{\xi}+\bar{z}) + c_g\beta_{\star,j_0}^2\}\mu(d\beta)\}\mu(d\beta)}{\E_{\tilde \xi,\bar \xi}\int\exp\{ -\frac{a_{dp}}{2}\beta^2 + \beta (\tilde{\xi} + \tilde{z}') + \beta \beta_{\star,j_0} t_\gamma + \beta_{\star,j_0} (\bar{\xi}+\bar{z}) + c_g\beta_{\star,j_0}^2\}\mu(d\beta)}\right];
    \end{equation}
    Recall the formula, valid for all $a,b\in\R$, for the generating function of two jointly Normal random variables $\xi$ and $\xi'$ of means $\mu_1,\mu_2\in\R$ and covariance matrix $\Sigma\in\R^2$
    \begin{equation}\label{eq:normal_mgf}
        \E_{\xi,\xi'}( \exp\{a \xi+b\xi'\} ) = \exp\{a\mu_1+b\mu_2 +(a,b)\Sigma (a,b)^{\top}/2\}.
    \end{equation}
    We can then use \eqref{eq:covariance_thetas} and \eqref{eq:normal_mgf} to integrate $\tilde \xi$ and $\bar \xi$ obtaining that \eqref{eq:aux_expect} is equal to
    \begin{equation}\label{eq:expectation}
       \E_{\tilde{z}}\left[\frac{\int f(\beta)\exp\{ -(a_{dp}+\tilde{q}-\tilde{v})\frac{\beta^2}{2}+(\tilde{m}-\bar{m}+t_\gamma)\beta\beta_{\star,j_0}+\sqrt{\tilde{q}}\beta\tilde{z}\}\mu(d\beta)}{\int\exp\{ -(a_{dp}+\tilde{q}-\tilde{v})\frac{\beta^2}{2}+(\tilde{m}-\bar{m}+t_\gamma)\beta\beta_{\star,j_0}+\sqrt{\tilde{q}}\beta\tilde{z}\}\mu(d\beta)}\right],
    \end{equation}
    where $\tilde{z} \sim \mathcal{N}(0,1)$, independent of everything else. 
 %\ps{A recall of equation numbers where these constants were defined wouldbe useful.}
 %with the constants above defined by equations \eqref{eq:def_tp}, \eqref{eq:def_lim_overlaps}, and \eqref{eq:def_zeta}; and $\tilde z := \tilde{z}'/\sqrt{\tilde{q}}$ which is a standard Gaussian variable independent of everything else. 
 Observe that, in this last equation, the terms $\beta_{\star,j_0} \bar{z}$ and $c_g$ in the exponents are missing as we have simplified the ones occurring in the numerator with those of the denominator. 
 %\textcolor{red}{First, it needs to be checked the above is correct, it seems to me it is? But remember we changed the $r_1,r_2,r_3$ definition to the current system of equations, to make things compact, so the below definition is wrong. }
 %Because this means \eqref{eq:alphasv} is correct which it should be. However, this means both \eqref{eq:FPE_all2} and Proposition \ref{prop:conv_order_param} are wrong (!!) and that needs to be fixed.}
 
From the above, we then conclude that
    \begin{small}
     \begin{equation*}
        \lim_{K\to\infty}\lim_{n\to\infty} \E\thermal{f(\beta_{j_0})}'_K = \E_{\tilde{z}}\left[\frac{\int f(\beta)\exp\{ -(a_{dp}+\tilde{q}-\tilde{v})\frac{\beta^2}{2}+(\tilde{m}-\bar{m}+t_\gamma)\beta\beta_{\star,j_0}+\sqrt{\tilde{q}}\beta\tilde{z}\}\mu(d\beta)}{\int\exp\{ -(a_{dp}+\tilde{q}-\tilde{v})\frac{\beta^2}{2}+(\tilde{m}-\bar{m}+t_\gamma)\beta\beta_{\star,j_0}+\sqrt{\tilde{q}}\beta\tilde{z}\}\mu(d\beta)}\right].
    \end{equation*}
    \end{small}
    Combining this with Lemmas \ref{lem:aprox_K} and \ref{lem:taylor_control},
     %\textcolor{red}{Explain in more detail relation of $(r_1,r_2,r_3)$ to order parameters.}\ms{I just added a comment pointing out that, from this derivation, we can see that the limits of the order parameters determine these constants. But I don't know what mroe to add.}   
    %noticing that 
   %and defining   $r_1 := a_{dp}+\tilde{q}-\tilde{v}$, $r_2:=\tilde{m}-\bar m$, and $r_3 := \tilde{q}$, 
   we conclude that
    \begin{equation*}
        \lim_{n\to\infty}\E\thermal{f(\beta_{j_0})} = \E_{Z}\thermal{f(\beta)}_{h,j_0},
    \end{equation*}
    where $\thermal{\cdot}_{h,j_0}$ is given by \eqref{eq:pbdef}.
    %refers to expectation under density proportional to 
    %$$e^{-(r_1/2)(b-m_{j_0})^2}, \,\,\text{where} \,\,m_{j_0}=\frac{r_2+t_\gamma}{r_1}\beta_{\star,j_0}+\frac{\sqrt{r_3}}{r_1}\tilde{Z}$$

%    \textcolor{red}{Final plan on outlining how this proposition statement should be: So statement of Prop 5a should be define all constants a certain way (including $r_1,r_2,r_3$ the way they are defined at the end of the proof here), then show the above result as final result. Then comes next subsection which will be Prop 5b. And at the end, the fact that this means the constants need to have a certain fixed point description. All of this should be a single master proposition.;}  
%As we can see from the derivation, the constants $(r_1,r_2,r_3)$ that determine the asymptotic marginal of the posterior are determined by the limits of the order parameters \eqref{eq:overlaps2}.
    
    This proves the proposition in the case in which $r=1$ and $f$ is bounded. The extension to general $r\geq1$ follows in the same way. 
    
%\textcolor{red}{First, revisit and fix statement of Prop 5 after this; main decision is where and how should the constants be defined for the first time. And remember to do same to prop 6 after done with next section, and write a line re proofs of props 3 and 4}
%By Propositions \ref{prop:2mon_Y} and \ref{prop:taylor_control} it is enough to prove it for $\bE\thermal{\cdot}'_K$ and then take the appropriate limits. We will do this for this expectation in a similar fashion as it was done in the proof of Proposition \ref{prop:hat_fix_point}.

\subsection{Proof of Proposition \ref{prop:hat_fix_point2}: leave-an-observation-out argument}\label{sec:leave_obs}

Recall definitions \eqref{eq:def_tT_u} and \eqref{eq:looposterior}--\eqref{eq:mean_observation} and that $\ba_\star=(\bX_1^{\top}\bbeta_\star,\hdots,\bX_n^{\top}\bbeta_\star)^{\top}, \ba_m = (\bX_1^{\top}\bbeta^{(m)},\hdots,\bX_n^{\top}\bbeta^{(m)})^{\top}$ Isolating the dependence of the log-likelihood on the last observation we get that
\begin{equation}
    \logl_{n,p}(\bB) = \logl_{n-1,p}(\bB) + u_n(\bX_n^\top\bB,a_{\star,n}),
\end{equation} 
where $a_{\star,n}$ denotes the $n$-th coordinate of $\bm{a}_\star$.
This implies that we can express the following expectation under the full posterior in terms of the leave-an-observation-out measure, defined in 
\eqref{eq:mean_observation}, according to
%\textcolor{red}{The following section had dummy index as $\ell$ and total index as $r$, I am changing them to $m$ and $\ell$ resp. Manuel, in your pass, please check I didnt miss anything.}
\begin{equation}\label{eq:loo_cavity}
    \E\langle g(a_{\star,n},a_{1,n},\dots,a_{\ell,n}) \rangle = \E \left[\frac{\left\langle g(a_{\star,n},a_{1,n},\dots,a_{\ell,n}) \exp\{\sum_{m\in[\ell]} u_n(a_{m,n},a_{\star,n})\}\right\rangle_o}{\left\langle \exp u_n(a_{1,n},a_{\star,n})\right\rangle_o^\ell}\right],
\end{equation}
where note that the denominator is obtained by noticing that $\langle \prod_{m=1}^{\ell} \exp \{u_n(a_{m,n},a_{\star,n}) \}\rangle_o = \langle \exp u_n(a_{1,n},a_{\star,n})\rangle_o^\ell $. %\textcolor{red}{TILL HERE.}
% \ps{If we added $a_{\star,i}$ on the left as the $(r+1)$-th argument, would anything change other than including the same as the $r+1$-th term on the RHS numerator within the $g$. If thats the only change, we should both write this above equation and Proposition 5 in this version. This more general result can be useful to statisticians, since they often want to compute joint functions of a $\bX_i^{\top}\bbeta^{(m)},\bX_i^{\top}\bbeta_{\star}$. Also it seems this general version is the one we use to obtain our system of equations. But I am leaving the changes for you since I am afraid introducing errors if I do it myself.}\ms{Yes, I agree that we need that generalisation to derive the equations. We should check it goes through in the same way.}

Note that on the right hand side, within the expectation, we have expectations w.r.t.~the leave-an-observation-out measure. The latter does not involve the $n$-th observation, thus a sample from the posterior induced by this measure will be independent of $\bX_n$. 
For calculating the limits of these expectations, it will be enough to track the joint distribution of the following $\ell+1$ variables

\begin{equation*}
    (a_{\star,n},a_{1,n},\hdots, a_{\ell,n})=(\bX_n^{\top}\bbeta_{\star},\bX_n^{\top}\bbeta^{(1)}, \hdots, \bX_n^{\top}\bbeta^{(\ell)}),
\end{equation*}
where $\bbeta^{(1)},\bbeta^{(2)},\hdots, \bbeta^{(\ell)}$ are independent samples from the leave-an-observation-out measure.
%\textcolor{red}{It is good to use the same notation but then in notation section, when we introduce $\bar{a}$'s and $\tilde{a}$'s we shouldnt. Using a single set of $a$'s to denote both full and LOO out measure samples notationally is fine since it is usually clear from context which one it means.}\ms{Now in the introduction, the definition of the $\bar{a}$ has been removed as it is not used. We still keep the $\tilde{a}$ as I think it is important to remark these do not depend on the removed coordinate. Do you agree with this?}
%\ps{should we use the same notation? $\bbeta^{(m)}$ has earlier been used to denote sample from the full posterior. Lets use a different notation? Like $\bbeta_o^{(m)}$ or something like this?}\ms{I sent you an email about this} 
By Gaussianity of $\bX_n$, and the fact that $\bbeta^{(m)}$ is independent of $\bX_n$, these are zero mean multivariate Gaussian, so it suffices to track only their covariance matrix.

Now the asymptotic representation theorem of Appendix \ref{app:asymp_representation} (Lemma \ref{lem:normal_approx_dist}), by Proposition \ref{prop:RS}, holds under the hypothesis of this proposition which states that the expectations of $Q_{11}$, $Q_{12}$, and $Q_{1\star}$ converge, and we obtain under the leave-an-observation-out measure, %yields that under the full posterior (i.e.~when $\bbeta^{(1)},\bbeta^{(2)},\hdots, \bbeta^{(\ell)}$ is sampled from the full posterior), the corresponding $a_{\star,n},a_{1,n},\dots,a_{\ell,n}$ satisfy \textcolor{red}{Is this under the full measure or LOO measure}\ms{I added one phrase specifying the measure in question.}
\begin{align}\label{eq:fittedvaluedecoupling}
    (a_{\star,n},a_{1,n},\hdots, a_{\ell,n}) \xrightarrow{d} (\theta_\star,\theta_1,\hdots, \theta_\ell).
\end{align}
 %and  $v_B,c_B,c_{BB_\star}$ defined as in \eqref{eq:FPE_all1}. 
We will use \eqref{eq:fittedvaluedecoupling} later in this section. Before we present the rest of the proof we comment on the different sources of randomness in the aforementioned display.

Next, we will need the following auxiliary result that proves concentration of   $Q_{11},Q_{12},Q_{1\star}$ defined in \eqref{eq:overlaps} under the leave-an-observation-out measure.
\begin{lemma}\label{lem:same_limit_lvo}     There exists a positive constant $C>0$ such that 
    \begin{equation*}
        \E\thermal{(Q_{11}-v_B)^2}_o, \, \E\thermal{(Q_{12}-c_B)^2}_o, \, \E\thermal{(Q_{1\star}-c_{BB_\star})^2}_o \leq \frac{C}{\sqrt{n}},
    \end{equation*}
%    \textcolor{red}{Trying to understand the meaning of the above: do you mean these are the order parameters when the $\bbeta$'s that come up in their definition is drawn from the leave-an-observation-out measure?}\ms{Exactly, this says that, under the loo measure, the order parameters converge to the same constant as in the full measure.}
    where $v_B,c_B,c_{BB_\star}$ are defined as in \eqref{eq:intermlimit}.
\end{lemma}
In the above lemma, since the expectations are taken under the leave-an-observation-out measure, it means that $Q_{11} = \bbeta^{\top}\bbeta/n$, where $\bbeta$ is a sample from the leave-an-observation-out measure.
\begin{proof}
%\textcolor{red}{TILL HERE}   
This proof is analogous to that of Lemma \ref{lem:equivalence_mean}. Here we will prove the conclusion for $Q_{11}$ but the result follows analogously for the others parameters. First notice that Proposition \ref{prop:RS} applies to the leave-an-observation-out posterior. We then have that, for some $v'_B$,
    \begin{equation*}
        \E\thermal{(Q_{11}-v'_B)^2}_o \leq \frac{C}{\sqrt{n}}.
    \end{equation*}
We thus only need to prove that $v'_B=v_B$.
    
    To see this, for every $t\in[0,1]$, let $\thermal{\cdot}_t$ be expectation with respect to the posterior corresponding to the log-likelihood
    \begin{equation*}
        \logl_{n,p,t}(\bB) := t u_n(\bX_n^\top\bB,a_{\star,n}) + \sum_{i \in[n-1]} u_i(\bX_i^\top\bB,a_{\star,i}).
    \end{equation*}
    That is, for each $f:\R^p\mapsto\R$, we define
    \begin{equation*}
        \thermal{f(\bbeta)}_t := \frac{\int f(\bbeta)\exp\{\logl_{n,p,t}(\bbeta)\} \prod_{j\in[p]}\mu(d\beta_j)}{\int \exp\{\logl_{n,p,t}(\bbeta)\} \prod_{j\in[p]}\mu(d\beta_j)}.
    \end{equation*}
    Clearly $\thermal{\cdot}_0$ is equal to $\thermal{\cdot}_o$ (expectation under the leave-an-observation-out posterior) and $\thermal{\cdot}_1$ to $\thermal{\cdot}$ (expectation under the full posterior). Now, notice that 
    \begin{equation*}
        \begin{split}
            \frac{d}{ds}\E\thermal{\hat{Q}}_s\big|_{s=t} & = \E\left[\frac{d}{ds}\frac{\int \hat{Q}\exp{\{\logl_{n,p,s}(\bbeta)\}} \prod_{j\in[p]}\mu(d\beta_j)}{\int \exp{\{\logl_{n,p,s}(\bbeta)\}} \prod_{j\in[p]}\mu(d\beta_j)}\right]\big|_{s=t} \\
            & = \E\left[\frac{\int \hat{Q}(\Delta L + E_n)\exp{\{\logl_{n,p,s}(\bbeta)\}} \prod_{j\in[p]}\mu(d\beta_j)}{\int \exp{\{\logl_{n,p,s}(\bbeta)\}} \prod_{j\in[p]}\mu(d\beta_j)}\right]\big|_{s=t} \\
            & \quad \quad - \E\left[ \frac{\int \hat{Q}\exp{\{\logl_{n,p,s}(\bbeta)\}} \prod_{j\in[p]}\mu(d\beta_j)\int u_n(a_{1,n},a_{\star,n})\exp{\{\logl_{n,p,s}(\bbeta)\}} \prod_{j\in[p]}\mu(d\beta_j)}{\left(\int \exp{\{\logl_{n,p,s}(\bbeta)\}} \prod_{j\in[p]}\mu(d\beta_j)\right)^2} \right]\big|_{s=t} \\
            & = \E\thermal{(\Delta L + E_n)(\hat{Q}-\langle \hat{Q} \rangle_t)}_t;
        \end{split}
    \end{equation*}
    where, for the first equality, we used \cite[Proposition A.2.1]{talagrand2010mean} to exchange differentiation and expectation and, for the second one, we used that $\frac{d}{ds}\logl_{n,p,s}(\bbeta) = (\Delta L + E_n)$. Then, for every fixed $t\in[0,1]$,
    \begin{equation*}
        \begin{split}
            \Big|\frac{d}{ds}\E\thermal{\hat{Q}}_s\big|_{s=t}\Big| 
             \leq \sqrt{\E\thermal{u_n^2(a_{\star,n},a_{1,n})}_t \E\thermal{(\hat{Q}-\langle \hat{Q} \rangle_t)^2}_t} \leq \frac{C}{n^{1/4}},
        \end{split}
    \end{equation*}
%\textcolor{red}{Please elaborate fully how the first equality is achieved. Also I dont understand the thing on the RHS of the equality since $u_n$ is usually something with two arguments, what did you actually mean here?}\ms{I added explicitly the argument of $u_n$ above. I also added the derivation for the equality in question.}{\color{blue}
where, for the last inequality, we used that Propositions \ref{prop:RS} and \ref{prop:control_mom_conv} can be straightforwardly adapted to $\thermal{\cdot}_t$. From this and the Mean Value Theorem we conclude that $v'_B=v_B$ which proves the lemma. 
\end{proof}

\subsubsection{Proof of Proposition \ref{prop:hat_fix_point2}}
\label{subsec:prop2}
%\textcolor{red}{This section is essentially the same argument as in B.1. So am skipping it, will come back to it once B.1 is done and I agree with everything there. Then my pass on below will be faster. The index changes to $m,l$ where necessary need to be done in below, I havent done that yet.}
 We first prove this for $g$ continuous and bounded.
   For simplicity we will write the proof with $l = 1$ but the derivation follows in a completely similar way for larger values of $l$. Let $M\geq1$. For conciseness of notation, below we use   $u(x,a_{\star,n})$ to denote $u_n(x,a_{\star,n})$ as defined in \eqref{eq:glmhamilton}.  For each $k\le M$ define
    \begin{equation}\label{eq:fkloo}
        f_{k}(a_{1,n},\dots,a_{k,n},a_{\star,n}) := g(a_{1,n}) \exp{u(a_{1,n},a_{\star,n})} \prod_{m=2}^{k} (1 - \exp{u(a_{m,n},a_{\star,n})});
    \end{equation}
    where we let $\prod_{j=2}^1 (\cdots) =1$. By Hypothesis \ref{hyp1} and the fact that $g$ is continuous and bounded, the functions $f_k$ are also continuous and bounded.

    By equation \eqref{eq:loo_cavity} and the geometric series formula,
    \begin{equation}
        \begin{split}
            \E\langle g(a_{1,n})\rangle &= \E\left[\langle g(a_{1,n})\rangle(1-(1-\langle e^{u(a_{1,n},a_{\star,n})}\rangle_o)^M)\right] + \E\left[\langle g(a_{1,n})\rangle(1-\langle e^{u(a_{1,n},a_{\star,n})}\rangle_o)^M\right]\\
            &= \E\left[\langle g(a_{1,n})e^{u(a_{1,n},a_{\star,n})}\rangle_o\frac{1-(1-\langle e^{u(a_{1,n},a_{\star,n})}\rangle_o)^M}{\langle e^{u(a_{1,n},a_{\star,n})}\rangle_o}\right] + \E\langle g(a_{1,n})\rangle(1-\langle e^{u(a_{1,n},a_{\star,n})}\rangle_o)^M\rangle\\
            & = \sum_{1\leq k\le M} \E \langle f_k(a_{1,n},\dots,a_{k,n},a_{n\star}) \rangle_o + \E\left[\langle g(a_{1,n}) \rangle(1 -\langle e^{u(a_{1,n},a_{\star,n})} \rangle_o )^M\right].\label{eq:geom_series}
        \end{split}
    \end{equation}
    As noted above, by Hypothesis \ref{hyp1} and the fact that $g$ is continuous and bounded, the $f_k$ are all continuous and bounded. Then, by \eqref{eq:fittedvaluedecoupling}, 
        \begin{equation}\label{eq:geom_sum_lim}
        \sum_{1\leq k \leq M} \E \langle f_k(a_{1,n},\dots,a_{k,n},a_{\star,n}) \rangle_o \xrightarrow{n\to\infty} \sum_{1\leq k \leq M} \E f_k(\theta_1,\dots,\theta_k,\theta_\star),
    \end{equation}
%    {\color{blue} Without loss of generality, assume $M$ is even.} 
where recall the definition of $\theta_m$'s and $\theta_\star$ from \eqref{eq:thetas}. 
For a proper constants $K > 0$, the second term can be bounded in the following way
    \begin{equation}\label{eq:bound_error}
        \begin{split}
            \Big|\E \left[ \langle g(a_{1,n}) \rangle (1 -\langle \exp{u(a_{1,n},a_{\star,n})} \rangle_o )^M\right]\Big| & \leq K \E  (1 -\langle \exp{u(a_{1,n},a_{\star,n})} \rangle_o )^M\\
            & \leq K \E \langle (1 - \exp{u(a_{1,n},a_{\star,n})} )^M\rangle_o,
        \end{split}
    \end{equation}
    where in the first line we used that $g$ is bounded and in the second Jensen's inequality. %\textcolor{red}{Minor point but dont you need M to be either odd or even for the correct Jensen to apply?}
    Now, because $(1 - \exp{u(a_{1,n},a_{\star,n})} )^M$ is a continuous and bounded function of $(a_{1,n},a_{\star,n})$ by Lemma \ref{lem:normal_approx_dist} we have that
    \begin{equation}\label{eq:n_lim_error}
        \E \langle (1 - \exp{u(a_{1,n},a_{\star,n})} )^M\rangle_o \xrightarrow{n\to\infty} \E  (1 -\exp{u(\theta_1,\theta_\star)} )^M,
    \end{equation}
    where on the RHS we use $u(x,\theta_\star)$ to denote $\tT(\theta_\star)x - A(x)$
 with $\tT(\theta_\star) = f(\theta_\star,e)$ and $e \sim \text{Unif}(0,1)$ independent of everything else.
 %As in Section \ref{sec:auxiliary}, let $G$ stand for  the joint distribution of $(\xi_{B_\star},z_{BB_\star})$. 
Thus, combining equations \eqref{eq:geom_series} through \eqref{eq:n_lim_error}, we have so far that for every fixed $M\geq 1$,
    \begin{equation}\label{eq:limit1}
        \lim_{n\to\infty} \E\langle g(a_{1,n})\rangle = \sum_{1\leq k\leq M} \E f_k(\theta_1,\dots,\theta_k,\theta_\star) + \mathcal{O}\Big(\E(1 -\exp{u(\theta_1,\theta_\star)} )^M\Big).
    \end{equation}
    Because this holds for every fixed $M\geq1$, taking the limit $M\to\infty$ after the limit in $n$, by dominated convergence theorem, which holds because the functions $f_k$ are continuous and bounded, we get that 
    \begin{equation}\label{eq:limit2}
        \sum_{1\leq k\leq M} \E f_k(\theta_1,\dots,\theta_k,\theta_\star) \xrightarrow{M\to\infty} \sum_{k\geq1} \E f_k(\theta_1,\dots,\theta_k,\theta_\star) \quad \text{and} \quad \E(1 -\exp{u(\theta_1,\theta_\star)} )^M \xrightarrow{M\to\infty} 0.
    \end{equation}
 Note that
 $\E(\cdot)$ denotes expectation w.r.t. $(G,e,\xi_1,\xi_2,...)$, and we see that 
 %\textcolor{red}{I will change notation for the $\xi$ integrals below.}
    \begin{equation}\label{eq:limit3}
        \begin{split}
            \sum_{k\geq1} \E f_k(\theta_1,\dots,\theta_k,\theta_\star) & = \sum_{k\geq1} \E_{G\otimes e}\int \phi(d\xi_1)\cdots \phi(d\xi_k)g(\theta_1) \exp{u(\theta_1,\theta_\star)} \prod_{m=2}^{k} (1 - \exp{u(\theta_m,\theta_\star)}) \\
            & = \sum_{k\geq1} \E_{G\otimes e}\int \phi(d\xi_1) g(\theta_1) \exp{u(\theta_1,\theta_\star)} \prod_{m=2}^{k} \int \phi(d\xi_m)(1 - \exp{u(\theta_m,\theta_\star)}) \\
            & = \sum_{k\geq1} \E_{G\otimes e}\int \phi(d\xi_1) g(\theta_1) \exp{u(\theta_1,\theta_\star)} \left(\int \phi(d\xi_2)(1 - \exp{u(\theta_2,\theta_\star)})\right)^{k-1} \\
            & = \E_{G\otimes e}\int \phi(d\xi_1) g(\theta_1) \exp{u(\theta_1,\theta_\star)} \sum_{k\geq1} \left(\int \phi(d\xi_2)(1 - \exp{u(\theta_2,\theta_\star)})\right)^{k-1} \\
            & = \E_{G\otimes e}\left(\frac{\int \phi(d\xi_1) g(\theta_1) \exp{u(\theta_1,\theta_\star)}}{\int \phi(d\xi_2) \exp{u(\theta_2,\theta_\star)}}\right)\\
            & = \E_{G\otimes e}\thermal{g(\theta_1)}_s,
        \end{split}  \end{equation}\label{eq:torefearlier}
where $\phi(d\xi_m)/d\xi_m$ denotes the standard Gaussian density and for  
%used to denote the standard Gaussian density,    where in
the fourth equality we used dominated convergence to exchange summation and integration and in the fifth we used the formula for geometric series.
%\textcolor{red}{Me to include reference to the $\langle \rangle_s$ definition suitably and to fix Prop 6 statement also suitably.}
%    \textcolor{red}{My main point was there are exponential terms on RHS of \eqref{eq:fkloo}. However, they vanish in RHS below. This must be because of suitable integration by parts that you performed and because of definition of $\langle \rangle_s$, like in the previous subsection. I guess this just needs showing more computations was my original point, and that still stands.}
        
        %\textcolor{red}{Like in the previous subsection, having the details for all the computations along the way for the below will be very important to not lose the reader. Please include.}
    Then, combining equations \eqref{eq:limit1}--\eqref{eq:limit3} we obtain that
    \begin{align}\label{eq:loofinalcalc}
         \sum_{1\leq k\leq M} \E f_k  \xrightarrow{M\to\infty} =\E_{G \otimes e}\thermal{g(\theta_1)}_{s}.
    \end{align}
    Finally, to see that the second term goes to $0$, note that $1 - \exp{u(\theta_1,\theta_\star)} < 1$ since $u$ is non-positive a.s. The limit then follows by another use of dominated convergence.

    Notice that if $g$ is polynomially bounded, by Proposition \ref{prop:control_mom_conv} we have that $g(a_n)$ is uniformly integrable. We can thus, for computing its expectation, approximate it by a continuous and bounded function. The same proof as above subsequently applies. 
    %The conclusion of Proposition \ref{prop:hat_fix_point} thus follows by the preceding argument for continuous and bounded functions.

   % \textcolor{red}{I need to check the last part of this section after Manuel elaborates, and go back and fix statement of Proposition 6.}

%\textcolor{red}{Comment on July 13 for myself: Till here; Appendices A and B will be final once my outstanding comments are addressed in the aforementioned sections. A main issue is the proof of Proposition 1 -- is this circular? If it is not, then the proofs are correct and I am quite happy with the above two sections in this case.}
   % \textcolor{red}{The Proposition 3,4,5,6 problem still remains, I need to figure out how to fix this -- will need one last pass on this appendix B.}

%\ps{Done till here, modulo checking two geometric series steps.}

% ------------------------------------------------------------------------
% ------------------------------------------------------------------------
% ------------------------------------------------------------------------

\section{Proofs of intermediate results for Proposition \ref{prop:fix_point2}}\label{app:approximation_cavity_measure}
\subsection{Proof of Lemma \ref{lem:aprox_K}}\label{subsec:conditioning}
For this result, we require some uniform bounds on $\E\thermal{Y_1^2},\dots,\E\thermal{Y_5^2}$, as established in the following.
\begin{lemma}\label{lem:2mon_Y}
  Assume Hypothesis \ref{hyp1} and \ref{hyp2} hold. Then, there exists a constant $C > 0$ s.t., for all $n\geq1$, we have $\E\thermal{Y_1^2},\dots,\E\thermal{Y_5^2} \leq C$. 
\end{lemma}
\begin{proof}
    We present the proof for $Y_1$. In this case, we have
    \begin{equation*}
        \E\thermal{Y_1^2} = \sum_{i\in[n]}\E\thermal{X_{ij_0}^2F_{ii}} + \sum_{i\neq i'\in[n]}\E\thermal{X_{ij_0}X_{i'j_0}F_{ii'}} = n \E\thermal{X_{1j_0}^2F_{11}} + n(n-1) \E\thermal{X_{1j_0}X_{2j_0}F_{12}};
    \end{equation*}
    where, for $i,j\in[n]$, we denote $(\tT_i(\tasi)-A'(\ta_{i}))(\tT_j(\ta_{\star,j})-A'(\ta_j))$  by $F_{ij}$; where $\tT_i$ is defined as in \eqref{eq:def_tT_u} and $\ta_i,\ta_{\star,i}$ are defined as in  \eqref{eq:ta_is}, but with $\tbbeta^{(i)}$ referring to a random sample from the full posterior with the $j_0$-th coordinate left out. The first term on the RHS is easy to bound. Indeed, by Hypothesis \ref{hyp1} there exists $K > 0$ such that 
    \begin{equation}\label{eq:cont_diag_Y1}
        n \E\thermal{X_{1j_0}^2F_{11}} \leq  n K (\E X_{1j_0}^2 + \E( X_{1j_0}^2 \tT^2_1(\ta_{\star,1}) )) \leq  n K \left(\E X_{1j_0}^2 + \sqrt{\E( X_{1j_0}^4) \E(\tT^4_1(\ta_{\star,1}) )}\right) \leq K',
    \end{equation}
    where we use that, because $\tTp_1$ is bounded, $\E(\tT^4_1(\ta_{\star,1}) ) \leq K ( 1 + \E\ta_{\star,1}^4 )$ to see that it is a bounded sequence. For the second term, by Gaussian integration by parts, as in \cite[Appendix A.4]{talagrand2010mean}, with respect to $X_{1j_0}$ and $X_{2j_0}$ we obtain that $n^2\E\thermal{X_{1j_0}X_{2j_0}F_{12}}$ is equal to %\textcolor{red}{How do you arrive at the below? is it a sequential Gaussian integration by parts conditioning on one $X$ at a time? Also several terms below retain $F_{12}$ how are you bounding that? Wasnt there a direct C-S way of showing this, is the Gaussian integration by parts really necessary? I am also very confused how $\beta_{\star,j_0}$ and $\beta_{j_0}$ appear in this isolated manner below.}\ms{This is multivariate Gaussian integration by parts. I added a citation to the proposition from Talagrand's book. Here C-S yields a bound that is $\mathcal{O}(n)$ which is not good enough. Gaussian integration by parts lets you show a much tighter bound. The terms that appear with $F_{12}$ are bounded as before. Notice that this is okay to do after GIP has introduced the factor $1/n^2$ that cancels the $n(n-1)$. If oyu do it before GIP, this does not work. The $\beta_{\star,j_0}$ and $\beta_{j_0}$ just appear because this are the coefficients that multiply $X_{1j_0}$ and $X_{2j_0}$ with respect to which you do GIP.}
    \begin{equation*}
        \begin{split}
             & \beta_{\star,j_0}^2\E\thermal{F_{12}\tT'_1\tT'_2 a_{1,1}a_{2,1}}
            - \beta_{\star,j_0}^2\E\thermal{F_{12}\tT'_1\tT'_2 a_{1,1}a_{2,2}} 
            - \beta_{\star,j_0}\E\thermal{F_{12}\tT'_1 a_{1,1}{A_2'}^{(1)} \beta^{(1)}_{j_0}} \\
            & + \beta_{\star,j_0}\E\thermal{F_{12}\tT'_1 a_{1,1}{A_2'}^{(1)} \beta^{(2)}_{j_0}}
            - \beta_{\star,j_0}\E\thermal{F_{12}{A_1'}^{(1)}\beta^{(1)}_{j_0}\tT'_2 a_{2,1}}
            + \beta_{\star,j_0}\E\thermal{F_{12}{A_1'}^{(1)}\beta^{(1)}_{j_0}\tT'_2 a_{2,2}} \\
            & - \E\thermal{F_{12}{A_1'}^{(1)}{A_2'}^{(1)}\beta^{(1)2}_{j_0}}
            + \E\thermal{F_{12}{A_1'}^{(1)}{A_2'}^{(2)}\beta^{(1)}_{j_0}\beta^{(2)}_{j_0}}
            - \beta_{\star,j_0}^2\E\thermal{F_{12}\tT'_1\tT'_2 a_{1,2}a_{2,1}} \\
            & - \beta_{\star,j_0}^2\E\thermal{F_{12}\tT'_1\tT'_2 a_{1,2}a_{2,2}}
            + 2 \beta_{\star,j_0}^2\E\thermal{F_{12}\tT'_1\tT'_2 a_{1,2}a_{2,3}}
            + \beta_{\star,j_0}\E\thermal{F_{12}\tT'_1 a_{1,2}{A_2'}^{(1)} \beta^{(1)}_{j_0}} \\
            & + \beta_{\star,j_0}\E\thermal{F_{12}\tT'_1 a_{1,2}{A_2'}^{(2)} \beta^{(2)}_{j_0}}
            - 2 \beta_{\star,j_0}\E\thermal{F_{12}\tT'_1 a_{1,2}{A_2'}^{(3)} \beta^{(3)}_{j_0}} 
            + \beta_{\star,j_0}\E\thermal{F_{12}{A_1'}^{(2)}\beta^{(2)}_{j_0}\tT'_2 a_{2,1}} \\
            & + \beta_{\star,j_0}\E\thermal{F_{12}{A_1'}^{(2)}\beta^{(2)}_{j_0}\tT'_2 a_{2,2}}
            - 2 \beta_{\star,j_0}\E\thermal{F_{12}{A_1'}^{(2)}\beta^{(2)}_{j_0}\tT'_2 a_{2,3}}
            - \E\thermal{F_{12}{A_1'}^{(2)}\beta^{(2)}_{j_0}{A_2'}^{(1)}\beta^{(1)}_{j_0}} \\
            & - \E\thermal{F_{12}{A_1'}^{(2)}{A_2'}^{(2)}\beta^{(2)2}_{j_0}}
            + 2 \E\thermal{F_{12}{A_1'}^{(2)}\beta^{(2)}_{j_0}{A_2'}^{(3)}\beta^{(3)}_{j_0}};
        \end{split}
    \end{equation*}
    where we heavily relied in notations introduced in previous sections to compress this expression. Despite the complexity of this expression, it is easy to notice that in every term, all the finite moments of all factors inside the expectation are uniformly bounded. These bounds follow, similarly to \eqref{eq:cont_diag_Y1}, from Hypothesis \ref{hyp1}, Lemma \ref{lem:cont_mom}, and Proposition \ref{prop:control_mom_conv}. From this and multiple applications of Cauchy-Schwarz, we conclude that for some $K'>0$
    \begin{equation}\label{eq:cont_nondiag_Y1}
        n(n-1) \E\thermal{X_{1j_0}X_{2j_0}F_{12}} \leq K'.
    \end{equation}
    Combining equations \eqref{eq:cont_diag_Y1} and \eqref{eq:cont_nondiag_Y1} we see that $\E\thermal{Y_1^2} \leq C$. A completely analogous argument proves the result for $Y_2$.

    The random variables $Y_3$ and $Y_4$ have second moments that are easier to bound. For example, by Hypothesis \ref{hyp1} we have $\E\thermal{Y_3^2} \leq d_2 \E ||\bX_{\bullet j_0}||^4$; where the RHS~is easily seen to be bounded. A similar argument applies for $Y_4$.

    Finally, for $Y_5$ observe that by Hypothesis \ref{hyp1} and Cauchy-Schwarz 
$|Y_5| \leq d_1 \sqrt{||\tbX\tbbeta||^2||\bX_{\bullet j_0}||_4^4}$. We then have
    \begin{equation}\label{eq:Y5def}
        \E\thermal{Y_5^2} \leq d_1^2   \sqrt{\E\thermal{||\tbX\tbbeta||^4} \E\left(\sum_{i\in[n]}X_{ij_0}^4\right)^2}.
    \end{equation}
    %\textcolor{red}{Should be $d_1^2$ on RHS?}
    where $\E\left(\sum_{i\in[n]}X_{ij_0}^4\right)^2$ is easily seen to be of order $\mathcal{O}(n^{-2})$. Finally, combining the fact that $||\tbX||_{op}$ is bounded with very high probability by $3+\delta$ and Corollary \ref{cor:bound_mom_norm}, we obtain that $n^{-2}\E\thermal{||\tbX\tbbeta||^4}$ is uniformly bounded. This concludes the proof.
\end{proof}

The conclusion of Lemma \ref{lem:aprox_K} is then given by an application of Proposition \ref{prop:comparison_conditional} in the special case of $||f||_\infty<\infty$ and letting $\varepsilon_n:= 2\sqrt{\proba(B^c)}$ and $\varepsilon'_K:=2\sqrt{{\bar{\E}}\thermal{\mathbb{I}_{A^c_K}}}$. The fact that $\varepsilon'_K$ goes to $0$ as $K\to\infty$ is a direct consequence of Lemma \ref{lem:2mon_Y}.
%\textcolor{red}{To check Proposition 10 applies when I get there.}

% ------------------------------------------------------------------------
% ------------------------------------------------------------------------
% ------------------------------------------------------------------------

\subsection{Proof of Lemma \ref{lem:taylor_control}}

For proving this we will use the following version of \cite[Proposition 4.5]{barbier2022marginals}.
\begin{lemma}\label{lem:approx_hamil}
    For every $f:\R\mapsto\R$ bounded we a.s. have that
    \begin{equation*}
        |\thermal{f(\beta_{j_0})}_K-\thermal{f(\beta_{j_0})}'_K| \leq 2 ||f||_\infty (\thermal{|E_n|}_K+\thermal{|E_n|}'_K),
    \end{equation*}
    where $E_n$ as defined in \eqref{eq:taylor_error}.
    \end{lemma}

By Lemma \ref{lem:approx_hamil}, it suffices to control $\bE\thermal{|E_n|}_K$ and $\bE\thermal{|E_n|}'_K$ to bound the effect of dropping the term $E_n$. If we define
\begin{equation*}
    \begin{split}
         E_{n,1} & :=  -\frac{\beta_{j_0}^3}{6}\sum_{i\in[n]} A'''(\xi_i) X_{ij_0}^3, \,\, E_{n,2} := \frac{\beta_{\star,j_0}^3}{6}\sum_{i\in[n]} \tT'''(\chi_i) \tbX_i^{\top}\tbbeta X_{ij_0}^3, \\
         E_{n,3} := & \frac{\beta_{\star,j_0}^2\beta_{j_0}}{2}\sum_{i\in[n]} \tT''(\tbX_{i}^{\top}\tbbeta_\star) X_{ij_0}^3, \mbox{ and } E_{n,4} := \frac{\beta_{\star,{j_0}}^3\beta_{j_0}}{6}\sum_{i\in[n]} \tilde{T}'''(\chi_i)X_{i{j_0}}^4
    \end{split}
\end{equation*}
we then have that $E_n = E_{n,1} + E_{n,2} + E_{n,3} + E_{n,4}$ and we can thus bound the absolute values of these four terms separately. We will now do this in an orderly fashion.

For the first term, by Hypothesis \ref{hyp1} and exchangeability, we have that there is some constant $C > 0$ such that 
\begin{equation*}    
    \bE\thermal{|E_{n,1}|}_K \leq \frac{d_2}{6} \sum_{i\in[n]} \bE\left(\thermal{|\beta_{j_0}|^3} |X_{ij_0}|^3\right) \leq \frac{d_2K^3}{6} \, n\bE |X_{1j_0}|^3 \leq \frac{C}{n^{1/2}}.
\end{equation*}

Notice that Proposition \ref{prop:exp_bound_norm} and Corollary \ref{cor:bound_mom_norm} easily extend to $\thermal{\cdot}_K$. For the second term, we use Hypothesis \ref{hyp1} and Caucy-Schwarz to see there is a constant $C' > 0$ such that
\begin{equation*}
    \begin{split}    \bE\thermal{|E_{n,2}|}_K & \leq \frac{d_1|\beta_{\star,j_0}|^3}{6} \sum_{i\in[n]} \bE\left( \thermal{|\tbX_i^{\top}\tbbeta|}_K |X_{ij_0}|^3\right) \\
        & \leq \frac{d_1|\beta_{\star,j_0}|^3}{6} \left( \bE\left(\thermal{||\tbX \tbbeta||^2}_K \sum_{i\in[n]}X_{ij_0}^6 \right)\right)^{1/2} \\
        & \leq \frac{3d_1|\beta_{\star,j_0}|^3}{2}\left( n\bE\left(\thermal{||\tbbeta||^2}_K X_{1j_0}^6 \right)\right)^{1/2} \\
        & \leq \frac{3d_1 |\beta_{\star,j_0}|^3}{6} \left(n^2\bE\thermal{||\tbbeta||^4}_K \bE X_{1j_0}^{12} \right)^{1/4} \leq \frac{C'}{n^{1/2}},
    \end{split}
\end{equation*}
%\textcolor{red}{Wish to clarify details in the above set of inequalities, how did you get rid of $\tbX$ in particular?}\ms{Under $\bE(\cdot)$ the operator norm of $\tbX$ is bounded by $3$ which means that $||\tbX\tbbeta||\leq3||\tbbeta||$.}
where in the last inequality we used the extension of Corollary \ref{cor:bound_mom_norm} to $\thermal{\cdot}_K$. While for the third term, we similarly have for some constant $C''>0$
\begin{equation*}
    \bE\thermal{|E_{n,3}|}_K \leq \frac{d_1\beta_{\star,{j_0}}^2}{2} \sum_{i\in[n]} \bE\left(\thermal{|\beta_{j_0}|}_K |X_{i{j_0}}|^3\right) \leq \frac{d_1\beta_{\star,{j_0}}^2K}{2} \, n\bE|X_{1{j_0}}|^3 \leq \frac{C''}{n^{1/2}}.
\end{equation*}
Finally, for the fourth, we readily see there exists a constant $C''''>0$ such that
\begin{equation*}
    \bE\thermal{|E_{n,4}|}_K \leq \frac{d_1|\beta_{\star,{j_0}}^3|}{6} \bE\left(\thermal{|\beta_{j_0}|}_K \sum_{i\in[n]} X_{i{j_0}}^4\right) \leq \frac{d_1|\beta_{\star,{j_0}}^3|K}{6} n \bE X_{i{j_0}}^4 \leq \frac{C''''}{n}.
\end{equation*}

Analogous bounds for $\bE\thermal{|E_{n,1}|}'_K$, $\bE\thermal{|E_{n,2}|}'_K$, $\bE\thermal{|E_{n,3}|}'_K$, and $\bE\thermal{|E_{n,4}|}'_K$ follow from straightforward adaptations of the previous bounds and generalizations of Proposition \ref{prop:exp_bound_norm} and Corollary \ref{cor:bound_mom_norm} to the measure $\thermal{\cdot}'_K$. 

Finally, by the definition of event $A_K$, $|\beta_{j_0}|\leq K$ under $\thermal{\cdot}_K$ and $\thermal{\cdot}'_K$.  Let $\bar{f}$ be the restriction of $f$ to the interval $[-K,K]$. Because $|\beta_{j_0}|\leq K$, then 
\begin{equation*}
    \bE\thermal{f(\beta_{j_0})}'_K = \bE\thermal{\bar{f}(\beta_{j_0})}'_K \,\, \text{and} \,\, \bE\thermal{f(\beta_{j_0})}_K = \bE\thermal{\bar{f}(\beta_{j_0})}_K.
\end{equation*}
Finally, $\|\bar{f}\|_\infty < \infty$ because $f$ is continuous and $[-K,K]$ is compact. We then have the desired result.

\subsection{Concentration of functions of the fitted values}\label{sec:conv_spin_RS}
%\textcolor{red}{Point to where this appendix results are used. }

From a straight forward adaptation of the proof of Proposition \ref{prop:hat_fix_point} for the case in which two observations are left out and using Lemma \ref{lem:normal_approx_dist_2spin}, we obtain the following which proves the asymptotic decorrelation of the different coordinates of functions of  $a_1,\dots,a_n$. For this, notice that if $f,g:\R^3\mapsto\R$ are continuous and polynomially bounded, then $F:\R^6\mapsto\R$ given by $F(x,y)=f(x)g(y)$ is also continuous and polynomially bounded.
\begin{lemma}\label{lem:decorr}
 Under the assumptions of Proposition \ref{prop:hat_fix_point}, let $f,g:\R^3\mapsto\R$ be continuous and polynomially bounded functions. We then have that 
    \begin{equation*}\label{eq:two-obs_out}
\lim_{n\to\infty}\E\thermal{f(a_{\star,1},a_{1,1},a_{2,1})g(a_{\star,2},a_{1,2},a_{2,2})} = \E_{G \otimes e}\thermal{f(\theta_\star,\theta_1,\theta_2)}_{s}\E_{G \otimes e}\thermal{g(\theta_\star,\theta_1,\theta_2)}_{s},
    \end{equation*}
    with $\E_{G \otimes e}(\cdot)$ defined as in Section \ref{sec:auxiliary} and $\thermal{\cdot}_{s}$ as in Proposition \ref{prop:hat_fix_point}.
\end{lemma}

From this we have the following concentration result used in the proof of Lemma \ref{lem:conv_dist_Z}.
\begin{proposition}\label{prop:conv_RS}
Under the assumptions of Proposition \ref{prop:hat_fix_point}, for $r\geq1$, let $f_1,\dots,f_r:\R^2\mapsto\R$ be continuous and polynomially bounded functions  and $\boldsymbol{V}^{(1)},\dots,\boldsymbol{V}^{(r)}\in\R^n$ be random vectors with  coordinates, for each $m\in[r]$ and $i\in[n]$, given by $V_i^{(m)} := f_m(a_{\star,i},a_{m,i})$. Then for every $m,m'\in[r]$, if $\hat Q_{m,m'}:=n^{-1}\boldsymbol{V}^{(m)\top}\boldsymbol{V}^{(m')}$, we have that
    \begin{equation*}
        \lim_{n\to\infty}\E\langle(\hQ_{m,m'}-\E\langle \hQ_{m,m'}\rangle)^2\rangle = 0.
    \end{equation*}
\end{proposition}
\begin{proof}
First notice that by the exchangeability of the variables $a_{\star,1},\dots,a_{\star,n}$ and $a_{m,1},\dots,a_{m,n}$ we have that
    \begin{small}
        \begin{equation*}
            \begin{split}
                \E\langle(\hQ_{m,m'}-\E\langle \hQ_{m,m'}\rangle)^2\rangle & = \frac{1}{n^2} \sum_{i\in[n]} \E\thermal{V^{(m)2}_iV^{(m')2}_i} + \frac{1}{n^2} \sum_{i\neq i'\in[n]} \E\thermal{V^{(m)}_iV^{(m')}_iV^{(m)}_{i'}V^{(m')}_{i'}} \\
                & \,\,\,\,\,\,\,\, -\frac{1}{n^2} \sum_{i\in[n]} \E^2\thermal{V^{(m)}_iV^{(m')}_i} - \frac{1}{n^2} \sum_{i\neq i'\in[n]} \E\thermal{V^{(m)}_iV^{(m')}_i}\E\thermal{V^{(m)}_{i'}V^{(m')}_{i'}}\\
                & = \frac{1}{n} \left(\E\thermal{V^{(m)2}_nV^{(m')2}_n}- \E^2\thermal{V^{(m)}_nV^{(m')}_n}\right)\\ & \,\,\,\,\,\,\,\, +\frac{n(n-1)}{n^2}\left(\E\thermal{V^{(m)}_nV^{(m')}_nV^{(m)}_{n-1}V^{(m')}_{n-1}}-\E\thermal{V^{(m)}_nV^{(m')}_n}\E\thermal{V^{(m)}_{n-1}V^{(m')}_{n-1}}\right).
            \end{split}
        \end{equation*}    
    \end{small}
%\textcolor{red}{The below is to check once with fresh mind tomorrow.}\textcolor{red}{TILL HERE.}
    It is easy to see that the first term of the last line converges to zero by the polynomial boundedness condition in the hypothesis and Proposition \ref{prop:control_mom_conv}. Moreover, by Lemma \ref{lem:decorr}, the last line converges to zero. 
\end{proof}

\subsection{Proof of Lemma \ref{lem:conv_dist_Z}}\label{subsec:limitloov}

By a simple interpolation argument (analogous to the proof of Lemma \ref{lem:same_limit_lvo}), Proposition \ref{prop:conv_RS}, and the assumption that the assumptions of Proposition \ref{prop:fix_point2}, we have that, for $m\neq m'\geq1$,
\begin{equation*}
    \E\thermal{(\tilde{Q}_{mm}-\tilde{v})^2}_v,  \,\, \E\thermal{(\bar{Q}_{mm}-\bar{v})^2}_v, \,\, \E\thermal{(M_{mm}-\tilde{m})^2}_v, 
\end{equation*}
\begin{equation*}
    \E\thermal{(\tilde{Q}_{mm'}-\tilde{q})^2}_v,  \,\, \E\thermal{(\bar{Q}_{mm'}-\bar{q})^2}_v, \,\, \text{and} \,\, \E\thermal{(M_{mm'}-\bar{m})^2}_v, 
\end{equation*}
all go to $0$ as $n\to\infty$.

%\textcolor{red}{Referring to places where these things were defined will be useful plus if any of the limiting constant letters were already used, we cant use the same one again.}
Recall definition \eqref{eq:mean_variable} of the leave-a-variable-out measure. It is clear then that, if $\bbeta$ is a sample from the measure induced by $\thermal{\cdot}_v$, the random vector $\tbbeta\in\R^{p-1}$ (the vector without the $j_0$-th coordinate of $\bbeta$) is independent of $\bX_{\bullet j_0}$.
%This implies that the random vectors $\bS^{(1)},\bS_\star^{(1)},\dots,\bS^{(l)},\bS_\star^{(l)}$ (recall definitions from \eqref{eq:scores}) are independent of $\bX_{\bullet j_0}$. 
Moreover, by the $L^2$ concentrations of $\tilde{Q}_{mm}$, $\tilde{Q}_{mm'}$, $M_{mm}$, and $M_{mm'}$ argued above and a simple adaptation of Lemma \ref{lem:normal_approx_dist} to this setting we obtain that
\begin{equation*}
    (Y_1^{(1)},Y_2^{(1)},\dots,Y_1^{(l)},Y_2^{(l)})\xrightarrow{d}(\tilde{\theta}_1,\bar{\theta}_1,\dots,\tilde{\theta}_l,\bar{\theta}_l),
\end{equation*}
under the measure induced by $\mathbb{E}[\langle \cdot \rangle_v$],
where $\tilde{\theta},\bar{\theta}$ are defined as in \eqref{eq:covariance_thetas}. 
From the independence of $\tbbeta$ from $\bX_{\bullet j_0}$ we can also easily prove that, for every $m\leq l$, 
\begin{equation*}
    \E\thermal{(Y_3-a_{dp}/2)^2}_v,\,\,\E(Y_4-t_\gamma)^2,\mbox{ and }\E\thermal{(Y_5-c_g)^2}_v
\end{equation*}
all go to $0$ as $n\to\infty$. % For example, we can see that \textcolor{red}{This is incomplete, also section heading should probably say Lemma 3 proof or something like that. }
%\begin{equation*}
 %   \E\thermal{(Y_5-c_g)^2}_v ...
%\end{equation*}
%\textcolor{red}{This is where doing the exact derivations showing exactly what limits these $Y_i$'s attain would be very desirable. Like after reading this, reader shouldnt have to check anything themselves regarding the form of the $\zeta's$ Just like you have already done for Appendix B.1.}
Finally, under the leave-a-variable-out measure $\beta^{(1)}_{j_0},\dots,\beta^{(l)}_{j_0}$ are all independent and distributed according to $\mu(\cdot)$. 
We then conclude the convergence stated in Lemma \ref{lem:conv_dist_Z}.

    \section{Concentration in $L^2$ of order parameters \eqref{eq:overlaps}}\label{app:replica_sym}

%\textcolor{red}{Given how I have structured the current Appendix A, where is this appendix result used for the first time? Generally }
In this section we will prove that, under $\E\thermal{\cdot}$, the order parameters $Q_{11}$, $Q_{12}$, and $Q_{1\star}$, as defined in \eqref{eq:overlaps}, concentrate. The formal statement is provided in Proposition \ref{prop:RS} at the end of this section. We use this concentration result critically in the proof of Proposition \ref{prop:conv_order_param} earlier. 

By the Hypothesis \ref{hyp1}, the Hessian $H[\logl_{n,p}]$ of $\logl_{n,p}$ is positive definite. This is so because, for all $j,j'\in[p]$,
\begin{equation*}
    \partial^2_{b_j,b_{j'}} \logl_{n,p} = - \sum_{i\in[n]} X_{ij}X_{ij'} A''(\bX_i^\top\bB)
\end{equation*}
and $A'' > 0$. Furthermore, by Hypothesis \ref{hyp1}, the coordinate-wise prior $\mu(\cdot)$ is strongly log-concave. Then, there are some $\varepsilon > 0$ and $a\in\R$ such that
\begin{equation}\label{eq:prior_eps_bound}
    \log \mu(x) < - \varepsilon (x-a)^2.
\end{equation}
Because this constant can be eliminated by a simple translation of the coordinates, we will assume w.l.g. that $a = 0$. All this implies the crucial fact that the posterior measure is strongly log-concave. 
% \ps{Which exact parts of the hypothesis do you use for this?}\ms{Hypothesis 1 (i, strong convexity) and (iii, convexity) imply this.}
To show that the order parameters concentrate we will introduce the following tilted log-likelihood
\begin{equation}\label{eq_pert_ham}
    \tilde{\logl}_{n,p}(\bB^{(1)},\bB^{(2)}) := \logl_{n,p}(\bB^{(1)}) + \logl_{n,p}(\bB^{(2)}) + \lambda_n ||\bB^{(1)}||^2 + \gamma_n \bB^{(1)\top}\bB^{(2)} + \chi_n \bbetas^{\top}\bB^{(1)},
\end{equation}
for $(\lambda_n)_{n\geq1},(\gamma_n)_{n\geq1},(\chi_n)_{n\geq1}\subseteq[0,\veps/2]$, three vanishing sequences to be selected later. Note that, unlike the log-likelihoods introduced before, this one involves two $p$-dimensional vectors. Given a function $f:\R^{2p}\mapsto\R$, then\footnote{Note here we use $(\bbeta^{(1)},\bbeta^{(2)})$ to refer to a sample from the tilted posterior \eqref{eq:tilted}, instead of referring to two independent samples from the full posterior as in the rest of the manuscript. But, since we refer to these under the expectation $\langle \cdot \rangle_T$, the meaning is clear from context.}
\begin{equation}\label{eq:tilted}
    \thermal{f(\bbeta^{(1)},\bbeta^{(2)})}_T := \frac{1}{\tZ_{\lambda_n,\gamma_n,\chi_n}}\int \exp\{\tilde{\logl}_{n,p}(\bbeta^{(1)},\bbeta^{(2)})\} \prod_{j\in[p]} \mu(d\beta_j^{(1)})\mu(d\beta_j^{(2)}),
\end{equation}
is the posterior expectation associated to \eqref{eq_pert_ham} with $\tZ_{\lambda_n,\gamma_n,\chi_n} > 0$ a suitable normalizing constant. We will either write $\bGamma$ or $(\bbeta^{(1)},\bbeta^{(2)})$
to refer to samples form the tilted posterior. In this section, we will denote the scaled log-normalizing constant of the posterior measure by
\begin{equation}\label{eq:log-norm}
    F_n := n^{-1}\log Z_n.
\end{equation}
In the same way, we will let the log-normalizing constant of the tilted posterior be
\begin{equation}\label{eq:log-norm-tilt}
    \tF_n := n^{-1}\log \tZ_{\lambda_n,\gamma_n,\chi_n}.
\end{equation}

First, we will give an exponential control of the norm $||\bbeta^{(1)}||^2$ under $\thermal{\cdot}_T$ (i.e., we control the square-norm of the first $p$ coordinates of a sample from the tilted posterior). We will then use this bound to control all the finite moments of $n^{-1}||\bbeta^{(1)}||^2$  under $\thermal{\cdot}_T$. This control is needed to account for the fact that the samples from this measure are not necessarily bounded.
% \ps{Note to myself: This is new, with this exponential moment control, we do not need the previous version with $R_n$, I should come back and re-read this part.}
\begin{proposition}\label{prop:exp_bound_norm}
    Assume Hypothesis \ref{hyp1} holds. Then there exists some constant $C \geq 0$ such that for all $t,(\lambda_n)_{n\geq1},(\gamma_n)_{n\geq1},(\chi_n)_{n\geq1}$ in $[0,\veps/4]$ and $n\geq1$ it holds that
\begin{equation*}
        \log \thermal{\exp{t ||\bbeta^{(1)}||^2}}_T, \log \thermal{\exp{t ||\bbeta^{(2)}||^2}}_T \leq Cn(||\bX||^2_{op}+n^{-1}||\tbT(0)||^2+1)^2,
    \end{equation*}
    where, as above, $\bbeta^{(1)}$ and $\bbeta^{(2)}$ are random vectors containing the first and last $p$ coordinates of a sample from the tilted posterior \eqref{eq:tilted}, respectively.
\end{proposition}
%\textcolor{red}{TILL HERE.}
\begin{proof}
    Here we prove this for $\bbeta^{(1)}$. The proof for $\bbeta^{(2)}$ follows in a similar way. 
        For any $\bB^{(1)},\bB^{(2)}$, because $0\leq t,\lambda_n \leq \varepsilon/4$,
    \begin{equation*}
        (t+\lambda_n)||\bB^{(1)}||^2 \leq \frac{\varepsilon}{2}||\bB^{(1)}||^2  
    \end{equation*}
    Similarly, because $2\bB^{(1)\top}\bB^{(2)}\leq||\bB^{(1)}||^2+||\bB^{(2)}||^2$ and $0\leq \gamma_n \leq \varepsilon/4$,
    \begin{equation*}
        \gamma_n \bB^{(1)\top}\bB^{(2)} \leq \frac{\varepsilon}{8}(||\bB^{(1)}||^2+||\bB^{(2)}||^2).
    \end{equation*}
    And because $2\bB^{(1)\top}\bbetas\leq||\bB^{(1)}||^2+||\bbetas||^2$ and $0\leq \chi_n \leq \varepsilon/4$,
    \begin{equation*}
        \chi_n \bB^{(1)\top}\bbetas \leq \frac{\varepsilon}{8}(||\bB^{(1)}||^2+||\bbetas||^2).
    \end{equation*}
    Finally, notice that by \eqref{eq:prior_eps_bound},
    \begin{equation*}
        \sum_{m\in\{1,2\}}\sum_{j\in[p]} \log \mu(b^{(m)}_j) \leq - \varepsilon(\|\bB^{(1)}\|^2 + \|\bB^{(2)}\|^2).
    \end{equation*}
    By the above bounds,
    \begin{equation*}
        (t+\lambda_n)||\bB^{(1)}||^2 + \gamma_n \bB^{(1)\top}\bB^{(2)} + \chi_n \bB^{(1)\top}\bbetas + \sum_{m\in\{1,2\}}\sum_{j\in[p]} \log \mu(b^{(m)}_j) 
    \end{equation*}
    is then smaller or equal to
    \begin{equation*}
        - \frac{\varepsilon}{8}\left( 2||\bB^{(1)}||^2 + 7 ||\bB^{(2)}||^2 - ||\bbetas||^2\right).
    \end{equation*}
    Also, as $A(\cdot)$ is by Hypothesis \ref{hyp1}(iii)  non-negative,
    \begin{equation}\label{eq:product_bound}
        \sum_{m\in\{1,2\}}\sum_{i\in[n]} \tT_i(\bX_i^\top\bbetas) \bX_i^\top\bB^{(m)} - A(\bX_i^\top\bB^{(m)})\leq \tau^\top (\bB^{(1)}+\bB^{(2)});
    \end{equation}
    where, for $j\in[p]$, $\tau_j := \sum_{i\in[n]}\tT_i(\bX_i^\top\bbetas) X_{ij}$. Because Hypothesis \ref{hyp1}(iv) states that $\sup_{n\geq1}||\bbetas||_\infty < \infty$, there is some $L > 0$ such that $||\bbetas||\leq \sqrt{n}L$ for all $n\geq1$. Then, because $||\tTp_i||_\infty\leq d_1$ by Hypothesis \ref{hyp1}(ii), it follows that for all $x\in\R$, $\tT_i(x) \leq d_1 |x| + \tT_i(0)$. Then, for some $K, K' > 0$
 \begin{equation}\label{eq:bound_tTnorm}
        \begin{split}
            ||\tbT(\ba_\star)||^2 = \sum_{i\in[n]} \tT_i^2(a_{\star,i}) & \leq \sum_{i\in[n]} (d_1 |a_{\star,i}| +\tT_i(0))^2 \\
            & \leq K \left(||\bX\bbetas||^2 + ||\tbT(0)||^2\right) \\
            & \leq K' n\left(||\bX||^2_{op}+ n^{-1}||\tbT(0)||^2\right),
        \end{split}
    \end{equation}
    where, in the last line, we used that $||\bbetas||^2\leq n L^2$. We then obtain that, for some appropriate $K'',K'''>0$,

    \begin{equation*}
        \log\int \exp\{t ||\bbeta^{(1)}||^2 + \tilde \logl_{n,p}(\bbeta^{(1)},\bbeta^{(2)})\} \prod_{j\in[p]}\mu(d\beta^{(1)}_j)\mu(d\beta^{(2)}_j)
    \end{equation*}
    is smaller or equal to
    {\small
    \begin{equation}\label{eq:bound_num}
        \begin{split}
             \log\int \exp\{\tau^\top (\bbeta^{(1)}+\bbeta^{(2)}) - \frac{\varepsilon}{2}||\bGamma||^2\} \prod_{j\in[p]}d\beta^{(1)}_jd\beta^{(2)}_j & \leq \frac{K''}{\veps}||\tau||^2 + \frac{||\bbetas||^2}{8} -p\log(K''\veps) \\
            & \leq \frac{K''}{\veps}||\tbT(\ba_\star)||^2\,||\bX||_{op}^2 + n\left(\frac{L}{8}-\kappa\log(K''\veps)\right)\\
            & \leq n K'''\left(||\bX||_{op}^2 + n^{-1} ||\tbT(0)||^2 + 1\right)^2;
        \end{split}
    \end{equation}}
    where we used \eqref{eq:prior_eps_bound} and \eqref{eq:product_bound} and that, for all $a\in\R_{>0}$ and $b\in\R$, $\int e^{-ax^2 + b x} dx = (\pi/a)^{1/2} e^{\frac{b^2}{4a}}$.

    Now we need to show an analogous lower bound on the log-normalizing constant. Define the set $B_n :=\{||\bB^{(1)}||_\infty+||\bB^{(2)}||_\infty\leq 1\}\subseteq\R^{2p}$ and $M := \int_{-1}^1\mu(dx)$. Without loss of generality, we may assume that $M > 0$. If this were not the case, it would suffice to adapt the argument by changing the interval $[-1,1]$ to one with positive $\mu$-mass. Notice that, because the integrand is non-negative,
    \begin{equation}\label{eq:rest_partf}
        Z_{\lambda_n,\gamma_n,\chi_n} \geq \int_{B_n}\exp\{\tilde \logl_{n,p}(\bbeta^{(1)},\bbeta^{(2)})\} \prod_{j\in[p]}\mu(d\beta^{(1)}_j)\mu(d\beta^{(2)}_j).
    \end{equation}
    Furthermore, in the set $B_n$ we have, by Cauchy-Schwarz and the fact that in $B_n$ the norms of $\bB^{(1)}$ and $\bB^{(2)}$ are bounded by $\sqrt{p}$,
    \begin{equation*}
        |\bB^{(1)\top}\bB^{(2)}| \leq \kappa n, \,\,\mbox{ and } \,\, |\bbetas^{\top}\bB^{(1)}| \leq \sqrt{\kappa n}||\bbetas||.
    \end{equation*}
    This implies that inside $B_n$ 
    \begin{equation}\label{eq:Bn_bound_pert}
        \tilde \logl_{n,p} \geq \logl_{n,p} - (\gamma_n + L \chi_n) \kappa n;
    \end{equation}
where $L>0$ is as above. On the other hand, inside $B_n$, by Cauchy-Schwarz and the fact that $||\bB{(1)}||\leq\sqrt{p}$, there is some $K^{(4)}>0$ for which
    
\begin{equation}\label{eq:Bn_bound_hamil}
        \begin{split}
            \sum_{i\in[n]}\tT_i(\bX_i^\top\bbetas) \bX_i^\top \bB^{(1)} & \geq - ||\tbT(\ba_\star)|| \, ||\bX\bB^{(1)}|| \\
            & \geq -\sqrt{n} ||\tbT(\ba_\star)|| \, ||\bX||_{op} \\
            & \geq -K^{(4)}n ||\bX||_{op} \sqrt{||\bX||^2_{op}+ n^{-1}||\tbT(0)||^2} \\
            & \geq -K^{(4)}n \sqrt{||\bX||^4_{op}+ n^{-1}||\tbT(0)||^2||\bX||^2_{op}} \\
            & \geq -K^{(4)}n \sqrt{||\bX||^4_{op}+ 2n^{-1}||\tbT(0)||^2||\bX||^2_{op}+n^{-2}||\tbT(0)||^4} \\
            & = -K^{(4)}n \left(||\bX||^2_{op}+n^{-1}||\tbT(0)||^2\right)
        \end{split}
    \end{equation}
    where, for the third bound, we used \eqref{eq:bound_tTnorm}. In a similar way, by Hypothesis \ref{hyp1}(iii) we have that $||A''||_\infty\leq 1$, which implies that there are constants $a,b, > 0$ such that, for all $x\in\R$ we have that $-A(x)\geq-(x^2+b|x|+c)$. Thus inside $B_n$ we have some $K^{(5)}>0$ s.t., for $m=1,2$,
   \begin{equation*}
        \begin{split}
            -\sum_{i\in[n]} A(\bX_i^\top\bB^{(m)}) & \geq - K^{(5)}\left(||\bX\bB^{(m)}||^2+\sum_{i\in[n]}\Big|\bX_i^\top\bB^{(m)}\Big|+n\right) \\
            & \geq  - K^{(5)}\left(||\bX\bB^{(m)}||^2+\sqrt{n}||\bX\bB^{(m)}||+n\right)\\
            & \geq  - K^{(5)}\left(||\bX||_{op}^2||\bB^{(m)}||^2+\sqrt{n}||\bX||_{op}||\bB^{(m)}||+n\right)\\
            & \geq-K^{(5)}n\left(||\bX||_{op}+1\right)^2;
        \end{split}
    \end{equation*}
    where we used that, by Cauchy-Schwarz, $\sum_{i\in[n]}|\bX_i^\top\bB^{(m)}|\leq \sqrt{n}||\bX\bB^{(m)}||$. Combining this last inequality with \eqref{eq:rest_partf}, \eqref{eq:Bn_bound_pert}, and \eqref{eq:Bn_bound_hamil} we obtain that for some $K^{(6)}$
    %\textcolor{red}{How was the integral totally dropped?}\ms{The Lebesgue measure of $B_n$ is exactly $4^{p}$ and so it is incorporated in the constant $+1$ in the expression below. I added a comment pointing to this.}
   \begin{equation}\label{eq:bound_denom}
        \log \tilde Z_{\lambda_n,\gamma_n,\chi_n} \geq -K^{(6)}n\left(||\bX||^2_{op}+n^{-1}||\tbT(0)||^2+1\right)^2,
    \end{equation}
    where we used that the Lebesgue measure of $B_n$ is $4^{p}$. Putting equations \eqref{eq:bound_num} and \eqref{eq:bound_denom} together produces the result.
\end{proof}

From this, a uniform bound for the finite moments of $||\bbeta^{(1)}||^2/n$ under $\thermal{\cdot}_T$ can be derived as follows.
\begin{corollary}\label{cor:bound_mom_norm}
    Assume Hypothesis \ref{hyp1} holds. For all $k\geq1$ there is some $C_k >0$ such that, for all $(\lambda_n)_{n\geq1},(\gamma_n)_{n\geq1},(\chi_n)_{n\geq1}$ in $[0,\veps/4]$ and $n\geq1$, $\E\thermal{||\bbeta^{(1)}||^{2k}}_T,\E\thermal{||\bbeta^{(2)}||^{2k}}_T\leq C_k n^k$.
\end{corollary}
\begin{proof}
    This follows directly from Proposition \ref{prop:exp_bound_norm}, \cite[Lemma 3.1.8]{talagrand2010mean}, Hypothesis \ref{hyp1}, and the fact that all the finite moments of $||\bX||_{op}$ are uniformly bounded in $n$ (which follows from \cite[Corollary 5.35]{vershynin2010introduction}).
\end{proof}

Next, we will produce a bound for the variance of the log-normalizing constant of the tilted model.
\begin{proposition}\label{prop:var_fe}
    Assume Hypothesis \ref{hyp1} holds. Then, there exists some fixed $C > 0$ s.t., for all $(\lambda_n)_{n\geq1},(\gamma_n)_{n\geq1},(\chi_n)_{n\geq1}$ in $[0,\veps/4]$, $\Var{\tF_n} \leq C/\sqrt{n}$.
\end{proposition}
\begin{proof}
    Define the vector $\be :=(e_1,\dots,e_n)$. By the law of total variance we have that
    \begin{equation}\label{eq:total_var}
        \Var{\tF_n} = \E\left(\Var{\tF_n\Big|\be}\right) + \Var{\E\left(\tF_n\Big|\be\right)}.
    \end{equation}
    We will now analyze the two terms separately.
    
    Given a differentiable function on $\R^{n\times n}$, we let $\nabla_{\bX}$ be the gradient over this space. Notice that, letting $\be$ be fixed, then $\tF$ is only a function of $\bX$. By Gauss-Poincaré inequality we have that for any given sequence $\bbetas$ as in Hypothesis \ref{hyp1}
    \begin{equation*}
        \Var{\tF_n\Big|\be} \leq \frac{1}{n}\E\left(||\nabla_{\bX}\tF_n||_F^2\Big|\be\right);
    \end{equation*}
    where $||\cdot||_F$ stands for the \emph{Frobenius norm}. 

    Note that, for all $i\in[n]$ and $j\in[p]$, the $ij$-th coordinate of $\nabla_{\bX}\tF$ is given, interchanging differentiation and integration using \cite[Proposition A.2.1]{talagrand2010mean}, by
    \begin{equation*}
        \begin{split}
            \partial_{X_{ij}}\tF_n & = \frac{1}{n\tilde{Z}_{\lambda_n,\gamma_n,\chi_n}} \int \partial_{X_{ij}} \left(\logl_{n,p}(\bB^{(1)}) + \logl_{n,p}(\bB^{(2)}) \right) \prod_{j\in[p]} \mu(db^{(1)}_j)\mu(db^{(2)}_j) \\
            & = \frac{1}{n}\sum_{m\in\{1;2\}} \thermal{\tTp_i\bX_i^\top\bbeta^{(m)}\beta_{\star,j} + \tT_i\beta^{(m)}_j - {A'(\bX_i^\top\bbeta^{(m)})}\beta^{(m)}_j}_T.
        \end{split}
    \end{equation*}

    Thus, 
    \begin{equation*}
        \nabla_{\bX}\tF_n = 2 \sum_{m\in\{1,2\}}(\bm{U}^{(m)} + \bm{V}^{(m)} - \bm{W}^{(m)})
    \end{equation*}
    where, for each $i\in[n]$ and $j\in[p]$, the $ij$-th component of these matrices are given by
    \begin{equation*}
        U^{(m)}_{ij} = n^{-1} \tTp_i\beta_{\star,j}\thermal{\bX_i^\top\bbeta^{(m)}}_T, \,\, V^{(m)}_{ij} = n^{-1} \tT_i\thermal{\beta^{(m)}_j}_T, \,\, \text{and} \,\, W^{(m)}_{ij} = n^{-1} \thermal{{A'(\bX_i^\top\bbeta^{(m)})}\beta^{(m)}_j}_T.
    \end{equation*}
    We then have, $||\nabla_{\bX}\tF_n||_F^2\leq 36\sum_{m\in\{1,2\}}||\bm{U}||_F^2 + ||\bm{V}||_F^2 +||\bm{W}||_F^2$. In a similar way as in Propositon \ref{prop:exp_bound_norm}, we can see that there is some large enough constant $K > 0$ s.t.
    \begin{equation*}
        ||\bm{U}^{(m)}||_F^2 \leq K ||\bX||^2_{op}\frac{\thermal{||\bbeta^{(m)}||^2}_T}{n}, \,\, ||\bm{V}^{(m)}||_F^2 \leq K (||\bX||_{op}+1)^2\frac{\thermal{||\bbeta^{(m)}||^2}_T}{n}, 
    \end{equation*}
\begin{equation*}
        \, \mbox{ and } \, ||\bm{W}^{(m)}||_F^2 \leq \frac{1}{n^2}\thermal{||{\bm{A}^{(m)}}'||^2||\bbeta^{(m)}||^2}_T \leq K \thermal{||\bX||^2_{op}\frac{||\bbeta^{(m)}||^4}{n^2} + \frac{||\bbeta^{(m)}||^2}{n}}_T.
    \end{equation*}
    In the bound for $||\bm{W}^{(m)}||_F^2$ we used that, by Hypothesis \ref{hyp1}(iii), $||A''||_\infty\leq 1$, which implies that there exists a constant $a > 0$ such that, for all $x\in\R$ we have $|A'(x)|\leq |x|+a$.
    Therefore, by the uniform bound for the moments of $||\bbeta^{(m)}||^2/n$ under $\thermal{\cdot}_T$ given in Corollary \ref{cor:bound_mom_norm} and the fact that all the finite moments of $||\bX||_{op}$ are uniformly bounded in $n$ (direct consequence of the exponential bound in \cite[Corollary 5.35]{vershynin2010introduction}), the mean of $||\nabla_{\bX}\tF_n||_F^2$ is $\mathcal{O}(1)$.
    
    For bounding the second term of \eqref{eq:total_var} define $g(\be) := \E(\tF_n|\be)$. We will use Efron-Stein's inequality (see \cite[Theorem 3.1]{boucheron2013concentration}) to control the variance of this function. Given $\be$ define a new random vector according to $\be'_i :=(e_1,\dots,e_{i-1},e'_i,e_{i+1},\dots,e_n)$ with $e_i'\sim\mbox{Unif}[0,1]$ independent of everything else. Also denote by $\logl_{n,p,i}'$ the resulting log-likelihood obtained by replacing $\be$ by $\be'_i$ and let $\tF_{n,i}$ be the corresponding log-normalizing constant. In this setting, the Efron-Stein's inequality takes the form %\textcolor{red}{Not sure if the Jensen is going the right way here. Very worried about this}
    %bounded differences inequality states that if we let, for $i\in[n]$, $c_i$ be the supremum of $|g(\be)-g(\be'_i)|$ over $\be\in[0,1]^n$ and $e_i'\in[0,1]$ then
    \begin{equation}\label{eq:es_ineq}
        \Var{g(\be)} \leq \frac{1}{2} \sum_{i\in[n]} \E(g(\be)-g(\be'_i))^2\leq \frac{1}{2} \sum_{i\in[n]} \E(\tF_n-\tF'_{n,i})^2, 
    \end{equation}
    where in the last relation we just used Jensen's inequality. For estimating the terms on the right hand side of the inequality, first notice that for all $i\in[n]$, 
    \begin{equation*}
        \begin{split}
            \tF_n - \tF'_{n,i} & = - \frac{1}{n}\log \thermal{\exp\{ \tilde{\logl}'_{n,p,i}(\bbeta^{(1)},\bbeta^{(2)})-\tilde{\logl}_{n,p}(\bbeta^{(1)},\bbeta^{(2)}) \}}_T; \\
            & \leq \frac{1}{n} \thermal{\tilde{\logl}_{n,p}(\bbeta^{(1)},\bbeta^{(2)}) - \tilde{\logl}'_{n,p,i}(\bbeta^{(1)},\bbeta^{(2)}) }_T
        \end{split}
    \end{equation*}
    by Jensen's inequality for $-\log(x)$. Let $\thermal{\cdot}_{T,i}$ be the expectation resulting from replacing $e_i$ by $e'_i$ in $\thermal{\cdot}_T$
     Similarly as above,
    \begin{equation*}
        \begin{split}
            \tF_n - \tF'_{n,i} & = \frac{1}{n}\log \thermal{\exp\{ \tilde{\logl}_{n,p}(\bbeta^{(1)},\bbeta^{(2)})-\tilde{\logl}'_{n,p,i}(\bbeta^{(1)},\bbeta^{(2)}) \}}_{T,i}; \\
            & \geq \frac{1}{n} \thermal{\tilde{\logl}_{n,p}(\bbeta^{(1)},\bbeta^{(2)})-\tilde{\logl}'_{n,p,i}(\bbeta^{(1)},\bbeta^{(2)}) }_{T,i}
        \end{split}
    \end{equation*}
    by the concave version of Jensen's inequality for $\log(x)$. In what follows, we omit the arguments of the log-likelihoods to have more compact equations. Combining the above and using Jensen's inequality one more time for the square, 
    \begin{equation}\label{eq:es_term}
        \begin{split}
            n^2 \E(\tF_n-\tF'_{n,i})^2 & \leq \E\thermal{\tilde{\logl}_{n,p}-\tilde{\logl}'_{n,p,i}}^2_T + \E\thermal{\tilde{\logl}_{n,p}-\tilde{\logl}'_{n,p,i}}^2_{T,i}  \\
            & \leq 2 \E\thermal{(\tilde{\logl}_{n,p}-\tilde{\logl}'_{n,p,i})^2}_T.
        \end{split}
    \end{equation}
    In the last inequality we used that, because $e_i$ and $e'_i$ are i.i.d., then $\E\thermal{\cdot}_{T,i}$ is equal to $\E\thermal{\cdot}_{T}$. Now, the difference of the two log-likelihoods can be readily computed and yields 
\begin{equation}\label{eq:dif_logl}
        \begin{split}
            (\tilde{\logl}_{n,p}-\tilde{\logl}'_{n,p,i})^2 & = (T\circ f(a_{\star,i},e_i)-T\circ f(a_{\star,i},e'_i))^2(a_{1,i}+a_{2,i})^2 \\
            & \leq K' (2|a_{\star,i}|+T\circ f(0,e_i)-T\circ f(0,e'_i))^2(a_{1,i}^2+a_{2,i}^2) . 
        \end{split}
    \end{equation}
    Observe that, by exchangeability, $\E\thermal{a_1^4}_T = n^{-1} \sum_{i\in[n]} \E\thermal{a_i^4}_T$. Then,
    \begin{equation*}
        \sum_{i\in[n]} \E\thermal{a_i^4}_T \leq \E\thermal{||\bX\bbeta||^4}_T \leq \sqrt{\E||\bX||_{op}^8\E\thermal{||\bbeta||^8}_T}.
    \end{equation*}
    From this and the bound over the moments of $||\bbeta||^2/n$ of Corollary \ref{cor:bound_mom_norm} we have that $\E\thermal{a_1^4}_T=\mathcal{O}(n)$. This together with Hypothesis \ref{hyp1} and the easy to check fact that $\E(a_{\star,1}+1)^4=\mathcal{O}(1)$ imply that, for some $K''>0$, for every $i\in[n]$
    \begin{equation}\label{eq:var_logl}
        \E\thermal{(\tilde{\logl}_{n,p}(\bbeta^{(1)},\bbeta^{(2)})-\tilde{\logl}'_{n,p,i}(\bbeta^{(1)},\bbeta^{(2)}))^2}_T \leq K'' \sqrt{n}.
    \end{equation}
    Combining equations \eqref{eq:es_ineq}-\eqref{eq:var_logl} we obtain that $\Var{g(\be)} \leq K''/\sqrt{n}$. This concludes the proof.
\end{proof}

In the following lemma we prove that the tilted posterior is strongly log-concave. For this we will consider its log-density as a function from $\R^{2p}$ to $\R$. As such, along the proof, we will use a generic deterministic vector $(\boldsymbol{b}_1,\boldsymbol{b}_2)\in\R^{2p}$ such that $\boldsymbol{b}_1,\boldsymbol{b}_1\in\R^p$ to evaluate likelihoods and priors. To simplify notations, in the following we will write $\log \mu(\boldsymbol{b}_1)$ instead of $\sum_{j=1}^p\log\mu(b_{1,j})$ and $\log \mu(\boldsymbol{b}_2)$ instead of $\sum_{j=1}^p\log\mu(b_{2,j})$.

\begin{lemma}\label{lem:logconc_tilted}
    If $0 < \lambda_n,\gamma_n < \veps/4$, then $\tilde{\logl}_{n,p}(\boldsymbol{b}_1,\boldsymbol{b}_2) + \log \mu(\boldsymbol{b}_1) + \log \mu(\boldsymbol{b}_2)$ is an $\varepsilon/2$-strongly concave function.
\end{lemma}
\begin{proof}
    Remember that
    \begin{equation*}
        \tilde{\logl}_{n,p}(\boldsymbol{b}_1,\boldsymbol{b}_2) = \logl_{n,p}(\boldsymbol{b}_1) +  \logl_{n,p}(\boldsymbol{b}_2) + \lambda_n \|\boldsymbol{b}_1\|^2 + \gamma_n \boldsymbol{b}_1^\top \boldsymbol{b}_2 + \chi_n \bbetas^\top \boldsymbol{b}_1.
    \end{equation*}
    By Hypothesis \ref{hyp1} we already know that $\logl_{n,p}$ is a concave function and $\log \mu(\boldsymbol{b}_1)$ and $\log \mu(\boldsymbol{b}_2)$ are $\varepsilon$-strongly concave. It is then enough to prove that the perturbation terms $\lambda_n \|\boldsymbol{b}_1\|^2$, $\gamma_n \boldsymbol{b}_1^\top \boldsymbol{b}_2$, and $\chi_n \bbetas^\top \boldsymbol{b}_1$ do not ``break'' the strong concavity of the posterior. For this, it is enough to bound their Hessians, in the Lowner order, by matrices proportional to the identity. In the case of $\chi_n\bbetas^\top \boldsymbol{b}_1$, this is trivial because its Hessian is the null matrix. Now, let $\mathcal{H}^{(11)}$ be the Hessian of $\lambda_n\|\boldsymbol{b}_1\|^2$ and $\mathcal{H}^{(12)}$ the one of $\gamma_n\boldsymbol{b}_1^\top\boldsymbol{b}_2$. Then, if $v\in S^{2p-1}(1)$, it is easy to see that
    \begin{equation*}
        \boldsymbol{v}^T \mathcal{H}^{(11)} \boldsymbol{v} = \lambda_n||\boldsymbol{v}_1||^2 \leq \lambda_n \ \ \ \ \ \ \mbox{ and } \ \ \ \ \ \  \boldsymbol{v}^T \mathcal{H}^{(12)} \boldsymbol{v} = \gamma_n (\boldsymbol{v}_1 \cdot \boldsymbol{v}_2)^2 \leq \gamma_n,
    \end{equation*}
    where $\boldsymbol{v}_1,\boldsymbol{v}_2\in\R^p$ are such that $\boldsymbol{v} = (\boldsymbol{v}_1,\boldsymbol{v}_2)$. Because $0< \lambda_n,\gamma_n < \veps/4$, we then have that
    \begin{equation*}
        \mathcal{H}^{(11)} - \frac{\varepsilon}{4} \mathbb{I} \quad \text{and} \quad \mathcal{H}^{(12)} - \frac{\varepsilon}{4} \mathbb{I}
    \end{equation*}
    are both negative-definite matrices for all values of $\boldsymbol{b}_1$ and $\boldsymbol{b}_2$. By the $\varepsilon$-strong concavity of $\logl_{n,p}(\boldsymbol{b}_1) + \log\mu(\boldsymbol{b}_1) +  \logl_{n,p}(\boldsymbol{b}_2) + \log\mu(\boldsymbol{b}_2)$, we then conclude that
    \begin{equation*}
        \tilde{\logl}_{n,p}(\boldsymbol{b}_1,\boldsymbol{b}_2) + \log \mu(\boldsymbol{b}_1) + \log \mu(\boldsymbol{b}_2)
    \end{equation*}
    is $\varepsilon/2$-strongly concave.
\end{proof}

The next proposition proves the concentration of some of the order parameters. For this, recall the definitions of $Q_{11}$, $Q_{12}$, and $Q_{1\star}$ from \eqref{eq:overlaps}. In this proposition, the definitions are the same but $(\bbeta^{(1)},\bbeta^{(2)})$ are taken to be samples form the tilted posterior.

\begin{proposition}\label{prop:rep_symm_pert}
    Let $(\lambda_n)_{n\geq1},(\gamma_n)_{n\geq1},(\chi_n)_{n\geq1} \in [0,\veps/4]$. Then, for some fixed $C > 0$,
    \begin{equation*}
        \E\langle(Q_{11}-\E\langle Q_{11}\rangle_T)^2\rangle_T, \ \E\langle(Q_{12}-\E\langle Q_{12}\rangle_T)^2\rangle_T, \ \E\langle(Q_{1\star}-\E\langle Q_{1\star}\rangle_T)^2\rangle_T \leq \frac{C}{n^{1/4}}.
    \end{equation*}
\end{proposition}
\begin{proof}
    We will present the proof for $Q_{11}$ but the arguments for $Q_{12}$ and $Q_{1\star}$ are completely analogous. Since the tilted posterior is $\varepsilon/2$-strongly log-concave, by Brascamp-Lieb's inequality (see \cite[Theorem 5.1]{brascamp2002extensions}), we have that 
\begin{equation}\label{eq:thm_conv_q}
        \langle(Q_{11} - \langle Q_{11}\rangle_T)^2\rangle_T \leq \frac{\langle||\nabla Q_{11}||^2\rangle_T}{\veps/2} = 8\frac{\langle||\bbeta^{(1)}||^2\rangle_T}{\veps n^2};
    \end{equation}
where we used that $\nabla Q_{11}=p^{-1}\nabla\|\bbeta^{(1)}\|^2 = 2p^{-1}\bbeta$.

    Taking expectation on both sides, we obtain by Corollary \ref{cor:bound_mom_norm} that there is some constant $K >0$ s.t.
    \begin{equation}\label{eq:therm_var}
        \E\langle(Q_{11} - \langle Q_{11}\rangle_T)^2\rangle_T \leq \frac{K}{n}.
    \end{equation}
    Now, note that 
    \begin{equation*}
        \partial_{\lambda_n} \tF_n = \langle Q_{11}\rangle_T, \ \ \ \partial^2_{\lambda_n} \tF_n = n \langle(Q_{11} - \langle Q_{11}\rangle_T)^2\rangle_T \geq 0,
    \end{equation*}
    \begin{equation*}
        \partial_{\lambda_n} \E \tF_n = \E\langle(Q_{11}\rangle_T, \ \ \ \mbox{ and } \ \ \  \partial^2_{\lambda_n} \E\tF_n = n \E\langle(Q_{11} - \langle Q_{11}\rangle_T)^2\rangle_T \geq 0.
    \end{equation*}
    Also, by a standard lemma (see for example \cite[Lemma 3.2]{panchenko2013sherrington}), we have that if $G,g:\R\mapsto\R$ are two convex functions, for any $\delta > 0$ it holds that
    \begin{equation}\label{eq:convexderivdiff}
            |G'(\theta) - g'(\theta)| \leq g'(\theta+\delta) - g'(\theta-\delta) + \delta^{-1} \sum_{y \in \{\theta-\delta, \theta, \theta+\delta\}} |G(y)-g(y)|.
    \end{equation}
    For the calculation below, we rename $\tF_n$ to be $\tF_{\lambda_n\gamma_n,\chi_n}$ to emphasise its dependence on $\lambda_n$, $\gamma_n$, and $\chi_n$. Note that $\tF_{\lambda_n\gamma_n,\chi_n}$ and $\E \tF_{\lambda_n\gamma_n,\chi_n}$ are convex in $\lambda_n$, so taking $g = \tF_{\lambda_n\gamma_n,\chi_n}$ and $G = \E \tF_{\lambda_n\gamma_n,\chi_n}$ we have that for all $\delta > 0$
    \begin{equation}\label{eq:conv_bound_f}
            |\langle Q_{11}\rangle_T - \E \langle Q_{11}\rangle_T| \leq \langle Q_{11}\rangle_{\lambda_n+\delta} - \langle Q_{11}\rangle_{\lambda_n-\delta} + \delta^{-1} \sum_{y \in \mathcal{Y}} |\tF_{y,\gamma_n,\chi_n} - \E \tF_{y,\gamma_n,\chi_n}|,
    \end{equation}
    with $\mathcal{Y}:=\{\lambda_n-\delta, \lambda_n, \lambda_n+\delta\}$. In this last equation, $\langle \cdot\rangle_{\lambda_n\pm\delta}$ stands for a posterior mean of the form $\langle\cdot\rangle_T$ where the perturbation parameter $\lambda_n$ has been replaced by $\lambda_n\pm\delta$ respectively. By the mean value theorem  $\exists x\in(\lambda_n-\delta,\lambda_n+\delta)$ such that 
    \begin{equation*}
        \langle Q_{11}\rangle_{\lambda_n+\delta} - \langle Q_{11}\rangle_{\lambda_n-\delta} = 2 \delta n \langle(Q_{11} - \langle Q_{11}\rangle_{\lambda_n=x})^2\rangle_{\lambda_n=x}.
    \end{equation*}
Since an analogous bound as in \eqref{eq:thm_conv_q} holds for $\langle\cdot \rangle_{\lambda_n=x}$, we have that for all $\delta > 0$, small enough,
    \begin{equation}\label{eq:conv_q_aux1}
        \langle Q_{11}\rangle_{\lambda_n+\delta} - \langle Q_{11}\rangle_{\lambda_n-\delta} \leq \frac{2\delta}{\veps n}\langle||\bbeta^{(1)}||^2\rangle_{\lambda_n=x}.
    \end{equation}
    % \ps{Why do you have $\delta$ in the numerator aboe, I thought this should just be $\alpha/\varepsilon.$}\ms{$\alpha/\varepsilon$ bounds uniformly the derivative of the function $f(x):=\langle Q_{11}\rangle_{x}$ and $\delta$ is the difference in the two arguments of the function on the left hand side. So I think it is correct.}
    On the other hand, since the bound on the variance in Proposition \ref{prop:var_fe} does not depend on $\lambda_n$, we have that %\textcolor{red}{How are you controlling the cross terms in the below summation? I agree if the summation were outside the square since you have concentration for each $\tilde{F}$.}\ms{Here I am using that $(\sum_{i\in I} a_i)^2 \leq |I|^2 \max_{i\in I} a_i^2 \leq |I|^2\sum_{i\in I}a_i^2$. The factor $9$ is thus just $|\mathcal{Y}|^2$.}
    \begin{equation}\label{eq:conv_q_aux2}
        \E\left(\sum_{y \in \mathcal{Y}} |\tF_{y,\gamma_n} - \E \tF_{y,\gamma_n}|\right)^2 \leq \frac{9K'}{\sqrt{n}}; 
    \end{equation}
where we used that
    \begin{equation}\label{eq:bound_square}
        \left(\sum_{i\in \mathcal{I}} a_i\right)^2 \leq |\mathcal{I}|^2 \max_{i\in \mathcal{I}} a_i^2 \leq |\mathcal{I}|^2\sum_{i\in \mathcal{I}}a_i^2
    \end{equation}
    which holds for the arbitrary finite index sets $\mathcal{I}$ and terms $(a_i)_{i\in\mathcal{I}}$. Squaring \eqref{eq:conv_bound_f} and using equations \eqref{eq:conv_q_aux1}-\eqref{eq:conv_q_aux2}, inequality \eqref{eq:bound_square}, and Proposition \ref{prop:var_fe} we have that, for some $K' > 0$, %\textcolor{red}{The middle term below is not adding up, within square root the denominator should be $\sqrt{n}$. And $\langle \|\bbeta \|^2_T \rangle$ should have $\lambda_n=x$--this later one is a minor point. But my concern is $\delta = n^{-1/8} $ makes the last term below $O(n^{1/4})$ instead of $O(n^{-1/4})$!! What am I missing?}\ms{I re-wrote this using the same bound as before: $(\sum_{i\in I} a_i)^2 \leq |I|^2\sum_{i\in I}a_i^2$. The factor $9$ is thus just $|\mathcal{Y}|^2$.}
  \begin{equation*}
        \begin{split}
            \E(\langle Q_{11}\rangle_T - \E \langle Q_{11}\rangle_T)^2 & \leq \E\left( \langle Q_{11}\rangle_{\lambda_n+\delta} - \langle Q_{11}\rangle_{\lambda_n-\delta} + \delta^{-1} \sum_{y \in \mathcal{Y}} |\tF_{y,\gamma_n,\chi_n} - \E \tF_{y,\gamma_n,\chi_n}| \right)^2 \\
            & \leq 4\left( \E(\langle Q_{11}\rangle_{\lambda_n+\delta} - \langle Q_{11}\rangle_{\lambda_n-\delta})^2 + \frac{1}{\delta^2} \E\left(\sum_{y \in \mathcal{Y}} |\tF_{y,\gamma_n} - \E \tF_{y,\gamma_n}|\right)^2  \right) \\
            & \leq 4\left( \frac{4\delta^2}{\veps^2}\frac{\E\thermal{||\bbeta||^4}_T}{n^2} + \frac{9K'}{\delta^2 \sqrt{n}} \right).
        \end{split}
    \end{equation*}
    By Corollary \ref{cor:bound_mom_norm},  $n^{-2}\E\thermal{||\bbeta||^4}_T$ is a bounded sequence. Then, letting $\delta_n = n^{-1/8}$ we obtain there is a positive constant $K'' > 0$ such that
    \begin{equation}\label{eq:dis_conv_q}
        \E(\langle Q_{11}\rangle_T - \E \langle Q_{11}\rangle_T)^2 \leq \frac{K''}{n^{1/4}}.
    \end{equation}
  Combining \eqref{eq:therm_var} and \eqref{eq:dis_conv_q} concludes the proof.
\end{proof}
%\textcolor{red}{TILL HERE}
The last step to prove the concentration of the order parameters of the model is to show that the conclusion of Proposition \ref{prop:rep_symm_pert} extends to our original untilted posterior measure. We establish this conclusion by a simple interpolation argument.

\begin{proposition}\label{prop:RS}
    Assume Hypothesis \ref{hyp1} holds. Then, there exists a positive constant $C > 0$ such that
    \begin{equation*}
        \E\langle(Q_{11}-\E\langle Q_{11}\rangle)^2\rangle, \ \E\langle( Q_{12}-\E\langle Q_{12}\rangle)^2\rangle, \ \E\langle(Q_{1\star}-\E\langle Q_{1\star}\rangle)^2\rangle \leq \frac{C}{n^{1/4}}.
    \end{equation*}
\end{proposition}
\begin{proof}
    Again we will prove this for $Q_{11}$ but the proofs for $Q_{12}$ and $Q_{1\star}$ are completely analogous. For simplicity, we will set $\gamma_n$ and $\chi_n$ to be equal to $0$. The proof may be extended to non null values of $\gamma_n$ and $\chi_n$ in a straightforward way. 
    
    Define an interpolating log-likelihood $\logl_{n,p,t} := \logl_{n,p} + t\lambda_n n Q_{11} $ for $t\in[0,1]$. Clearly, $\logl_{n,p,0} = \logl_{n,p}$ and $\logl_{n,p,1} = \tilde{\logl}_{n,p}$. Let $\langle \cdot\rangle_t$ be the posterior expectation with respect to this log-likelihood. That is, for each $f:\R^p\mapsto\R$ and $t\in[0,1]$, we let
    \begin{equation*}
        \thermal{f(\bbeta)}_t := \frac{\int f(\bB) \exp\{\logl_{n,p,t}(\bB)\}\prod_{j\in[p]}\mu(db_j)}{\int \exp\{\logl_{n,p,t}(\bB)\}\prod_{j\in[p]}\mu(db_j)}.
    \end{equation*}
In this way we have that $\thermal{\cdot}_1$ is equal to $\thermal{\cdot}_T$ and $\thermal{\cdot}_0$ to $\thermal{\cdot}$. Note that by Brascamp-Lieb inequality (see \cite[Theorem 5.1]{brascamp2002extensions}) we have
    \begin{equation}\label{eq:brasc_bound}
        \langle(Q_{11}-\langle Q_{11}\rangle_t)^2\rangle_t \leq 2\frac{\thermal{||\bbeta||^2}_t}{\veps n^2}.
    \end{equation}
    Then, for all $t\in[0,1]$
    \begin{equation*}
        \frac{d}{dt}\langle Q_{11}\rangle_t = n \lambda_n \langle(Q_{11} -\langle Q_{11}\rangle_t)^2\rangle_t \leq \frac{2\lambda_n}{\veps n}\thermal{||\bbeta||^2}_t
    \end{equation*}
    and
    \begin{equation*}
        \frac{d}{dt}\E\langle Q_{11}\rangle_t = n \lambda_n \E\langle(Q_{11} -\langle Q_{11}\rangle_t)^2\rangle_t \leq \frac{2\lambda_n}{\veps n}\E\thermal{||\bbeta||^2}_t,
    \end{equation*}
    which implies by the Mean Value Theorem that, for some $K>0$,
    \begin{equation}\label{eq:dif_mean}
        |\langle Q_{11}\rangle - \langle Q_{11}\rangle_T| \leq \frac{\lambda_n K}{\veps}(||\bX||_{op}+1)^2 \ \ \ \mbox{ and } \ \ \ |\E\langle Q_{11}\rangle - \E\langle Q_{11}\rangle_T| \leq \frac{\lambda_n K}{\veps},
    \end{equation}

    where for the first relation we used Proposition \ref{prop:exp_bound_norm} together with Jensen's inequality while for the second we used Corollary \ref{cor:bound_mom_norm}. We then have that 
    \begin{equation*}
        \begin{split}
            \sqrt{\E\langle(Q_{11}-\E\langle Q_{11}\rangle)^2\rangle} & \leq \sqrt{\E\langle(Q_{11}-\langle Q_{11}\rangle)^2\rangle} + \sqrt{\E(\langle Q_{11}\rangle-\langle Q_{11}\rangle_T)^2} \\ 
            & \ \ \ \ \ + \sqrt{(\E\langle Q_{11}\rangle - \E\langle Q_{11}\rangle_T)^2} + \sqrt{\E(\langle Q_{11}\rangle_T - \E\langle Q_{11}\rangle_T)^2}.
        \end{split}
    \end{equation*}
    %\ps{In the first expression on the RHS above, in my calculations, I obtain the expression with Gibbs bracket $\langle \rangle$ instead of $\langle \rangle_T$. I think thats the one it should be? But didnt wish to change without you cross-checking.}\ms{Yes, I agree. In fact, as it was before it was not clear how to control the first term.}
    By choosing $\lambda_n = n^{-1/4}$, we have that owing to Proposition \ref{prop:rep_symm_pert}, the fact that $||\bX||_{op}$ has uniformly bounded moments, and equations \eqref{eq:brasc_bound}-\eqref{eq:dif_mean} the RHS of this last inequality is $\mathcal{O}(n^{-1/4})$. This concludes the proof.
\end{proof}

    \section{Distributional limits of projections of random vectors}\label{app:asymp_representation}
As seen earlier, quantifying properties of the posterior in GLMs requires analyzing the joint behavior of terms of the form $\bX_i^{\top}\bbeta^{(1)},\bX_{i}^{\top}\bbeta^{(2)}$, for suitable choices of $\bbeta^{(1)}$ and $\bbeta^{(2)}$. If these vectors were independent of $\bX_i$, quantifying the joint distribution of these terms would be easy. In this appendix, we record some such basic results, considering general sequences of random vectors (without using our specific problem structure), which we use in intermediate steps in our proofs. 
%Here we will prove some adaptations of \cite[Lemma 5.3]{barbier2021performance} to our setting. Throughout this section, for some fixed 
For $\ell \in \mathbb{N}$, we will let $\bV_0,\bV_1,\dots,\bV_\ell\in\R^p$ be  sequences of random vectors. In Lemmas \ref{lem:normal_approx_dist} and \ref{lem:normal_approx_dist_2spin} below, we prove an asymptotic representations for products of normal vectors with these sequences of random vectors.

\begin{assumption}[order parameters concentration]\label{ass:RS}
    The sequence of random vectors $\bV_0,\hdots,\bV_\ell$ satisfies the following property: there exist constants $0\leq c_0, c \leq v,v_0$ such that for all distinct $m , m'\in[r]$ we have that $\E( p^{-1}\bV_0^\top\bV_m - c_0 )^2$, $\E( p^{-1}\bV_m^\top\bV_{m'} - c )^2$, $\E( p^{-1}\bV_0^{\top}\bV_0 - v_0 )^2$, and $\E( p^{-1}\bV_m^\top\bV_{m} - v )^2$ go to $0$ as $p\to\infty$. \end{assumption}

In this section, results are proved for general sequences of random vectors $\bV_m$'s that satisfy Assumption \ref{ass:RS}. These results are used in Section \ref{sec:proofs} to prove Proposition \ref{prop:fix_point2} and \ref{prop:hat_fix_point2} where the $\bV_m$'s are chosen to be specific random vectors of interest.

Let $\bX\in\R^p$ be some sequence of standard normal vectors independent of $(\bV_0,\bV_1,\dots,\bV_\ell)$. And for every $m\in[\ell]\cup\{0\}$, let $S_m := \bX^\top \bV_m/\sqrt{p}$. 
\begin{lemma}\label{lem:normal_approx_dist}
Assume $0 < c \leq v$ and $c_0^2\leq v_0 c$ and let $\theta_0 := \sqrt{v_0-c_0^2/c} \xi_0 + c_0/\sqrt{c}\, z$ and (for $m\in[\ell]$) $\theta_m := \sqrt{v - c} \xi_m + \sqrt{c}\, z$; where $\xi_0,\dots,\xi_\ell,z$ are i.i.d. standard normal random variables. Under Assumption \ref{ass:RS} we then have that, as $p\to\infty$,
    \begin{equation*}
        (S_0,S_1,\dots,S_\ell) \xrightarrow{d} (\theta_0,\theta_1,\dots,\theta_\ell).
    \end{equation*}
\end{lemma}
\begin{proof}
    Throughout the proof, $f:\R^{m+1}\mapsto\R$ will be some continuous bounded function. Notice that for fixed values of $\bV_0,\bV_1,\dots,\bV_\ell$, the quantity $(S_0,S_1,\dots,S_\ell)\in\R^{\ell+1}$ is a normal random vector with covariance matrix $\Sigma_p \in \mathbb{R}^{(\ell+1) \times (\ell+1)}$ with entries $\Sigma_{p,m m'} = p^{-1}\bV_m^\top\bV_{m'}$. Because $f$ is continuous and bounded, we thus have that there exists some continuous and bounded $g: \R^{(\ell+1)\times(\ell+1)} \mapsto \R$ such that
    \begin{equation}\label{eq:normal_mean}
        \E \left[f(S_0,S_1,\dots,S_\ell)\right] = \E\left[ g(\Sigma_p)\right].
    \end{equation}
    By Assumption \ref{ass:RS}, $\Sigma_p \xrightarrow{L^2} \Sigma$ as $p\to\infty$; where $\Sigma\in\R^{(\ell +1)\times (\ell+1)}$ is a symmetric matrix s.t. for all distinct $m,m'\in[\ell]\cup\{0\}$ its elements are given by $\Sigma_{00} = v_0$, $\Sigma_{0m} = c_0$, $\Sigma_{mm'} = c$, and $\Sigma_{mm} = v$. In other words,
    \begin{equation*}
        \Sigma = 
            \begin{pmatrix}
                v_0 & c_0 & \dots & c_0 \\ 
                c_0 &  &  & \\ 
                \vdots &  & \Sigma'& \\ 
                c_0 &  &  & 
            \end{pmatrix},
    \end{equation*}
    with $\Sigma'\in\R^{m\times m}$ a symmetric matrix with elements equal to $v$ on the diagonal and $c$ on the off the diagonal. Then, by the continuous mapping theorem, $\E\left[ g(\Sigma_p) \right] \to g(\Sigma)$. Notice that the normal vector $(\theta_0,\theta_1,\dots,\theta_r)$ has covariance matrix $\Sigma$. By the definition of $g$, we then have that $\E\left[f(\theta_0,\theta_1,\dots,\theta_\ell)\right] = g(\Sigma)$. This proves that $\E \left[f(S_0,S_1,\dots,S_\ell)\right]$ converges to $\E\left[f(\theta_0,\theta_1,\dots,\theta_\ell)\right]$ and, because $f$ is a generic continuous bounded function, it concludes the proof.
\end{proof}

\begin{remark}
    The other scenario in which we will need a result of this kind is when $c_0 = c = 0$. In this case, the result trivially holds for $\theta_0 := \sqrt{v_0}\xi_0$ and $\theta_l := \sqrt{v}\xi_l$.
\end{remark}

Now, let $\bX'\in\R^p$ be some sequence of standard normal vectors independent of $(\bV_0,\bV_1,\dots,\bV_\ell)$ and $\bX$. Analogous to the notations before, let (for $m\in[\ell]\cup\{0\}$) $S'_m := {\bX'}^T \bV_m/\sqrt{p}$, $\theta'_0 := \sqrt{v_0-c_0^2/c} \xi'_0 + c_0/\sqrt{c}\, z'$, and (for $m\in[\ell]$) $\theta'_m := \sqrt{v - c} \xi'_\ell + \sqrt{c}\, z'$; where $\xi'_0,\dots,\xi'_\ell,z'$ are i.i.d. standard normal variables independent of $\xi_0,\dots,\xi_\ell,z$. We then have the following extension.

\begin{lemma}\label{lem:normal_approx_dist_2spin}
Assume $0<c\leq v$ and $c_0^2 \leq v_0 c$ and let $(\theta_0,\theta_1,\dots,\theta_\ell)$ and $(\theta'_0,\theta'_1,\dots,\theta'_\ell)$ be as above. Under Assumption \ref{ass:RS} we then have that, as $p\to\infty$,
    \begin{equation*}
        (S_0,S_1,\dots,S_\ell,S'_0,S'_1,\dots,S'_\ell) \xrightarrow{d} (\theta_0,\theta_1,\dots,\theta_\ell,\theta'_0,\theta'_1,\dots,\theta'_\ell).
    \end{equation*}
\end{lemma}
\begin{proof}
    For each fixed value of $(\bV_0,\bV_1,\dots,\bV_\ell)$, $(S_0,S_1,\dots,S_\ell,S'_0,S'_1,\dots,S'_\ell)$ is a normal vector such that, for every $m,m'\in[r]\cup\{0\}$, the covariance of $S_m$ and $S'_{m'}$ is zero. The result then follows in a similar way as in the proof of Lemma \ref{lem:normal_approx_dist}.
\end{proof}

\begin{remark}
    Again, the analogous result for $c_0 = c = 0$ also trivially holds.
\end{remark}

% ------------------------------------------------------------------------------------------------------------------------------------
% ------------------------------------------------------------------------------------------------------------------------------------
% ------------------------------------------------------------------------------------------------------------------------------------

\section{Conditional posterior measures}\label{app:conditioning}

In Appendix \ref{subsubsec:approx_cavity}, we require bounding expectations under posterior measures restricted to specific sets. Here we present some basic conditioning arguments that where used as auxiliary results there.  We discuss the presented results in broad generality here without reference to  specific posteriors. To this end, let $\bX\in\R^{n\times p}$ and $\by\in\R^n$ be a random matrix and vector, respectively, and $\E(\cdot)$ be the expectation induced by their joint distribution. Also, define the measure 
%\ps{Okay so I have comments on the below. First of all, if $\bY$ is as you define it is, it is not clear where a parameter $\bbeta$ is coming in from. In a supervised problem, we have two kinds of variables, $\by, \bx$ and then $\bbeta$ is trying to estimate the conditional distribution of $\by|\bx$ which is assumed to be parametric. Note, dimensions of $\bx$ and $\bbeta$ match. But here it seems $\bY$ is n-dimensional , there is no p, so what is the following posterior trying to estimate?}
\begin{equation*}
    \proba(d\bbeta|\bX,\by) := \frac{1}{Z_n}\exp{\mathcal{H}_{n,p}(\bbeta|\bX,\by)}\nu(d\bbeta);
\end{equation*}
%\ps{If you wish to keep this section general, then we cant use the notations $\mathcal{L}_{n,p}$ and $\mu$ since we have already used this before and this already means something specific in the context of this paper. So this needs to be changed through this section}
where, fixing $\bX$ and $\by$, $\mathcal{H}_{n,p}:\R^p\mapsto\R$ is some general log-likelihood, $\nu(\cdot)$ is some appropriate Borel prior. Let $\bbeta$ denote a random sample from this measure
%\ps{No we cant have this; $\mu$ referred to the prior we actually used, same comment about the log-likelihood, so far you have not said what lets use a different later that has not been},
and, assuming the right hand side of this definition is integrable, $Z_n > 0$ is a suitable normalizing constant. For the rest of this section $\thermal{\cdot}$ will refer to expectation under this generic posterior measure. We use this abuse of notation to avoid introducing new notations that will not be used later on.

%\ps{Again abuse of notation. I actually think it may not bad to just work with our full posterior here, without  trying to say something for a general setting. Every notation will need to be reinvented otherwise, we dont really need that.}\ms{I agree, trying to prove this in generality could be ocnfusing. I just removed the first paragraphs and leave it as a result for our full posterior. When dealing with general random vectors I used the same $\bV$ notation as in the section on the asymptotic representation.}

Let $f:\R^p\mapsto\R$ and $g:\R^{n\times p}\mapsto\R$ be some functions. Define the events $A:=\{f(\bbeta)\geq 0\}$ %\textcolor{red}{These specific $A,B$ definitions seems weird. Should they just be generic events so that $B$ has a strictly positive probability and  $\thermal{\mathbb{I}_A} >0$ a.s.?}\ms{The above paragraph explains the context better.} 
and $B:=\{g(\bX)\geq 0\}$ and assume that $B$ has a strictly positive probability. Moreover, assume that $\thermal{\mathbb{I}_A} >0$ a.s. Let $\bar{\E}(\cdot)$ be the expectation with respect to $(\bX,\by)$ conditional on the event $B$. Similarly, for each $h:\R^p\mapsto\R$, define
\begin{equation*}
    \begin{split}
        \thermal{h(\bbeta)}' := 
        \begin{cases}
            \frac{\thermal{h(\bbeta)\mathbb{I}_A}}{\thermal{\mathbb{I}_A}} & \mbox{ if } \thermal{\mathbb{I}_A} >0 \\
            0 & \mbox{otherwise}.
        \end{cases}
    \end{split}
\end{equation*}
%\textcolor{red}{To check Lemma 1 (Appendix C.1) last three lines proof satisfies conditions for this.}
%\textcolor{red}{There are quite some inconsistencies in the notations in this section and in Appendix B.1, these need to be made consistent. For now, I am reading this section keeping in mind the definitions given here. We should revisit B.1 and make these consistent.}\ms{As explained in the introductory paragraph, I think it might be a good idea to keep the notation of this secition somewhat independent. For now I will do that and we can discuss it later on if you are not happy about it.}
\begin{proposition}\label{prop:comparison_conditional}
    For all $h:\R^p\mapsto\R$ s.t. $\E\thermal{h^2(\bbeta)} < \infty$, we have that
    \begin{equation*}
        |\E\thermal{h(\bbeta)} - \bE\thermal{h(\bbeta)}'| \leq (\bE\thermal{\mathbb{I}_{A^c}})^{1/2} \left(\bE^{1/2}\thermal{h^2} + \bE^{1/2}\thermal{h^2}'\right) + \left(\E\thermal{h^2}\proba(B^c)\right)^{1/2}+\bE\thermal{|h|}\proba(B^c).
    \end{equation*}
\end{proposition}
\begin{proof}
    This bound follows easily by the definitions above. By hypothesis, $\E\thermal{h^2(\bbeta)} < \infty$ and this implies that
    \begin{equation*}
        \bE\thermal{h^2(\bbeta)}, \, \bE\thermal{h^2(\bbeta)}' < \infty.
    \end{equation*}
    First, notice that 
    \begin{equation*}
        \E\thermal{h(\bbeta)} - \bE\thermal{h(\bbeta)} = \E[\thermal{h(\bbeta)}\mathbb{I}_{B^c}] - \bE\thermal{h(\bbeta)}\proba(B^c).
    \end{equation*}
    By Cauchy-Schwarz, we then have
    \begin{equation*}
        |\E\thermal{h(\bbeta)} - \bE\thermal{h(\bbeta)}| \leq \left(\E\thermal{h^2}\proba(B^c)\right)^{1/2} + \bE\thermal{|h(\bbeta)|}\proba(B^c).
    \end{equation*}    
    Similarly, we have that
    \begin{equation*}
        \bE\thermal{h(\bbeta)} - \bE\thermal{h(\bbeta)}' = \bE\thermal{h(\bbeta)\mathbb{I}_{A^c}} - \bE[\thermal{h(\bbeta)}'\thermal{\mathbb{I}_{A^c}}].
    \end{equation*}
    As before, this implies that
    \begin{equation*}
        |\bE\thermal{h(\bbeta)} - \bE\thermal{h(\bbeta)}'| \leq (\bE\thermal{\mathbb{I}_{A^c}})^{1/2} \left(\bE^{1/2}\thermal{h^2} + \bE^{1/2}\thermal{h^2}'\right).
    \end{equation*}    
\end{proof}
%\textcolor{red}{Can you show exacly how this simplification works? It seems one term was fully dropped? Maybe you used something basic that I am missing}\ms{I think this is okay. Here I used that $(\proba(B^c))^{1/2} \geq \proba(B^c)$ to combine the bounds for the last two terms into a single term.}
\begin{remark}
    If $||h||_\infty < \infty$, this simplifies to
    \begin{equation}\label{eq:bound_cond_mean}
         |\E\thermal{h(\bbeta)} - \bE\thermal{h(\bbeta)}'| \leq 2||h||_\infty \left[(\bE\thermal{\mathbb{I}_{A^c}})^{1/2}+(\proba(B^c))^{1/2}\right].
    \end{equation}
\end{remark}

In Section \ref{sec:proofs}, we use this kind of conditioned measures to compare the expectations under the full posterior and suitable conditional posteriors. In that section, we also require a convergence in distribution version of the aforementioned proposition. For this, we first prove two auxiliary results below. 

For some fixed $d\geq1$, let $(\bV_n)_{n\geq1}$ be a sequence of random vectors in $\R^d$ and $\bV\in\R^d$ another random vector such that $\bV_n\xrightarrow{d}{\bV}$. Denote by $\proba_n(\cdot)$ the probability measure induced by $\bV_n$, $\proba(\cdot)$ the one by ${\bV}$, and let $A\subseteq\R^d$ be some set of positive probability w.r.t all the measures $(\proba_n)_{n\geq1}$. Let $(\bar{\bV}_n)_{n\geq1}$ be the sequence of random vectors obtained by conditioning $(\bV_n)_{n\geq1}$ on $A$ and $\bar{\bV}$ the one obtained by conditioning $\bV$ on $A$. We then have the following.
\begin{lemma}\label{lem:cond_w_conv}
     If $A$ is a closed continuity set of $\proba$ and $\proba(A)>0$, then $\bar{\bV}_n\xrightarrow{d}\bar{\bV}$.
\end{lemma}
\begin{proof}
    Along the proof we will make several uses of Portmanteau's Theorem. In particular, we will prove the conclusion by showing that for all $C\subseteq\R^d$ closed,
    \begin{equation*}
        \lim\sup_{n\to\infty} \proba_n(C) \leq \proba(C).
    \end{equation*}
    For this, first notice that, by the convergence in distribution and the fact that $A$ is a continuity set of $\proba$, we have that
    \begin{equation}\label{eq:lim_cont_set}
        \lim_{n\to\infty}\proba_n(A) = \proba(A).
    \end{equation}
    Also, by the fact that $A$ is closed and the convergence in distribution, for every closed $C\subseteq\R^d$
    \begin{equation}\label{eq:lim_sup_closed}
        \lim\sup_{n\to\infty} \proba_n(C\cap A) \leq \proba(C\cap A)
    \end{equation}
    Let $\bar{\proba}_n(\cdot)$ be the probability measure induced by $\bar{\bV}_n$ and $\bar{\proba}(\cdot)$ the one induced by $\bar{\bV}$. Equations \eqref{eq:lim_cont_set} and \eqref{eq:lim_sup_closed} imply that, for every closed $C\subseteq\R^d$, there exists a vanishing sequence $(\epsilon_n)_{n\geq1}$ such that
    \begin{equation*}
        \bar{\proba}_n(C) = \frac{\proba_n(C\cap A)}{\proba_n(A)} \leq \frac{\proba(C\cap A)+\epsilon_n}{\proba(A)-\epsilon_n}.
    \end{equation*}
    The result is then obtained by taking $\lim\sup_n$ on both sides of the last equation and noticing that the right hand side tends to $\bar{\proba}(C)$. 
\end{proof}
From this convergence in distribution result and the fact that $\proba_n(A)\to\proba(A)$ we then have the following corollary.
\begin{corollary}\label{cor:f_ind_conv}
     If $A$ is a closed continuity set of $\proba$ and $\proba(A)>0$, then for all $f:A\mapsto\R$ continuous and bounded 
     \begin{equation*}
         \E(f(\bV_n)\mathbb{I}_A(\bV_n))\xrightarrow{n\to\infty}\E(f(\bV)\mathbb{I}_A(\bV)).
     \end{equation*}
\end{corollary}

    \section{General results to establish Hypothesis \ref{hyp2}}\label{app:conv_mon_control}

In this appendix we will derive two general results to prove the uniform boundedness for the moments of $a_{n}$ assumed in Hypothesis \ref{hyp2}.

\subsection{Moment control of fitted values under log-polynomial condition}

The main result of this subsection will be the following proposition.

\begin{proposition}\label{prop:control_mom_conv}
Assume that Hypothesis \ref{hyp1} holds. Furthermore, suppose for a constant $\tilde{C}$ that $|\tT| < \tilde{C}$ a.s.~and that there exists some polynomial $p(x,y)$ such that $A(x) = \log p(e^x,e^{-x})$. Then, for all $k\geq1$ there exists some $C_k>0$ such that, for all $n\geq1$, $\E\thermal{a_n^{2k}} \leq C_k$.
\end{proposition}

The proof will be based on the following lemma that bounds the exponential moments of $a_n$ under the leave-an-observation-out measure $\thermal{\cdot}_o$ defined in \eqref{eq:mean_observation}. In this subsection, $\E_n(\cdot)$ will be expectation with respect to the single normal vector $\bX_n$ and $\tE(\cdot)$ with respect to $(\bX_1,\dots,\bX_{n-1})$. We then have that $\tE(\E_n(\cdot))$ is the expectation with respect to $\bX$. Recall that, for $m\geq1$ and $i\in[n]$, $a_{\star,i}=\bX_i^\top\bbeta_\star$, $a_{m,i}=\bX_i^\top\bbeta^{(m)}$ and further that $a_n=\bX_n^\top\bbeta^{(1)}$.

\begin{lemma}
    For all $\alpha>0$ there exists $C>0$ such that $\log \En\thermal{\exp\{\alpha a_n\}}_o \leq C(||\tbX||_{op}+1)^2$.
\end{lemma}
\begin{proof}
    Note that
    \begin{equation*}
        \En\thermal{\exp\{\alpha a_n\}}_o = \En\thermal{\exp\Big\{\alpha \sum_{j\in[p]} X_{jn}\beta_j\Big\}}_o = \thermal{\exp\Big\{\frac{\alpha^2 ||\bbeta||^2}{2n}\Big\}}_o.
    \end{equation*}
    Here, as before, we used \cite[Proposition A.2.1]{talagrand2010mean} to exchange derivatives and integrals. Analogously as in Proposition \ref{prop:exp_bound_norm}, we have that for every sufficiently small $t > 0$,  for some $K>0$,
    \begin{equation*}
        \log\thermal{\exp\{t ||\bbeta||^2\}}_o \leq n K (||\tbX||_{op}+1)^2.
    \end{equation*}
    Choosing some $t > 0$ such that this holds and letting $n\geq1$ be large enough so that $2nt>\alpha^2$, we obtain that
    \begin{equation*}
        \log\thermal{\exp\Big\{\frac{\alpha^2}{2nt}t||\bbeta||^2\Big\}}_o \leq \frac{\alpha^2}{2nt}\log\thermal{\exp\{t||\bbeta||^2\}}_o \leq \frac{\alpha^2 K}{2t}(||\tbX||_{op}+1)^2,
    \end{equation*}
    where we used the concave version of Jensen's applied to the function $x^a$ (with $a:=\alpha^2/2nt<1$) to obtain the first inequality. We then have the conclusion  choosing $C := \alpha^2K/2t$.
\end{proof}

From here we can obtain the desired bound for the moments of $a_n$, as argued below. %\textcolor{red}{TILL HERE.}

\begin{proof}[Proof of Proposition \ref{prop:control_mom_conv}]
    Notice that, for every $f:\R\mapsto\R$,
    \begin{equation}\label{eq:aux_cav_eq}
        \thermal{f(a_n)} = \frac{\thermal{f(a_n)\exp\{\tT_n a_n - A(a_n)\}}_o}{\thermal{\exp\{\tT_n a_n - A(a_n)\}}_o}.
    \end{equation}
    Now note that
    \begin{equation}\label{eq:log_exp_an}
        \begin{split}
            \log\En\thermal{\exp\{a_n\}} &\leq \frac{1}{2}\log\En\thermal{\exp\{(1+\tT_n) a_n - A(a_n)\}}^2_o + \frac{1}{2}\log\En\thermal{\exp\{\tT_n a_n - A(a_n)\}}^{-2}_o \\
            & \leq \frac{1}{2}\log\En\thermal{\exp\{2(1+\tT_n) a_n\}}_o + \frac{1}{2}\log\En\thermal{\exp\{-2\tT_n a_n +2 A(a_n)\}}_o,
        \end{split}
    \end{equation}
    %\textcolor{red}{A bit confused with the first line above--since the $E_n$ was outside of the fraction and log outside of that for the LHS-- then how does that become additive? Can you please write some in between steps?}\ms{This is just CS and the properties of log. Indeed,
  %  \begin{equation*}
   %     \begin{split}
    %        \log\E\left[ \frac{\text{Numerator}}{\text{Denominator}}\right] & \leq  \log\left(\sqrt{\E(\text{Numerator})^2 \E(\text{Denominator})^{-2}}\right) \\
   %         & = \frac{1}{2} \log\E(\text{Numerator})^2 + \frac{1}{2} \log\E(\text{Denominator})^2.
    %    \end{split}
    %\end{equation*}
    %}
    where in the first line we used  \eqref{eq:aux_cav_eq} and Cauchy-Schwarz, and in the second one that $A\geq0$ and Jensen's on both terms for $x^2$ and $1/x^2$, respectively.   
    We will now upper bound these terms by functions of $||\tbX||_{op}$. For the first we have, for some $K\geq0$,
    \begin{equation*}
        \begin{split}
            \log\En\thermal{\exp\{2(1+\tT_n) a_n\}}_o & \leq \log\En\thermal{\exp\{2(1+\tilde{C}) a_n\}\ind{a_n\geq0}+\exp\{2(1-\tilde{C}) a_n\}\ind{a_n<0}}_o \\
            &\leq \log2\max\Big\{\En\thermal{\exp\{2(1+\tilde{C}) a_n\}}_o,\En\thermal{\exp\{2(1-\tilde{C}) a_n\}}_o\Big\} \\
            & \leq K (||\tbX||_{op}+1)^2,
        \end{split}
    \end{equation*}
%\textcolor{red}{how do you get the first line of the above since the $\tilde{C}$ bound was on the derivatives of $\tT_n$?} 
where in the last line we used the previous lemma. In an analogous manner, we may obtain a similar bound for the second term of the second line of \eqref{eq:log_exp_an}. For the second term we explicitly use that $A(x)=\log p(e^x,e^{-x})$ for some polynomial $p$ to arrive at    \begin{equation*}
        \En\thermal{\exp\{-2\tT_n a_n +2 A(a_n)\}}_o = \En\thermal{e^{-2\tT_n a_n}p^2(e^{a_n},e^{-a_n})}_o.
    \end{equation*}
    Bounding in a similar fashion as done for the first term shows that there is another constant $K'\geq0$ such that
    \begin{equation*}
        \log\En\thermal{\exp\{-2\tT_n a_n +2 A(a_n)\}}_o \leq K' (||\tbX||_{op}+1)^2.
    \end{equation*}

    Define $Z:=K'(||\tbX||_{op}+1)^2$. Finally, by \cite[Lemma 3.1.8]{talagrand2010mean}, which states that if a random variable $X$ is such that $\log(\E(e^X)) \leq C$, then (for every $k\geq1)$ $\E(X^k) \leq 2^k(k^k + C^k)$, we have that there is some $K''>0$ such that %\ps{Maybe instead of simply calling Talagrand here, we can maybe summarize the lemma statement or something like that? Otherwise reader has to open it tocheck if our proof is true.}\ms{Done.}
        \begin{equation*}
            \E\thermal{a_n^{2k}} \leq K''(1+\tE Z^{2k}).
        \end{equation*}
        The conclusion follows by fixing $C_k := K''(1+\sup_{n\geq1}\tE Z^{2k})$ which is finite.
    \end{proof}

% ------------------------------------------------------------------------------------------------------------------------------------
% ------------------------------------------------------------------------------------------------------------------------------------
% ------------------------------------------------------------------------------------------------------------------------------------

\section{Moment control of coordinates of posterior samples}\label{app:moment_control}

In this appendix we will prove that, for all $k\geq1$, the $k$-th moment of each coordinate of a sample $\bbeta$ from the full posterior is a bounded sequence in $n$. 
%\ps{I would forego the following sentence. Where we use something specific from Talagrand, we can maybe refer}\ms{I commented out the phrase.}
% Such an argument was presented in  \cite{talagrand2010mean} for a highly stylized model known as  the Shcheribina-Tirozzi model in statistical physics which corresponds to a logistic model with true signal equals $\bm{0}$, that is the global null model. We adapt the proof strategy from herein, however, the existence of a  non-zero signal in our case  presents significant additional challenges that we address here. 

\begin{lemma}\label{lem:cont_mom}
    Suppose Hypothesis \ref{hyp1} holds. Then, for every $k\geq 1$ there exists some $C_k > 0$ such that, for all $n\geq1$ and $j_0\in[p]$, $\E\thermal{\beta_{j_0}^{2k}} \leq C_k$.
\end{lemma}

For this, we will first prove an intermediate result. This will require the definition of the following auxiliary measure.  Recall that $\tbX$ is the design matrix obtained by removing the ${j_0}$-th column and 
%coordinate of each covariate, let $\tbbeta\in\R^{p-1}$ and 
$\tbbetas\in\R^{p-1}$ is the vector obtained by removing the $j_0$-th coordinate of %the sample $\bbeta$ from the posterior and 
the true signal $\bbetas$. Recall definitions \eqref{eq:def_tT_u} and \eqref{eq:ta_is}--\eqref{eq:mean_variable}
%\ps{Actually I added the description of $\tbbeta$ to the previous line, but I am not sure if it right. $\tbbeta$ could either be the $j_0$-th coordinate removed from $\bbeta$, a sample from the full posterior, or could be a sample from the leave-a-variable-out posterior. These are two different things, and we need to be careful, we are not using the same notation to mean these two different things at different parts of the manuscript. Also by this notation in the lemma statement below $\tT_i(a_{\star,i})$ should actually be $\tT_i(1)$. Lets discuss how we wish to fix this notation} 
Fix ${j_0}\in[p]$. For $i\in[n]$ and $t\in[0,1]$, define $\tT_{i,t} := \tT_i(t X_{ij_0}\beta_{\star,j_0} +\tasi)$, $\tT'_{i,t} := \tT'_i(t X_{ij_0}\beta_{\star,j_0} +\tasi)$, and $\tT''_{i,t} := \tT''_i(t X_{ij_0}\beta_{\star,j_0} +\tasi)$, where recall that $\tasi=\tbX_i^{\top}\tbbetas$. For $f:\R^{p-1}\mapsto\R$ define
%\textcolor{red}{I think this $\tT_{i}(t)$ notation is confusing because $\tT$ was already a well-defined function before so $\tT_i(x)$ at any point x has a specific meaning in the context of all our previous stuff.}\ms{We agreed to use the sub-index notation that is present in the formulas now.}

\begin{equation}\label{eq:interpolLOOV}
    \langle f(\tbbeta) \rangle_{{j_0},t} := \frac{1}{Z_{{j_0},t}}\int f(\tbbeta) \exp\left\{\sum_{i\in[n]} \tT_{i,t} \tbX_i^\top\tbbeta - A(\tbX_i^\top\tbbeta)\right\} \prod_{j\in[p]\backslash\{j_0\}}\mu(d\tilde{\beta}_j);
\end{equation}
%\ps{This is odd $d \mu (\tbbeta)$ probably doesnt make sense since support of $\mu$ is 1-d. I guess you meant 
%$\mu^{\otimes p-1}$?}\ms{corrected}
with $Z_{{j_0},t}$ a suitable normalization constant. 

\begin{lemma}\label{lem:bet_exp_bound}
    Under the same conditions as the previous lemma, there is a constant $C > 0$ s.t. for every $j_0\in[p]$
    \begin{equation*}
        \log \langle \exp \frac{1}{4}\beta_{j_0}^2 \rangle \leq C \left(\sum_{i\in[n]} \langle(\tT_i(a_{\star,i}) - A'(\ta_i)) X_{i{j_0}}\rangle_{{j_0},1} + 1\right)^2.
    \end{equation*}
\end{lemma}
\begin{proof}
    The proof follows by adapting \cite[Lemma 3.2.6]{talagrand2010mean} to our setting. According to \cite[Theorem 3.1.4]{talagrand2010mean}, if the logarithm of the density of the posterior (which we will denote by $\phi(\cdot)$) is such that, for all $\bbeta,\bbeta'\in\R^p$ and some fixed $K > 0$,
    \begin{equation}\label{eq:tal_condition}
        \phi\left(\frac{\bbeta+\bbeta'}{2}\right) - \frac{\phi(\bbeta)+\phi(\bbeta')}{2} \geq K ||\bbeta-\bbeta'||^2,
    \end{equation}
    then
    \begin{equation}\label{eq:tal_exp_bound}
        \thermal{e^{(\beta_{j_0}-\thermal{\beta_{j_0}})^2}} \leq 4.
    \end{equation}
    Because of the strong log-concavity of the prior and the convexity of $A(\cdot)$ there is some constant $\delta > 0$ such that $\phi(\cdot) + \delta ||\cdot||^2$ is concave. 
  We therefore have that
    \begin{equation*}
        \phi\left(\frac{\bbeta+\bbeta'}{2}\right) + \delta\left\|\frac{\bbeta+\bbeta'}{2}\right\|^2 - \frac{\phi(\bbeta)+\phi(\bbeta')}{2}-\delta \frac{\|\bbeta\|^2+\|\bbeta'\|^2}{2}\geq 0
    \end{equation*}
    or, what is the same,
    \begin{equation*}
        \phi\left(\frac{\bbeta+\bbeta'}{2}\right) - \frac{\phi(\bbeta)+\phi(\bbeta')}{2} \geq -\delta\left\|\frac{\bbeta+\bbeta'}{2}\right\|^2 + \delta \frac{\|\bbeta\|^2+\|\bbeta'\|^2}{2}.
    \end{equation*}
    Now, notice that
    \begin{equation*}
        -\delta\left\|\frac{\bbeta+\bbeta'}{2}\right\|^2 + \delta\frac{\|\bbeta\|^2+\|\bbeta'\|^2}{2} = -\frac{\delta}{4}\|\bbeta+\bbeta'\|^2 + \frac{\delta}{2}(\|\bbeta\|^2+\|\bbeta'\|^2) = \frac{\delta}{4}\|\bbeta-\bbeta'\|^2.
    \end{equation*}
    Thus, inequality \eqref{eq:tal_condition} holds with $K=\delta/4$.
    The other result we will use is \cite[Lemma 3.2.2]{talagrand2010mean} which states that if $w:\R\mapsto\R$ is concave and achieves its maximum at some $x_*\in\R$, then for all $K'>0$ such that $w''\leq -K'$
    \begin{equation}\label{eq:bound_fluct}
        K' \int (x -x_*)^2 e^{w(x)} dx \leq \int e^{w(x)} dx.
    \end{equation}
        Define $w:\R\mapsto\R$ to be $w(x) = f(x) + \log \mu(x)$; where
    \begin{equation*}
        f(x) := \log \int \exp\left\{\sum_{i\in[n]} \tT_i(a_{\star,i}) (\tbX_i^\top\tbbeta + x X_{i{j_0}})  - A(\tbX_i^\top\tbbeta + x X_{i{j_0}})\right\} \prod_{j\in[p]\backslash\{j_0\}}\mu(d\beta_j). 
    \end{equation*}
    Now, notice that the marginal of $\beta_{j_0}$ will be proportional to $e^{w(\beta_{j_0})}$ and that, for some $\epsilon > 0$, $w'' \leq -\epsilon$ because of the strong log-concavity of $w$. From \eqref{eq:bound_fluct} and \cite[Lemma 3.2.2]{talagrand2010mean} we obtain that
    \begin{equation*}
        \langle(\beta_{j_0}-x_*)^2\rangle \leq \frac{1}{\epsilon}
    \end{equation*}
    with $x_*$ the maximum of $w$. This then implies that $|\langle \beta_{j_0} \rangle| \leq \epsilon^{-1/2} + |x_*|$. By the strong concavity of $w$ we also have that $|w'(0)| = |w'(x_*) - w'(0)| \geq \epsilon |x_*|$. Here, with some abuse of notation, refer to the density of $\mu(\cdot)$ as $\mu(x)$. We will also assume without loss of generality that $\mu(0)>0$. If this is not the case, the argument can be easily adapted to a generic point in the support of the density of $\mu(x)$. Combining the above, we then obtain that 
    \begin{equation}\label{eq:sec_exp_bound}
        \begin{split}
            \exp\langle\beta_{j_0}\rangle^2 & \leq \exp\left(\frac{\sqrt{\epsilon} + w'(0)}{\epsilon}\right)^2\\
            & = \exp\left\{\epsilon^{-2}\left(\sqrt{\epsilon}+ f'(0) + \frac{\mu'(0)}{\mu(0)}\right)^2\right\}\\
            & = \exp\left\{\epsilon^{-2}\left(\sqrt{\epsilon}+\sum_{i\in[n]} \langle(\tT_i - \tilde{A'_i}) X_{ij_0}\rangle_{{j_0},1} + \frac{\mu'(0)}{\mu(0)}\right)^2\right\}.
        \end{split}
    \end{equation}
    By the convexity of $e^{\frac{x^2}{2}}$ and bounds \eqref{eq:tal_exp_bound} and \eqref{eq:sec_exp_bound} we obtain 
    \begin{equation*}
        \begin{split}
            \langle \exp \frac{1}{4}\beta_{j_0}^2 \rangle & = \left\langle \exp\left\{ \frac{1}{2}\thermal{\beta_{j_0}} + \frac{1}{2}(\beta_{j_0}-\thermal{\beta_{j_0}}) \right\}^2\right\rangle \\
            & \leq \frac{1}{2}\langle \exp \thermal{\beta_{j_0}}^2 \rangle + \frac{1}{2}\langle \exp (\beta_{j_0}-\thermal{\beta_{j_0}})^2 \rangle \\
            & \leq \frac{1}{2}\exp\left\{\epsilon^{-2}\left(\sqrt{\epsilon}+\sum_{i\in[n]} \langle(\tT_i - \tilde{A'_i}) X_{ip}\rangle_{{j_0},1} + \frac{\mu'(0)}{\mu(0)}\right)^2\right\} + 2.
        \end{split}
    \end{equation*}
    The result then follows by using that if $K'' > 0$ and $x \geq1$ then $\log(K''+x) \leq K'' + \log x$.
\end{proof}

From this result we can then achieve uniform control for the moments of $\beta_{j_0}$. By Lemma \ref{lem:bet_exp_bound} and \cite[Lemma 3.1.8]{talagrand2010mean}, it is enough to show that there is some $K> 0$ s.t., for all $n\geq1$,
\begin{equation*}
    \E\left(\sum_{i\in[n]}\tT_i(a_{\star,i})  X_{i{j_0}}\right)^{2k}, \,\, \E\left(\sum_{i\in[n]}\langle A'(\ta_i)\rangle_{{j_0},1} X_{i{j_0}}\right)^{2k} \leq K.
\end{equation*}
To prove this, we will need the following auxiliary proposition.
\begin{proposition}\label{prop:control_mom_conv_t}
    For all $k\geq1$, there is some constant $C_k >0$ such that, for all $n\geq1$ and $t\in[0,1]$, 
    \begin{equation*}
        \E\thermal{a_1^{2k}}_{j_0,t} \leq C_k.
    \end{equation*}
\end{proposition}
\begin{proof}
    For $K > 0$, define 
    \begin{equation}\label{eq:def_DK_B}
        D_K:=\{||\tbbeta||\leq\sqrt{n}K\} \quad \text{and} \quad B:=\{||\tbX||_{op}\leq3\}    
    \end{equation}
    Let $K$ be large enough so that $\mu([-K,K]) > 0$. Notice that, for every $\bX$, there is a $K' > 0$ such that
    \begin{equation}\label{eq:def_cond_DK}
        \thermal{\mathbb{I}_{D_K}}_{j_0,t} \geq K' \mu^{\otimes p-1}([-K,K]^{p-1}).
    \end{equation}
As in \eqref{eq:def_mean_bE}, we define $\bE(\cdot)$ as the conditional expectation with respect to $(\bX,\mathbf{e})$ conditional on event $B$. Moreover, for every $f:\R^{p-1}\mapsto\R$ we let
    \begin{equation*}
        \thermal{f(\tbbeta)}'_{j_0,t} := \frac{\thermal{f(\tbbeta)\mathbb{I}_{D_K}}_{j_0,t} }{\thermal{\mathbb{I}_{D_K}}_{j_0,t} }
    \end{equation*}
where $\thermal{\mathbb{I}_{D_K}}_{j_0,t} > 0$ a.s. 
since, for almost all $\bX$, the posterior density is strictly non-negative on $D_K$. In the zero probability event $\{\thermal{\mathbb{I}_{D_K}}_{j_0,t} = 0\}$ we just define the ratio as $0$. From bounds in \cite{vershynin2010introduction} for $\proba(B^c)$ and an adaptation of Proposition \ref{prop:exp_bound_norm} to $\thermal{\cdot}_{j_0,t}$ we can easily see that:
    \begin{enumerate}[(i)]
        \item for some $\alpha > 0$, $\proba(B^c) \leq e^{-\alpha n}$;
        \item for $K > 0$ large enough, there is an $\alpha' > 0$ s.t. $\bE\thermal{\mathbb{I}_{D^c_K}}_{j_0,t} \leq e^{-\alpha' n}$;
        \item and, for all $k'\geq1$, $\E\thermal{\ta_1^{2k'}}_{j_0,t} \leq n^{-1}\E\thermal{||\bX\bbeta||^{2k'}}_{j_0,t}  \leq n^{k'-1}$.
    \end{enumerate}

Notice that, by (iii), moments of $\ta_i$ grow at most polynomially under the measure $\E\thermal{\cdot}_{j_0,t}$ while upper bounds in (i) and (ii) are exponentially decaying. This readily implies that, for all $k\geq1$, $\E\thermal{\ta_i^{2k}}_{j_0,t}$ and $\bE\thermal{\ta_i^{2k}}_{j_0,t}$ differ by an $o(1)$ quantity. Similarly, by Proposition \ref{prop:comparison_conditional}, we have that the moments $\bE\thermal{\ta_i^{2k}}_{j_0,t}$ and $\bE\thermal{\ta_i^{2k}}'_{j_0,t}$ also differ by an $o(1)$ quantity. We will make repeated uses of this in the rest of the proof.
By taking a derivative and using a Gaussian integration by parts we obtain
    \begin{equation*}
        \begin{split}
            \frac{d}{dt}\E\thermal{\ta_1^{2k}}_{j_0,t} & = \sum_{i\in[n]} \E\thermal{\ta_{1,1}^{2k}(\ta_{1,i}-\ta_{2,i})\tTp_{i,t} X_{ij_0}\beta_{\star,j_0}}_{j_0,t} \\
            & = \frac{t}{n}\sum_{i\in[n]} \E\thermal{\ta_{1,1}^{2k}(\ta_{1,i}-\ta_{2,i})\tT''_{i,t}\beta^2_{\star,j_0}}_{j_0,t} \\
            & \hspace{2cm}+ \E\thermal{\ta_{1,1}^{2k}(\ta_{1,i}-\ta_{2,i})(\ta_{1,i}+\ta_{2,i}-\ta_{3,i})\tTp^2_{i,t}\beta^2_{\star,j_0}}_{j_0,t} \\
            & = \frac{t}{n}\sum_{i\in[n]} \bE\thermal{\ta_{1,1}^{2k}(\ta_{1,i}-\ta_{2,i})\tT''_{i,t}\beta^2_{\star,j_0}}'_{j_0,t} \\
            & \hspace{2cm}+ \bE\thermal{\ta_{1,1}^{2k}(\ta_{1,i}-\ta_{2,i})(\ta_{1,i}+\ta_{2,i}-\ta_{3,i})\tTp^2_{i,t}\beta^2_{\star,j_0}}'_{j_0,t} + o(1);
        \end{split}
    \end{equation*}
where for the second equality we used Gaussian integration by parts with respect to $X_{ij_0}$: the first terms comes from the derivative of $\tTp_{i,t}$ w.r.t. $X_{ij_0}$ and the second term comes from the derivatives of $\exp\left\{\sum_{i\in[n]} \tT_{i,t} \tbX_i^\top\tbbeta - A(\tbX_i^\top\tbbeta)\right\}$ and the normalization constant $Z_{j_0,t}$; and for the third equality we used (i)-(iii) along with Proposition \ref{prop:comparison_conditional} to approximate $\E(\cdot)$ by $\bE(\cdot)$ and, in the second term, $\thermal{\cdot}_{j_0,t}$ by $\thermal{\cdot}'_{j_0,t}$, as discussed above. By grouping the terms on the last line correctly and using Hypothesis \ref{hyp1} we see that for some appropriate constants $K',K'' > 0$ 
\begin{equation*}
        \begin{split}
            \frac{t}{n}&\sum_{i\in[n]} \bE\thermal{\ta_{1,1}^{2k}(\ta_{1,i}-\ta_{2,i})\tT''_{i,t}\beta^2_{\star,j_0}}'_{j_0,t} + \bE\thermal{\ta_{1,1}^{2k}(\ta_{1,i}-\ta_{2,i})(\ta_{1,i}+\ta_{2,i}-\ta_{3,i})\tTp^2_{i,t}\beta^2_{\star,j_0}}'_{j_0,t}\\
            &\hspace{1cm}=\frac{t}{n} \bE\thermal{\ta_{1,1}^{2k}\sum_{i\in[n]}\ta_{1,i}\tT''_{i,t}\beta^2_{\star,j_0}}'_{j_0,t} - \frac{t}{n} \bE\left[\thermal{\ta_{1,1}^{2k}}'_{j_0,t}\sum_{i\in[n]}\thermal{\ta_{2,i}\tT''_{i,t}\beta^2_{\star,j_0}}'_{j_0,t}\right] \\
            & \hspace{2cm}+ \frac{t}{n} \bE\thermal{\ta_{1,1}^{2k}||\tbX\tbbeta||^2\tTp^2_{i,t}\beta^2_{\star,j_0}}'_{j_0,t} - \frac{t}{n} \bE\left[\thermal{\ta_{1,1}^{2k}}'_{j_0,t}\thermal{||\tbX\tbbeta||^2\tTp^2_{i,t}\beta^2_{\star,j_0}}'_{j_0,t}\right]\\
            & \hspace{2cm} - \frac{t}{n} \bE\thermal{\ta_{1,1}^{2k}\sum_{i\in[n]}\ta_{1,i}\ta_{3,i}\beta^2_{\star,j_0}}'_{j_0,t} + \frac{t}{n} \bE\thermal{\ta_{1,1}^{2k}\sum_{i\in[n]}\ta_{2,i}\ta_{3,i}\beta^2_{\star,j_0}}'_{j_0,t}\\
            & \hspace{1cm} \leq K''\Big(\bE\thermal{\ta_{1,1}^{2k}n^{-1/2}||\tbX\tbbeta||}'_{j_0,t} + \bE\left(\thermal{\ta_{1,1}^{2k}}'_{j_0,t}n^{-1/2}\thermal{||\tbX\tbbeta||}'_{j_0,t}\right) \\
            & \hspace{2cm}+ \bE\thermal{\ta_{1,1}^{2k}n^{-1}||\tbX\tbbeta||^2}'_{j_0,t} + \bE\left[\thermal{\ta_{1,1}^{2k}}'_{j_0,t}n^{-1}\thermal{||\tbX\tbbeta||^2}'_{j_0,t}\right]\\
            & \hspace{2cm} + \bE\thermal{\ta_{1,1}^{2k}n^{-1}||\tbX\tbbeta^{(1)}||\,||\tbX\tbbeta^{(2)}||}'_{j_0,t} +\bE\thermal{\ta_{1,1}^{2k}n^{-1}||\tbX\tbbeta^{(2)}||\,||\tbX\tbbeta^{(3)}||}'_{j_0,t}\Big)\\
            & \hspace{1cm}\leq K' \bE\thermal{\ta_{1}^{2k}}'_{j_0,t};
        \end{split}
    \end{equation*}
   where we used that $\tTp_{i,t}$ and $\tT''_{i,t}$ are bounded, that $\sum_{i\in[n]}|\ta_i|\leq \sqrt{n}||\tbX\bbeta||$, $\sum_{i\in[n]}\ta^2_i = ||\tbX\tbbeta||^2$, and that, by the definitions \eqref{eq:def_DK_B} of the sets $D_K$ and $B$, under $\bE\thermal{\cdot}'_{j_0,t}$ the quantity $n^{-1}||\tbX\tbbeta||^2$ is uniformly bounded by a constant.  Approximating, using Lemmas \ref{lem:aprox_K} and \ref{lem:taylor_control}, $\bE(\cdot)$ by $\E(\cdot)$ and $\thermal{\cdot}'_{j_0,t}$ by $\thermal{\cdot}_{j_0,t}$, we then have that
    \begin{equation*}
        \frac{d}{dt}\E\thermal{\ta_1^{2k}}_{j_0,t} \leq K' \E\thermal{\ta_1^{2k}}_{j_0,t} + o(1)
    \end{equation*}
Let $X(t)$ be a non-negative and differentiable function and $K,K'>0$ two constants such that, for all $t\in[0,T],$
    \begin{equation*}
        \frac{dX(t)}{dt} \leq K' X(t) + K.
    \end{equation*}
    Then Gronwall's inequality, \cite[Theorem 1.2.2]{ames1997inequalities}, implies that
    \begin{equation*}
        X(t) \leq X(0)e^{K't} + \frac{K}{K'}(e^{K't}-1).
    \end{equation*}
    In our case, because $K$, as a sequence of $n$, is $o(1)$, this becomes
    \begin{equation*}
        X(t) \leq X(0)e^{K't} + o(1).
    \end{equation*}
    We then obtain that
    \begin{equation*}
\E\thermal{\ta_1^{2k}}_{j_0,t} \leq  \E\thermal{\ta_1^{2k}}_{j_0,0} \exp\{{K' t}\} +o(1).
    \end{equation*}
    %\textcolor{red}{I am confused at the end--something seems circular here or I am missing something--To get $\exp\{K't \}$ on the RHS, you actually need $\int_0^t \mathbb{E}<\tilde{a}^{2k}_1>_{j_0,s}ds \leq C$ however, arent you trying to prove the integrand is bounded by a constant? Here it seems you are already using that which seems circular. Like the integrad could depend on s and nothing in the proof before this line so far rules it out. So whats going on?}\ms{I don't think there is circularity here. I am using the following form of Gronwall's inequality. Let $X(t)$ be a non-negative and differentiable function and $K,K'>0$ two constants such that, for all $t\in[0,T]0,$
  %  \begin{equation*}
   %     \frac{dX(t)}{dt} \leq K X(t) + K'.
   % \end{equation*}
 %   Then [Gronwall],
 %   \begin{equation*}
  %      X(t) \leq X(0)e^{Kt} + \frac{K'}{K}(e^{Kt}-1).
  %  \end{equation*}
  %  In our case, because $K'$, as a sequence of $n$ is $o(1)$ this becomes
  %      \begin{equation*}
  %      X(t) \leq X(0)e^{Kt} + o(1).
  %  \end{equation*}
  %  Notice that I only assumed that $X(t)$ is non-negative and differentiable. 
   % }
   % \textcolor{red}{This explanation is helpful, can you please include it in the manuscript suitably?}\ms{Done!}
    But, by Hypothesis \ref{hyp2}, the right hand side is a bounded sequence. We then conclude the proof.
\end{proof}

\begin{proof}[Proof of Lemma \ref{lem:cont_mom}]
    As pointed out before, it is enough to show that there is some $K> 0$ s.t., for all $n\geq1$,
    \begin{equation}\label{eq:bound_needed}
        \E\left(\sum_{i\in[n]}\tT_i(a_{\star,i})  X_{i{j_0}}\right)^{2k}, \,\, \E\left(\sum_{i\in[n]}\langle A'(\ta_i)\rangle_{{j_0},1} X_{i{j_0}}\right)^{2k} \leq K.
    \end{equation}
    We will then prove both these inequalities separately.
    
    Notice that, by a first order approximation of $\tT$, for the first bound it is enough to show that, for some $K > 0$,
    \begin{equation*}
        \E\left(\sum_{i\in[n]}\tT_i(\tasi)  X_{i{j_0}}\right)^{2k},\,\,\E\left(\sum_{i\in[n]} \beta_{\star,{j_0}} X^2_{i{j_0}}\right)^{2k} \leq K(1+\beta^{2k}_{\star,{j_0}});
    \end{equation*}
    where we used that $\tT'$ is a bounded function. From the hypothesis of the lemma and the fact that $\beta_{\star,j_0}$ is fixed and the $X_{ij_0}$ are i.i.d. of mean $0$ and variance $1/n$ we directly see that there is some $K' > 0$ s.t. 
    \begin{equation*}
        \E\left(\sum_{i\in[n]} \beta_{\star,{j_0}} X^2_{i{j_0}}\right)^{2k} \leq K' \beta_{\star,{j_0}}^{2k} .
    \end{equation*}
    For the other bound notice that $\tT_i(\tasi)$ is independent of $X_{i{j_0}}$. Then expanding the power and again using that $X_{i{j_0}}$ are i.i.d. normal of mean $0$ and variance $1/n$ we have that, for some $K''>0$
    \begin{equation*}
        \begin{split}
            \E\left(\sum_{i\in[n]}\tT_i(\tasi) X_{i{j_0}}\right)^{2k} & = \sum_{i_1,\dots,i_{2k}\in[n]}\E\big(\tT_i(\ta_{\star,i_1})\cdots\tT_i(\ta_{\star,i_{2k}})\big) \,\,  \E(X_{i_1{j_0}}\cdots X_{i_{2k}{j_0}}) \\
            & = \sum_{i_1,\dots,i_{k}\in[n]}\E\big(\tT^2(\ta_{\star,i_1})\cdots\tT^2(\ta_{\star,i_{k}})\big) \,\,  \E(X^2_{i_1{j_0}}\cdots X^2_{i_{k}{j_0}}) \\
            & \leq K'' \sum_{i_1,\dots,i_{k}\in[n]}  \E(X^2_{i_1{j_0}}\cdots X^2_{i_{k}{j_0}}) = K'' \E\left(\sum_{i\in[n]} X_{i{j_0}}\right)^{2k} = K''(2k-1)!!;
        \end{split}
    \end{equation*}
where, for the last inequality, we used that, as a consequence of Hypothesis \ref{hyp1}, all moments of $\tT_1(\ta_{\star,1}),\dots,\tT_n(\ta_{\star,n})$ are bounded. This is indeed the case because Hypothesis \ref{hyp1} implies the at most linear growth of $\tT$ and the moments of $a_{\star,i}$ are easily seen to be finite. Thus, for some $K>0$, we have
    \begin{equation*}
        \E\left(\sum_{i\in[n]}\tT_i(a_{\star,i})  X_{i{j_0}}\right)^{2k} \leq K(1+\beta_{\star,{j_0}}^{2k}).
    \end{equation*}
        For the second bound of \eqref{eq:bound_needed}, by the Mean Value Theorem, for some $\xi\in[0,1]$
    \begin{equation*}
        \begin{split}
            \E\left(\sum_{i\in[n]}\langle A'(\ta_i)\rangle_{{j_0},1} X_{i{j_0}}\right)^{2k} & \leq \,\,\,2^{2k} \E\left(\sum_{i\in[n]}\langle A'(\ta_i)\rangle_{{j_0},0} X_{i{j_0}}\right)^{2k}\\
            & \hspace{1cm} +  2^{2k}\E\left(\sum_{i\in[n]}\langle (A'(\ta_{1,i})-A'(\ta_{2,i}))\ta_{1,i}\rangle_{{j_0},\xi}\tTp_{i,t}(\xi) X^2_{i{j_0}}\beta_{\star,j_0}\right)^{2k}\\
            & \leq K''' \left[\E\left(\sum_{i\in[n]}\langle A'(\ta_i)\rangle_{{j_0},0} X_{i{j_0}}\right)^{2k}+ \E\left(\sum_{i\in[n]} \thermal{|\ta_i|}_{{j_0},\xi} X^2_{i{j_0}}\right)^{2k}\right];
        \end{split}
    \end{equation*}
    for $K''' > 0$ some constant only depending on $k\geq1$ and where we used that, by Hypothesis \ref{hyp1}, $A''$ and $\tT'$ are bounded functions and $\beta_{\star,j_0}$ is a bounded sequence. It is enough then to show that the two expectations on the last line are uniformly bounded in $n$. For the first expectation, notice that $(\langle A'(\ta_i)\rangle_{{j_0},0})_{i\in[n]}$ are independent of the $(X_{i{j_0}})_{i\in[n]}$. Furthermore, because the $(X_{i{j_0}})_{i\in[n]}$ are i.i.d. centered Gaussian r.v.s of variance $1/n$ and by Hypothesis \ref{hyp2} we have as above that, for some $K^{(4)}>0$,
    \begin{equation*}
        \begin{split}
            \E\left(\sum_{i\in[n]}\langle A'(\ta_i)\rangle_{{j_0},0} X_{i{j_0}}\right)^{2k} & = \sum_{i_1,\dots,i_k\in[n]} \E\left(\langle A'(\tbX_{i_1}^\top \tbbeta)\rangle^2_{{j_0},0}\cdots\langle A'(\tbX_{i_k}^\top \tbbeta)\rangle^2_{{j_0},0}\right) \E\left(X^2_{i_1{j_0}}\cdots X^2_{i_k{j_0}}\right) \\
            & \leq K^{(4)} \sum_{i_1,\dots,i_k\in[n]} \E\left(X^2_{i_1{j_0}}\cdots X^2_{i_k{j_0}}\right) = K^{(4)} \E\left(\sum_{i\in[n]} X_{i{j_0}}\right)^{2k} \\
            & = K^{(4)}(2k-1)!!;
        \end{split}
    \end{equation*}
 Now, note that the proof of Proposition \ref{prop:exp_bound_norm} directly extends, for ${j_0}\in[p]$ and $t\in[0,1]$, to the measures $\thermal{\cdot}_{{j_0},t}$. Therefore, by the analogue of Corollary \ref{cor:bound_mom_norm}, for every $k\geq1$ there is a $C_k > 0$ such that for all $n\geq1$ and $t\in[0,1]$ it holds that
    \begin{equation*}
        \E\thermal{||\tbbeta||^{2k}}_{j_0,t} \leq C_k n^k.
    \end{equation*}
    Moreover, the moments of $n\sum_{i\in[n]} X_{i{j_0}}^4$ are uniformly bounded in $n$. Thus, for the second term it is enough to see that, for some $K^{(5)} > 0$,
    \begin{equation*}
        \begin{split}
            \E\left(\frac{1}{n}\sum_{i\in[n]} \thermal{|\ta_i|}_{{j_0},\xi} nX^2_{i{j_0}}\right)^{2k} & \leq \frac{1}{n}\sum_{i\in[n]}\E\left( \thermal{|\ta_i|}_{{j_0},\xi} nX^2_{i{j_0}}\right)^{2k} \\
            & \leq n^{2k}\E\thermal{\ta^{2k}_1 X^{4k}_{1{j_0}}}_{{j_0},\xi} \leq n^{2k}\left(\E\thermal{\ta^{4k}_1}_{{j_0},\xi} \E X^{8k}\right)^{1/2} \leq K^{(5)};
        \end{split}
    \end{equation*}
    where we used that $\E X^{8k}$ is $\mathcal{O}(n^{-4k})$ and Proposition \ref{prop:control_mom_conv_t}.
\end{proof}
    \section{Smooth approximation for logistic regression}\label{app:smooth_app}
%\textcolor{red}{I am very confused about this section from the beginning, First your probably meant $\tT$ here--next $\tT_i(x) = \boldsymbol{1}\{e_i\leq A'(x) \}$. My confusion is in the Hypothesis your assumptions are on $\tT$ and not $T$ so then why are you worried about T here? And by the way for logistic, $T(x) = x$ and not what is written here. We need to clarify this before I can move forward.}
In the case of logistic regression, $\tT_i(x) = \mathbb{I}_{\{x \geq \sigma^{-1}(e_i)\}}$, where $\sigma(\cdot)$ denotes the sigmoid function and $e_i \sim \text{Unif}(0,1)$, independent of everything else. This function is indeed bounded but is not $\mathcal{C}^3$ and therefore does not obey Hypothesis 1(ii). For each $\delta>0$, let $f_\delta(x):=(\tanh(x/\delta)+1)/2$. However, for each $\delta > 0$, we can approximate $\tT_i(x)$ by a function $\tT_{\delta,i}(x)$ given by $\tT_{\delta,i}(x) = f_\delta(x-\sigma^{-1}(e_i))$. And Hypothesis \ref{hyp1}(ii) does hold for $\tT_{\delta,i}$. Furthermore, when $\delta\approx0$ we have that $\tT_{\delta,i}\approx \tT_i$.

In the rest of this appendix, we will let $\proba(\cdot|\bX,\by)$ denote the logistic regression posterior given by
\begin{equation}\label{eq:full_posterior}
    \proba(d\bbeta|\bX,\by) = \frac{1}{Z_n} \exp\left\{\sum_{i\in[n]}\tT_i(a_{\star,i})\bX_i^\top\bbeta - \log(1+e^{\bX_i^\top\bbeta})\right\} \prod_{j\in[p]}\mu(d\beta_j).
\end{equation}
In a similar way, we will let $\proba_\delta(\cdot|\bX,\by)$ be the posterior measure associated to the smooth approximation used. That is, $\proba_\delta$ is the measure defined according to
\begin{equation}\label{eq:smooth_posterior}
    \proba_\delta(d\bbeta|\bX,\by) = \frac{1}{Z_{n,\delta}} \exp\left\{\sum_{i\in[n]}\tT_{\delta,i}(a_{\star,i})\bX_i^\top\bbeta - \log(1+e^{\bX_i^\top\bbeta})\right\} \prod_{j\in[p]}\mu(d\beta_j),
\end{equation}
where $Z_{n,\delta} > 0$ is a normalizing constant. We will then let $\thermal{\cdot}_\delta$ denote expectation with respect to \eqref{eq:smooth_posterior}. Finally, for each $j\in[p]$, we will let $\proba_j$ and $\proba_{\delta,j}$ denote the $j$-th marginals of $\proba$ and $\proba_\delta$, respectively.

\subsection{Convergence of order parameters}

In this section we will prove that the values of the asymptotic expectations of the order parameters \eqref{eq:overlaps} with respect to $\thermal{\cdot}_\delta$ approximate the ones with respect to $\thermal{\cdot}$. This is the content of the following proposition.
\begin{proposition}\label{prop:conv_order_logistic}
    In the above context, for every $m,m'\in\{1,2,\star\}$, we have that
    \begin{equation*}
        \lim_{n\to\infty} \E\thermal{Q_{mm'}}=\lim_{\delta\to0}\lim_{n\to\infty} \E\thermal{Q_{mm'}}_\delta.
    \end{equation*}
\end{proposition}

For simplicity we will write the proof for $Q_{1\star}$ while the ones for the other variables work in the same way. 
Consider the posterior corresponding to the log-likelihood
\begin{equation}\label{eq:approx_delta}
    \hat{\logl}_{n,p}(\bB) :=  \chi \bB^\top \bbetas + \sum_{i \in[n]} ((1-t)\tT_i(a_{\star,i})+t\tT_{\delta,i}(a_{\star,i}))\bX_i^\top\bB - A(\bX_i^\top\bB)
\end{equation}
and the same prior as the original model. The corresponding posterior expectation will be denoted by $\thermal{\cdot}_{\chi,\delta,t}$. We then have that, for $f:\R^p\mapsto\R$,
\begin{equation*}
    \thermal{f(\bbeta)}_{\chi,\delta,t} := \frac{\int f(\bB)\exp\{\hat{\logl}_{n,p}(\bB)\}\prod_{j\in[p]}\mu(db_j)}{\int\exp\{\hat{\logl}_{n,p}(\bB)\}\prod_{j\in[p]}\mu(db_j)}.
\end{equation*}
Notice that the logistic regression log-likelihood is recovered when $\chi$ and $t$ are both zero. In the rest of the section let $\epsilon > 0$ be as in Section \ref{app:moment_control}.

A key observation will be that the proof of Proposition \ref{prop:exp_bound_norm} holds with direct adaptations for this posterior measure. Because for every $t\in[0,1]$ and $\delta>0$
\begin{equation*}
    0 \leq (1-t)\tT_i(a_{\star,i})+t\tT_{\delta,i}(a_{\star,i}) \leq 1,
\end{equation*}
the same derivation proves that there are $a,K>0$ such that, for every $t\in[0,1]$, $\delta>0$ small enough, $0<\chi<\epsilon/4$, and $n\geq1$,
\begin{equation*}
    \log \thermal{\exp{a ||\bbeta||^2}}_{\chi,\delta,t} \leq K(||\bX||_{op}+1)^2.
\end{equation*}
This, as in Corollary \ref{cor:bound_mom_norm}, implies that there is some $K' >0$ for which
\begin{equation}\label{eq:bound_norm_param}
    n^{-1}\E\thermal{||\bbeta||^2}_{\chi,\delta,t},\, n^{-2}\E\thermal{||\bbeta||^4}_{\chi,\delta,t} \leq K'.
\end{equation}
From this we get that
\begin{equation}\label{eq:bound_conv_spin_param}
    \E\thermal{a_1^2}_{\chi,\delta,t} = n^{-1}\E\thermal{||\bX\bbeta||^2}_{\chi,\delta,t} \leq \left(\E||\bX||_{op}^4 n^{-2}\E\thermal{||\bbeta||^4}_{\chi,\delta,t}\right)^{1/2} \leq {K'}^{1/2} \left(\E||\bX||_{op}^4\right)^{1/2} \leq K'';
\end{equation}
for some $K''>0$ not depending on $n$, $\chi$, $\delta$, nor $t$. Lemma \ref{lem:logconc_tilted} of Section \ref{app:moment_control} also holds (uniformly on the parameters $\chi$, $\delta$, and $t$) for the measure $\thermal{\cdot}_{\chi,\delta,t}$.

As in equation \eqref{eq:brasc_bound} in the proof of Proposition \ref{prop:RS}, we have that if $0<\chi<\epsilon/4$, then for some $K'''>0$ 
\begin{equation}\label{eq:bound_means1}
    \begin{split}
        |\E\thermal{Q_{1\star}}_{\chi,\delta,1}-\E\thermal{Q_{1\star}}_{\delta}| & = \left| \frac{d}{d\chi} \E\thermal{Q_{1\star}}_{\chi,\delta,1} \Big|_{\chi=\xi} \right| \ \chi \\
        & = n \E\thermal{\big(Q_{1\star} - \thermal{Q_{1\star}}_{\xi,\delta,1}\big)^2}_{\xi,\delta,1}  \ \chi \\
        & \leq \frac{n\E\thermal{||\bbeta||^2}_{\xi,\delta,1}}{\epsilon n^2} \chi \leq K''' \chi;
    \end{split}
\end{equation}
where, for in the first line we used the Mean Value Theorem, for the second the explicit expression for the derivative, and for the last Brascamp-Lieb, valid by Lemma \ref{lem:logconc_tilted}, and equation \eqref{eq:bound_norm_param}. In an analogous way, again Lemma \ref{lem:logconc_tilted}, Brascamp-Lieb, and equation \eqref{eq:bound_norm_param} we get
\begin{equation}\label{eq:bound_means2}
|\E\thermal{Q_{1\star}}_{\chi,\delta,0}-\E\thermal{Q_{1\star}}| \leq K^{(4)} \chi;
\end{equation}
where $K^{(4)}>0$ is some constant and we used that, for every $\delta>0$, $\E\thermal{Q_{1\star}}_{0,\delta,0} = \E\thermal{Q_{1\star}}$.

Below, Lemma \ref{lem:ord_param_smooth} will connect in a similar way the expectation of $Q_{1\star}$ with respect to $\E\thermal{\cdot}_{\chi,\delta,1}$ and $\E\thermal{\cdot}_{\chi,\delta,0}$. Before proving this, we establish an auxiliary lemma.

\begin{lemma}\label{lem:boundTTdelta}
    There is a constant $C > 0$ such that for all $n\geq1$ and $\delta>0$,
    \begin{equation*}
        \sum_{i\in[n]} \E\thermal{(\tT_{\delta,i}(a_{\star,i})-\tT_i(a_{\star,i}))a_i}_{\chi,\delta,t} \leq C \delta^{1/4} n.
    \end{equation*}
\end{lemma}
\begin{proof}
    Let $I_{\delta,i} := (-\sqrt{\delta}+\sigma^{-1}(e_i),\sqrt{\delta}+\sigma^{-1}(e_i))$. Notice that
    \begin{equation*}
        \left|\mathbb{I}_{\{x\geq0\}}-f_\delta(x)\right| = \frac{1}{1 + e^{2|x|/\delta}}.
    \end{equation*}
We then have that, for $a_{\star,i}\in I_{\delta,i}^c$, $|\tT(a_{\star,i})-\tT_{\delta,i}(a_{\star,i})|$ is of order $\mathcal{O}(e^{-2/\sqrt{\delta}})$. To control the mean of the first term we use that by Hypothesis \ref{hyp2} and the above, for some $K > 0$,
    \begin{equation*}
        \E \thermal{(\tT_{\delta,i}-\tT_i)a_i}_{\chi,\delta,t} \leq K \sqrt{\E{(\tT_{\delta,i}-\tT_i)^2}} \leq K \sqrt{P(a_{\star,i}\in I_{\delta,i})+\mathcal{O}(e^{-4/\sqrt{\delta}})}.
    \end{equation*}
    And $P(a_{\star,i}\in I_{\delta,i})$ is easily seen to be $\mathcal{O}(\sqrt{\delta})$. 
    We then have that
    \begin{equation*}
        \sum_{i\in[n]} \thermal{(\tT_{\delta,i}-\tT_i)a_i}_{\chi,\delta,t} \leq K \delta^{1/4} n.
    \end{equation*}
\end{proof}

\begin{lemma}\label{lem:ord_param_smooth}
    There is some fixed $C>0$ such that, for every $\delta>0$ small enough and $\delta^{1/4}<\chi<\epsilon/4$,
    \begin{equation*}
        |\E\thermal{Q_{1\star}}_{\chi,\delta,1}-\E\thermal{Q_{1\star}}_{\chi,\delta,0}|\leq C \delta^{1/4}.
    \end{equation*}
\end{lemma}
\begin{proof}
    Let $F_{n,\chi,\delta,t}$ be the log-normalizing constant  associated with the density corresponding to the log-likelihood $\hat{\logl}_{n,p}$. We then have that, by Lemma \ref{lem:boundTTdelta},
    \begin{equation*}
        \left|\frac{d}{dt}\E F_{n,\chi,\delta,t}\right| = \left|\sum_{i\in[n]}\E\thermal{(\tT_i(a_{\star,i})-\tT_{\delta,i}(a_{\star,i})) a_i}_{\chi,\delta,t}\right| \leq C \delta^{1/4} n.
    \end{equation*}
    Then, for some constant $K^{(5)}>0$,
    \begin{equation}\label{eq:dif_eps_FE}
        \left|\E F_{n,\chi,\delta,1} - \E F_{n,\chi,\delta,0} \right| \leq n K^{(5)} \delta^{1/4}.
    \end{equation}
By taking the first two derivatives of $n^{-1}\E F_{n,\chi,\delta,1}$ with respect to $\chi$ we get
    \begin{equation*}
        n^{-1}\frac{d}{d\chi}\E F_{n,\chi,\delta,1} = \E\thermal{Q_{1\star}}_{\chi,\delta,1} \,\,\,\,\, \mbox{ and } \,\,\,\,\, n^{-1}\frac{d^2}{d\chi^2}\E F_{n,\chi,\delta,1} = n^{-1}\E\thermal{(Q_{1\star}-\thermal{Q_{1\star}}_{\chi,\delta,1})^2}_{\chi,\delta,1}\geq0.
    \end{equation*}
    This proves that, seen as a function of $\chi$, $n^{-1}\mathbb{E} F_{n,\chi,\delta,1}$ is a convex function whose first derivative gives $\E\thermal{Q_{1\star}}_{\chi,\delta,1}$. Similarly, $n^{-1}\mathbb{E} F_{n,\chi,\delta,0}$ is a convex function whose first derivative gives $\E\thermal{Q_{1\star}}_{\chi,\delta,0}$.

  As in \eqref{eq:convexderivdiff}, for every $d > 0$ sufficiently small we then have that
    \begin{equation}\label{eq:convex_lem_eps}
        |\E\thermal{Q_{1\star}}_{\chi,\delta,1}-\E\thermal{Q_{1\star}}_{\chi,\delta,0}| \leq \E\thermal{Q_{1\star}}_{\chi+d,\delta,1}-\E\thermal{Q_{1\star}}_{\chi-d,\delta,1}+\frac{1}{nd}\sum_{y\in\mathcal{Y}}|\E F_{n,y,1}-\E F_{n,y,0}|;
    \end{equation}
    with $\mathcal{Y}:=\{\chi-d,\chi,\chi+d\}$. As in Lemma \ref{lem:logconc_tilted}, as $\chi<\epsilon/4$, the measure associated to $\thermal{\cdot}_{\chi,\delta,1}$ is $\epsilon$-strongly log-concave for $\epsilon>0$. Then, by Brascamp-Lieb we have that for every such $\chi$
    \begin{equation}\label{eq:brascamp_eps}
        \frac{d}{d\chi}\thermal{Q_{1\star}}_{\chi,\delta,1} = n\thermal{(Q_{1\star}-\thermal{Q_{1\star}}_{\chi,\delta,1})^2}_{\chi,\delta,1} \leq \frac{||\bbetas||^2}{\epsilon n} \leq K^{(6)}
    \end{equation}
    for some $K^{(6)}>0$. Combining equations \eqref{eq:dif_eps_FE}, \eqref{eq:convex_lem_eps}, and \eqref{eq:brascamp_eps} and using the Mean Value Theorem we finally get that there is some $K^{(7)} > 0$ such that, for all $d > 0$ sufficiently small,
    \begin{equation*}
        |\E\thermal{Q_{1\star}}_{\chi,\delta,1}-\E\thermal{Q_{1\star}}_{\chi,\delta,0}| \leq K^{(7)}\left(d + \frac{\delta^{1/2}}{d}\right).
    \end{equation*}
    The conclusion of the lemma is reached by fixing the value $d = \delta^{1/4}$.
\end{proof}

\begin{proof}[Proof of Proposition \ref{prop:conv_order_logistic}]
    As mentioned before, we will prove this for $Q_{1\star}$ but the proofs for $Q_{11}$ and $Q_{12}$ follow in a similar way. By equations \eqref{eq:bound_means1} and \eqref{eq:bound_means2} and the lemma above, we have that for all $\delta>0$ sufficiently small and $\delta^{1/4}<\chi<\epsilon/4$
    \begin{equation*}
        \begin{split}
            |\E\thermal{Q_{1\star}}_{\delta}-\E\thermal{Q_{1\star}}| & \leq |\E\thermal{Q_{1\star}}_{\delta}-\E\thermal{Q_{1\star}}_{\chi,\delta,1}|+|\E\thermal{Q_{1\star}}_{\chi,\delta,1}-\E\thermal{Q_{1\star}}_{\chi,\delta,0}|+|\E\thermal{Q_{1\star}}_{\chi,\delta,0}-\E\thermal{Q_{1\star}}|\\
            & \leq K^{(8)}(\chi + \delta^{1/4}),
        \end{split}
    \end{equation*}
    for constants $K^{(8)}>0$ that do not depend on $n$, $\chi$, or $\delta$. The result then follows by taking first the limit of $n\to\infty$ and then $\delta\to0$ and $\chi\to0$.
\end{proof}

\subsection{Wasserstein convergence of marginals}

Recall the definition of the Wasserstein-2 metric between two Borel measures $\mu$ and $\nu$ on $\R^p$
\begin{equation*}
    W_2(\mu,\nu) := \inf_{\gamma \in \Gamma(\mu,\nu)} \left(\int  ||\boldsymbol{X}-\boldsymbol{Y}||^2 \gamma(d\boldsymbol{X},d\boldsymbol{Y})\right)^{1/2}
\end{equation*}
where $\Gamma(\mu,\nu)$ is the set of all couplings of $\mu$ and $\nu$.

Here we will prove the following bound on the Wasserstein-2 metric between the $j_0$-th marginal under $\mathbb{P}$ and $\mathbb{P}_{\delta}$ respectively (recall definitions from \eqref{eq:full_posterior} and \eqref{eq:smooth_posterior} respectively). Instead of the empirical measure convergence assumptions in Hypothesis 1(iv), for this section, we will assume that $\bbetas$ has i.i.d. coordinates drawn from $\pi(\cdot)$. This is stronger than Hypothesis 1(iv). The rest of our assumptions stay as earlier. For the rest of this section, we will assume that $\E(\cdot)$ stands for expectation with respect to $(\bX,\be)$, as before, but also with respect to the randomness in $\bbetas$.

\begin{proposition}\label{prop:marginal_smooth}
   % Assume Hypothesis \ref{hyp1} and \ref{hyp2} hold. Furthermore, suppose that the vector $\bbetas$ is such that every coordinate is i.i.d. with finite second moment. Then, 
   In the aforementioned setting, let $j_0 \in [p]$. There exists a constant $C > 0$ that does not depend on $n$ such that $\E W^2_2(\proba_{j_0},\proba_{\delta,j_0}) \leq C \delta^{1/4}$, where $\proba_{j_0},\proba_{\delta,j_0}$ denoted respectively the $j_0$-th marginal induced under $\proba$ and $\proba_\delta$. 
\end{proposition}

\begin{proof}
    For this proof, denote
    \begin{equation*}
        F_n := \log Z_n \quad \text{and} \quad F_{n,\delta} := \log Z_{n,\delta},
    \end{equation*}
    with $Z_n$ and $Z_{n,\delta}$ as in \eqref{eq:full_posterior} and \eqref{eq:smooth_posterior}. For two Borel measures $\mu$ and $\nu$ on $\R^p$ such that $\mu$ define their \emph{Kullback–Leibler divergence}, $D_{\rm KL},(\mu||\nu)$ according to
    \begin{equation*}
        D_{\rm KL}(\mu||\nu) := \int \log\left(\frac{d\mu}{d\nu}(\bB)\right) \mu(d\bB).
    \end{equation*}

    By \cite[Theorem 1]{otto2000generalization}, because the posterior is strongly log-concave, we have that Talagrand's transport inequality holds for it. This means that
    \begin{equation*}
        W^2_2(\proba,\proba_\delta) \leq 2 D_{\rm KL}(\proba_\delta||\proba);
    \end{equation*}
    Now expectation of the KL divergence is given by 
    \begin{equation*}
        \E D_{\rm KL}(\proba_\delta||\proba) = \sum_{i\in[n]} \E\thermal{(\tT_{\delta,i}-\tT_i)a_i}_\delta + \E(F_n - F_{\delta,n}).
    \end{equation*}
    
    The first term can be directly controlled using Lemma \ref{lem:boundTTdelta}. As in the previous subsection, let $F_{n,\chi,\delta,t}$ be the log-normalizing constant  associated with the density corresponding to the log-likelihood $\hat{\logl}_{n,p}$. In the rest of the proof we fix $\chi=0$ and thus omit its corresponding subscript. For the second term, notice that
    \begin{equation*}
        \begin{split}
            \frac{d F_{n,\delta,t}}{d t}\Big|_{t=\xi} & = \frac{1}{Z_{n,\delta,\xi}}\int \frac{d}{dt}\exp\big\{\sum_{i\in[n]}(t\tT_{i}+(1-t)\tT_{\delta,i})\bX_i^\top\bbeta - \log(1+e^{\bX_i^\top\bbeta})\big\}\Big|_{t=\xi}\prod_{j\in[p]}\mu(d\beta_j) \\
            & = \sum_{i\in[n]} \thermal{(\tT_i-\tT_{\delta,i}) a_i}_{\delta,\xi}.
        \end{split}
    \end{equation*}
    Then, by the Mean Value Theorem with respect to $t$, Can you please show an intermediate step for the following.
    \begin{equation*}
        F_n - F_{\delta,n} = \sum_{i\in[n]} \thermal{(\tT_i-\tT_{\delta,i}) a_i}_{\delta,\xi};
    \end{equation*}
    with $t=\xi$ some value in $(0,1)$. From this and another use of Lemma \ref{lem:boundTTdelta} we have that
    \begin{equation*}
        \E W_2^2(\proba,\proba_\delta) \leq C \delta^{1/4} n
    \end{equation*}

    To conclude, just notice that, because the coordinates of $\bbetas$ are exchangeable, then the measures $\proba$ and $\proba_\delta$ are symmetric under exchanges of the coordinates. %\textcolor{red}{Is the first inequality below a property of Wasserstein distances?}\ms{Yes, it is a property of $W_2$. The $W_2$ distance is defined as the infimum, with respect to couplings, of the $L^2$ distance between samples from the measures. In this case, if $\bbeta$ is a sample from $\proba$ and $\bbeta_\delta$ from $\proba_\delta$, then
    %\begin{equation*}
     %   \begin{split}
      %      W_2^2(\proba,\proba_\delta) & = \inf_{\Gamma(\bbeta,\bbeta_\delta)} \E_\Gamma||\bbeta-\bbeta_\delta||^2 \\
     %       & = \E_{\Gamma_\star}||\bbeta-\bbeta_\delta||^2 \\
      %      & = \sum_{j\in[p]} \E_{\Gamma_\star}(\beta_j-\beta_{\delta,j})^2 \\
       %     & \geq \sum_{j\in[p]} \inf_{\Gamma_j(\beta_j,\beta_{\delta,j})} \E_{\Gamma_j}(\beta_j-\beta_{\delta,j})^2 \\
     %       & = \sum_{j\in[p]} W_2^2(\proba_j,\proba_{\delta,j}).
     %   \end{split}
    %\end{equation*}
   % Here, for simplicity, I am assuming that there is a coupling $\Gamma_\star$ that exactly achieves the infimum in $\inf_{\Gamma(\bbeta,\bbeta_\delta)} \E_\Gamma||\bbeta-\bbeta_\delta||^2$ but the argument can be easily extended to cases in which there is no such coupling.}
    Then,
    \begin{equation*}
        \E W_2^2(\proba,\proba_\delta) \geq \sum_{j\in[p]} \E W_2^2(\proba_j,\proba_{\delta,j}) = p \E W_2^2(\proba_{j_0},\proba_{\delta,j_0}).
    \end{equation*}
    The conclusion then follows because $p/n\to \kappa$.
\end{proof}

    \section{Explicit fixed point equations for linear regression}\label{app:linear_regression}

\subsection{Proof of Proposition \ref{prop:lin_reg_eqs}}\label{sec:proof_lin_reg_eqs}

%\textcolor{red}{I feel I need to do this in last round, when I have checked accuracy and completeness of both current Section B.1 and Sections 3,4.}
For this subsection, we abbreviate $\theta(\xi_B), \xi_B$ as $\theta$ or $\theta(\xi)$ and $ \xi$ respectively.
Notice that, to have explicit expressions for these equations, it is enough to compute
\begin{equation*}
    \thermal{\beta}_{h}, \,\, \thermal{\beta^2}_{h}, \,\, \thermal{\theta(\xi)}_s \mbox{ and } \thermal{\theta^2(\xi)}_s.
\end{equation*}
We will compute these by repetitive use of \emph{Gaussian integration by parts}. The first part of the proposition will follow from calculating $\thermal{\theta(\xi)}_s$ and $\thermal{\theta^2(\xi)}_s$, which do not require the signal distribution $\mu(\cdot)$ to be Gaussian. 

First, notice that
\begin{equation*}
    \thermal{\theta}_s = \sqrt{\kappa(v_B-c_B)} \thermal{\xi}_s + \sqrt{\kappa c_B} z_{BB_\star}.
\end{equation*}
By suitable completion of squares, note that computing $\langle g(\xi) \rangle_s $ amounts to calculating $E(g(\xi))$ where 

\begin{equation}\label{eq:xidist}
\xi \sim \mathcal{N} \Big[\frac{c_1}{1+c_1^2}(\theta_\star + z - c_2z_{BB_{\star}}), \frac{1}{1+c_1^2} \Big], \,\, \text{with} \,\, c_1 = \sqrt{\kappa(v_B-c_B)}, c_2 =\sqrt{\kappa c_B}, 
\end{equation}
where $z= \Phi^{-1}(e) \sim \mathcal{N}(0,1)$, independent of everything else. 
%We therefore have that 
%Furthermore, by Gaussian integration by parts with respect to $\xi$ we have
%\begin{equation*}
%    \thermal{\xi}_s = - \kappa(v_B-c_B) \thermal{\xi}_s + \sqrt{\kappa(v_B-c_B)} (\theta_\star + z - \sqrt{\kappa c_B} z_{BB_\star});
%\end{equation*}
Hence, 
\begin{equation*}
        \thermal{\xi}_s = \frac{\sqrt{\kappa(v_B-c_B)}}{\kappa(v_B-c_B) + 1} (\theta_\star + z -  \sqrt{\kappa c_B} z_{BB_\star}). 
\end{equation*}
 This implies that

\begin{equation}\label{eq:mean_h}
    \thermal{\theta}_s = \frac{\kappa(v_B-c_B)}{\kappa(v_B-c_B) + 1} (\theta_\star + z) + \frac{\sqrt{\kappa c_B}}{\kappa(v_B-c_B) + 1}  z_{BB_\star}.
\end{equation}
We now compute the second moment of $\theta$.
\begin{equation*}
   \thermal{\theta^2}_s = \kappa(v_B-c_B) \thermal{\xi^2}_s + \kappa c_B z^2_{BB_{\star}} + 2 \kappa \sqrt{(v_B-c_B)c_B} \thermal{\xi}_s z_{BB_\star}.
\end{equation*}
%From yet another use of Gaussian integration by parts we have
%\begin{equation*}
 %   \thermal{\xi^2}_s = 1 - \kappa(v_B-c_B) \thermal{\xi^2}_s + \sqrt{\kappa(v_B-c_B)} (\theta_\star + z) \thermal{\xi}_s - \kappa \sqrt{(v_B-c_B)c_B} z_{BB_\star} \thermal{\xi}_s 
%\end{equation*}
 From \eqref{eq:xidist}, we know that 

\begin{equation*}
    \thermal{\xi^2}_s = \frac{1}{\kappa(v_B-c_B) + 1}\left( 1 + \frac{\kappa(v_B-c_B)}{\kappa(v_B-c_B) + 1} (\theta_\star + z -  \sqrt{\kappa c_B} z_{BB_\star})^2 \right).
\end{equation*}
Putting things together, we obtain 
\begin{equation}\label{eq:mom2_s}
    \begin{split}
        \thermal{\theta^2}_s & = \frac{\kappa(v_B-c_B)}{\kappa(v_B-c_B) + 1}\left( 1 + \frac{\kappa(v_B-c_B)}{\kappa(v_B-c_B) + 1} (\theta_\star + z -  \sqrt{\kappa c_B} z_{BB_\star})^2 \right) + \kappa c_B z^2_{BB_\star} \\
        & \hspace{2cm} + \frac{2\kappa(v_B-c_B)}{\kappa(v_B-c_B) + 1} (\theta_\star + z -  \sqrt{\kappa c_B} z_{BB_\star})  \sqrt{\kappa c_B} z_{BB_\star}.
    \end{split}
\end{equation}
%\textcolor{red}{$\langle\theta\rangle_s$ second term sign change needs to be arried through if I am right.}
By the independence of $\xi_{B_{\star}}$, $z$, and $z_{BB_\star}$ and using definitions \eqref{eq:thetasearlier}, \eqref{eq:scalarscores}, and \eqref{eq:FPE_all2}, we obtain that 
\begin{equation}\label{eq:matchingr1to3}
    r_1 = 1 + r_2 = 1 + \E_{G \otimes e}\thermal{\theta(\xi)}^2_s - \E_{G \otimes e}\thermal{\theta^2(\xi)}_s =  \frac{1}{\kappa(v_B-c_B)+1}
\end{equation}
\begin{equation*}
    \mbox{and } \,\, r_3 = \E_{G \otimes e}\thermal{\theta(\xi) - \theta_\star - z}^2_s = \frac{\gamma^2+c_B -2 c_{BB_\star}+\kappa^{-1}}{\kappa(v_B-c_B+\kappa^{-1})^2}.
\end{equation*}
We will now prove the second part of the proposition. To this end, we assume that the prior $\mu(\cdot)$ is standard Gaussian. 
and that $\pi(\cdot)$ is centered. In this case, $\langle f(\beta) \rangle_h$ is equivalent to $\mathbb{E}[(f(\beta)]$ where $\beta \sim \mathcal{N}(m/(1+v),v/(1+v))$, with $m,v$ defined as in \eqref{eq:alphasv}. Plugging in these definitions we have,
%Then, by Gaussian integration by parts, we have that
%\textcolor{red}{TILL HERE}
%\begin{equation*}
    %\thermal{\beta}_{h} = - r_1 \thermal{\beta}_h + (r_2+1) B_\star + \sqrt{r_3} z
%\end{equation*}
%from which we get that
\begin{equation}\label{eq:mean_h2}
    \thermal{\beta}_{h} = \frac{(r_2+1) B_\star + \sqrt{r_3} z}{r_1 + 1},
\end{equation}
%Similarly, we get
%\begin{equation*}%
    %\thermal{\beta^2}_{h} = 1 - r_1 \thermal{\beta^2}_h + \thermal{\beta}_h ((r_2+1) B_\star + \sqrt{r_3} Z)
%\end{equation*}
%from which we get
\begin{equation}\label{eq:mom2_h}
    \thermal{\beta^2}_{h} = \frac{1}{r_1 + 1} \left( 1 + \frac{((r_2+1) B_\star + \sqrt{r_3} z)^2}{r_1 + 1}\right).
\end{equation}
From this, and using the fact that $B_\star$ and $z$ are independent and centered we get that
\begin{equation}\label{eq:FPE1_lin}
    \begin{cases}
        v_B = \frac{1}{r_1 + 1} + c_B \\
        c_B = \frac{(r_2+1)^2 \gamma^2 + r_3}{(r_1 + 1)^2} \\
        c_{BB_\star} = \frac{(r_2+1) \gamma^2}{r_1 + 1}
    \end{cases}
\end{equation}
%In the particular case in which prior and signal are standard Gaussians, by the above we have that \textcolor{red}{Dont know why the below is needed, can directly skip to the $r_1$ equation}\ms{I agree that $C$ is not a per se meaningful quantity. This equation below can be removed.}
%\begin{equation*}
%    C = \frac{\kappa}{r_1 + \kappa + 1}.
%\end{equation*}
Plugging things in \eqref{eq:matchingr1to3}, we obtain that 
\begin{equation*}
    r^2_1 + \kappa r_1 - 1 = 0
\end{equation*}
from which we see that
\begin{equation}\label{eq:r1thm}
    r_1 = r_2+1 = \sqrt{\left(\frac{\kappa}{2}\right)^2 + 1} - \frac{\kappa}{2}.
\end{equation}

%\textcolor{red}{TILL HERE ON FINAL PASS.}
% ------------------------------------------------------------------------------------------------------------------------------------

\subsection{Derivation of $r_1$ and $r_2$ by direct arguments}\label{app:direct_integration_op}
In the setting of linear regression with Gaussian prior, that is, where $\mu(\cdot)$ is Gaussian (for simplicity we assume standard Gaussian), we can validate by direct arguments that Proposition \ref{prop:lin_reg_eqs} provides accurate expressions. While this could be formally proved for all the constants, here we show the computations for $r_1$ and $r_2$. We choose this example  since closed-form expressions exist in this case, allowing us to validate the accuracy of our results. By calculations similar to prior section, it can be shown that our main theorem says a single posterior marginal behaves as $\beta_j \sim \mathcal{N}(m/(1+v),v/(1+v))$ in this setting. The mean and variance simplify to 
\begin{equation}\label{eq:meanvar}
\frac{m}{1+v} = \frac{(r_2+1)/r_1 \beta_{\star,j}+\sqrt{r_3}/r_1Z}{1+1/r_1} = (r_2+1)/(r_2+2)\beta_{\star,j} + \sqrt{r_3}/(r_1+1)Z, \quad v/(1+v)=1/(r_1+1),
\end{equation}
where we used the fact that here $r_1 = r_2+1.$ Recall the generative model is 

%Here we will consider linear regression with Gaussian prior. By direct arguments, we will prove that the asymptotic marginals of a fixed coordinate are Gaussian distributions of the same variance predicted by Proposition \ref{prop:lin_reg_eqs}. The observations, in this setting, are given by
\begin{equation*}
    \boldsymbol{y} = \boldsymbol{X}\boldsymbol{\beta_\star} + \boldsymbol{z}
\end{equation*}
and the posterior is
\begin{equation*}
    \proba(\boldsymbol{\beta}|\boldsymbol{X},\boldsymbol{y}) = \frac{1}{Z_n} \exp\left\{ -\frac{1}{2}||\bX\bbeta - \boldsymbol{y}||^2 - \frac{1}{2}||\bbeta||^2 \right\}.
\end{equation*}
In this case, we can then rewrite the posterior according to
\begin{equation*}
    \proba(\boldsymbol{\beta}|\boldsymbol{X},\boldsymbol{y}) = \frac{1}{Z_n} \exp\left\{ -\frac{1}{2}\bbeta^\top(\bX^\top\bX + \iden)\bbeta + \boldsymbol{y}^\top \bX \bbeta \right\}.
\end{equation*}
This means that the posterior is exactly a Multivariate Normal with mean
\begin{equation*}
    \boldsymbol{\mu} = (\bX^\top\bX + \iden)^{-1} \bX^\top \boldsymbol{y}
\end{equation*}
and covariance matrix
\begin{equation*}
    \boldsymbol{\Sigma} = (\bX^\top\bX + \iden)^{-1}.
\end{equation*}
Notice that $\boldsymbol{\mu}$ is the ridge solution to the problem with tuning parameter equals $1$ and $\boldsymbol{\Sigma}$ is the resolvent of the Wishart matrix ($\bX^\top\bX$) evaluated at $z=-1$.

Suppose we wanted to compute the marginal associated with $\beta_1$. By the fact this is a a Multivariate distribution, the marginal will be a Normal random variable with mean
\begin{equation*}
    \mu_1 = [(\bX^\top\bX + \iden)^{-1} \bX^\top \boldsymbol{y}]_1
\end{equation*}
and variance
\begin{equation*}
    \Sigma_{11} = [(\bX^\top\bX + \iden)^{-1}]_{11}.
\end{equation*}

The easiest of these constants to analyze in the limit of $n$ going to infinity is the variance. For this, we have that
\begin{equation*}
    [(\bX^\top\bX + \iden)^{-1}]_{11} = M_n(-1) \xrightarrow{n\to\infty} a = m(-1);
\end{equation*}
where $M_n(z)$ is the Stieltjes transform of a Wishart matrix of aspect ratio $\kappa$ and, by \cite{marchenko1967distribution}, $m(z)$ is the solution to the quadratic equation
\begin{equation*}
    z\kappa m^2(z) - (1 - \kappa - z) m(z) + 1 = 0.
\end{equation*}
Plugging in $m(z)= a$ and $z=-1$, we obtain that
\begin{equation*}
    a^2 + \left(\frac{2}{\kappa}-1\right) a - \frac{1}{\kappa} = 0. 
\end{equation*}
The only positive solution to this last equation is
\begin{equation*}
    a = \frac{1}{2}\left(\sqrt{\left(\frac{2}{\kappa}-1\right)^2 + \frac{4}{\kappa}} + 1 - \frac{2}{\kappa}\right) = \frac{1}{2}\left(\sqrt{\left(\frac{2}{\kappa}\right)^2 + 1 } + 1 - \frac{2}{\kappa}\right)
\end{equation*}
Which implies that
\begin{equation*}
    \begin{split}
        a^{-1} & = \frac{2}{\sqrt{\left(\frac{2}{\kappa}\right)^2 + 1 } + 1 - \frac{2}{\kappa}} = \frac{2\left(\sqrt{\left(\frac{2}{\kappa}\right)^2 + 1 } + \frac{2}{\kappa} - 1 \right)}{\left(\frac{2}{\kappa}\right)^2 + 1 -  \left(1 - \frac{2}{\kappa}\right)^2} \\
        & = \frac{\kappa}{2} \left(\sqrt{\left(\frac{2}{\kappa}\right)^2 + 1 } + \frac{2}{\kappa} - 1 \right) = \sqrt{\left(\frac{\kappa}{2}\right)^2 + 1 } + 1 - \frac{\kappa}{2} 
    \end{split}
\end{equation*}
 From \eqref{eq:r1thm}, we see that this equals $r_1+1$ as well as $r_2+2$, thus indeed the variance of the marginal is $a=1/(r_1+1)$ as suggested by our main theorem (recall \eqref{eq:meanvar}). 
 
% we have that this must equal $r_1+1$, and solving for $r_1$, we obtain precisely \eqref{eq:r1thm}.
 %This also 
%validates our description of $r_1$. 

%This proves that the variance of the marginals computed from direct integration and random matrix results coincide with the prediction of Proposition \ref{prop:lin_reg_eqs}.

We next show that  part of the mean from direct computations and our main results match
%part of the mean from direct computations and our main result match, and in the process validate our description for $r_2$. 
%We will now focus on proving that the prediction for $r_2$ is also correct. For this notice that
\begin{equation*}
    \boldsymbol{\mu} = (\bX^\top\bX + \iden)^{-1} \bX^\top \boldsymbol{y} = (\bX^\top\bX + \iden)^{-1} \bX^\top \bX \bbetas + (\bX^\top\bX + \iden)^{-1} \bX^\top\bz.
\end{equation*}
Furthermore,
\begin{equation*}
    \begin{split}
        (\bX^\top\bX + \iden)^{-1} \bX^\top \bX \bbetas & = (\bX^\top\bX + \iden)^{-1} (\bX^\top \bX + \iden) \bbetas - (\bX^\top\bX + \iden)^{-1} \bbetas \\
        & = \bbetas - (\bX^\top\bX + \iden)^{-1} \bbetas.
    \end{split}
\end{equation*}
From this we see that
\begin{equation*}
    [(\bX^\top\bX + \iden)^{-1} \bX^\top \bX \bbetas]_1 = (1-\Sigma_{11}) \beta_{\star,1} - \sum_{j=2}^p [(\bX^\top\bX + \iden)^{-1}]_{1j} \beta_{\star,j}.
\end{equation*}
Notice, by the preceding argument, the first term converges to $(1-a) \beta_{\star,1}$. But recall that $a^{-1}=r_2+2$ from previous calculations. Thus the coefficient of the marginal is $1-a= \frac{r_2+1}{r_2+2}$ which matches the coefficient as predicted by our main theorem results for linear regression from \eqref{eq:meanvar}.

\section{Computation details}\label{app:computation}
To validate the theory of Theorems \ref{thm:contraction} and \ref{thm:marginals}, we ran numerical experiments as demonstrated by the figures in our main manuscript. They required numerically computing the solution to the fixed point equations \eqref{eq:FPE_all2}–\eqref{eq:FPE_all1}. Here we provide details regarding how we proceeded in this regard.

Our approach consisted in treating the right-hand side of \eqref{eq:FPE_all2}–\eqref{eq:FPE_all1} as a map, which was then applied iteratively to the values obtained at the previous step until the constants $(v,q,m,r_1,r_2,r_3)$ stabilized. To this end, the expectations appearing on the right-hand side of equations \eqref{eq:FPE_all2}–\eqref{eq:FPE_all1} were approximated by empirical averages computed from samples of the corresponding measures: $p_s(\cdot)$ for equation \eqref{eq:FPE_all2} and $p_h(\cdot)$ for equation \eqref{eq:FPE_all1}. In each case, $1000$ samples were drawn using the No-U-Turn Sampler (NUTS), employing the default PyMC configuration with $4$ chains and a tuning phase of $2000$ iterations. The smoothing parameter $\delta$ was fixed throughout at the value $1/1000$. The code used for all simulations is available in the repository \cite{github_manu}.

\end{document}